\documentclass[leqno,onefignum,onetabnum,eqref]{siamltex}
\usepackage{epsfig,amsmath,amssymb,subfigure,version,graphicx,fancybox,mathrsfs,amscd}
\usepackage{cases,bm,multirow}
\usepackage{algorithm,algorithmic}
\usepackage{url,xcolor}

\newtheorem{remark}{Remark}[section]
\graphicspath{{figures/}}

\graphicspath{{figures/}}

\title{Image Reconstruction via Discrete Curvatures
}

\author{Qiuxiang Zhong\thanks{Center for Applied Mathematics, Tianjin University, Tianjin, 300072, China, {\tt E-mail: zhongqiuxiang@tju.edu.cn}}
  \and {Ke Yin}\thanks{{Center for Mathematical Sciences, Huazhong Univeristy of Science and Technology, Wuhan, Hubei, 430070, China, {\tt E-mail: kyin@hust.edu.cn}}}
  \and {Yuping Duan}\thanks{Corresponding author. Center for Applied Mathematics, Tianjin University, Tianjin, 300072, China, {\tt E-mail: yuping.duan@tju.edu.cn}}}

\begin{document}
\maketitle
\slugger{sisc}{xxxx}{xx}{x}{x--x}

\begin{abstract}
The curvature regularities are well-known for providing strong priors in the continuity of edges, which have been applied to a wide range of applications in image processing and computer vision. However, these models are usually non-convex, non-smooth and highly non-linear, the first-order optimal condition of which are high-order partial differential equations. Thus, the numerical computation are extremely challenging. In this paper, we propose to estimate the discrete curvatures, i.e., mean curvature and Gaussian curvature, in the local neighborhood according to differential geometry theory. By minimizing certain functions of curvatures on all level curves of an image, it yields a kind of weighted total variation minimization problem, which can be efficiently solved by the proximal alternating direction method of multipliers (ADMM). Numerical experiments are implemented to demonstrate the effectiveness and superiority of our proposed variational models for different image reconstruction tasks.
\end{abstract}



\pagestyle{myheadings}
\thispagestyle{plain}
\markboth{ Image Reconstruction via Discrete Curvatures}{Image Reconstruction via Discrete Curvatures}

\section{Introduction}

Curvatures are important geometric concepts, which depict the amount of a curve from being straight as in the case of a line or a surface deviating from being a flat plane. In the literature, the curvature-based regularities have achieved great success for image processing tasks. Compared to the well-known total variation (TV) regularization, the curvature models can not only remove the staircase effect, but also keep the edges and corners of objects.

Nitzberg, Mumford and Shiota \cite{nitzberg1993filtering} observed that the line energies such as Euler's elastica can be used as regularization for the completion of missing contours in images. Masnou and Morel \cite{masnou1998level} used the level lines structure to minimize the Euler's elastica energy subject to certain boundary conditions. The Masnou-Moral Euler's elastica model for image denoising can be written as follows
\begin{align}\label{Euler's-elastica-denoising}
\min_{u}~\int_{\Omega}\bigg[1+\alpha \Big(\nabla\cdot\frac{\nabla u}{|\nabla u|}\Big)^2\bigg]|\nabla u|dx + \frac{\lambda}{2}\int_\Omega(u-f)^2dx,
\end{align}
where ${\Omega}$ is a bounded domain of $\mathbb{R}^n$ (a rectangle, typically), $f:\Omega\rightarrow \mathbb{R}$ is a given image defined on $\Omega$, $u:\Omega\rightarrow \mathbb{R}$ is the latent clean image, and $\lambda$, $\alpha$ are two positive parameters. According to the Euler's elastica energy, denoised images have smooth connections in the level curves of images. Due to the non-smoothness, nonlinearity and nonconvexity, the numerical minimization of Euler's elastica is highly challenging. The gradient flow was used to solve a set of coupled second order partial differential equations in \cite{ballester2001filling, shen2003euler} for minimizing the Euler's elastica energy, which usually takes high computational cost in imaging applications. Schoenemann, Kahl and Cremers \cite{schoenemann2009curvature} solved the associated linear programming relaxation and thresholded the solution to approximate the original integer linear program regarding to curvature regularity.
Discrete algorithms based on graph cuts methods were studied in \cite{el2010fast,bae2011graph}. Thanks to the development of operator splitting technique and augmented Lagrangian algorithm, fast solvers for Euler's elastica models have been presented in \cite{tai2011fast,duan2013fast,yashtini2016fast,bae2017augmented}. Recently, Deng, Glowinski and Tai \cite{deng2019new} proposed a Lie operator-splitting based time discretization scheme, which is applied to the initial value problem associated with the optimality system. A convex, lower semi-continuous, coercive approximation of Euler's elastica energy via functional lifting  was studied in \cite{bredies2015convex}. Later, Chambolle and Pock \cite{chambolle2019total} used a lifted convex representation of curvature depending variational energies in the roto-translational space, which yields a natural generalization of the total variation to the roto-translational space.

By considering the image surface or graph in $\mathbb R^3$ characterized by $z = u(x,y)$, $(x,y)\in\Omega$, the image minimization problems are then transferred to the corresponding surface minimization problems.
Both mean curvature (MC) and Gaussian curvature (GC) have been used as the regularization to preserve geometric features of the image surface for different image processing tasks. The mean curvature  was first introduced for noise removal as mean curvature driven diffusion algorithms \cite{el1997mean,yezzi1998modified}, which evolved the image surface at a speed proportional to its mean curvature. Zhu and Chan \cite{zhu2012image} proposed to employ the $L^1$-norm of mean curvature of the image surface for image denoising, i.e.,
\begin{equation}\label{mean-curvature-l1}
\min_{u}~\int_{\Omega}\bigg|\nabla\cdot\bigg(\frac{\nabla u}{\sqrt{1+|\nabla u|^2}}\bigg)\bigg|dx + \frac{\lambda}{2}\int_\Omega(u-f)^2dx,
\end{equation}
which has been proven can keep corners of objects and greyscale intensity contrasts of images and also remove the staircase effect. Originally, the smoothed MC model was numerically solved by the gradient decent method, which involves high order derivatives and converges slowly in practice. To deal with this difficulty, some effective and efficient numerical algorithms for MC model \eqref{mean-curvature-l1} were proposed based on augmented Lagrangian method \cite{zhu2013augmented,myllykoski2015new}. However, there always exists some inevitable problems in this kind of methods, such as the choices of the algorithm parameters and the slow convergence rate.

Gaussian curvature-driven diffusion was first studied in \cite{lee2005noise} for noise removal, which is shown to be superior in preserving image structures and details. Lu, Wang and Lin \cite{lu2011high} proposed a energy functional based on Gaussian curvature for image smoothing, which is solved by a diffusion process. Gong and Sbalzarini \cite{gong2013local} presented a variational model with local weighted Gaussian curvature as regularizer, which can be solved by the splitting techniques. In \cite{brito2016image}, the authors minimized the $L^1$-norm of gaussian curvature for image denoising, i.e.,
\begin{equation}\label{guass-curvature-l1}
\min_{u} ~\int_\Omega\frac{|\mbox{det } \nabla^2 u|}{(1+|\nabla u|^2)^2}dx + \frac{\lambda}{2}\int_\Omega(u-f)^2dx ,
\end{equation}
where $ \nabla^2 u$ is the Hessian of function $u$ and
\[\mbox{det } \nabla^2 u = \frac{\partial^2u}{\partial x^2}\frac{\partial^2u}{\partial y^2}-\Big|\frac{\partial^2u}{\partial x\partial y}\Big|^2.\]
Although these methods are desirable, MC and GC regularizer are limited by two main issues: Firstly, the algorithms available either converge slowly or contain too many parameters. Secondly,  such regularizers require the image to be at least twice differentiable function (cf. equation \eqref{mean-curvature-l1} and \eqref{guass-curvature-l1}).

Besides, Goldluecke and Cremers \cite{goldluecke2011introducing} used a convex approximation of the $p$-Menger-Melnikov curvature, called the total curvature, which measures theoretic formulation of curvature mathematically related to mean curvature.
Recently, Gong and Sbalzarini \cite{gong2017curvature} presented a filter-based approach to use the pixel-local analytical solutions to approximate the TV, MC and GC by enumerating the constant, linear and developable surfaces in the $3\times3$ pixel neighborhood. Although the curvature filter avoids to solve the high-order partial differential equations associated with the curvature-based variational models, it still has two crucial limitations: (i) There is no rigorous definition and accurate estimation of the curvatures,  which were numerically approximated by certain distances monotone with respect to curvatures. (ii) For specific image processing tasks, such as denoising, registration, etc., it requires to alternatively solve the curvature regularization and data fidelity term using the gradient descent, which is also time consuming.

In this work, we aim to precisely define the discrete curvatures for the points on image surface over a $3\times3$ pixel neighborhood, and consider the following curvature-based regularization for image denoising problem
\begin{equation}\label{curvature_model}
   \min_{u}~ \int_{\rm{\Omega}} g(\kappa)|\nabla u|dx + \frac{\lambda}{2} \int_{\rm{\Omega}}(u-f)^2dx,
\end{equation}
where $g(\kappa)$ denotes a function of curvature.
According to \cite{chambolle2019total}, the following three typical energies are adopted in this work, where $\alpha$ is a positive parameter to balance the curvature and arclength.
\begin{itemize}
\item[1)]Total absolute curvature (TAC): measures the sum of length and absolute curvature
\begin{align}
g_1(\kappa) & = 1+\alpha|\kappa|,
\label{kappa1}
\end{align}
which allows for sharp corners in the level sets of the images and has been studied in
\cite{nitzberg1993filtering,bredies2013convex,he2019segmentation}.
\item[2)] Total square curvature (TSC): penalizes the length and the squared curvature
\begin{align}
g_2(\kappa) & = 1+\alpha|\kappa^2|,
\label{kappa2}
\end{align}
which is equivalent to the Euler's elastica energy being discussed in our introduction. It is well-known the Euler's elastica energy favors long connectivity and smooth shapes in the images.
\item[3)] Total roto-translational variation (TRV): measures the length and curvature through an Euclidean metric
\begin{align}
g_3(\kappa) & = \sqrt{1+\alpha|\kappa^2|},
\label{kappa3}
\end{align}
which corresponds to the total variation of the lifted curve in the roto-translational space and has been explored in \cite{sarti2001subjective,chambolle2019total}. It prefers smooth shapes, but allows sharp discontinuities.
\end{itemize}

Because the curvature can be computed explicitly, we regard the minimization problem \eqref{curvature_model} as a re-weighted TV model, and use the ADMM to efficiently solve it.
We prove the existence of a solution and discuss the convergence of the ADMM algorithm under certain assumption. Numerous applications to image denoising and inpainting show the efficiency of the proposed method.
Compared to the-state-of-the-art variational curvature models, our method has the following advantages:
\begin{itemize}
\item[1)] By computing the normal curvatures in the local pixel neighborhood, we can estimate both MC and GC in terms of principal curvatures without requiring the image to be twice differentiable. Thus, by taking either MC or GC into \eqref{curvature_model}, our model can not only achieve good image restoration results but also preserve the geometric properties, such as edges, corners etc., very well.
\item[2)] Because we only introduce one artificial variable, our ADMM has less parameters than other curvature-based models. More specifically, our algorithm has only one parameter of the penalty term to be selected while the ALM for Euler's elastica model in \cite{tai2011fast} has three such kind of parameters.
\item[3)] Our model is more flexible to adapt with the different combinations of the function-type and curvature-type without affecting the way of the operator-splitting and the associated ADMM-based algorithm.
\item[4)] By evaluating our model with different functions of MC and GC, we conclude that the best choice for natural images denoising is the absolute GC regularity, while the absolute MC regularity usually achieves better restoration results on smooth images.
\item[5)] For the same stopping criterion, our model has lower computational cost per iteration and  requires less iterations than ALM in \cite{tai2011fast}. The advantages is significantly shown by the experiments on color image denoising such that our method only need around 1/2 of the CPU time used by ALM in \cite{tai2011fast}.
\end{itemize}

This paper is organized as follows. We introduce some neccessary definitions and notations of parametric curves and surface in differential geometry theory in section \ref{sect2}. The discrete curvatures, the curvature regularized model and ADMM-based algorithm are  discussed in section \ref{sect3}. Section \ref{sect4} is dedicated to numerical experiments on image reconstruction problems to demonstrate the efficiency and superiority of the proposed approach. Finally, we draw some conclusions in section \ref{sect5}.

To summarize this section, we would like to mention that the aforementioned curvature-based variational problems, i.e., \eqref{Euler's-elastica-denoising}-\eqref{guass-curvature-l1}, are largely mathematically formals. To the best of our knowledge, the proper functional framework to formulate these problems has not been identified yet. Similarly, we do not know much about the function space of our model \eqref{curvature_model}, which has to be a subspace of $L^2(\Omega)$. Obviously, the discrete problems largely ignore these functional analysis considerations. Thus, we discuss our model under the discrete setting in the followings.

\section{Parametric curves, surface and curvatures}
\label{sect2}
Since we are going to estimate discrete curvatures using the differential geometry theory, we first give a brief introduction of curve and surface to make the paper reasonably self-contained.

Let ${\bm r} = {\bm r}(x,y): \Omega\subset\mathbb R^2\rightarrow {\mathbb R}^3$ be a regular parametric surface $S$ and $(x,y)$ be the coordinates on surface $\Omega$. Therefore, an arbitrary continuous differentiable curve $C$ lying on $S$ can be denoted by parametric function ${{\bm c}(t)} = {\bm r}(x(t),y(t))~(a \leq t \leq b)$, the derivative of which is given as
\begin{equation*}
  {\bm c}'(t) = x'(t){\bm r}_x(x(t),y(t))+y'(t){\bm r}_y(x(t),y(t))
\end{equation*}
associated to the tangent vector of arbitrary point on the curve. All tangent vectors of a point $p$ on surface $S$ constitute the tangent space $T_pS$ with $\{{\bm r}_x,{\bm r}_y\}$ being its basis.

\begin{definition}\label{tangentplane}
In ${\mathbb R}^3$, the two-dimensional plane expanded in the basis $\{{\bm r}_x,{\bm r}_y\}$ is called the tangent plane of point $P$ on surface $S$, whose parametric function is
\begin{equation*}
   {\bm X}(\lambda,\mu) = {\bm r}(x,y)+\lambda{\bm r}_x(x,y)+\mu{\bm r}_y(x,y),
\end{equation*}
where $\lambda,\mu$ are the parameters of the moving point on the tangent plane.
\end{definition}

The length of parametric curve $C$ can be measured as
\[l(C) =\int_{C\in S}ds = \int_C\Big|\frac{d \bm r}{dt}\Big|dt = \int_C\Big|\bm r_x\frac{dx}{dt} + \bm r_y \frac{dy}{dt}\Big|dt = \int_C\sqrt{Edx^2 + 2F dxdy + G dy^2},\]
where $E = \bm r_x\cdot\bm r_x, ~ F= \bm r_x\cdot \bm r_y, ~ G =  \bm r_y\cdot\bm r_y$. The \emph{first fundamental form} is defined as
\begin{equation}\label{firstform}
{\rm{\uppercase\expandafter{\romannumeral1}}} = ds^2 =d\bm r\cdot d\bm r = Edx^2 + 2F dxdy + G dy^2,
\end{equation}
and the $E$, $F$, $G$ are called the first fundamental form coefficients, which plays important roles in many intrinsic properties of a surface.

In order to quantify the curvature of a surface $S$, we consider a curve $C$ on $S$ passing through point $O$ shown in FIG. \ref{Curvaturevector}. The \emph{curvature vector} is used to measure the rate of change of the tangent along the curve, which can be defined using the \emph{unit tangent} vector $\bm t$ and the \emph{unit normal} vector $\bm n$ of the curve $C$ at point $O$ as
\[\bm k = \frac{d\bm t}{ds} = \bm k_n + \bm k_g,\]
with $\bm k_n$ being the \emph{normal curvature} vector and $\bm k_g$ being the \emph{geodesic curvature} vector. Let $\bm N$ be the \emph{surface unit normal vector}, which is defined as
\[\bm N = \frac{{\bm r}_x \times {\bm r}_y}{|{\bm r}_x \times {\bm r}_y|}.\]
By differentiating $\bm N\cdot \bm t =0$ along the curve with respect to $s$, we obtain
\[\frac{d\bm t}{ds}\cdot \bm N + \bm t\cdot\frac{d\bm N}{ds} =0.\]
Thus, the normal curvature of the surface at $O$ in the direction $\bm t$ can be expressed as
\begin{equation}\label{normal_curvature}
\kappa_n = \frac{d\bm t}{ds}\cdot \bm N=-\bm t \cdot\frac{d\bm N}{ds} = -\frac{d\bm r}{ds}\cdot\frac{d\bm N}{ds} =\frac{Ldx^2+2Mdxdy+Ndy^2}{Edx^2+2Fdxdy +Gdy^2},
\end{equation}
where $L = \bm r_{xx}\cdot\bm N$, $M = \bm r_{xy}\cdot\bm N$, $N = \bm r_{yy}\cdot\bm N$. We call the numerator of \eqref{normal_curvature} the \emph{second fundamental form} such that
\begin{equation}\label{secondform}
{\rm{\uppercase\expandafter{\romannumeral2}}} = Ldx^2 + 2Mdxdy + Ndy^2,
\end{equation}
and $L$, $M$, $N$ are called second fundamental form coefficients.

\begin{figure}
\begin{minipage}[t]{0.5\textwidth}
\centering
\includegraphics[width=2.0in]{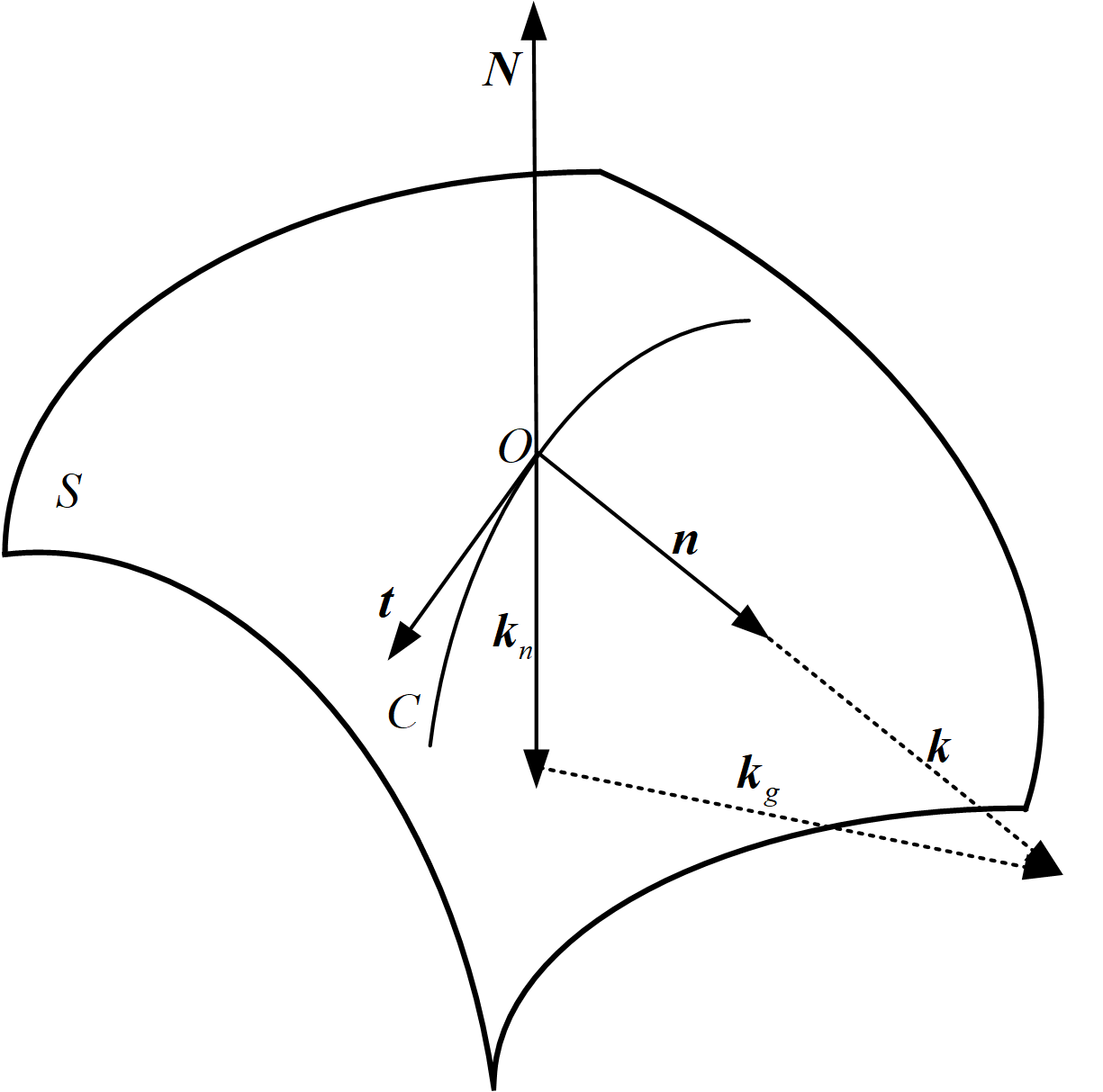}
\caption{Curvature vector}
\label{Curvaturevector}
\end{minipage}%
\begin{minipage}[t]{0.5\textwidth}
\centering
\includegraphics[width=2.3in]{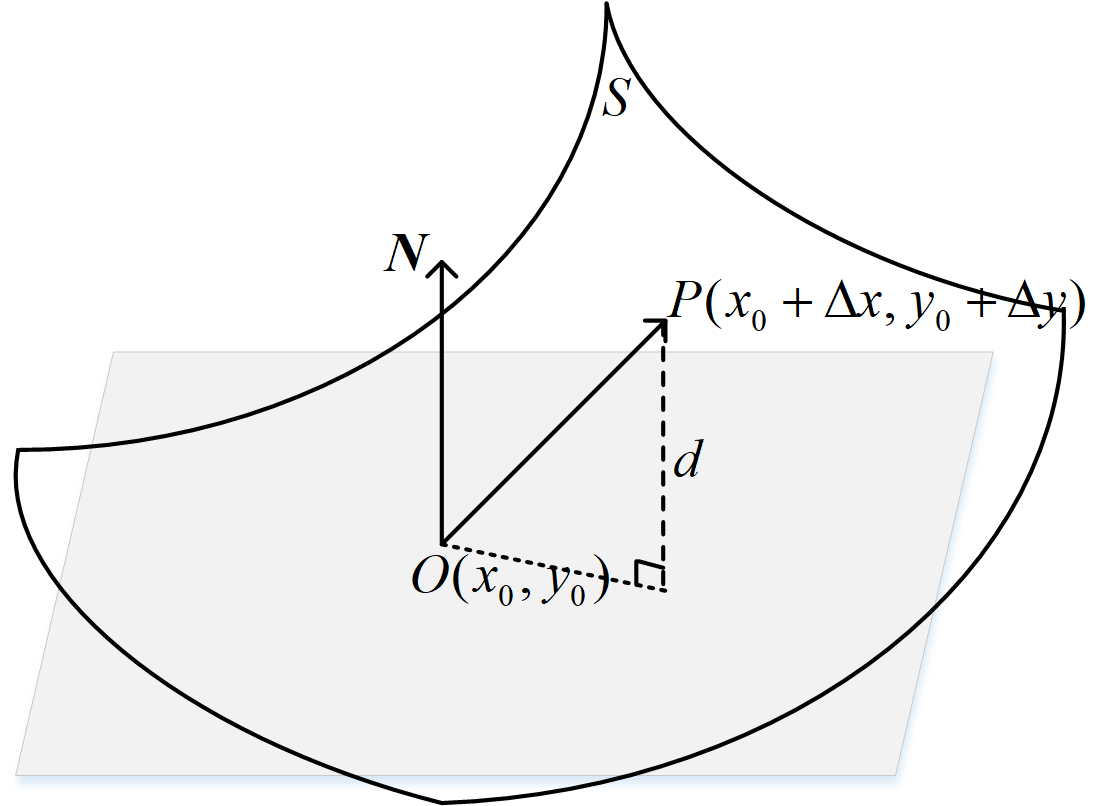}
\caption{The distance of a proximal point to its tangent plane}
\label{distance}
\end{minipage}
\end{figure}

\begin{proposition}\label{distancetheorem}
Suppose $S$: ${\bm r} = {\bm r}(x,y)$ is a regular parametric surface, $O(x_0,y_0)$ is a arbitrary point on $S$, then the distance of the proximal point $P(x_0+\Delta x,y_0+\Delta y)$ to its tangent plane can be estimated as follow
\begin{equation}\label{dist}
  d \approx \frac{1}{2} \rm{\uppercase\expandafter{\romannumeral2}},
\end{equation}
\vspace{.1in}
where $\rm{\uppercase\expandafter{\romannumeral2}}$ denotes the second fundamental form.
\end{proposition}

\begin{proof}\label{distanceproof}
As shown in FIG. \ref{distance}, the distance of the proximal point $P(x_0+\Delta x,y_0+\Delta y)$ to its tangent plane is obtained as follows
\[d(\Delta x,\Delta y) = ({\bm r}(x_0+\Delta x,y_0+\Delta y)-{\bm r}(x_0,y_0)) \cdot {\bm N}.\]
By Taylor's formula we have
\begin{align*}
   & {\bm r}(x_0+\Delta x,y_0+\Delta y)-{\bm r}(x_0,y_0) \\
   & =({\bm r}_x \Delta x + {\bm r}_y \Delta y)+\frac{1}{2}({\bm r}_{xx} (\Delta x)^2+2{\bm r}_{xy}\Delta x\Delta y+{\bm r}_{yy} (\Delta y)^2) + {\bm o}((\Delta x)^2+(\Delta y)^2),
\end{align*}
and
\begin{equation*}
  \lim_{(\Delta x)^2+(\Delta y)^2 \rightarrow 0} \frac{{\bm o}((\Delta x)^2+(\Delta y)^2)}{(\Delta x)^2+(\Delta y)^2} = 0.
\end{equation*}
Owing to ${\bm r}_x \cdot {\bm N} = {\bm r}_y \cdot {\bm N} = 0$, it follows that
\begin{equation*}
 d(\Delta x,\Delta y) = \frac{1}{2}[L(\Delta x)^2+2M\Delta x\Delta y+N(\Delta y)^2]+o((\Delta x)^2+(\Delta y)^2),
\end{equation*}
where the formula $L(\Delta x)^2+2M\Delta x\Delta y+N(\Delta y)^2$ is the \emph{second fundamental form}.
Therefore, when $\sqrt{(\Delta x)^2+(\Delta y)^2} \rightarrow 0$, we obtain
\[d(\Delta x,\Delta y) \approx \frac{1}{2}{\rm{\uppercase\expandafter{\romannumeral2}}},\]
which completes the proof.
\end{proof}

The two \emph{principal curvatures} of $S$ at point $P$ measure how the surface bends by different amounts in different directions at that point, which are defined as
\begin{align}\label{principal_curvatures}
\kappa_1 &= \kappa_1(P) = \emph{the maximum normal curvature at P},\\
\kappa_2 &= \kappa_2(P) = \emph{the minimum normal curvature at P}.
\end{align}
With the principal curvatures, we can further define the \emph{Guassian curvature} and \emph{mean curvature} as follows.
\begin{definition}\label{def2.3}
The Gaussian curvature of $S$ at point $P$, $K=K(P)$, and the mean curvature of $S$ at point $P$, $H=H(P)$ are defined as follows,
\begin{equation}
K = \kappa_{1}\kappa_{2},\quad H = \frac12(\kappa_{1}+\kappa_{2}).
\end{equation}
\end{definition}
The Gaussian curvature is also known as the \emph{curvature} of a surface, which is intrinsic measure of the curvature,  depending only on distances that measured on the surface, not on the way it is isometrically embedded in Euclidean space. Although the mean curvature is not intrinsic, a surface with zero mean curvature at all points is called the minimal surface.

\section{The curvature-based variational model and numerical algorithm}
\label{sect3}
Without loss of generality, we represent a gray image as an $m\times m$ matrix and the grid ${\rm{\Omega}}=\{(i,j):1\leq i\leq m,1\leq j\leq m\}$.

\begin{figure}[t]
\centering
\includegraphics[width=0.5\linewidth]{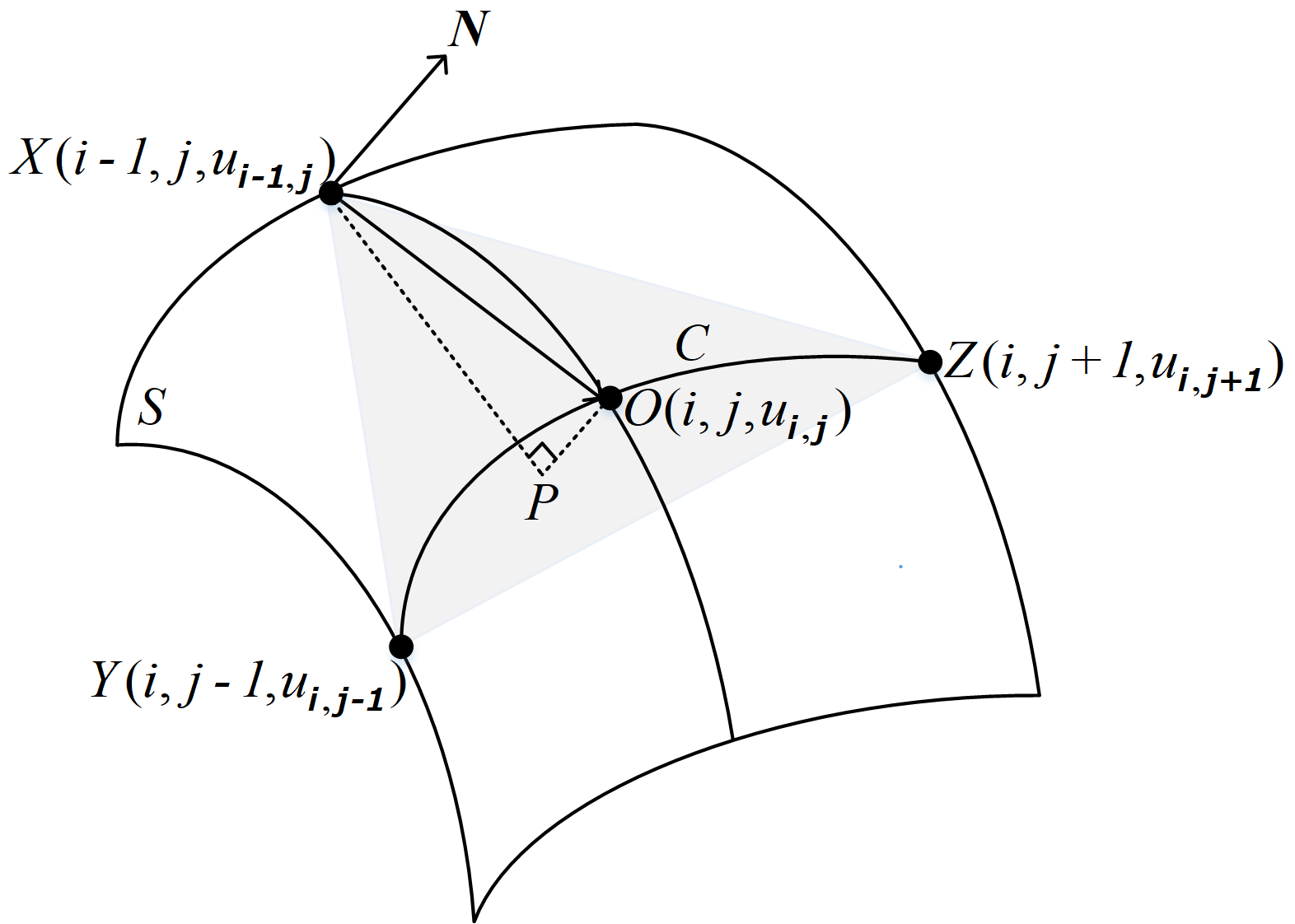}
\caption{Calculate the normal curvature on image surface.}
\label{surface}
\end{figure}

\subsection{Calculation of normal curvatures}
In order to quantify the curvatures of the image surface $S$, we can consider a curve $C$ on $S$ which passes through $O$ in a local window as shown in FIG. \ref{surface}, where $(i,j)$ indicates the coordinates and $u$ denotes the image intensity function. The normal curvature of the curve $C$ at point $O$ in the direction $\overrightarrow{OX}$ can be defined by the quotient of the second and the first fundamental form, i.e.,
\begin{equation}\label{curvature}
 \kappa_n = \frac{\rm{\uppercase\expandafter{\romannumeral2}}}{\rm{\uppercase\expandafter{\romannumeral1}}} = -\frac{d\bm r\cdot d\bm N}{ds^2} \approx \frac{2d}{ds^2} = \frac{2(\overrightarrow{PO} \cdot {\bm N})}{\widehat{OX}^2}.
\end{equation}

Because the normal vector ${\bm N}$ of tangent plane $T_{XYZ}$ can be decided by the cross product of the vector $\overrightarrow{XY}$ and $\overrightarrow{XZ}$, i.e.,
\begin{equation}\label{N}
 {\bm N} = \overrightarrow{XY} \times \overrightarrow{XZ} = (2u_{i-1,j}-u_{i,j-1}-u_{i,j+1},u_{i,j-1}-u_{i,j+1},2),
\end{equation}
we can approximate the projection distance $d$ using the point $O$ by computing its projection to the tangent plane $T_{XYZ}$ 
\begin{equation}\label{di}
 d= \overrightarrow{PO}\cdot {\bm N} =\frac{2u_{i,j}-u_{i,j-1}-u_{i,j+1}}{\sqrt{(2u_{i-1,j}-u_{i,j-1}-u_{i,j+1})^2+(u_{i,j-1}-u_{i,j+1})^2+4}}.
\end{equation}
On the other hand, the arclength $\widehat{OX}$ can be approximated in the following way
\begin{equation}\label{op}
\widehat{OX} \thickapprox \sqrt{(u_{i-1,j}-u_{i,j})^2+h^2},
\end{equation}
where $h$ is space step size along the $x$-axis and the $y$-axis.

Therefore, the normal curvature of point $O$ in direction $X$ can be expressed as follows

\begin{equation}\label{kn}
\kappa_n \thickapprox \frac{2(2u_{i,j}-u_{i,j-1}-u_{i,j+1})}{((u_{i-1,j}-u_{i,j})^2+h^2) \sqrt{(2u_{i-1,j}-u_{i,j-1}-u_{i,j+1})^2+(u_{i,j-1}-u_{i,j+1})^2+4}}.
\end{equation}
\subsection{Mean curvature and Gaussian curvature}
In order to compute the normal curvatures in the $3\times3$ local window, we first define eight triangular tangent planes (i.e., T1-T8) as shown in FIG. \ref{tangentplanes}, which are the physically nearest tangent planes to the center pixel (black one). It is important to calculate the distance of the center pixel to these tangent planes in order to estimate the normal curvatures in the $3\times3$ local window.
\begin{figure*}[t]
\centering
\includegraphics[width=1.0\linewidth]{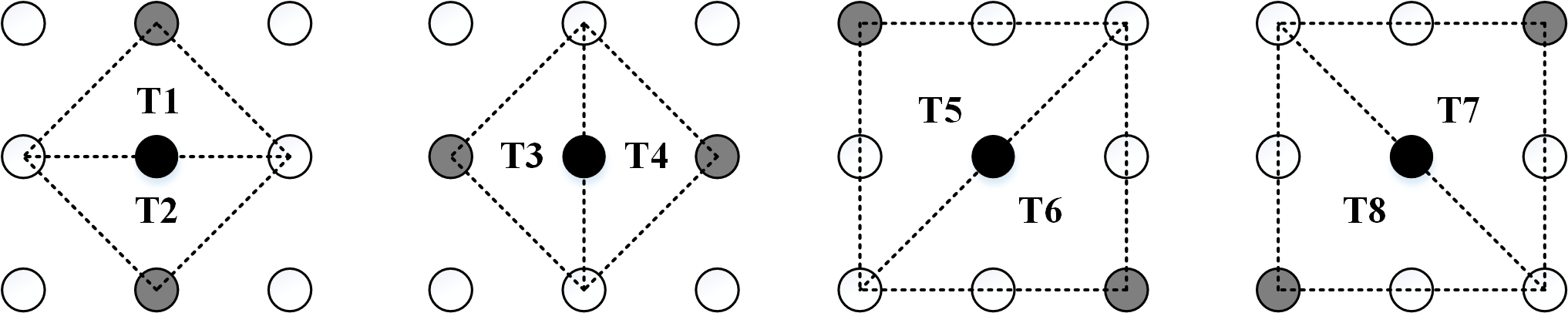}
\caption{The eight tangent planes of the center point in a $3\times3$ local window, where black node denotes the center point and the grey nodes denote the normal direction.}
\label{tangentplanes}
\end{figure*}

Similar to the computation of Euler's elastica energies \cite{shen2003euler,tai2011fast,chambolle2019total}, we use the staggered grid in the $x-y$ plane. Both the staggered grid and the corresponding image surface are shown in FIG. \ref{grid} (a) and (b), where the $\bullet$-nodes denote the original grids, and the $\Box$-nodes and $\vartriangle$-nodes are half grids. The intensity values on $\vartriangle$-nodes are estimated as the mean of its two neighboring $\bullet$-nodes, while on $\Box$-nodes are estimated as the mean of the four surrounding $\bullet$-nodes.

Now, we can calculate the distance $d_\ell$, $\ell=1,\ldots,8$, of $(i,j, u_{i,j})$ to its eight tangent planes according to \eqref{di}, which are given as
\[d_1 = \frac{2u_{i,j}-u_{i,j-1}-u_{i,j+1}}{\sqrt{(2u_{i-1,j}-u_{i,j-1}-u_{i,j+1})^2+(u_{i,j-1}-u_{i,j+1})^2+4}},\]
\[d_2 = \frac{u_{i,j-1}+u_{i,j+1}-2u_{i,j}}{\sqrt{(2u_{i+1,j}-u_{i,j-1}-u_{i,j+1})^2+(u_{i,j+1}-u_{i,j-1})^2+4}},\]
\[d_3 = \frac{u_{i-1,j}+u_{i+1,j}-2u_{i,j}}{\sqrt{(u_{i+1,j}-u_{i-1,j})^2+(u_{i-1,j}+u_{i+1,j}-2u_{i,j-1})^2+4}},\]
\[d_4 = \frac{2u_{i,j}-u_{i-1,j}-u_{i+1,j}}{\sqrt{(u_{i-1,j}-u_{i+1,j})^2+(u_{i-1,j}+u_{i+1,j}-2u_{i,j+1})^2+4}},\]
\[d_5 = \frac{u_{i-1,j+1}+u_{i+1,j-1}-2u_{i,j}}{\sqrt{(u_{i+1,j-1}-u_{i-1,j-1})^2+(u_{i-1,j+1}-u_{i-1,j-1})^2+4}},\]
\[d_6 = \frac{2u_{i,j}-u_{i-1,j+1}-u_{i+1,j-1}}{\sqrt{(u_{i-1,j+1}-u_{i+1,j+1})^2+(u_{i+1,j-1}-u_{i+1,j+1})^2+4}},\]
\[d_7 = \frac{2u_{i,j}-u_{i-1,j-1}-u_{i+1,j+1}}{\sqrt{(u_{i-1,j+1}-u_{i+1,j+1})^2+(u_{i-1,j-1}-u_{i-1,j+1})^2+4}},\]
\[d_8 = \frac{u_{i-1,j-1}+u_{i+1,j+1}-2u_{i,j}}{\sqrt{(u_{i+1,j-1}-u_{i-1,j-1})^2+(u_{i+1,j+1}-u_{i+1,j-1})^2+4}}.\]
Simultaneously, we estimate the arclength of the central point $(i,j)$ to the neighboring points in the $3\times3$ neighborhood, which is defined as the square root of the quadratic sum of two pixel differences and grid distance between two points according to \eqref{op}. As a result, the eight normal curvatures can be calculated using \eqref{kn}, which gives
\begin{align}\label{ki}
\kappa_\ell & \thickapprox
\begin{cases}
\frac{2d_\ell}{(u_\ell-u_{i,j})^2+h^2}, & \ell = 1,~2,~3,~4,\\
\frac{2d_\ell}{(u_\ell-u_{i,j})^2+2h^2},& \ell = 5,~6,~7,~8,
\end{cases}
\end{align}
with $u_\ell$ being the intensity of the grey node on the tangent plane as shown in FIG. \ref{tangentplanes}.

Then, the principal curvature $\kappa_1$ and $\kappa_2$ can be obtained as follows
\begin{equation}\label{kminmax}
 \kappa_{1} = {\rm{max}}\{\kappa_\ell\},~~ \kappa_{2} = {\rm{min}}\{\kappa_\ell\},\quad \mbox{for}~{\ell =1,2,\cdots,8}.
\end{equation}
According to Definition \ref{def2.3}, we can calculate the MC and GC on each point of the image surface using the principal curvatures from
\begin{equation}\label{MCGC}
 H= \frac{\kappa_{1} + \kappa_{2}}{2}~~\mbox{and}\quad K= \kappa_{1}\kappa_{2}.
\end{equation}

\begin{figure}[t]
      \centering
      \subfigure[The staggered grid]{
      \includegraphics[width=0.36\linewidth]{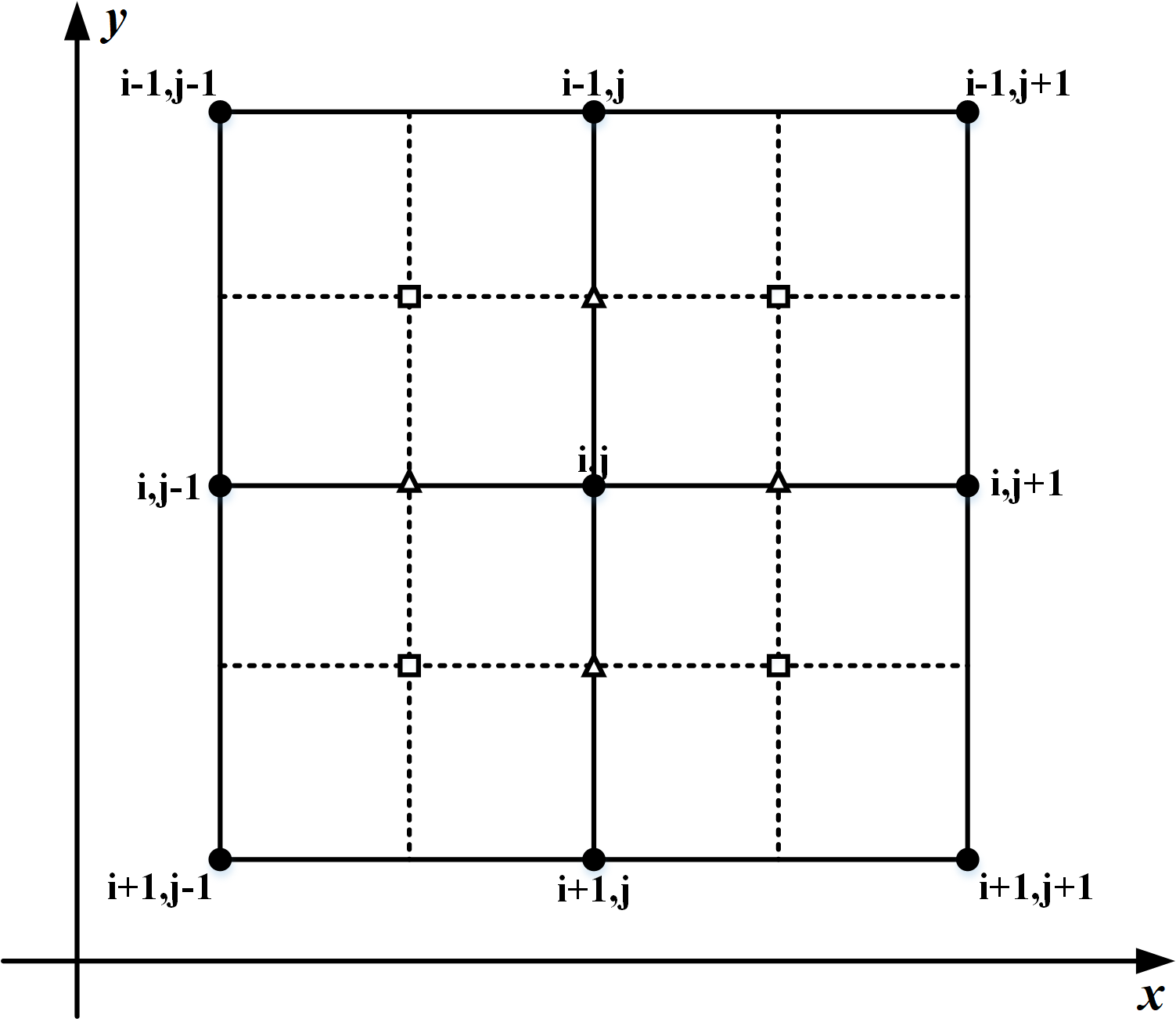}}
      \hspace{.2in}
      \subfigure[The 3-D grid]{
      \includegraphics[width=0.4\linewidth]{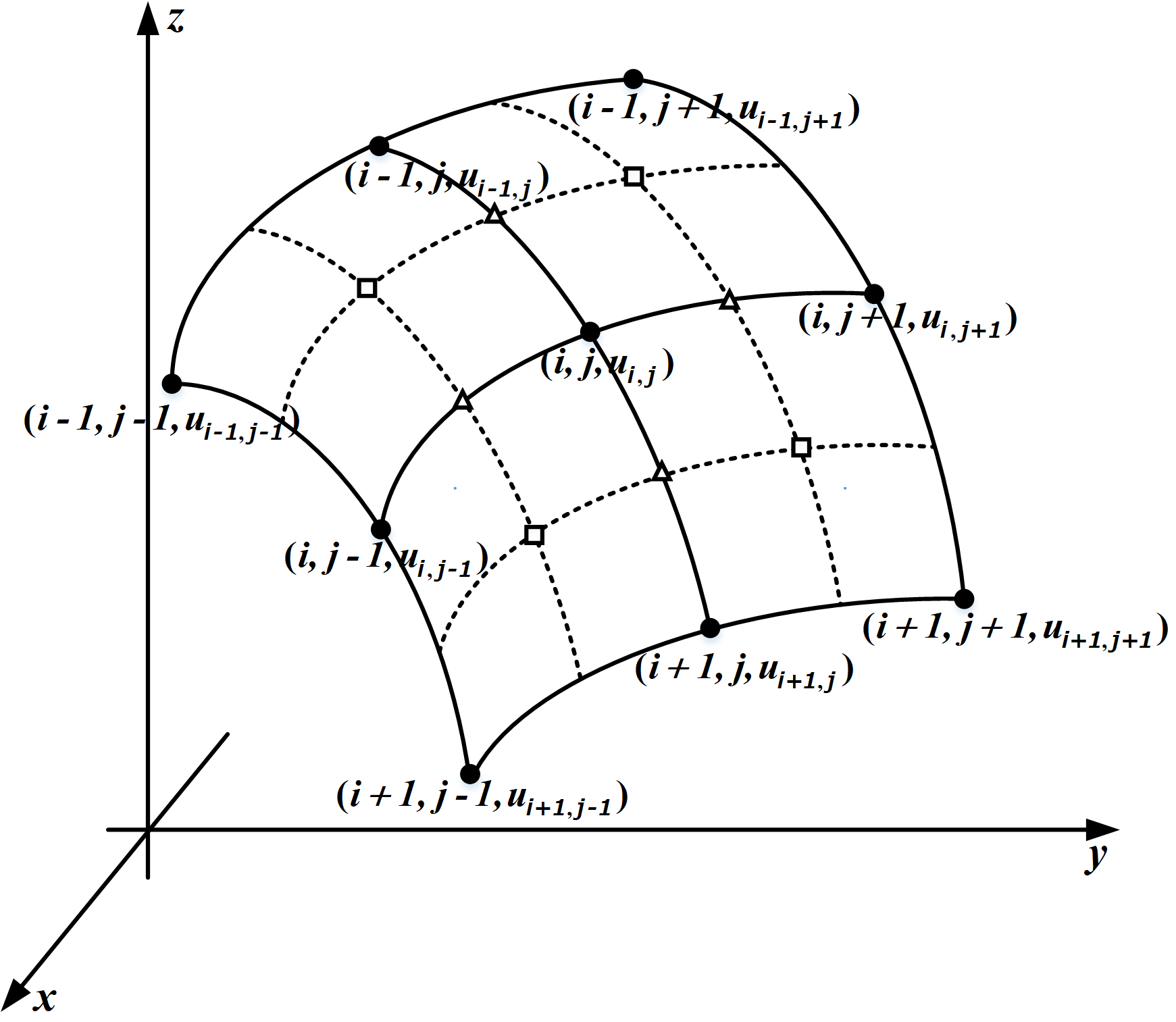}}
    \caption{The discrete staggered grid and 3-D grid.}
	\label{grid}
\end{figure}

\subsection{ADMM-based numerical Algorithm}
With the discrete curvatures, we can rewrite the minimization problem \eqref{curvature_model} into the following discrete form
\begin{equation}\label{dis_model}
\min_{u} \sum_{1\leq i,j\leq m} g(\kappa_{i,j}) |\nabla u_{i,j}| + \frac{\lambda}{2}\|u-f\|^2,
\end{equation}
which $\kappa_{i,j}$ denotes either mean curvature $H$ or Gaussian curvature $K$ in \eqref{MCGC} on point $(i,j)$, $|\cdot|$ is the usual Euclidean norm in $\mathbb R^2$ and $\|\cdot\|$ is the $L^2$ norm. Note that all the matrix multiplication and divisions in this paper are element-wise. The discrete gradient operator $\nabla:\mathbb R^{m^2}\rightarrow \mathbb R^{m^2\times m^2}$ is defined by
\[(\nabla u)_{i,j} = ((\nabla u)^x_{i,j}, (\nabla u)^y_{i,j})\]
with
\begin{align*}
(\nabla u_{i,j})^x =
\begin{cases}
u_{i+1,j}-u_{i,j} ,& \mathrm{if}~{1 \leq i<m},\\
u_{1,j}-u_{i,j}   ,& \mathrm{if}~{i=m},
\end{cases}~
(\nabla u_{i,j})^y =
\begin{cases}
u_{i,j+1}-u_{i,j} ,& \mathrm{if}~{1 \leq j<m},\\
u_{i,1}-u_{i,j}   ,& \mathrm{if}~{j=m},
\end{cases}
\end{align*}
for $i,j=1,\cdots,m$.

As long as the discrete MC and GC can be estimated based on the current value of the image, fast algorithms can be applied to the discrete re-weighted TV model such as split Bregman method \cite{goldstein2009split}, primal-dual splitting method \cite{chambolle2011first} and augmented Lagrangian method \cite{wu2010augmented}. Here, we adopt the proximal ADMM \cite{shefi2014rate,yashtini2016fast}, which can guarantee the convergence in theory.

More specifically, we introduce an auxiliary variable $\bm v$ to rewrite the original unconstrained optimization problem \eqref{curvature_model} into an equivalent discrete constrained minimization as follows
\begin{equation}\label{constrained version}
\begin{split}
& \min_{u,\bm v}~ \sum_{1\leq i,j\leq m}g(\kappa_{i,j})|\bm v_{i,j}| +\frac{\lambda}{2}\|u-f\|^2 \\
& ~\mathrm{s.t.}~~\bm v_{i,j}=\nabla u_{i,j}.
\end{split}
\end{equation}

Given some  $(u^k,\bm v^k)\in \mathbb R^{m^2} \times \mathbb R^{m^2\times m^2}$, the proximal augmented Lagrangian is defined as
\begin{align*}
  \mathcal{L}(u,\bm v; \bm\Lambda) = & \sum_{1\leq i,j\leq m}g(\kappa_{i,j})|\bm v_{i,j}|  + \frac{\lambda}{2}\|u-f\|^2 \\
  & + <\bm\Lambda,\bm v-\nabla u> + \frac{\mu}{2}\|\bm v-\nabla u\|^2 + \frac{\tau}{2}\|u-u^k\|^2+\frac{\sigma}{2}\|\bm v-\bm v^k\|^2,
\end{align*}
where $\bm\Lambda$ represents the Lagrangian multiplier, and $\mu,\tau,\sigma$ are the positive parameters. Then, we iteratively and alternatively solve the $u$- and $\bm v$-subproblem until reaching the terminating condition; see Algorithm 3.1.
\vspace{4mm} \hrule\hrule  \vspace{2mm} {\noindent\bf ADMM-based Algorithm 3.1} \nolinebreak
\vspace{2mm}
\hrule\hrule
\begin{itemize}
\item[1:] \textbf{Input:} Degraded image $f$, model parameter $\lambda$, $\mu, \tau,\sigma$, maximum iteration $T_{max}$, and stopping threshold $\epsilon$.
\item[2:] \textbf{Initialize:} $u^0=f$, ${\bm v}^0=0$, ${\bm\Lambda}^0=0$.
\item[3:] \textbf{while} (not converged and $k\leq T_{max}$) do
\begin{itemize}
\item[(i)] Compute $u^{k+1}$ from:
\begin{equation}
u^{k+1} = \arg\min_u\Big\{\frac{\lambda}{2}\|u-f\|^2+\frac{\mu}{2}\Big\|\nabla u-\bm v-\frac{\bm\Lambda}{\mu}\Big\|^2+ \frac{\tau}{2}\|u-u^k\|^2\Big\};
\label{sub-u}
\end{equation}
\item[(ii)]Compute $H(u^{k+1})$ or $K(u^{k+1})$ according to \eqref{MCGC} using the latest estimation $u^{k+1}$ and take it into $g(\kappa)$;
\item[(iii)] Compute $ \bm v^{k+1}$ from:
\begin{equation}
{\bm v}^{k+1}=\arg\min_{\bm v}\Big\{\sum_{1\leq i,j\leq m}g(\kappa_{i,j})|\bm v_{i,j}|+\frac{\mu}{2}\Big\|\bm v-\nabla u+\frac{\bm\Lambda}{\mu}\Big\|^2+\frac{\sigma}{2}\|\bm v-\bm v^k\|^2\Big\};
\label{sub-v}
\end{equation}
\item[(iv)] Update $ \bm\Lambda^{k+1}$ from:
\begin{equation}\label{multipliers}
 \bm\Lambda^{k+1} = \bm\Lambda^{k}+\mu(\bm v^{k+1}-\nabla u^{k+1});
\end{equation}
\item[(v)] Check convergence condition:
\begin{equation*}
  \frac{\|u^{k+1}-u^{k}\|_1}{\|u^{k}\|_1}\leq \varepsilon.
\end{equation*}
\end{itemize}
\item[4:] \textbf{end while}
\item[5:] \textbf{output:} Restored image.
\end{itemize}
\hrule\hrule  \vspace{4mm}

\subsubsection{The $u$-subproblem}
The first-order optimality condition of \eqref{sub-u} gives a linear equation as follows
\begin{equation*}
\big((\lambda+\tau)\mathcal I-\mu\nabla \cdot \nabla\big) u^{k+1} = \lambda f + \tau u^k-\nabla\cdot(\mu\bm v^k+\bm\Lambda^k)
\end{equation*}
with $\mathcal I$ being the identity matrix.
Under the periodic boundary condition, we can solve the above equation by the fast Fourier Transform (FFT), i.e.,
\begin{equation}\label{sol-u}
u^{k+1}= \mathcal{F}^{-1} \bigg(\frac{\mathcal{F}({\lambda f}-\nabla\cdot(\mu\bm v^k+\bm\Lambda^k)+ \tau u^k)}{(\lambda+\tau)\mathcal I-\mu\mathcal{F}(\Delta)} \bigg),
\end{equation}
where $\mathcal F$ and $\mathcal F^{-1}$ denote the Fourier transform and inverse Fourier transform, respectively.
\subsubsection{The $\bm v$-subproblem}
We first estimate the curvatures based on the latest value $u^{k+1}$ according to \eqref{ki}-\eqref{MCGC} and take them into the curvature functions. Then, the minimization problem w.r.t. $\bm v$ becomes straightforward, which has the unique minimizer by the shrinkage operator \cite{beck2009fast}
\begin{equation}\label{sol-v}
\bm v^{k+1} = {\rm{shrinkage}}\bigg(\frac{\mu\nabla u^{k+1}-\bm\Lambda^k+\sigma{\bm v}^k}{\mu+\sigma} ,\frac{g(\kappa(u^{k+1}))}{\mu+\sigma} \bigg)
\end{equation}
with the shrinkage operator being defined as
\begin{equation*}
{\rm{shrinkage}}(a,b)={\rm{max}}\{|a|-b,0\} \circ \frac{a}{|a|},
\end{equation*}
and $\circ$ being the element-wise multiplication.

\subsection{Convergence Analysis}
In this subsection, we give the  convergence result for Algorithm 3.1. First, we prove that a solution of the discrete curvature-based regularization model \eqref{dis_model} exists.
\begin{lemma}
\label{coercivelemma}
There exists a minimizer $u^*\in \mathbb{R}^{m^2}$ for the discrete minimization problem \eqref{dis_model}.
\end{lemma}
\begin{proof}
By the definitions of $g$ in \eqref{kappa1}, \eqref{kappa2} and \eqref{kappa3}, $g(\kappa(u))\geq 1$. According to Lemma 3.8 of \cite{Huang2009A}, we have $\sum\limits_{1\leq i,j\leq m}|\nabla u_{i,j}| +\frac{\lambda}{2}\|u-f\|^2$ is coercive. Then there is
\begin{equation}
\sum_{1\leq i,j\leq m}g(\kappa_{i,j})|\nabla u_{i,j}| +\frac{\lambda}{2}\|u-f\|^2 \geq \sum\limits_{1\leq i,j\leq m}|\nabla u_{i,j}| +\frac{\lambda}{2}\|u-f\|^2
\end{equation}
is also coercive.
By definition of $\kappa=H,K$ as defined in \eqref{kminmax}, \eqref{MCGC} and continuity of the min/max functions, $\kappa=\kappa(\{\kappa_1,\kappa_2\})$ is continuous on $\{\kappa_\ell:\ell=1,\cdots,8\}$. Moreover by \eqref{ki}, $\kappa_\ell$ ($\ell=1,\cdots,8$) are continuous functions on $u$. Therefore, $\sum\limits_{1\leq i,j\leq m}g(\kappa_{i,j})|\nabla u_{i,j}| +\frac{\lambda}{2}\|u-f\|^2$ is continuous on $u$. Together with coercivity and continuity, we have that the discrete minimization problem \eqref{dis_model} has a minimizer $u^*\in  \mathbb{R}^{m^2}$.
\end{proof}

In the followings, we analyze the convergence theoretically for the proposed ADMM-based numerical algorithm under certain conditions. We first give a useful lemma.

\begin{lemma}
\label{convexlemma}
Suppose $T(x)=\frac{1}{2}\|Ax-b\|^2+N(x)$ with a convex function $N$. Assuming $x^\ast$ be a stationary point of $T(x)$, i.e., $0\in\partial T(x^\ast)$, then we obtain
\begin{equation*}
 T(x)-T(x^\ast)\geq\frac{1}{2}\|A(x-x^\ast)\|^2.
\end{equation*}
\end{lemma}

\begin{proof}\label{convexproof}
Let $M(x)=\frac{1}{2}\|Ax-B\|^2$. Since $x^\ast$ is a stationary point, i.e., $0\in\nabla M(x^\ast)+\partial N(x^\ast)$, we have
\begin{equation*}
  N(x)-N(x^\ast)\geq \langle -\nabla M(x^\ast),x-x^\ast \rangle,~\forall x.
\end{equation*}
It follows that
\begin{equation*}
 T(x)-T(x^\ast)\geq M(x)-M(x^\ast)-\langle \nabla M(x^\ast),x-x^\ast \rangle = \frac{1}{2}\|A(x-x^\ast)\|^2,
\end{equation*}
which concludes the lemma.
\end{proof}

\begin{theorem}\label{convergenttheorem}
Assume $\{(u^k, \bm v^k; \bm\Lambda^{k})\}_{k \in \mathbb N}$ is the sequence generated by proposed ADMM-based Algorithm 3.1 and $({\bar u}, \bar {\bm v}; \bar {\bm\Lambda})$ is a point satisfying the first-order optimality conditions
\begin{eqnarray}\label{optimal conditions}
\left \{
\begin{array}{ll}
\lambda(u-f)+\nabla\cdot \bm\Lambda=0,\\
g(\kappa)s+\bm\Lambda=0,~{\rm where}~s\in \partial |\bm v|,\\
\bm v-\nabla u=0.
\end{array}
\right.
\end{eqnarray}
If for any $s^k\in \partial |\bm v^k|$ and any $\bar s\in \partial |\bar {\bm v}|$ satisfy
\begin{equation}\label{deltak}
\Delta_k:=\langle(g(\kappa^k)-g(\bar \kappa))s^k,\bm v^k-\bar {\bm v}\rangle \geq 0,~\forall k\in\mathbb{N}.
\end{equation}
Then, we have \\
$(\rm a)$ The Lagrangian functional is monotonically decreasing, i.e.,
\begin{equation}\label{functional}
\begin{split}
\mathcal{L}(u^k,\bm v^k;\bm\Lambda^k)-\mathcal{L}(u^{k+1},\bm v^{k+1}; \bm\Lambda^{k+1})&\geq\frac{\tau}{2}\|u^{k+1}-u^k\|^2+\frac{\mu}{2}\|\nabla u^{k+1}-\bm v^k\|^2\\&+\frac{\sigma}{2}\|\bm v^{k+1}-\bm v^k\|^2+\frac{1}{2\mu}\|\bm\Lambda^{k+1}-\bm\Lambda^{k}\|^2.
\end{split}
\end{equation}
\\
$(\rm b)$ The successive errors $u^{k+1}-u^k\rightarrow 0$, $\bm v^{k+1}-\bm v^k\rightarrow 0$, $\nabla u^{k+1}-\bm v^k\rightarrow 0$, and $\bm\Lambda^{k+1}-\bm\Lambda^k\rightarrow 0$ as $k\rightarrow\infty$.\\
$(\rm c)$ The sequence $\{(u^k, \bm v^k; \bm\Lambda^{k})\}_{k \in \mathbb N}$ converges to a limit point $(u^\ast, \bm v^\ast; \bm\Lambda^{\ast})$ that satisfies the first-order optimality conditions \eqref{optimal conditions}.
\end{theorem}
\begin{proof}\label{convergeproof}
$(\rm a)$ For $u$-subproblem, according to Lemma \ref{convexlemma}, it follows that
\begin{equation}\label{uconvex}
\mathcal{L}(u^k,\bm v^k;\bm\Lambda^k)-\mathcal{L}(u^{k+1},\bm v^{k}; \bm\Lambda^{k})\geq\frac{\tau}{2}\|u^{k+1}-u^k\|^2+\frac{\mu}{2}\|\nabla u^{k+1}-\bm v^k\|^2.
\end{equation}

Similarly, for $\bm v$-subproblem, by Lemma \ref{convexlemma} and the Theorem 3 in \cite{yashtini2016fast}, we have

\begin{align}\label{vconvex}
 \mathcal{L}(u^{k+1},\bm v^k;\bm\Lambda^k)-\mathcal{L}(u^{k+1},\bm v^{k+1}; \bm\Lambda^{k})& \geq\frac{\sigma}{2}\|\bm v^{k+1}-\bm v^k\|^2 \\ \nonumber
 & +\langle g(\kappa^{k+1})s^{k+1}-g(\bar \kappa)\bar s,\bm v^{k+1}-\bar{\bm v} \rangle.
\end{align}
Note that $\langle g(\kappa^{k+1})s^{k+1}-g(\bar \kappa)\bar s,\bm v^{k+1}-\bar{\bm v} \rangle=\Delta_{k+1}+g(\bar \kappa)\langle s^{k+1}-\bar s, \bm v^{k+1}-\bar{\bm v}\rangle$. Referring to Lemma 3.3 in \cite{chen2013bregman}, the term $\langle s^{k+1}-\bar s, \bm v^{k+1}-\bar{\bm v}\rangle \geq 0$ for any $s^k\in \partial |\bm v^k|$ and $\bar s\in \partial |\bar {\bm v}|$. In addition, $\Delta_{k} \geq 0$ for all $k$ by the assumption \eqref{deltak} of proposed theorem. Therefore $\langle g(\kappa^{k+1})s^{k+1}-g(\bar \kappa)\bar s,\bm v^{k+1}-\bar{\bm v} \rangle \geq 0$.

Using \eqref{multipliers} and Lemma \ref{convexlemma}, it is immediate that
\begin{equation}\label{lambdaconvex}
\mathcal{L}(u^{k+1},\bm v^{k+1}; \bm\Lambda^k)-\mathcal{L}(u^{k+1},\bm v^{k+1};\bm\Lambda^{k+1})
\geq \frac{1}{2\mu}\|\bm\Lambda^{k+1}-\bm\Lambda^{k}\|^2.
\end{equation}

Then, by adding \eqref{uconvex}-\eqref{lambdaconvex} and dropping the nonnegative term, we complete the proof of part $(\rm a)$.

$(\rm b)$ Due to the boundedness of the sequence $\mathcal{L}(u^k,\bm v^k;\bm\Lambda^k)$, we sum the inequality \eqref{functional} in part $(\rm a)$ from $k=1$ to $\infty$ to obtain
\begin{equation*}
\sum_{k=1}^\infty \|u^{k+1}-u^k\|^2+\|\nabla u^{k+1}-\bm v^k\|^2+\|\bm v^{k+1}-\bm v^k\|^2+\|\bm\Lambda^{k+1}-\bm\Lambda^k\|^2<\infty.
\end{equation*}
This further gives
\begin{equation*}
\lim_{k\rightarrow\infty}(\|u^{k+1}-u^k\|=\|\nabla u^{k+1}-\bm v^k\|=\|\bm v^{k+1}-\bm v^k\|=\|\bm\Lambda^{k+1}-\bm\Lambda^k\|)=0.
\end{equation*}

$(\rm c)$ According to part $(\rm a)$ and $(\rm b)$, the sequence $\{(u^k,\bm v^k;\bm\Lambda^k)\}_{k\in \mathbb N}$ generated by Algorithm 3.1 is uniformly bounded on $\rm{\Omega}$. Therefore, there exists a weakly convergent subsequence $\{(u^{k_l},\bm v^{k_l};\bm\Lambda^{k_l})\}_{l \in \mathbb N}$, which has a limit point $(u^\ast,\bm v^\ast;\bm\Lambda^\ast)$. Analogously, due to $\bm v^{k_l} \rightarrow \bm v^\ast$ a.e. in $\rm{\Omega}$ as $l \rightarrow \infty$ and $s^{k_l} \in \partial |\bm v^{k_l}|$, there exists a subsequence of $\{s^{k_l}\}_{l \in \mathbb N}$ that converges weakly to $s^\ast \in \partial |\bm v^\ast|$.

\begin{figure*}[t]
      \centering
      \subfigure[TAC-MC]{
			\includegraphics[width=0.35\linewidth]{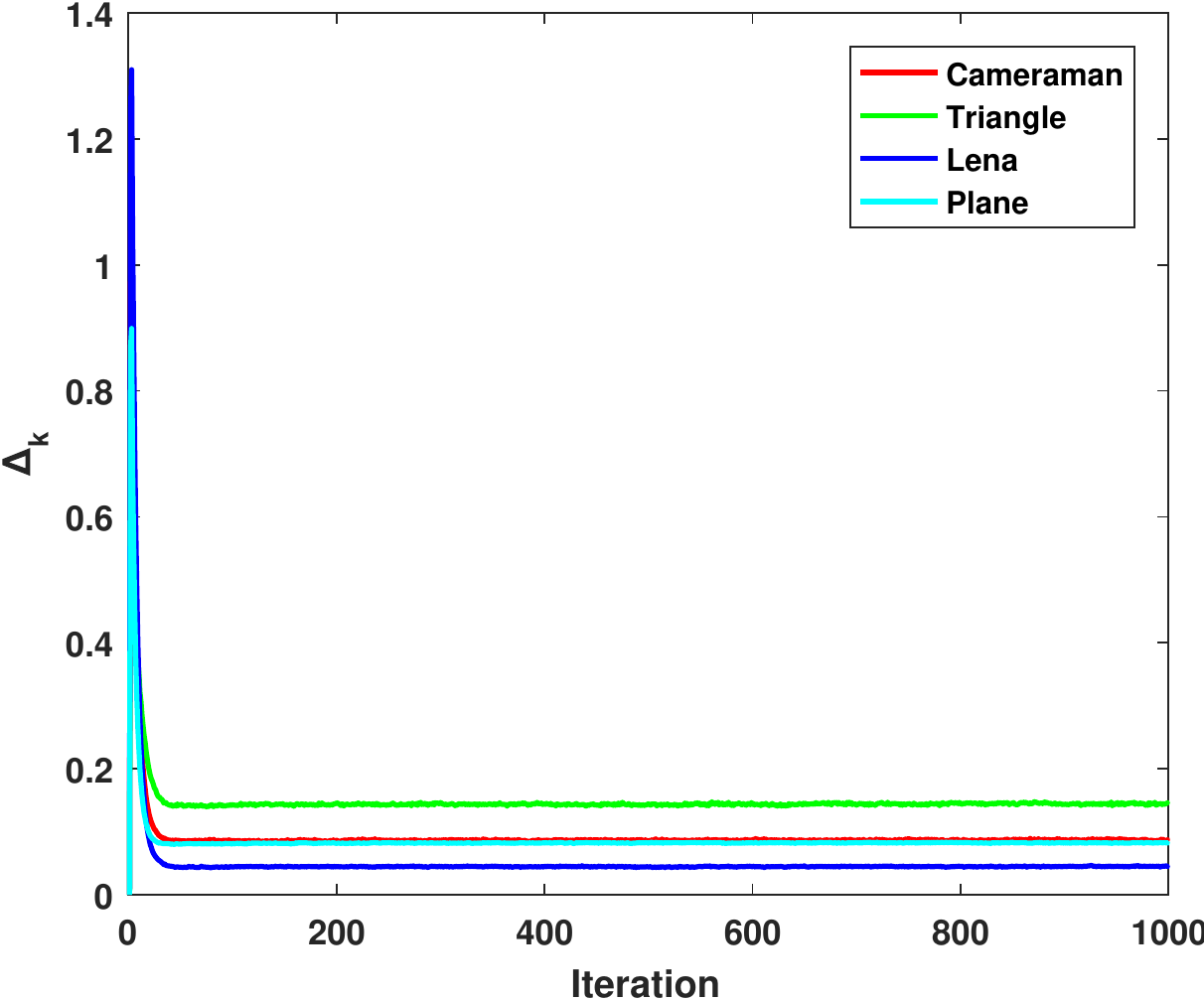}}
      \hspace{.2in}
      \subfigure[TAC-GC]{
			\includegraphics[width=0.35\linewidth]{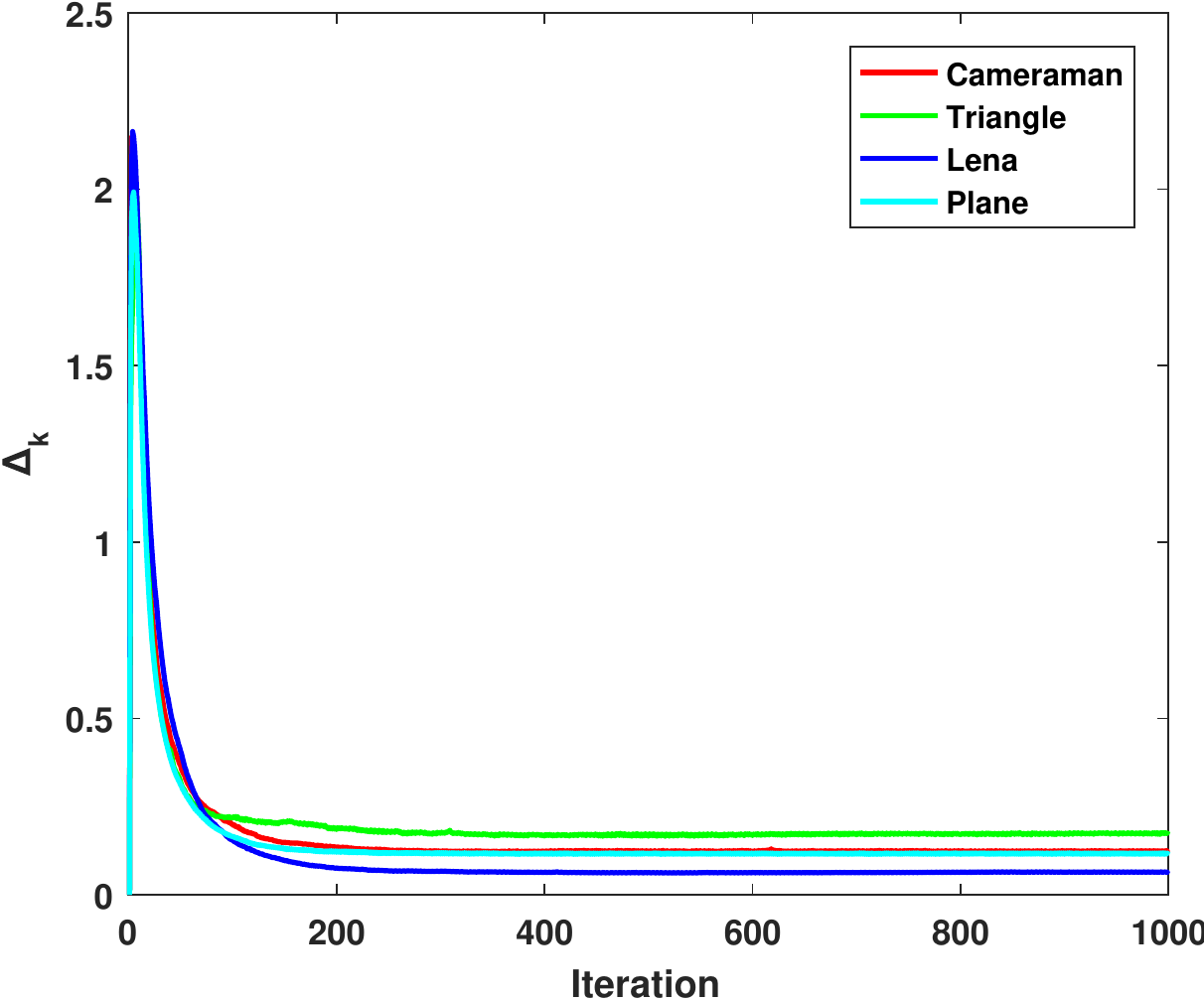}}
	\caption{The behavior of $\Delta_k$ with iteration numbers in TAC method on the different test images with $\tau=0$ and $\sigma=0$. Note that $\Delta_k\geq0$ for all iterations.}
	\label{DeltakTAC}
\end{figure*}

The sequence $\{(u^{k_l},\bm v^{k_l};\bm \Lambda^{k_l})\}_{l \in \mathbb N}$ satisfies the optimality conditions in Algorithm 3.1, i.e.,
\begin{eqnarray*}
\left \{
\begin{array}{ll}
{\lambda}(u^{k_l+1}-f)-\mu\nabla \cdot \big(\nabla u^{k_l+1}-\bm v^{k_l}-\frac{\bm\Lambda^{k_l}}{\mu}\big)+\tau (u^{k_l+1}-u^{k_l})= 0,\\
g(\kappa^{k_l+1})s^{k_l+1}+\mu\big(\bm v^{k_l+1}-\nabla u^{k_l+1}+\frac{\bm \Lambda^{k_l}}{\mu}\big)+\sigma(\bm v^{k_l+1}-\bm v^{k_l})= 0, \\
\bm\Lambda^{k_l+1}=\bm\Lambda^{k_l}+\mu(\bm v^{k_l+1}-\nabla u^{k_l+1}).
\end{array}
\right.
\end{eqnarray*}

Taking the limit from the convergent subsequence, we obtain
\begin{eqnarray*}
\left \{
\begin{array}{ll}
\lambda(u^\ast-f)+\nabla \cdot \bm\Lambda^\ast=0,\\
g(\kappa^\ast)s^\ast+\bm\Lambda^\ast=0,~s^\ast\in \partial |\bm v^\ast|,\\
\bm v^\ast-\nabla u^\ast=0,
\end{array}
\right.
\end{eqnarray*}
for almost every point in $\rm{\Omega}$. This implies that the generated limit point $(u^\ast, \bm v^\ast; \bm\Lambda^{\ast})$ by sequence $\{(u^k,\bm v^k;\bm\Lambda^k)\}_{k\in \mathbb N}$ satisfies the first-order optimality conditions \eqref{optimal conditions}.
\end{proof}

\begin{remark}
The proof of Theorem 3.1 requires $\Delta_k\geq 0$. Indeed, it is difficult to find any lower bound theoretically. As shown in FIG. \ref{DeltakTAC}, the numerical experiments show that the behavior of $\Delta_k$ satisfies the assumption even when $\tau$ and $\sigma$ are fixed as $0$. Thus, it is somehow reasonable to make such assumption on $\Delta_k$.
\end{remark}
\begin{remark}
We always set $\tau=0$ and $\sigma=0$ in the numerical implementations, which is the case in FIG. \ref{DeltakTAC}.
\end{remark}

\section{Experiments}
\label{sect4}

In this section, comprehensive experiments on both synthetic and real image restoration with different noise distributions are implemented to verify the efficiency and superiority of our curvature-based variational models. These experimental images are composed of different edges and texture structures as well as homogenous regions. All numerical experiments are performed utilizing Matlab R2016a on a machine with 3.40GHz Intel(R) Core(TM) i7-6700 CPU and 32GB RAM.

In our experiments, we adopt the popular peak signal-to-noise ratio (PSNR) and structural similarity (SSIM) \cite{wang2004image} to quantitatively evaluate the imaging performance under different image degradation conditions. In addition, the variation of the residuals as well as the relative errors and numerical energy are provided to illustrate the convergence of the ADMM algorithm versus the iterations, which are defined as
\[ R(\bm v^k, u^k) = \|{\bm v}^k-\nabla u^k\|_1,\]and
\[ReErr(\bm\Lambda^k) = \frac{\|\bm\Lambda^k-\bm\Lambda^{k-1}\|_{1}}{\|\bm\Lambda^{k-1}\|_{1}} ~~~\mbox{and}~~~ReErr(u^k)=\frac{\|u^k-u^{k-1}\|_{1}}{\|u^{k-1}\|_{1}},
\]
and
\[E(u^k) =\sum_{1\leq i,j\leq m} g\big(\kappa(u^{k-1}_{i,j})\big) |\nabla u^k_{i,j}| + \frac{\lambda}{2}\|u^k-f\|^2.\]

\subsection{Parameters discussing}

There are total three parameters in the proposed algorithm such that $\lambda,\alpha, \mu$. The most important parameter in our model is the $\lambda$, which is used to balance the contribution between the data fidelity and regularization term. The smaller the $\lambda$ is, the smoother the restoration is. If $\lambda$ is too large, the model fails to remove the noises, while if $\lambda$ is too small,  the restoration becomes over-smoothed and some features will be lost. The positive parameter $\alpha$ can balance the influence between the curvature and arclength, which should be chosen  appropriately to smooth the homogenous regions as well as preserve the image details. The penalty parameter $\mu$ controls the convergent speed and stability of the proposed algorithm, we notice that large $\mu$ reduces both efficiency of the algorithm and restoration quality, while too small $\mu$ can not guarantee the stability of proposed algorithm. The specific values of $\lambda$, $\alpha$ and $\mu$ are given in each experiment. Besides, we choose $h=1$ throughout the experiments for the best balance between the smoothness and fine details.

\begin{figure*}[t]
      \centering
      \subfigure{
			\includegraphics[width=0.18\linewidth]{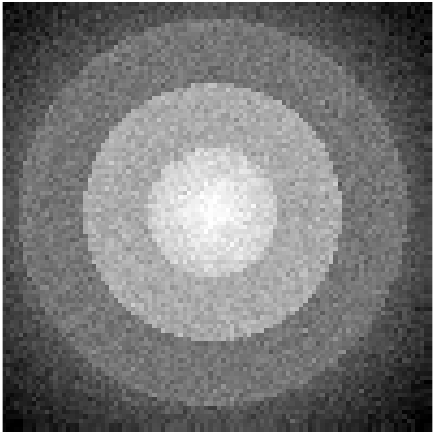}}\hspace{-1ex}
      \subfigure{
			\includegraphics[width=0.18\linewidth]{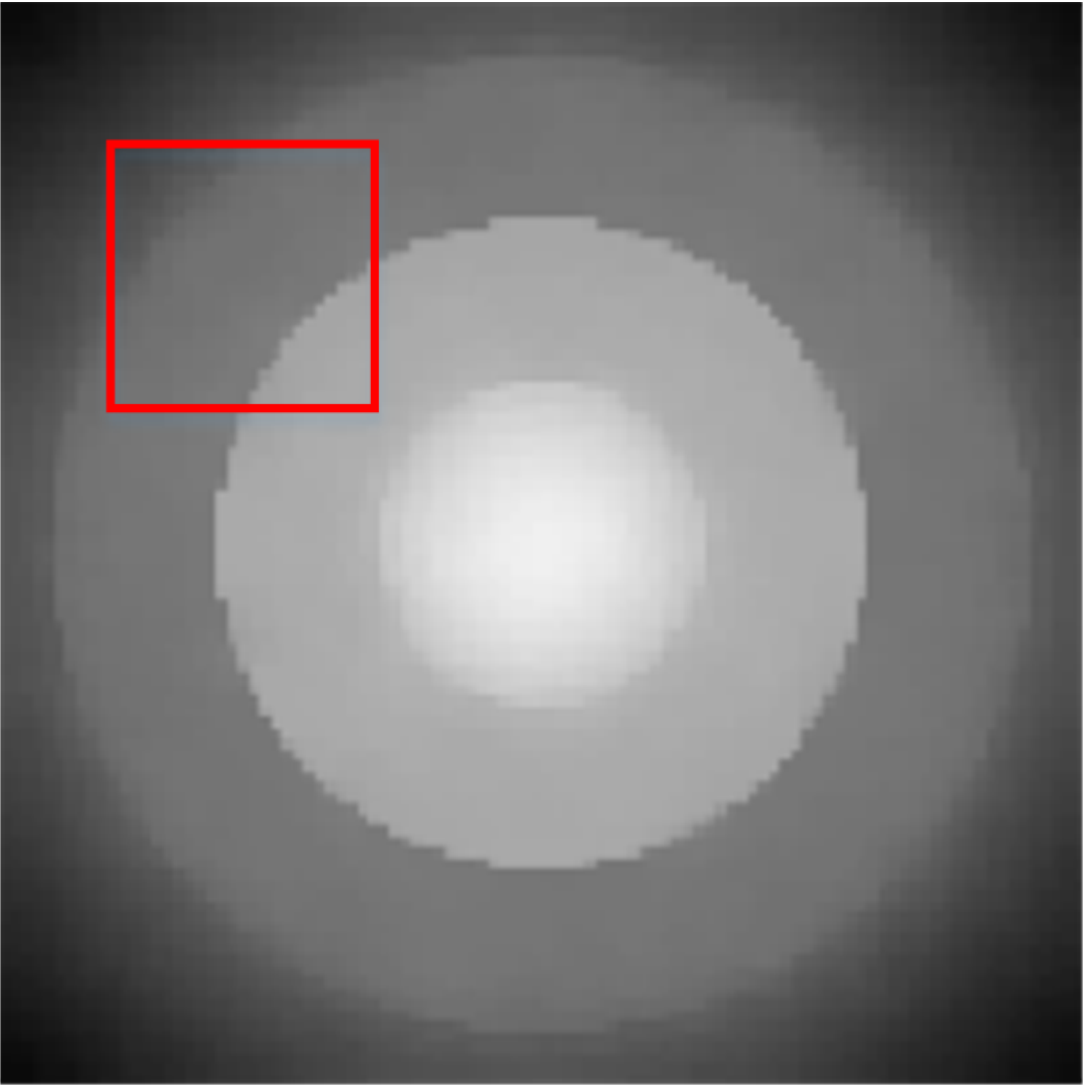}}\hspace{-1ex}
      \subfigure{
            \includegraphics[width=0.18\linewidth]{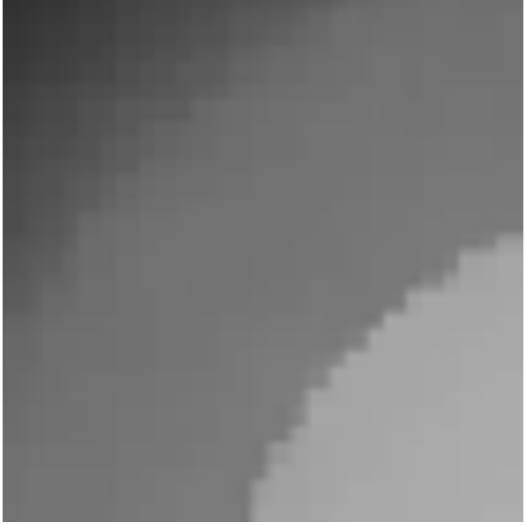}}\hspace{-1ex}
      \subfigure{
			\includegraphics[width=0.18\linewidth]{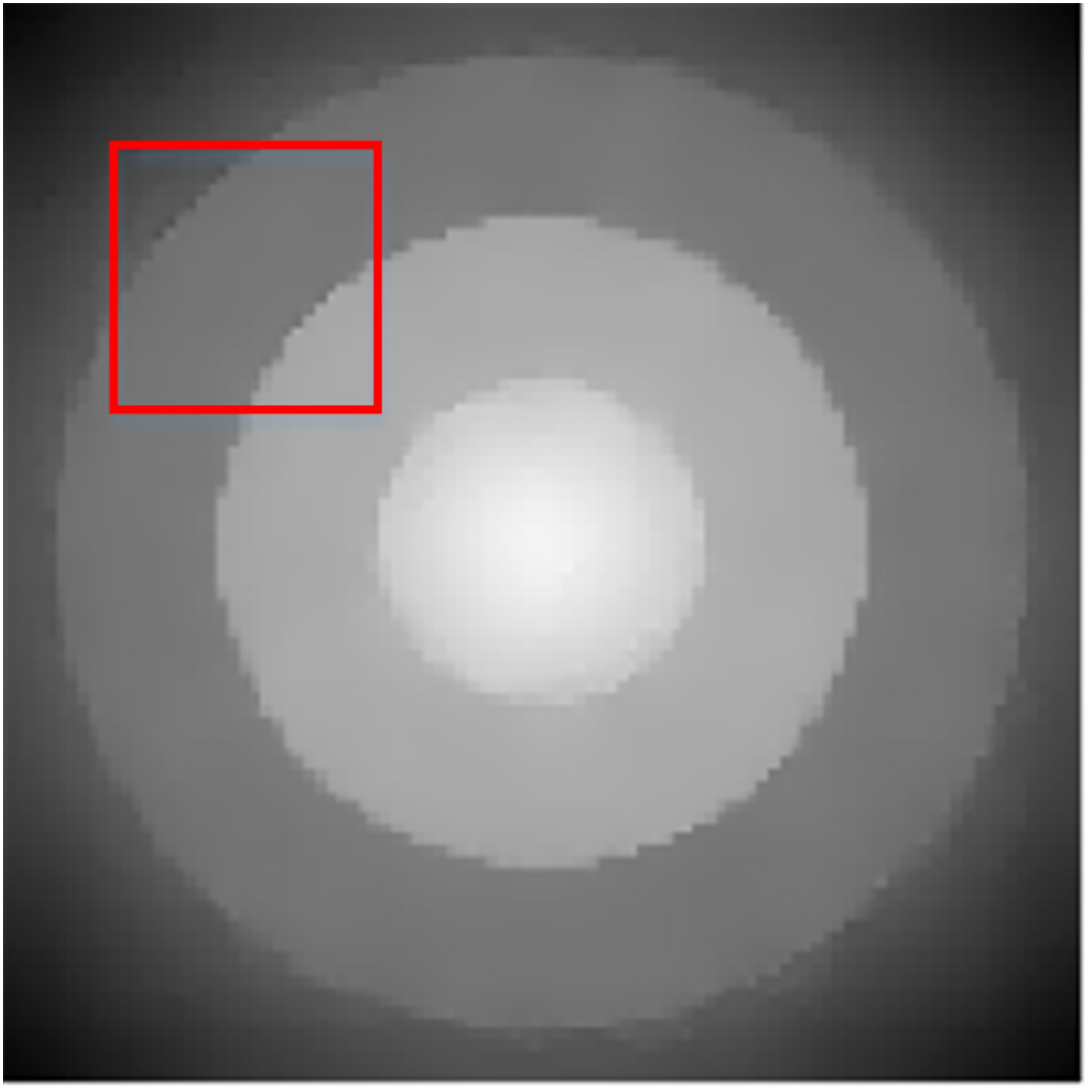}}\hspace{-1ex}
      \subfigure{
            \includegraphics[width=0.18\linewidth]{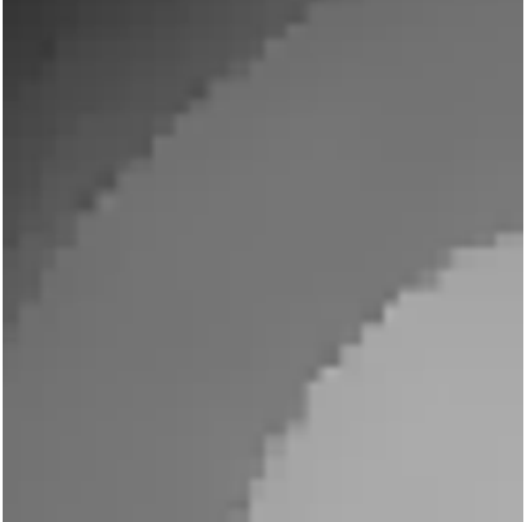}}
      \subfigure{
			\includegraphics[width=0.18\linewidth]{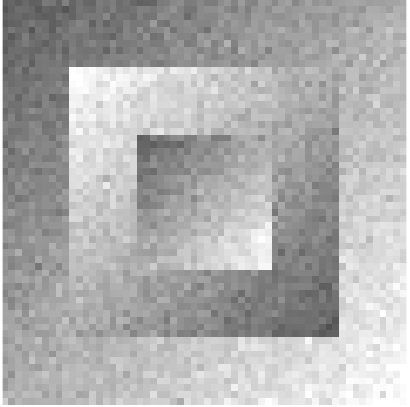}}\hspace{-1ex}
      \subfigure{
			\includegraphics[width=0.18\linewidth]{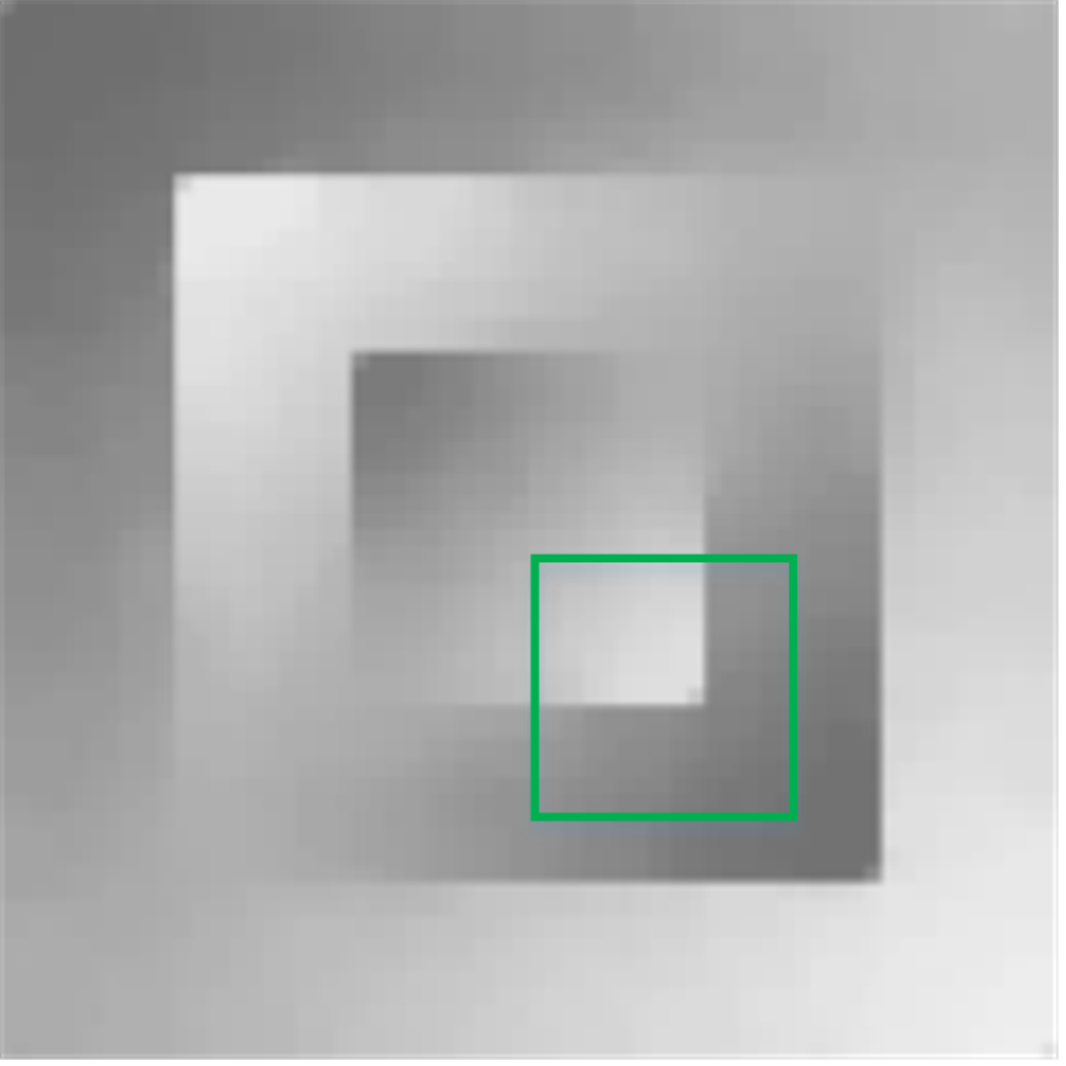}}\hspace{-1ex}
      \subfigure{
            \includegraphics[width=0.18\linewidth]{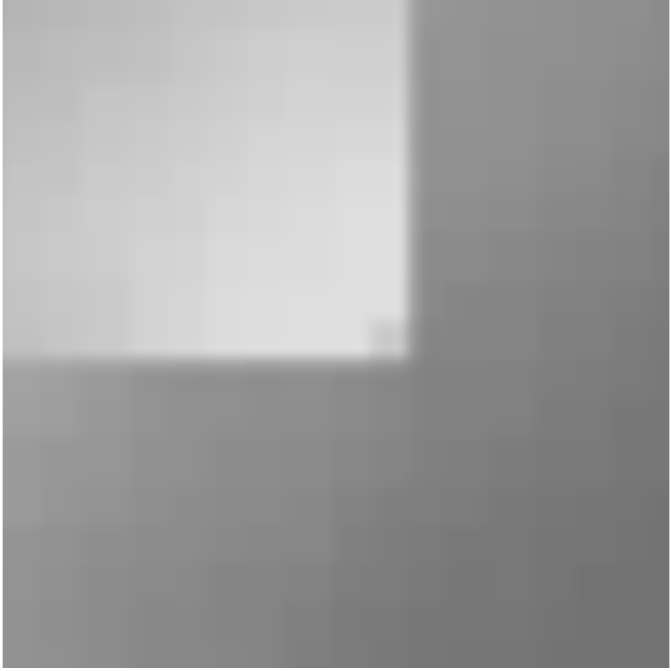}}\hspace{-1ex}
      \subfigure{
			\includegraphics[width=0.18\linewidth]{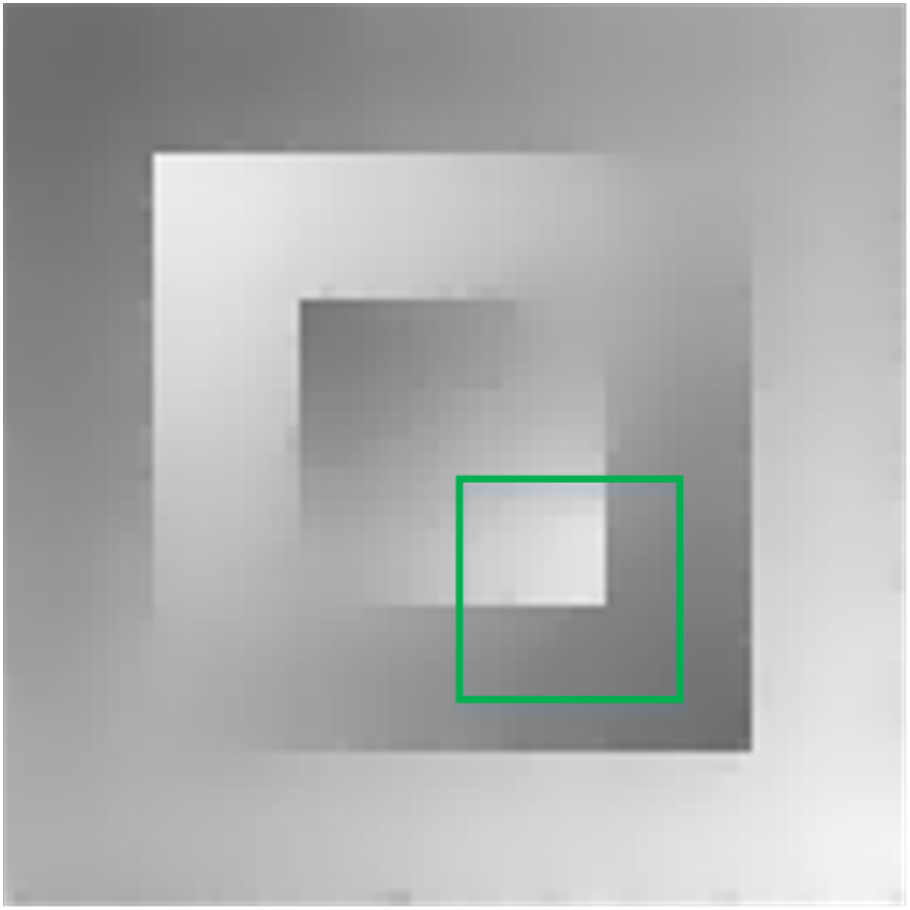}}\hspace{-1ex}
      \subfigure{
            \includegraphics[width=0.18\linewidth]{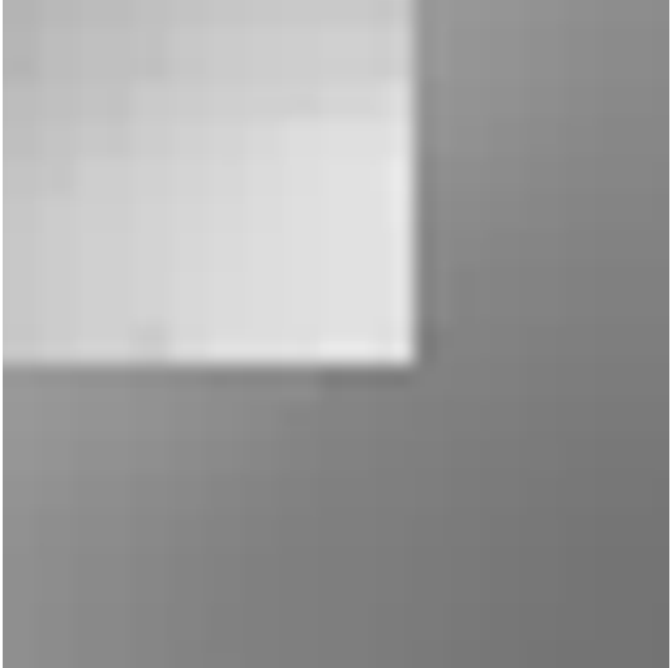}}
      \setcounter{subfigure}{0}
      \subfigure[Noisy]{
			\includegraphics[width=0.18\linewidth]{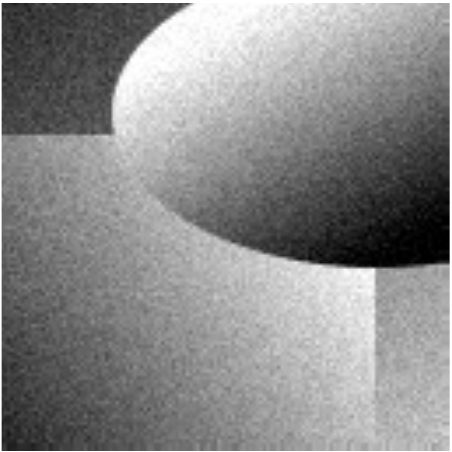}}\hspace{-1ex}
      \subfigure[Euler's elastica]{
			\includegraphics[width=0.18\linewidth]{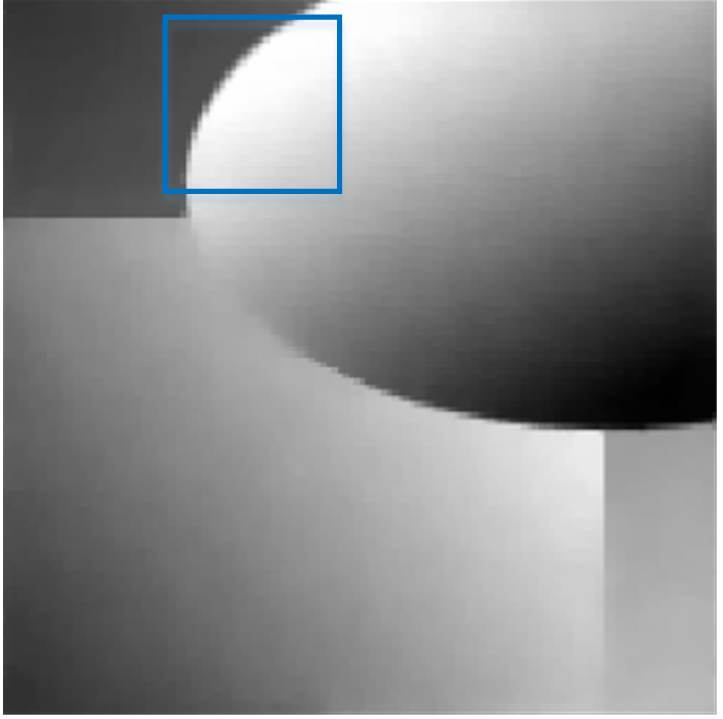}}\hspace{-1ex}
      \subfigure[zoomed]{
			\includegraphics[width=0.18\linewidth]{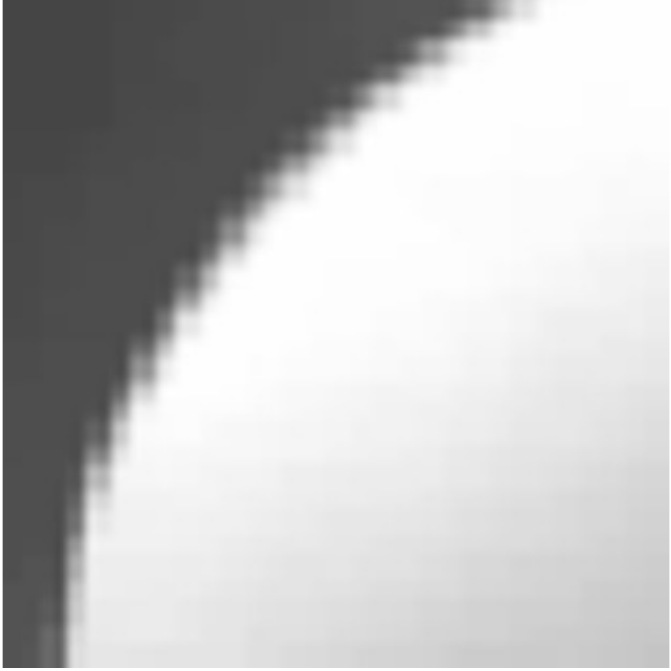}}\hspace{-1ex}
      \subfigure[TAC-MC]{
			\includegraphics[width=0.18\linewidth]{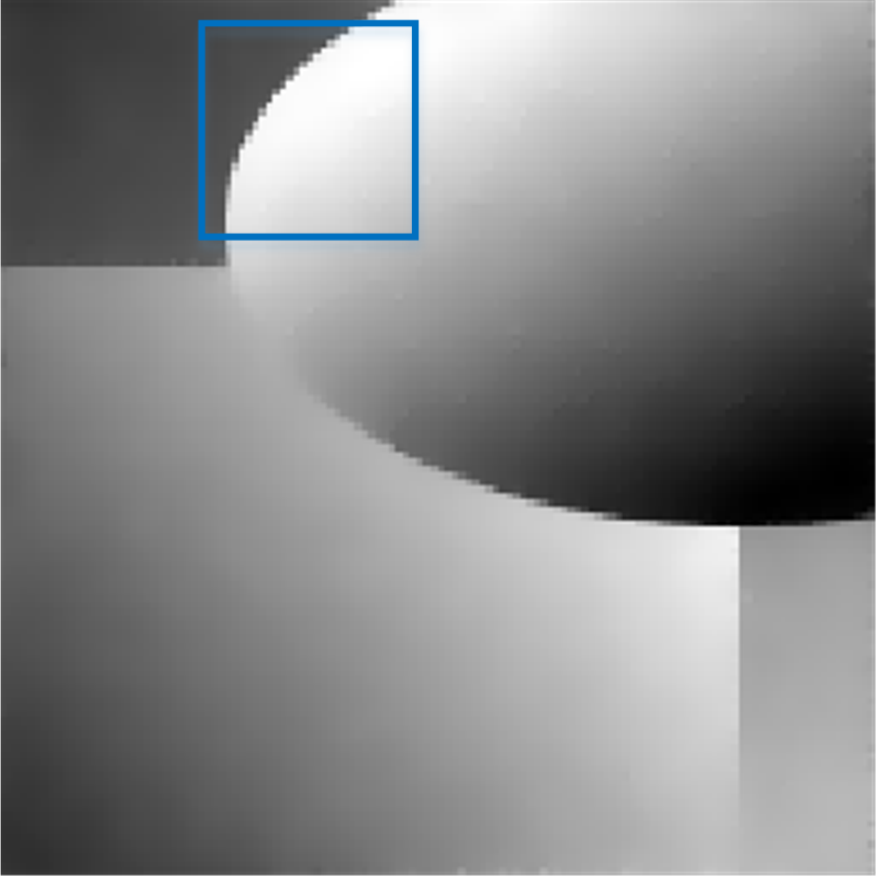}}\hspace{-1ex}
      \subfigure[zoomed]{
			\includegraphics[width=0.18\linewidth]{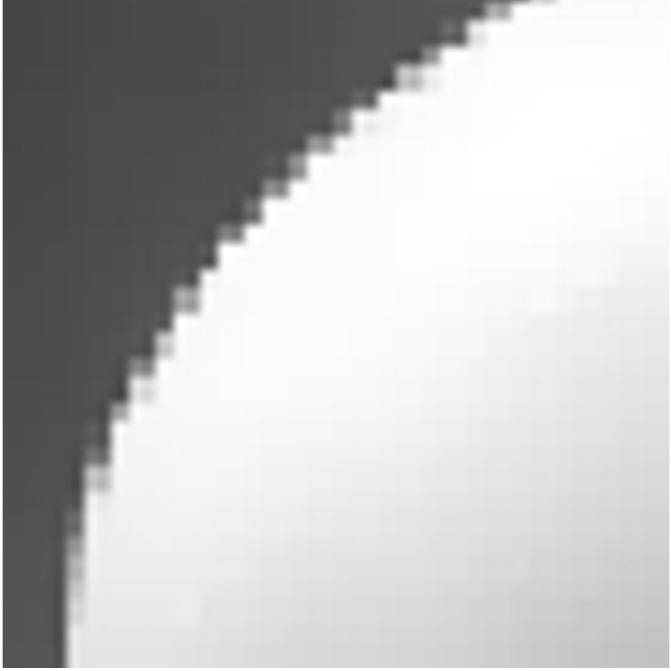}}
	\caption{The denoising results of the smooth images A1, A2 and A3 by the Euler's elastica and our TAC-MC model.}
	\label{smoothimages}
\end{figure*}

\begin{table*}[t]
\caption{The PSNR and SSIM of Gaussian noise removal for the Euler's elastica and our curvature-based models.}
\label{denoise1}
\tiny
\begin{center}
\begin{tabular}{c|c|c|c|c|c|c|c|c}
\hline\hline
Images & Noisy images & Euler & TAC-MC & TAC-GC & TSC-MC & TSC-GC & TRV-MC & TRV-GC \\
\hline\hline
A1($100\times100$) & 28.24 & 36.52 & \bf{38.04} & 37.86 & 37.92 & 37.78 & 37.81 & 37.71\\
PSNR/SSIM & 0.5925 & 0.9515 & \bf{0.9656} & 0.9646 & 0.9645 & 0.9633 & 0.9623 & 0.9616\\
\hline
A2($60\times60$) & 28.29 & 35.02 & \bf{35.98} & 35.74 & 35.70 & 35.63 & 35.62 & 35.84\\
PSNR/SSIM & 0.6458 & 0.9484 & \bf{0.9569} & 0.9552 & 0.9546 & 0.9540 & 0.9528 & 0.9555\\
\hline
A3($128\times128$) & 28.25 & 38.85 & \bf{39.70} & 39.58 & 39.54 & 39.39 & 39.41 & 39.62\\
PSNR/SSIM & 0.5164 & 0.9706 & \bf{0.9775} & 0.9770 & 0.9763 & 0.9750 & 0.9752 & 0.9768\\
\hline\hline
\end{tabular}
\end{center}
\end{table*}

\begin{figure*}[t]
      \centering
      \subfigure[A1-Clean]{
            \includegraphics[width=0.25\linewidth]{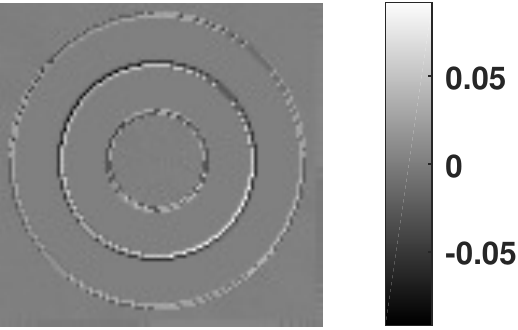}}\hspace{-1.2ex}
      \subfigure[A1-Euler's elastica]{
            \includegraphics[width=0.26\linewidth]{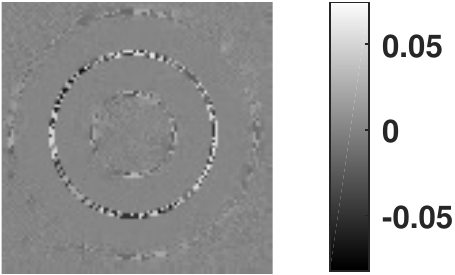}}\hspace{-1.2ex}
      \subfigure[A1-TAC-MC]{
            \includegraphics[width=0.25\linewidth]{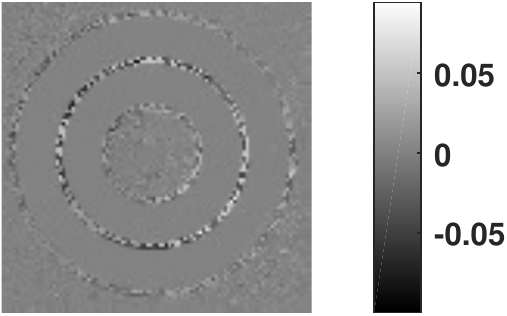}}\hspace{-1.2ex}
      \subfigure[A2-Clean]{
            \includegraphics[width=0.25\linewidth]{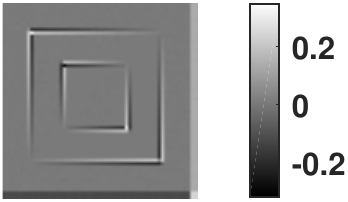}}\hspace{-1.0ex}
      \subfigure[A2-Euler's elastica]{
            \includegraphics[width=0.25\linewidth]{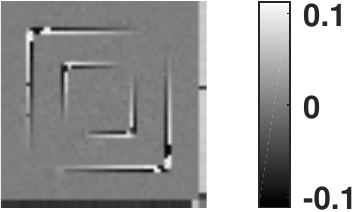}}\hspace{-1.0ex}
      \subfigure[A2-TAC-MC]{
            \includegraphics[width=0.25\linewidth]{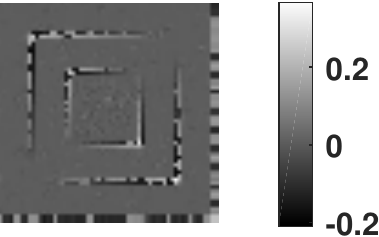}}
	\caption{The numerical MC of the clean images, denoising images obtained by the Euler's elastica and our TAC-MC model.}
	\label{MCimages}
\end{figure*}

\begin{figure*}[t]
      \centering
      \subfigure[A1-Clean]{
			\includegraphics[width=0.3\linewidth]{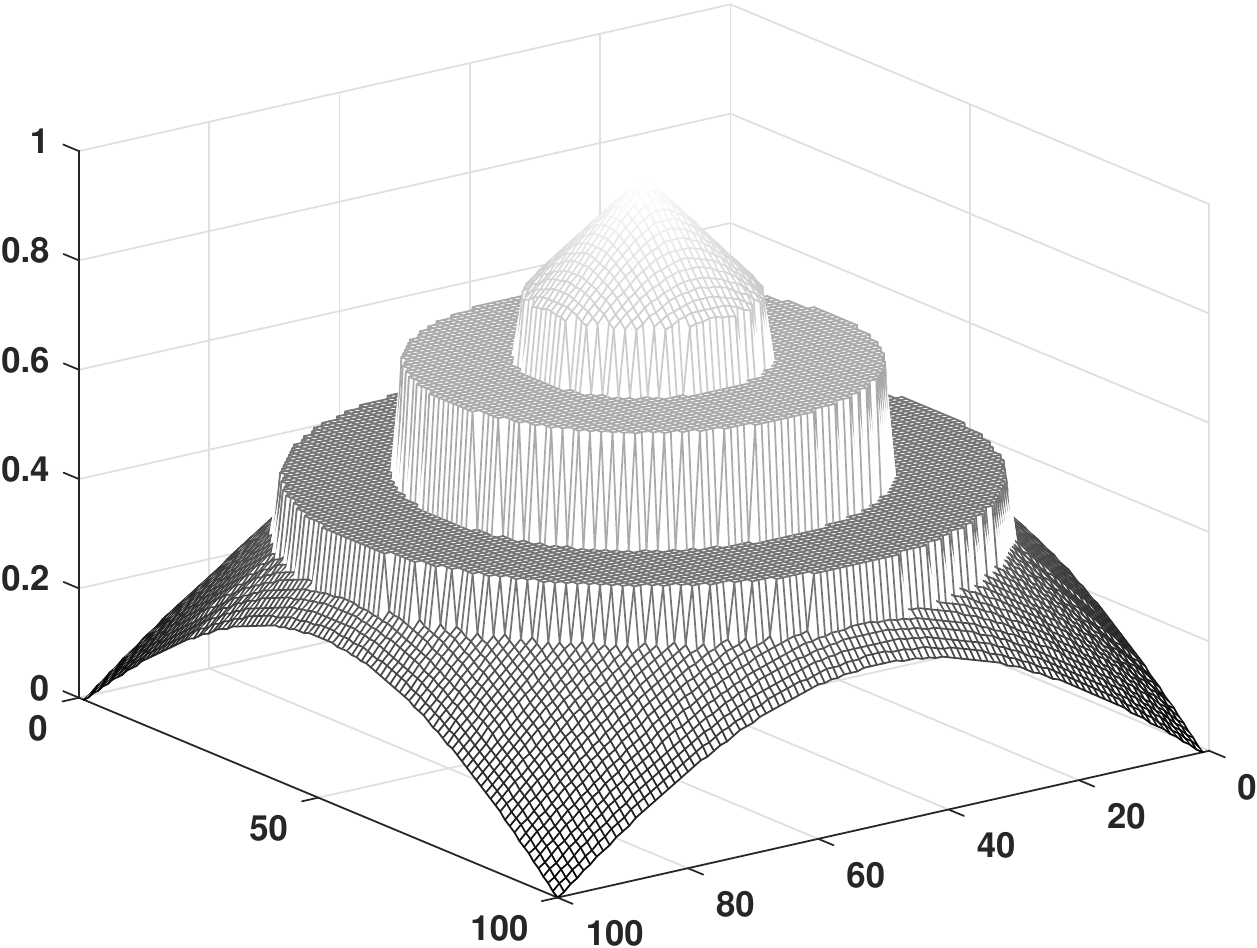}}
      \subfigure[A1-Euler's elastica]{
			\includegraphics[width=0.3\linewidth]{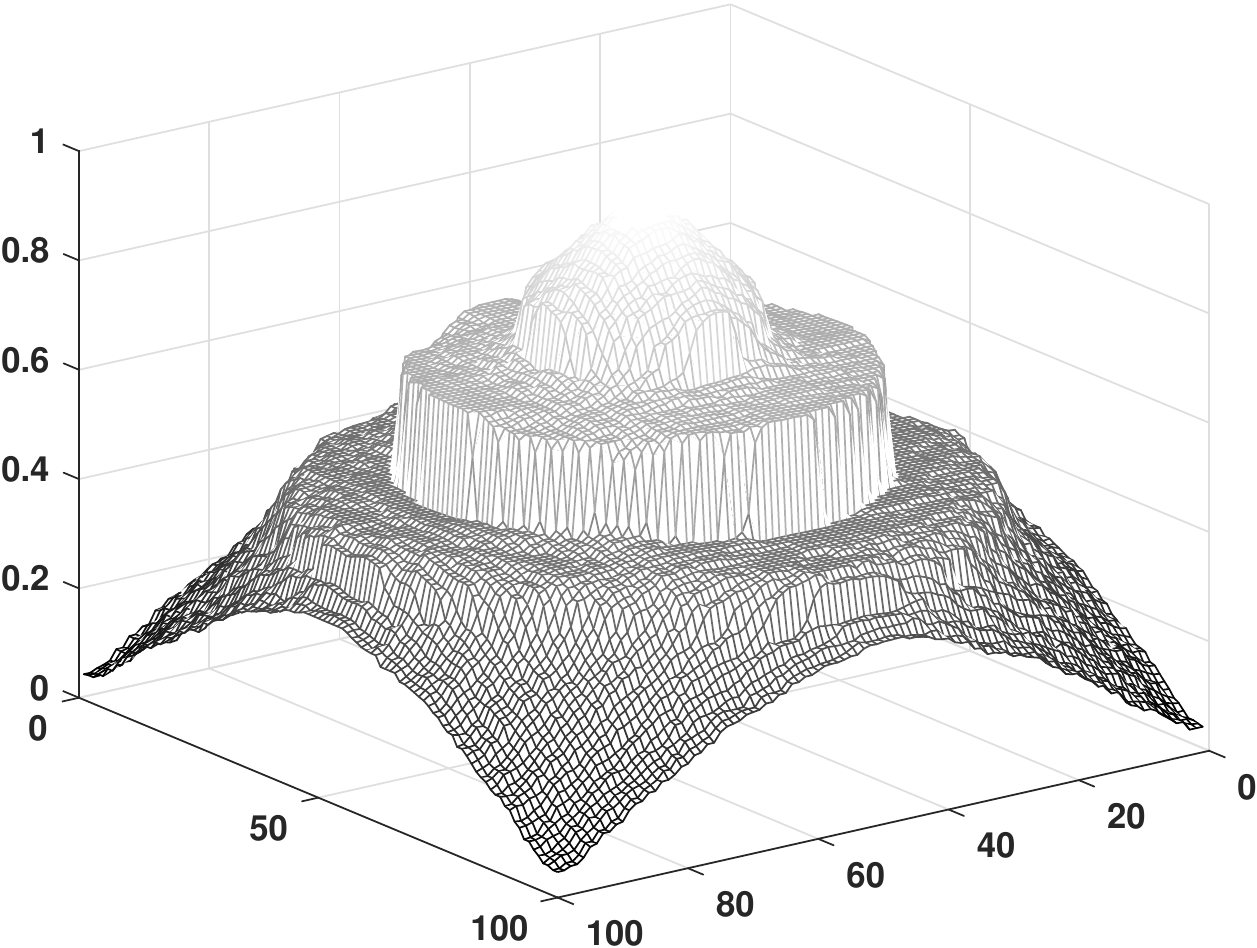}}
      \subfigure[A1-TAC-MC]{
			\includegraphics[width=0.3\linewidth]{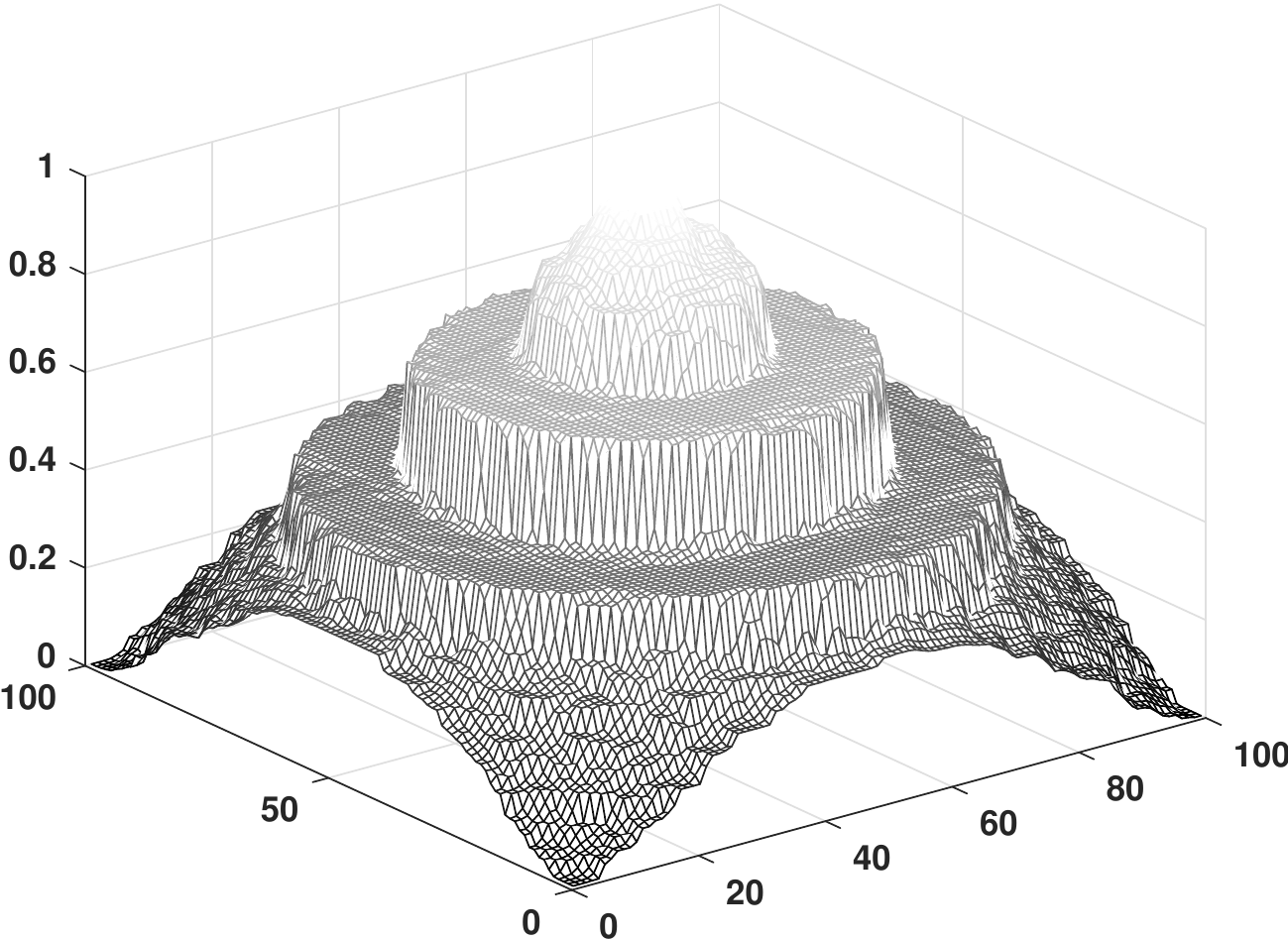}}
      \subfigure[A2-Clean]{
			\includegraphics[width=0.3\linewidth]{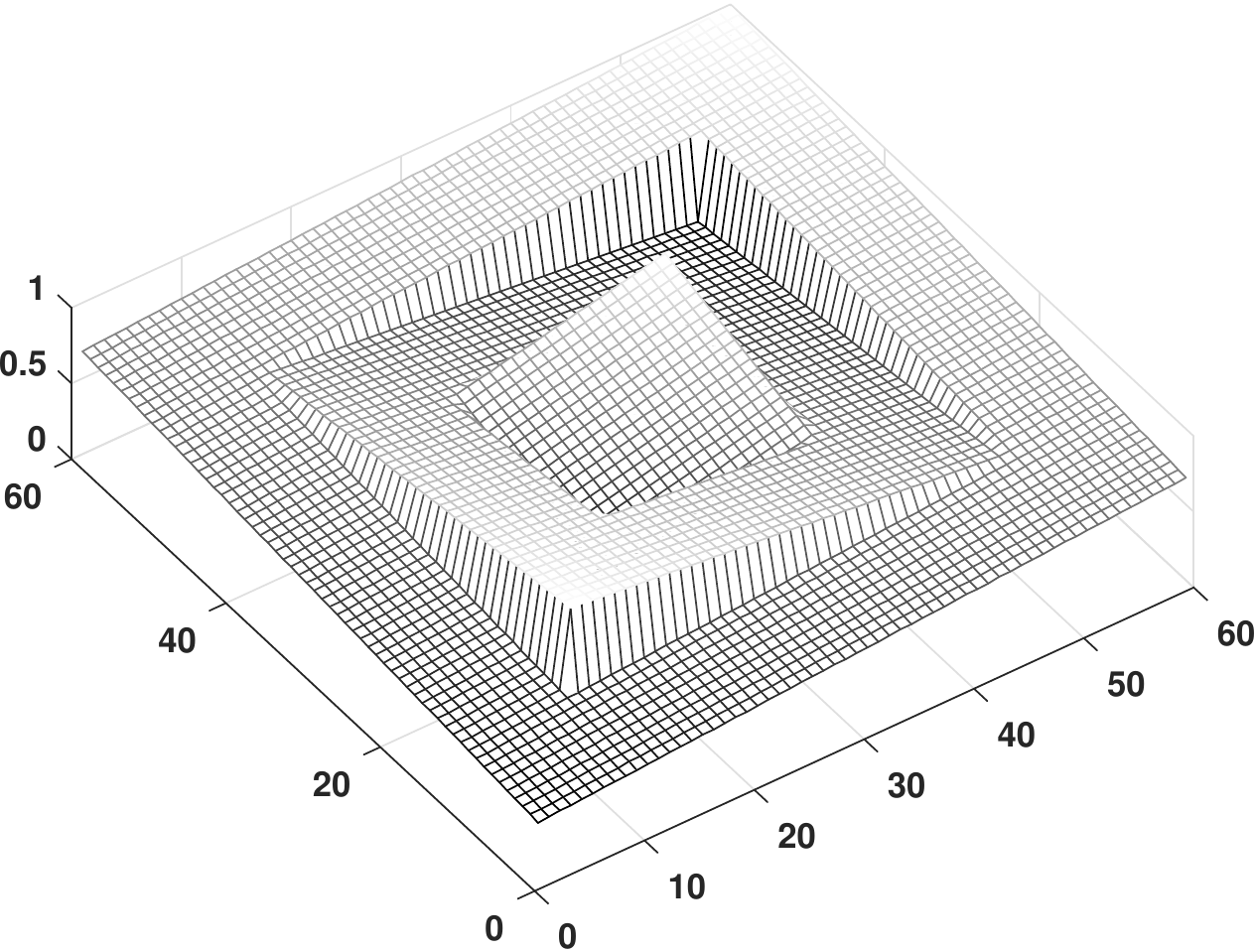}}
      \subfigure[A2-Euler's elastica]{
            \includegraphics[width=0.3\linewidth]{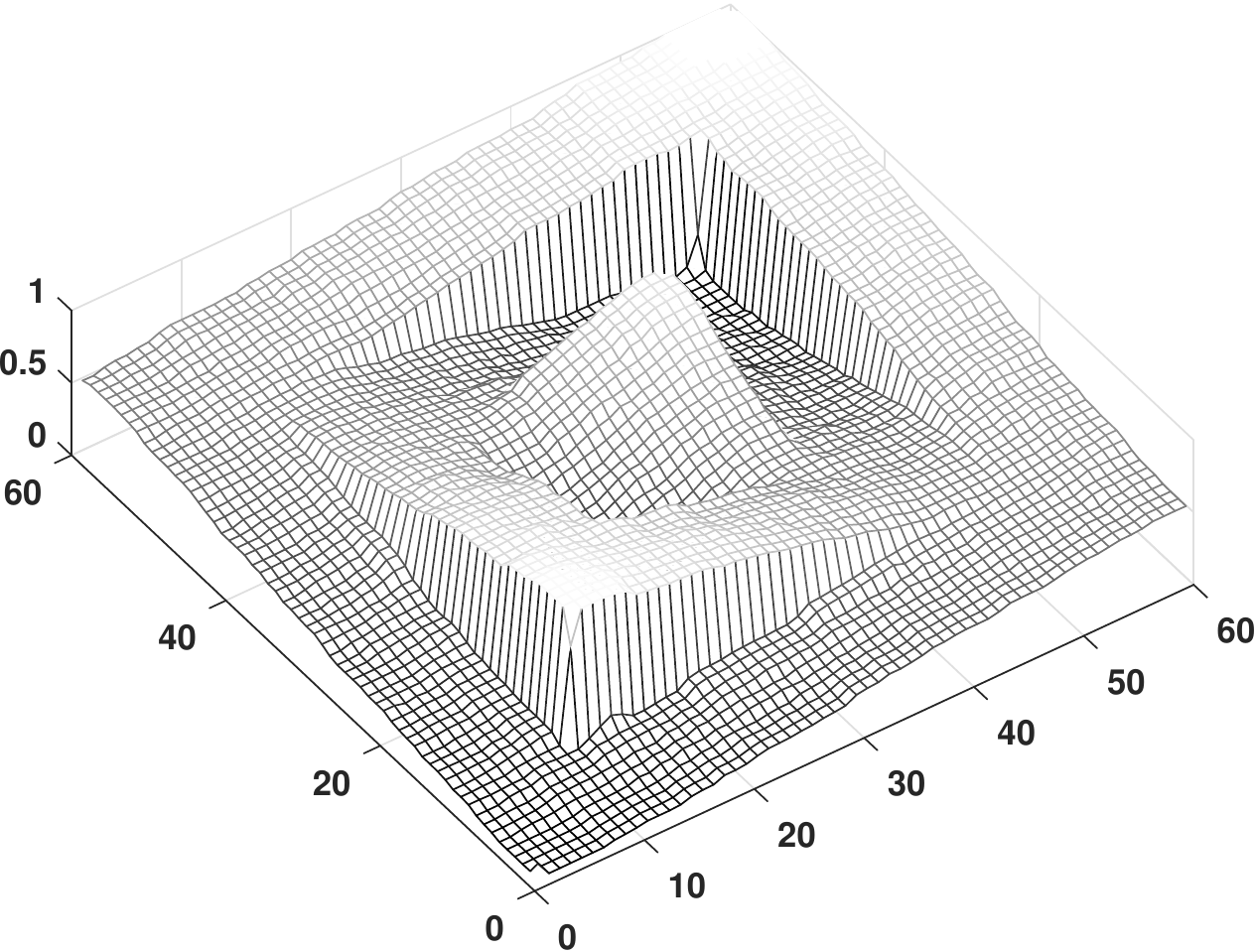}}
      \subfigure[A2-TAC-MC]{
			\includegraphics[width=0.3\linewidth]{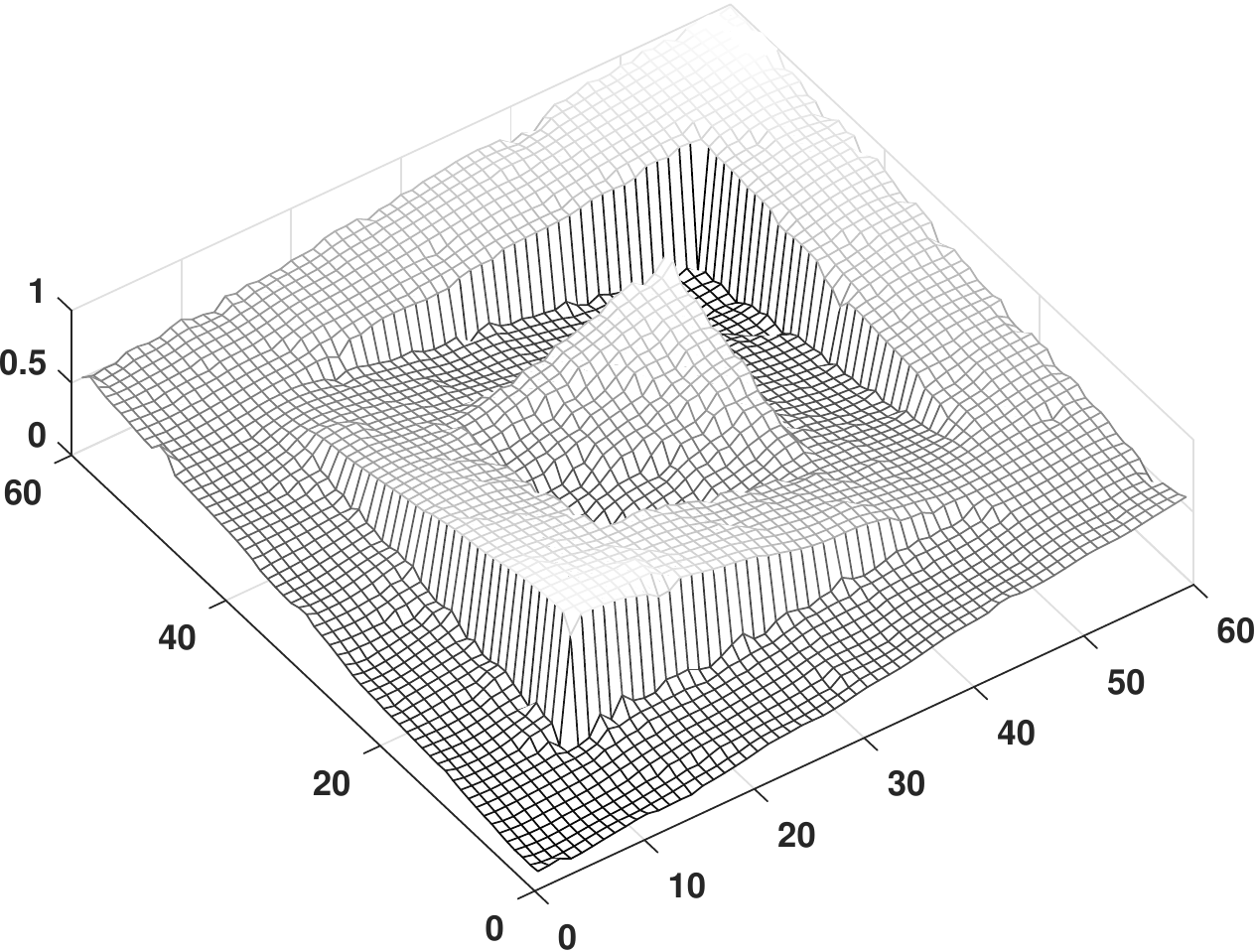}}
	\caption{The image surfaces of the clean images, denoising images obtained by the Euler's elastica and our TAC-MC model.}
	\label{imagesurfaces}
\end{figure*}

\begin{figure*}[t]
      \centering
      \subfigure{
			\includegraphics[width=0.16\linewidth]{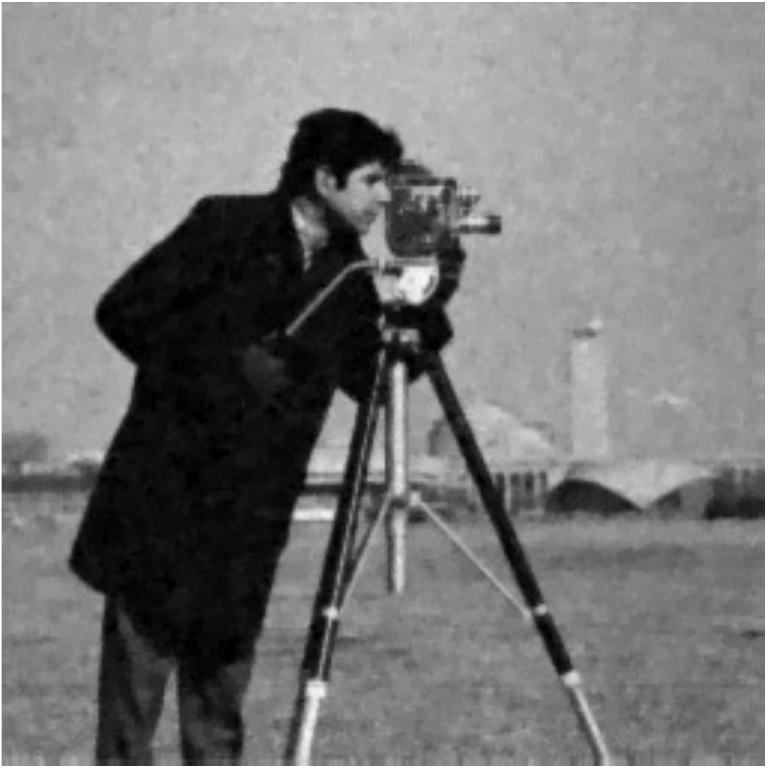}}\hspace{-1ex}
      \subfigure{
            \includegraphics[width=0.16\linewidth]{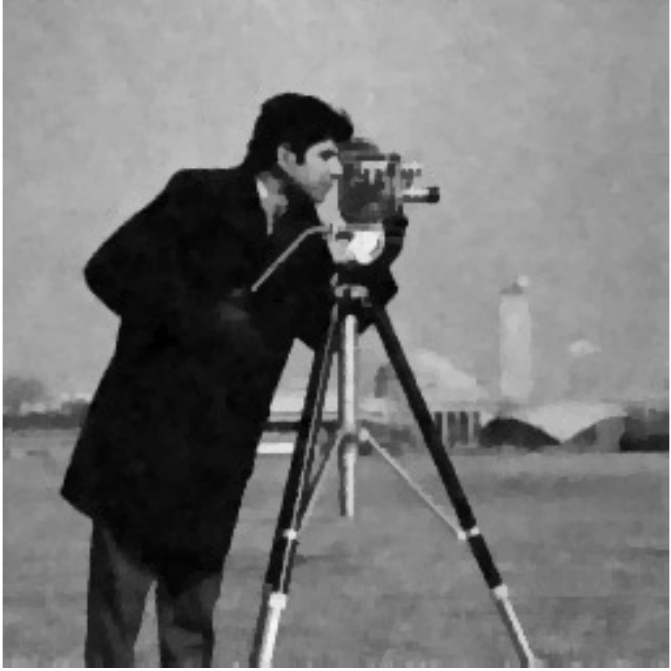}}\hspace{-1ex}
      \subfigure{
			\includegraphics[width=0.16\linewidth]{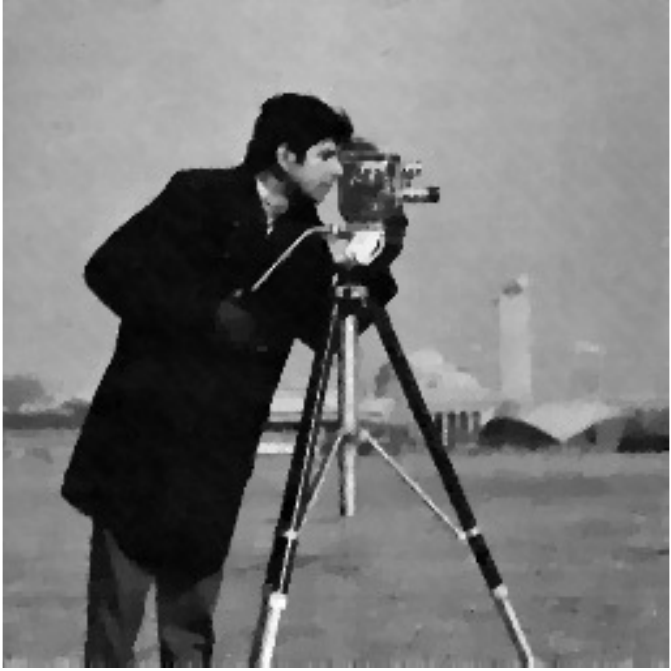}}\hspace{-1ex}
      \subfigure{
			\includegraphics[width=0.16\linewidth]{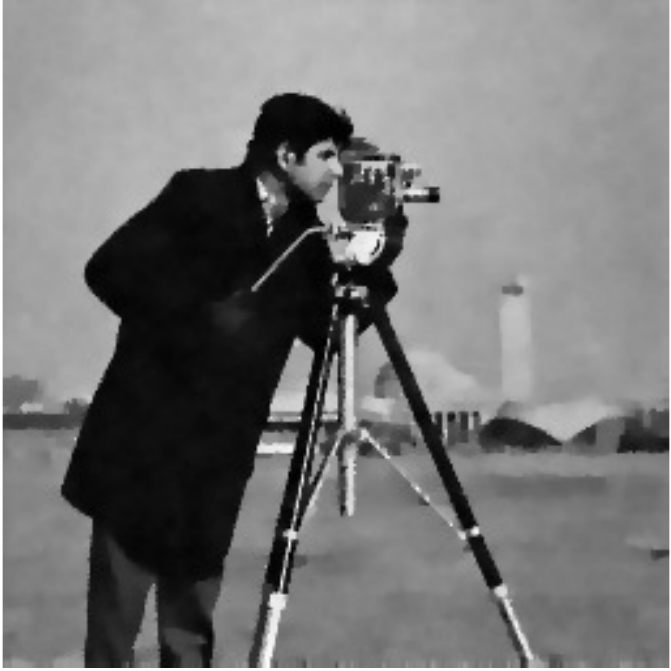}}\hspace{-1ex}
      \subfigure{
			\includegraphics[width=0.16\linewidth]{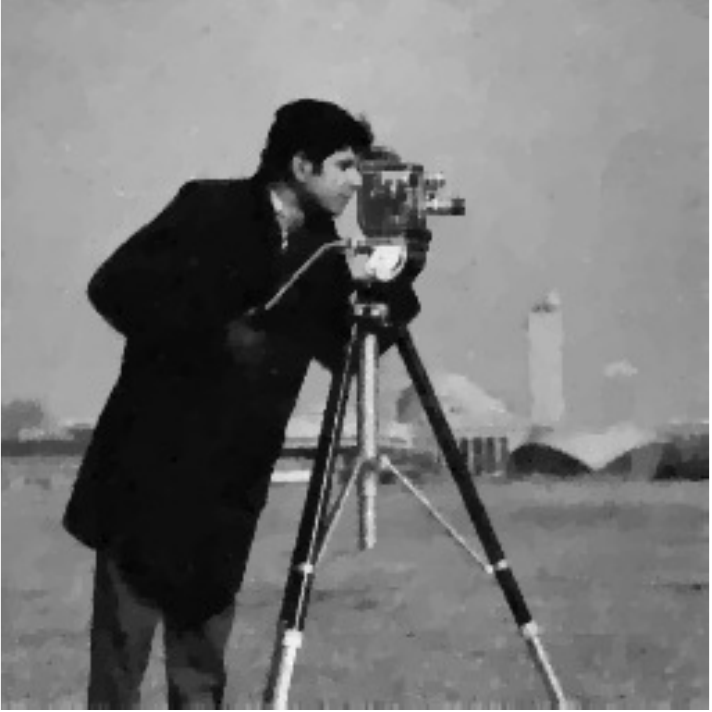}}\hspace{-1ex}
      \subfigure{
			\includegraphics[width=0.16\linewidth]{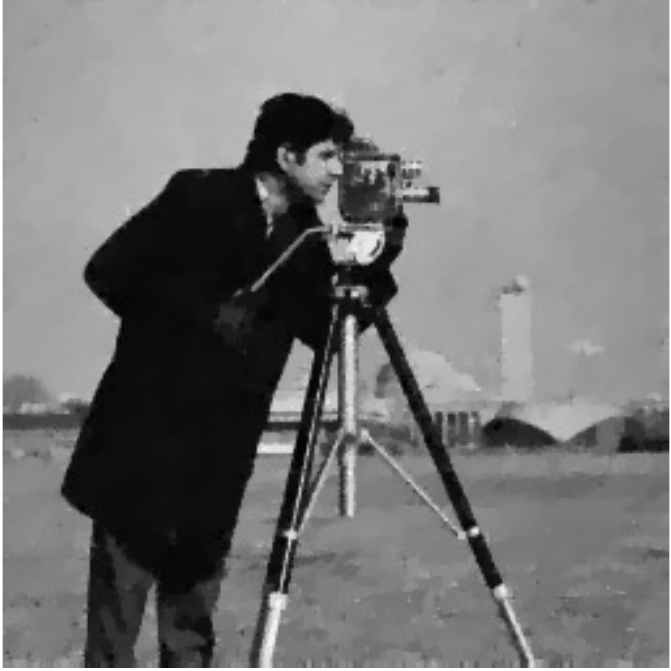}}
      \setcounter{subfigure}{0}
      \subfigure[TV]{
            \includegraphics[width=0.16\linewidth]{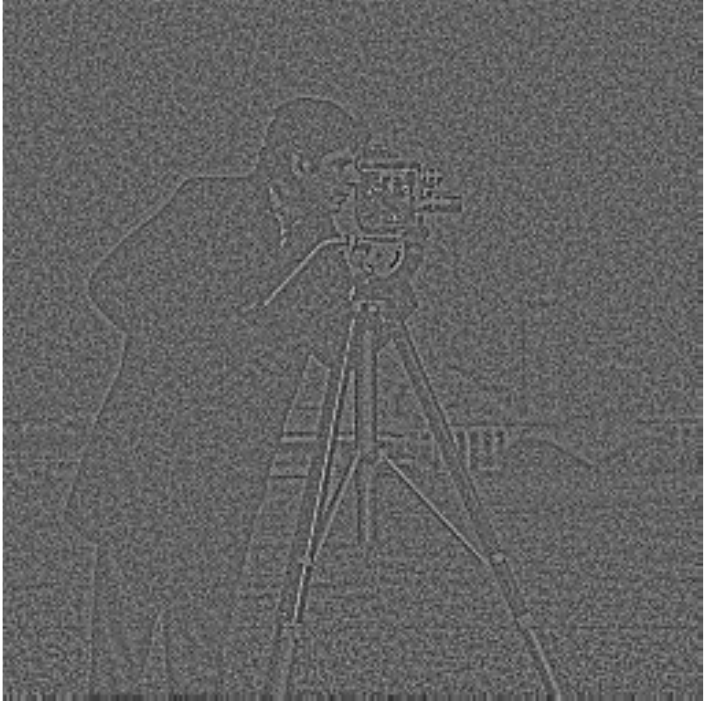}}\hspace{-1ex}
      \subfigure[Euler]{
            \includegraphics[width=0.16\linewidth]{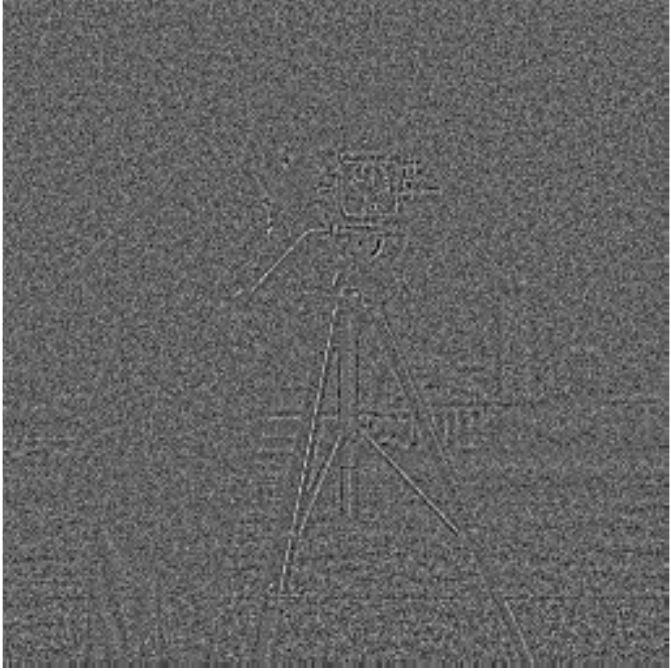}}\hspace{-1ex}
      \subfigure[TGV]{
            \includegraphics[width=0.16\linewidth]{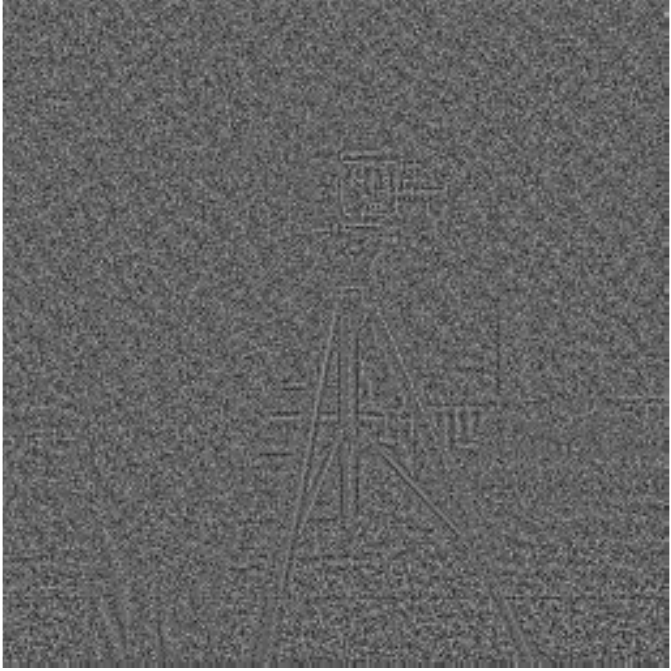}}\hspace{-1ex}
      \subfigure[MEC]{
            \includegraphics[width=0.16\linewidth]{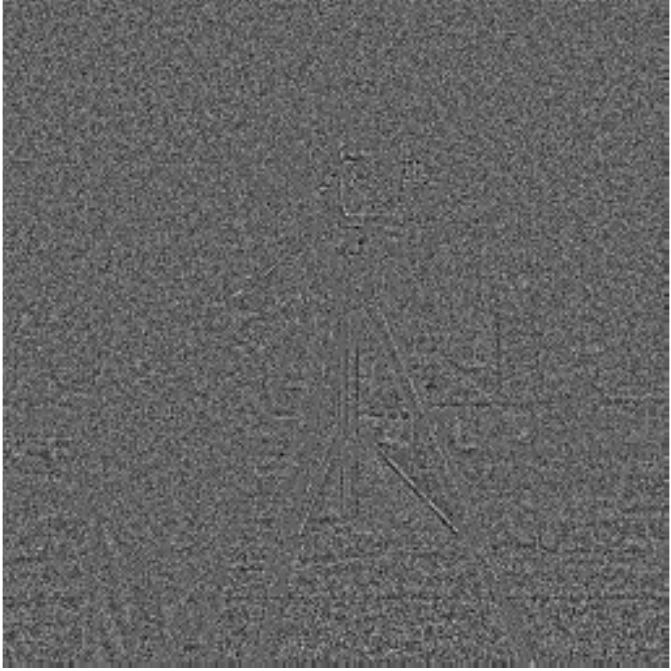}}\hspace{-1ex}
      \subfigure[TAC-MC]{
			\includegraphics[width=0.16\linewidth]{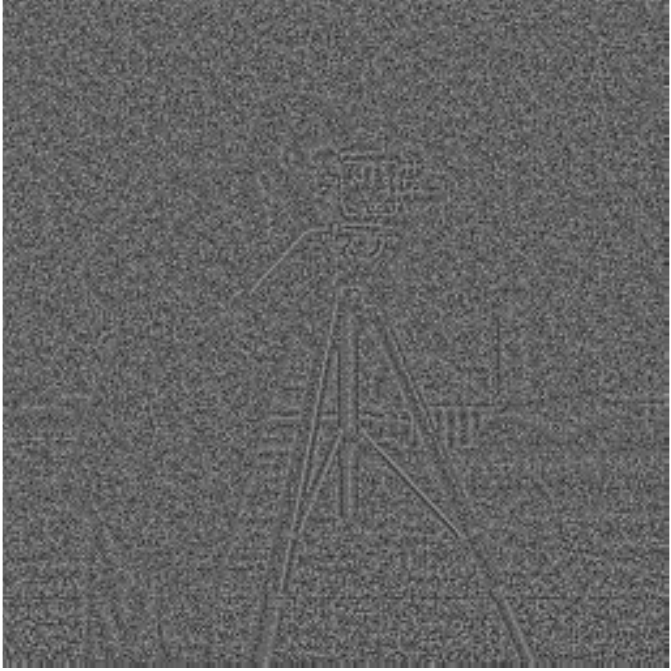}}\hspace{-1ex}
      \subfigure[TAC-GC]{
            \includegraphics[width=0.16\linewidth]{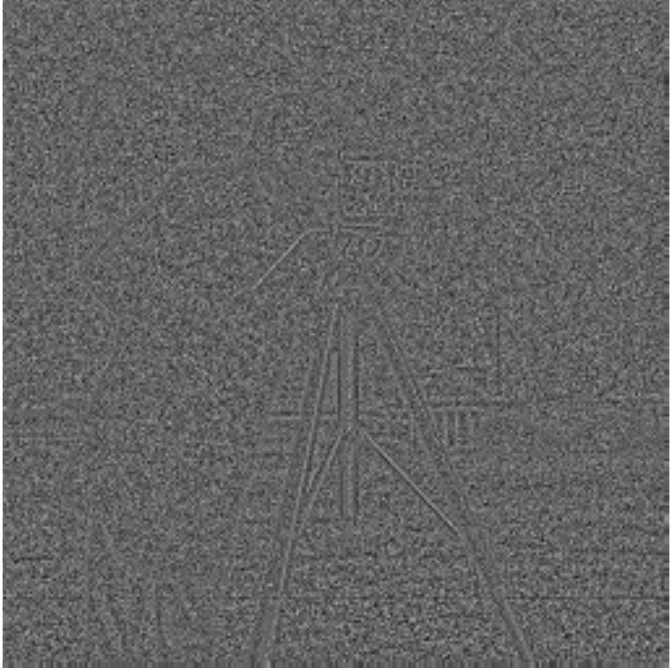}}
	\caption{The denoising results of `Cameraman' (top) and the corresponding residual images (bottom) of the comparative methods.}
	\label{DenoisingCameraman}
\end{figure*}

\begin{figure*}[t]
      \centering
      \subfigure[Relative error in $u^k$]{
			\includegraphics[width=0.23\linewidth]{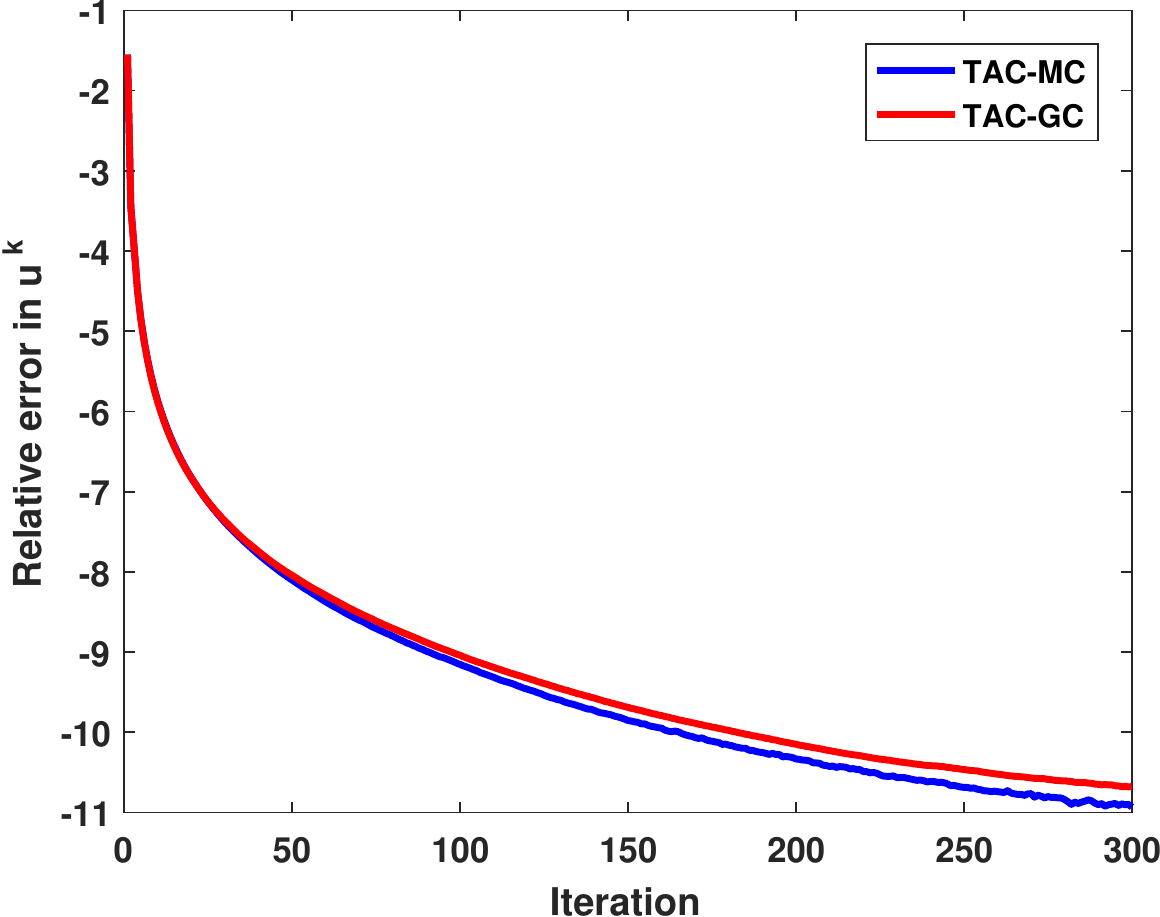}}
      \subfigure[Relative residual]{
			\includegraphics[width=0.23\linewidth]{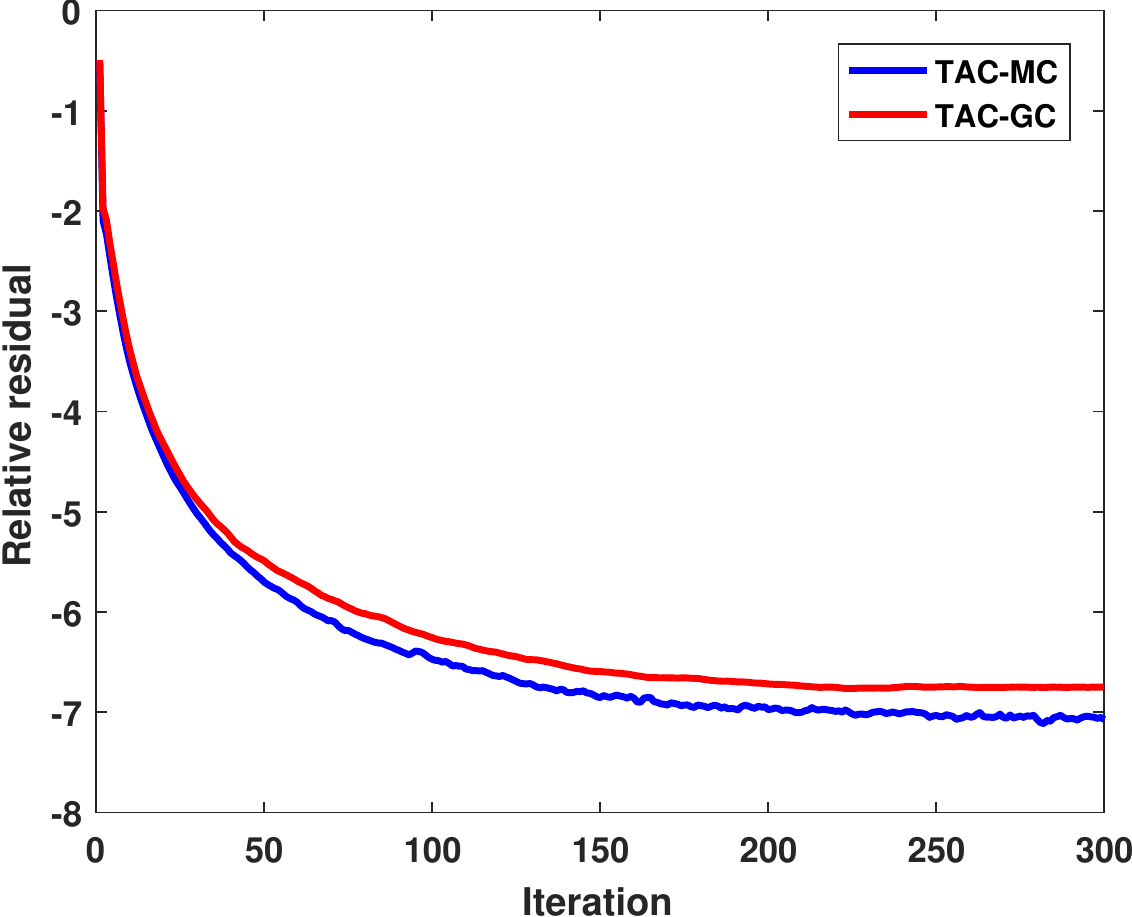}}
      \subfigure[Relative error in $\bm\Lambda^k$]{
            \includegraphics[width=0.23\linewidth]{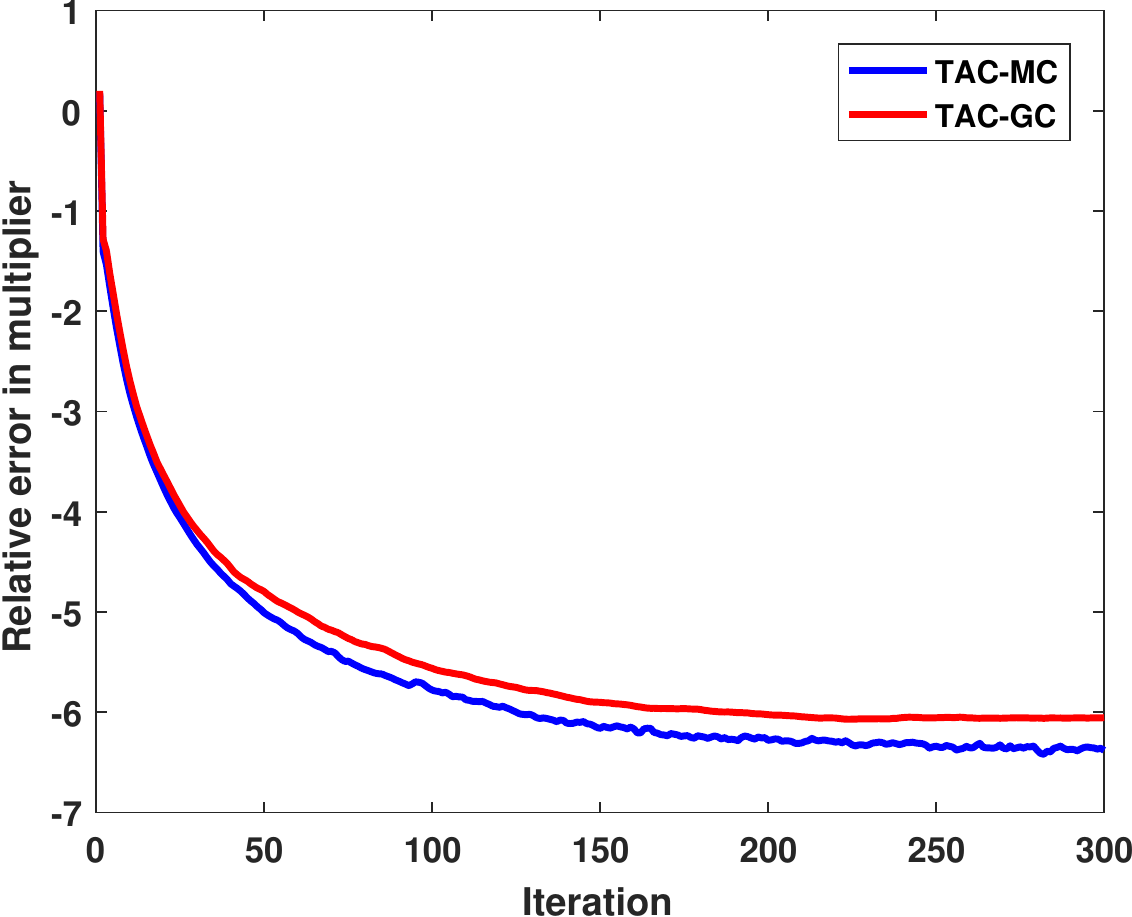}}
      \subfigure[Numerical energy]{
			\includegraphics[width=0.23\linewidth]{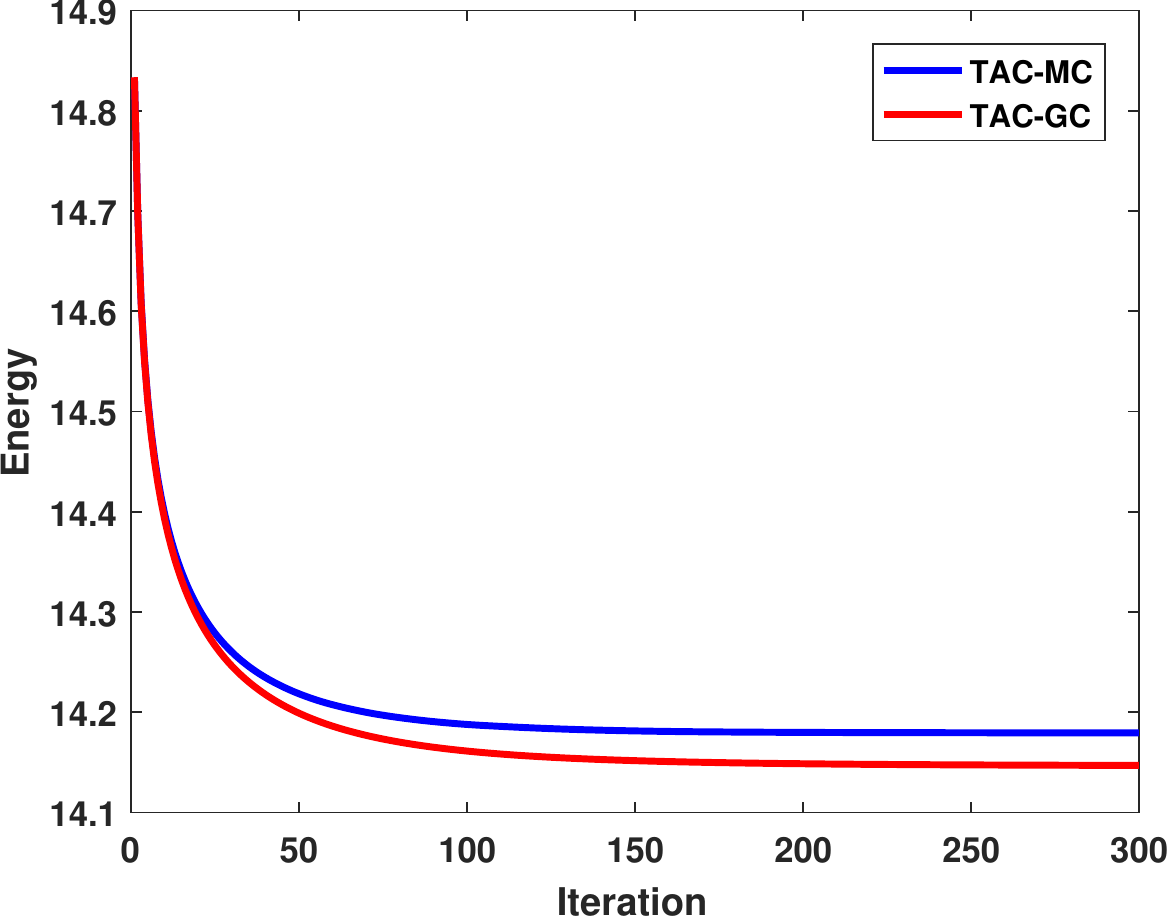}}
	\caption{Evaluations of `Cameraman' by the proposed methods. From left to right: Relative error in $u^k$, relative residual, relative error in multiplier and numerical energy, respectively.}
	\label{EvaluationsCameraman}
\end{figure*}

\begin{table*}[t]
\caption{The PSNR and SSIM of Gaussian noise removal for different methods.}
\label{denoise2}
\tiny
\begin{center}
\begin{tabular}{c|c|c|c|c|c|c|c}
\hline\hline
Images & Noisy images & TV & Euler's elastica & TGV & MEC & TAC-MC & TAC-GC \\
\hline\hline
Cameraman($256\times256$) & 22.45 & 27.29 & 27.93 & 28.22 & 28.38 & 28.65 & \bf{28.92} \\
PSNR/SSIM & 0.4087 & 0.7905 & 0.8187 & 0.8161 & 0.8203 & 0.8295 & \bf{0.8355} \\
\hline
Triangle($214\times254$) & 22.71 & 32.04 & 34.85 & 35.52 & \bf{36.65} & 36.02 & 36.35 \\
PSNR/SSIM & 0.2666 & 0.9247 & 0.9588 & 0.9504 & 0.9654 & 0.9705 & \bf{0.9749} \\
\hline
Lena($256\times256$) & 22.34 & 27.25 & 28.10 & 28.04 & 28.18 & 28.30 & \bf{28.54}   \\
PSNR/SSIM & 0.4855 & 0.8139 & 0.8335 & 0.8307 & 0.8352 & 0.8378 & \bf{0.8422}  \\
\hline
Plane($512\times512$) & 22.12 & 29.48 & 30.22 & 30.16 & 30.35 & 30.58 & \bf{30.85}   \\
PSNR/SSIM & 0.3555 & 0.8505 & 0.8681 & 0.8548 & 0.8719 & 0.8726 & \bf{0.8763}  \\
\hline\hline
\end{tabular}
\end{center}
\end{table*}

\begin{table*}[!ht]
\caption{The CPU time comparison of Gaussian noise removal for comparative methods.}
\label{efficiency}
\tiny
\begin{center}
\begin{tabular}{c|c|c|c|c|c|c|c|c}
\hline\hline
Images & \multicolumn{2}{|c|}{Cameraman($256\times256$)} & \multicolumn{2}{|c|}{Triangle($214\times254$)} & \multicolumn{2}{|c|}{Lena($256\times256$)} & \multicolumn{2}{|c}{Plane($512\times512$)}\\
\hline
Methods & Time & Iterations & Time & Iterations & Time & Iterations & Time & Iterations\\
\hline\hline
TV & {\color{red}6.17} & 300 & {\color{red}6.61} & 300 & {\color{red}5.86} & 300 & {\color{red}31.15} & 275 \\
\hline
Euler's elastica & 21.83 & 300 & 18.61 & 288 & 21.65 & 300  & 137.72 & 296 \\
\hline
TGV & 22.89 & 300 & 21.18 & 300 & 22.85 & 300 & 115.95 & 300 \\
\hline
MEC & 43.61 & 300 & 40.81 & 300 & 44.14 & 300 & 248.61 & 300 \\
\hline
TAC-MC & \bf{15.94} & \bf{232} & \bf{15.87} & \bf{252} & \bf{13.89} & \bf{201} & \bf{55.25} & \bf{144} \\
\hline
TAC-GC & {\color{blue}16.80} & {\color{blue}251} & {\color{blue}16.02} & {\color{blue}260} & {\color{blue}14.78} & {\color{blue}220} & {\color{blue}60.71} & {\color{blue}162} \\
\hline\hline
\end{tabular}
\end{center}
\end{table*}

\subsection{Computational complexity}
In this subsection, we analyze the computational complexity of the Algorithm 3.1. It is apparent that the main computationally expensive components include the calculation of discrete MC or GC, the FFT, inverse FFT and shrinkage operators. Generally speaking, calculating the MC or GC on image surface costs $\mathcal{O}(m^2)$. The computational complexity of FFT, inverse FFT in $u$-subproblem is well-known as $\mathcal{O}[m^2\log(m^2)]$ at each iteration. The $\bm v$-subproblem with two components can be computed at the cost $\mathcal{O}(2m^2)$ using the shrinkage operator. Therefore, the total computational complexity of Algorithm 3.1 is $\mathcal{O}[2m^2\log(m^2)+3m^2]$. On the other hand, the augmented Lagrangian method (ALM) of the Euler elastica model in \cite{tai2011fast} has four subproblems, which are solved by the FFT, inverse FFT and shrinkage operators. Its total computational complexity can be expressed as $\mathcal{O}[6m^2\log(m^2)+4m^2]$ per iteration. In addition, the augmented Lagrangian method for mean curvature regularization model in \cite{zhu2013augmented} has five subproblems, whose total computational complexity can be denoted as $\mathcal{O}[6m^2\log(m^2)+8m^2]$ per iteration. It is obvious that our proposed algorithm has lower computational complexity per iteration compared to the other two curvature-based models.

\begin{figure*}[t]
      \centering
      \subfigure{
			\includegraphics[width=0.16\linewidth]{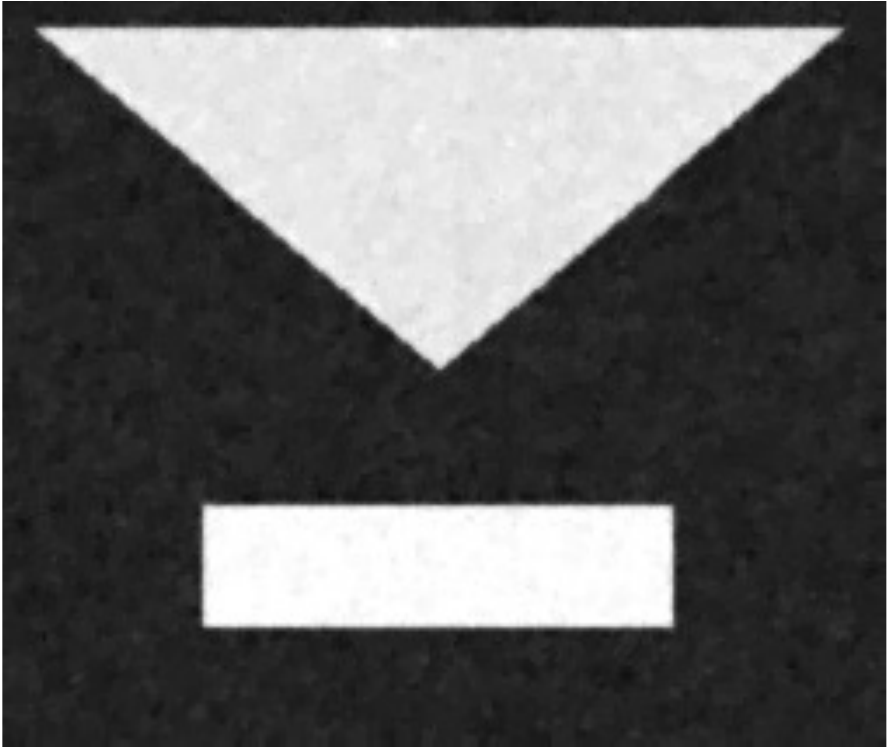}}\hspace{-1ex}
      \subfigure{
            \includegraphics[width=0.16\linewidth]{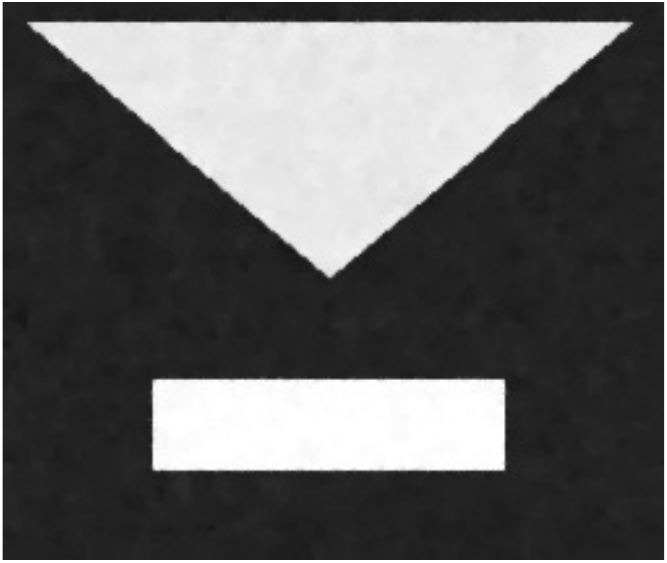}}\hspace{-1ex}
      \subfigure{
			\includegraphics[width=0.16\linewidth]{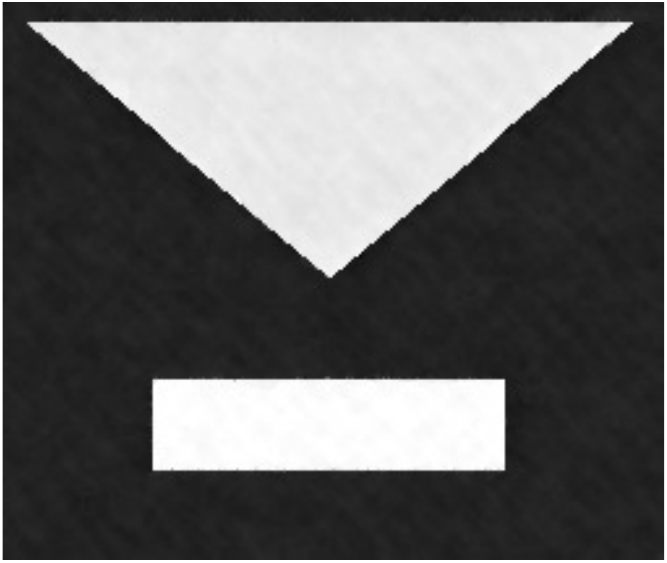}}\hspace{-1ex}
      \subfigure{
			\includegraphics[width=0.16\linewidth]{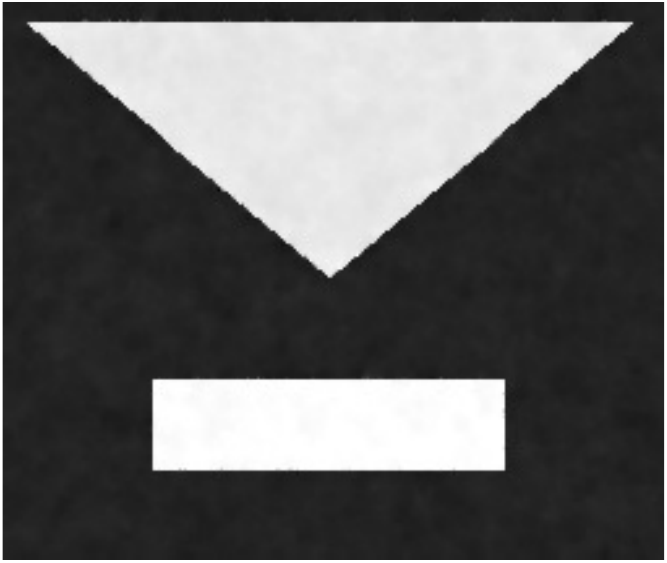}}\hspace{-1ex}
      \subfigure{
			\includegraphics[width=0.16\linewidth]{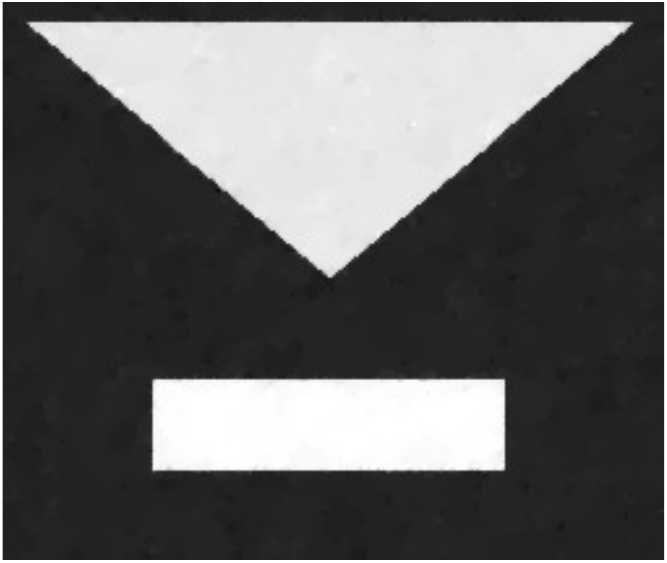}}\hspace{-1ex}
      \subfigure{
			\includegraphics[width=0.16\linewidth]{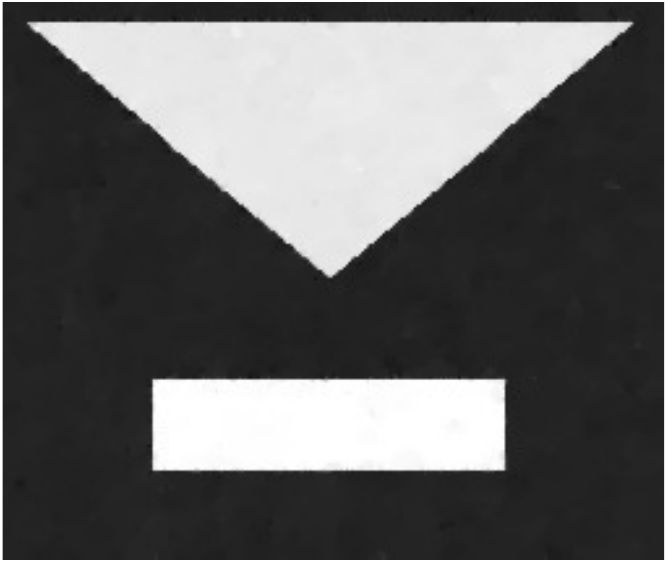}}
\setcounter{subfigure}{0}
      \subfigure[TV]{
            \includegraphics[width=0.16\linewidth]{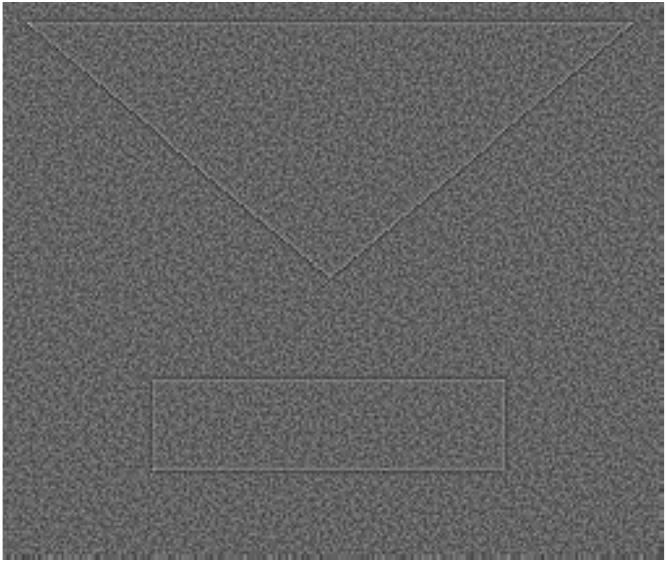}}\hspace{-1ex}
      \subfigure[Euler]{
            \includegraphics[width=0.16\linewidth]{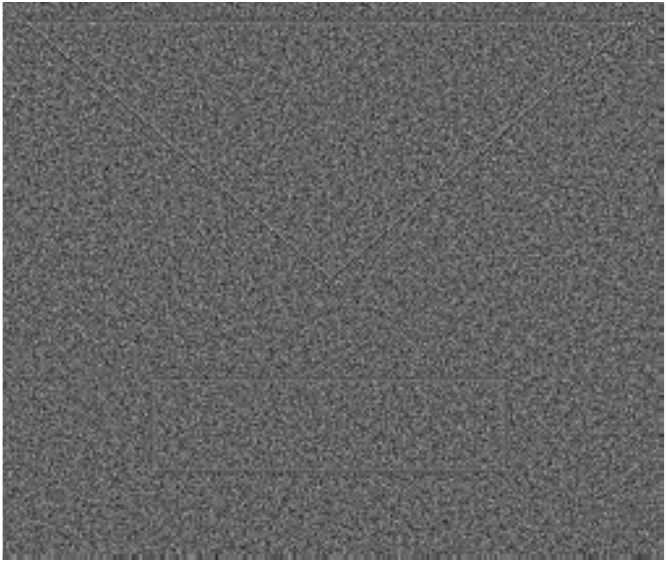}}\hspace{-1ex}
      \subfigure[TGV]{
            \includegraphics[width=0.16\linewidth]{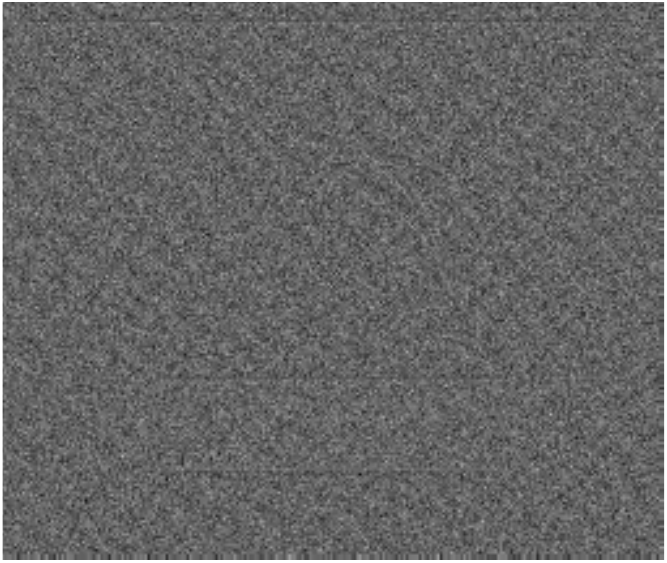}}\hspace{-1ex}
      \subfigure[MEC]{
            \includegraphics[width=0.16\linewidth]{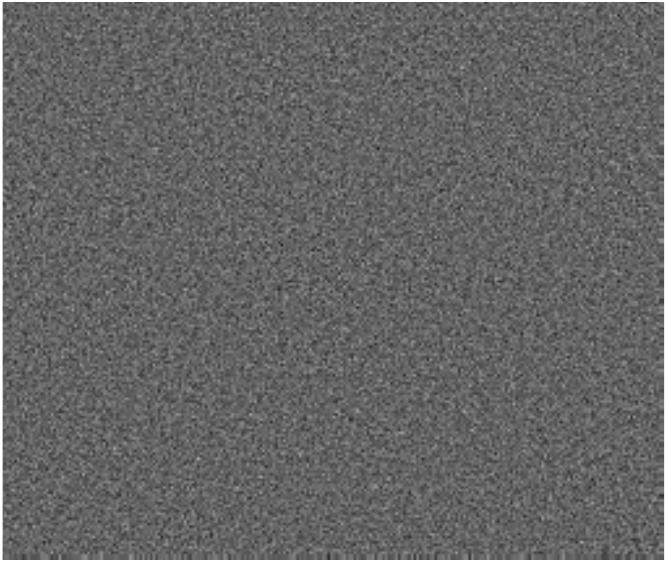}}\hspace{-1ex}
      \subfigure[TAC-MC]{
			\includegraphics[width=0.16\linewidth]{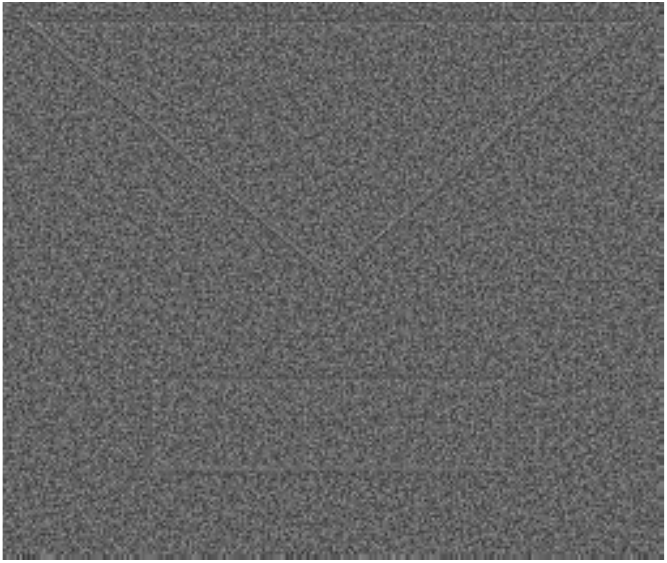}}\hspace{-1ex}
      \subfigure[TAC-GC]{
            \includegraphics[width=0.16\linewidth]{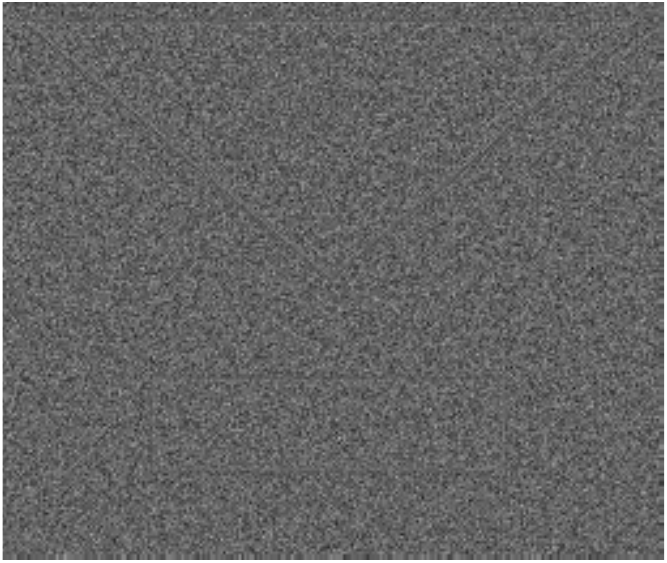}}
	\caption{The denoising results of `Triangle' (top) and the corresponding residual images (bottom) by the comparative methods.}
	\label{DenoisingTriangle}
\end{figure*}

\begin{figure*}[t]
      \centering
      \subfigure[Relative error in $u^k$]{
			\includegraphics[width=0.23\linewidth]{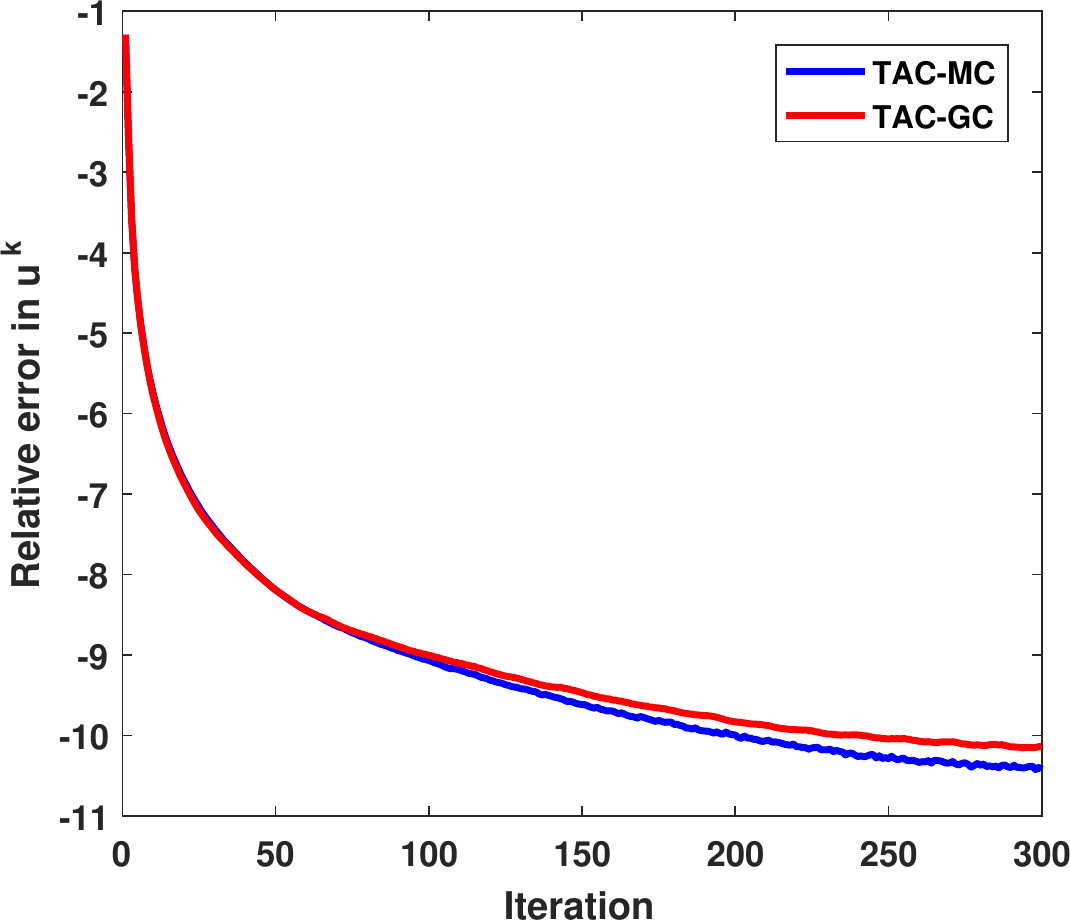}}
      \subfigure[Relative residual]{
			\includegraphics[width=0.23\linewidth]{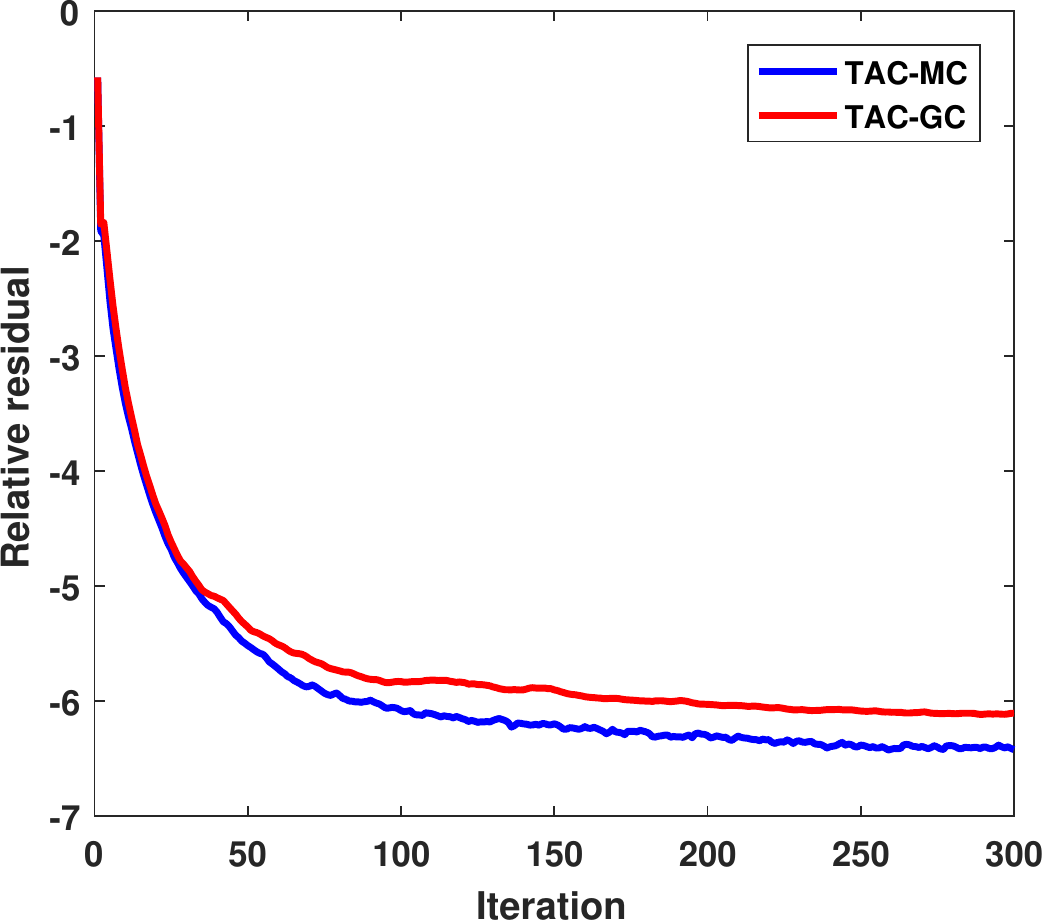}}
      \subfigure[Relative error in $\bm\Lambda^k$]{
            \includegraphics[width=0.23\linewidth]{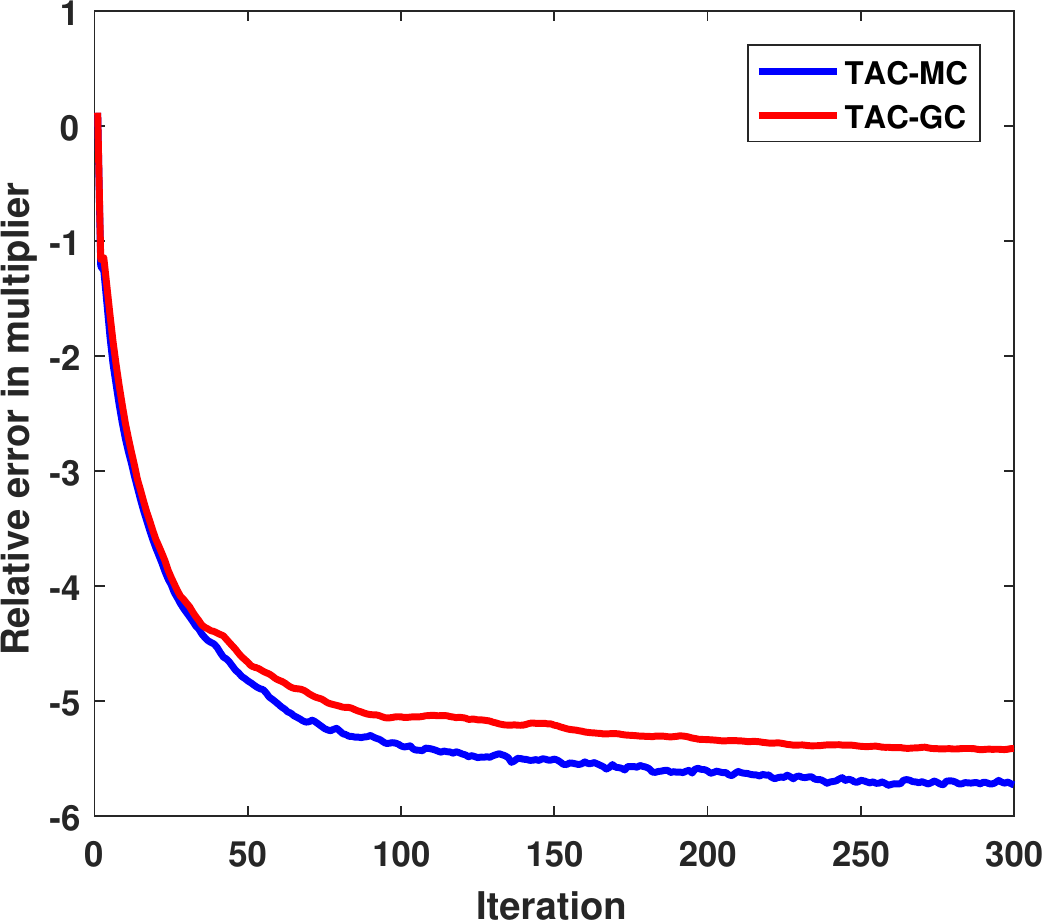}}
      \subfigure[Numerical energy]{
			\includegraphics[width=0.23\linewidth]{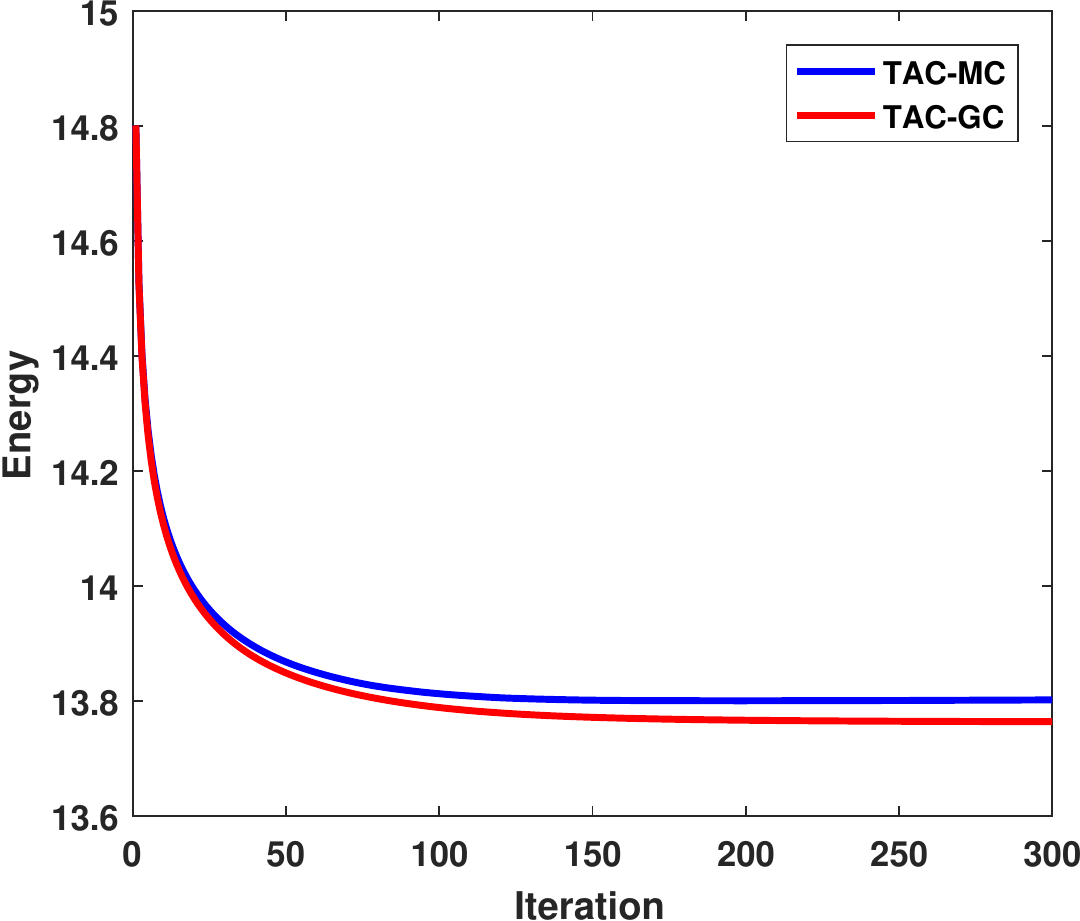}}
	\caption{Evaluations of `Triangle' by the proposed methods. From left to right: Relative error in $u^k$, relative residual, relative error in multiplier and numerical energy, respectively.}
	\label{EvaluationsTriangle}
\end{figure*}

\subsection{Gaussian denoising}
At the first place, we compare the proposed models relying on TAC, TSC and TRV, with the Euler's elastica model on image denoising problems. Three smooth images corrupted by Gaussian noise with zero mean and the standard deviation 10 are used in the evaluation. We fix the parameters $\lambda =0.09$, $\mu = 0.01$, $T_{max}=300$ and $\epsilon=4\times10^{-4}$ for our model, and set $\alpha = 0.1$ for the MC-based variational models (i.e., TAC-MC, TSC-MC and TRV-MC) and $\alpha = 5$ for the GC-based models (i.e., TAC-GC, TSC-GC and TRV-GC). On the other hand, we implement the ALM algorithm in \cite{tai2011fast} with the same parameters as the ones used in the original paper such that $\alpha=10$, $\eta=10^2$, $r_1=1$, $r_2=2\cdot10^2$, $r_4=5\cdot10^2$ and $\epsilon=10^{-2}, 1.3\cdot10^{-3}, 8\cdot10^{-3}$.

In Table \ref{denoise1}, we detail the comparison results in terms of PSNR and SSIM. It can be observed that our discrete curvature model always achieves higher PSNR and SSIM than the Euler's elastica method for all curvature function and curvature type combinations. Moreover, the TAC-MC model gives the best recovery results for all three images among the combinations. In FIG. \ref{smoothimages}, we display the restoration results obtained by the Euler's elastica model and our TAC-MC model, which clearly shows the our model can ideally preserve the structures such as edges and corners. The numerical MC of two of the test images are exhibited in FIG. \ref{MCimages}, which are calculated using the equations \eqref{ki}-\eqref{MCGC} on the clean images, restoration images of the Euler's elastica and our TAC-MC model. For fair comparison, we project all images into $[0,1]$ before calculating the numerical curvatures. It can be observed that the numerical MC is relative small in the homogeneous regions, and jumps across the edges, which give the evidence that MC regularity can preserve the edges and corners. We also find that the values of the MC obtained by our TAC-MC model are in the same range as the values calculated on the clean images, while the Euler's elastica model tends to underestimate the curvatures. Moreover, we display the image surfaces of the clean images and restored images of the Euler's elastica and TAC-MC model in FIG. \ref{imagesurfaces}, which clearly illustrate our discrete curvature regularizer can preserve the edges and sharp corners better than the Euler's elastica.

To further demonstrate the effectiveness and efficiency of the proposed curvature model,  we evaluate the performance on more natural images and compare with several state-of-the-art variational denoising methods including Total variation (TV) in \cite{yang2009efficient}, Euler's elastica (Euler) in \cite{tai2011fast}, the second-order total generalized variation (TGV) in \cite{bredies2010total} and mean curvature regularizer (MEC) in \cite{zhu2013augmented}. Four different test images (i.e., `Cameraman', `Lena', `Triangle' and `Plane') are degraded by the Gaussian noise with zero mean and the standard deviation 20. To setup the experimental comparison as fair as possible, the parameters of the comparative methods are selected as suggested in the corresponding papers, which are set as (a) TV: $r_1=10$ and $\lambda=15$; (b) Euler's elastica: $\alpha=10$, $r_1=1$, $r_2=2\cdot10^2$, $r_4=5\cdot10^2$ and $\eta=2\cdot10^2$; (c) TGV: $\alpha_0=1.5$, $\alpha_1=1.0$, $r_1=10$, $r_2=50$ and $\lambda=10$; (d) MEC: $r_1=40$, $r_2=40$, $r_3=10^5$, $r_4=1.5\cdot10^5$ and $\lambda=10^2$. The experience-dependent parameters in our model are set as $\lambda =0.07$, $\mu = 2$, $T_{max}=300$ and $\epsilon=3\times10^{-5}$. Similar to the previous experiment, we use $\alpha=0.5$ for TAC-MC model and $\alpha=5$ for TAC-GC model.

\begin{figure*}[t]
      \centering
      \subfigure{
			\includegraphics[width=0.16\linewidth]{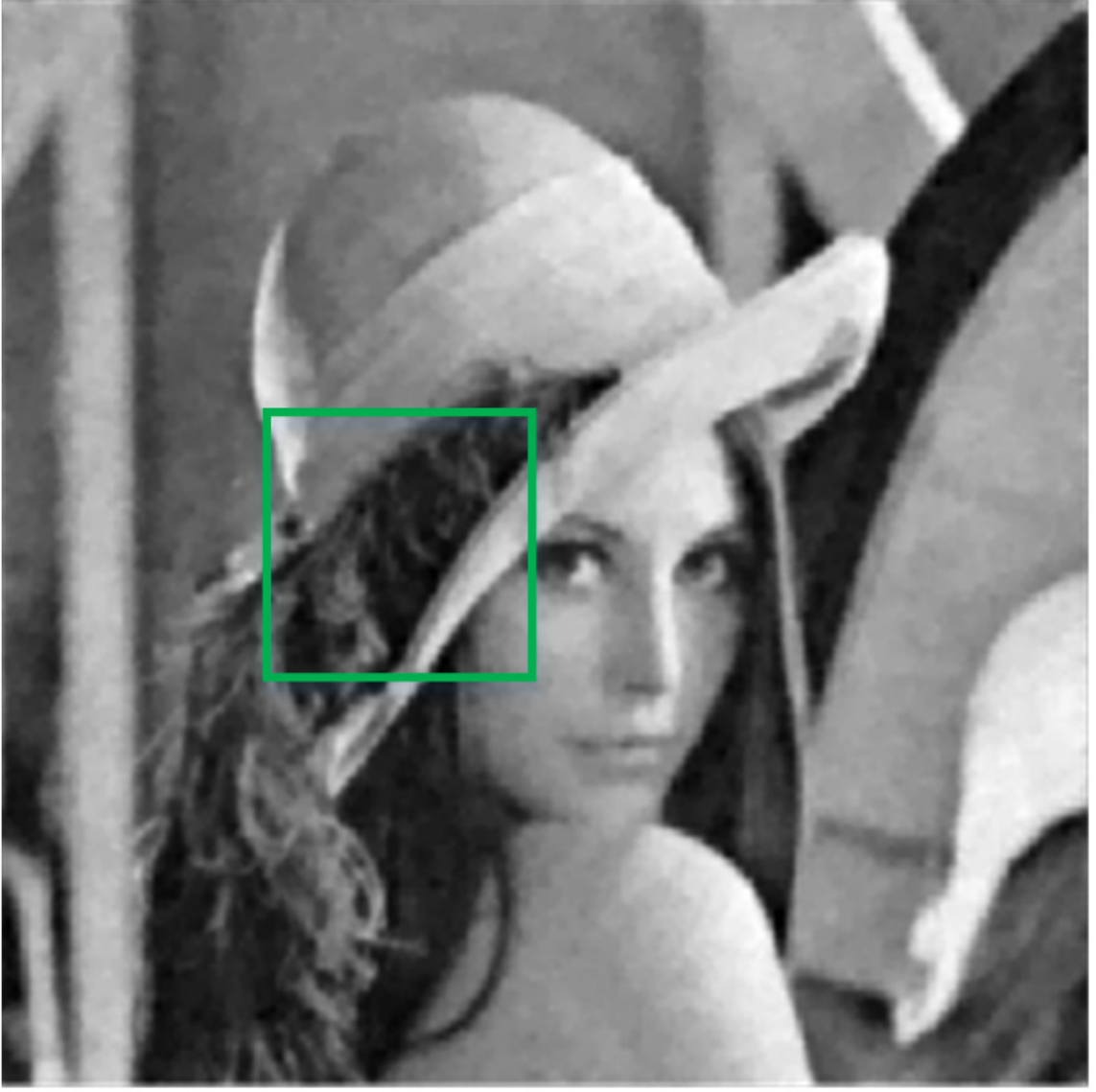}}\hspace{-1ex}
      \subfigure{
            \includegraphics[width=0.16\linewidth]{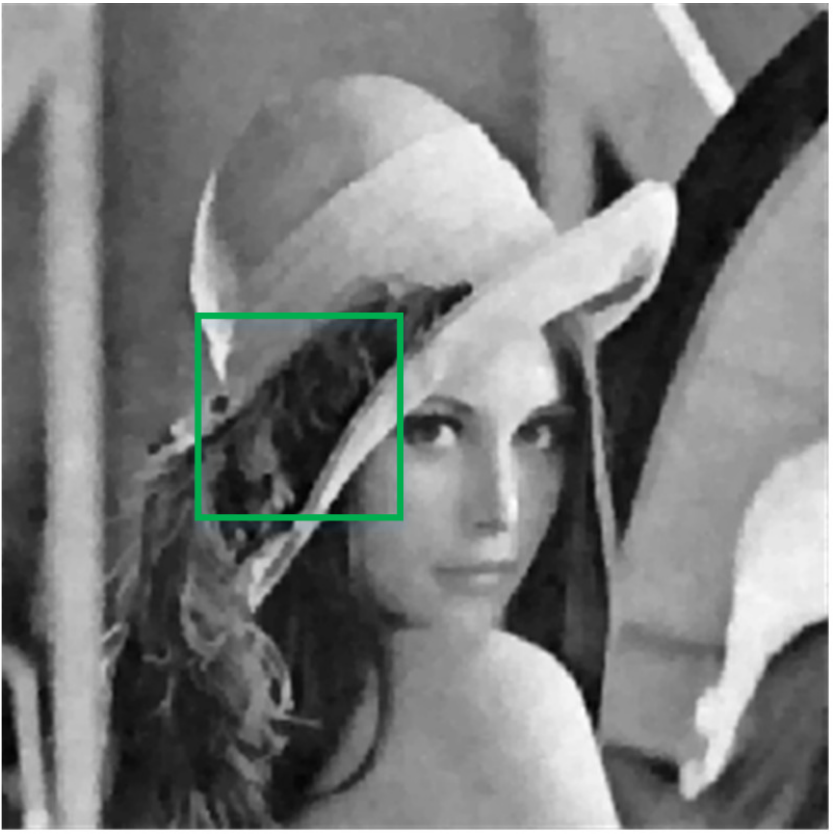}}\hspace{-1ex}
      \subfigure{
			\includegraphics[width=0.16\linewidth]{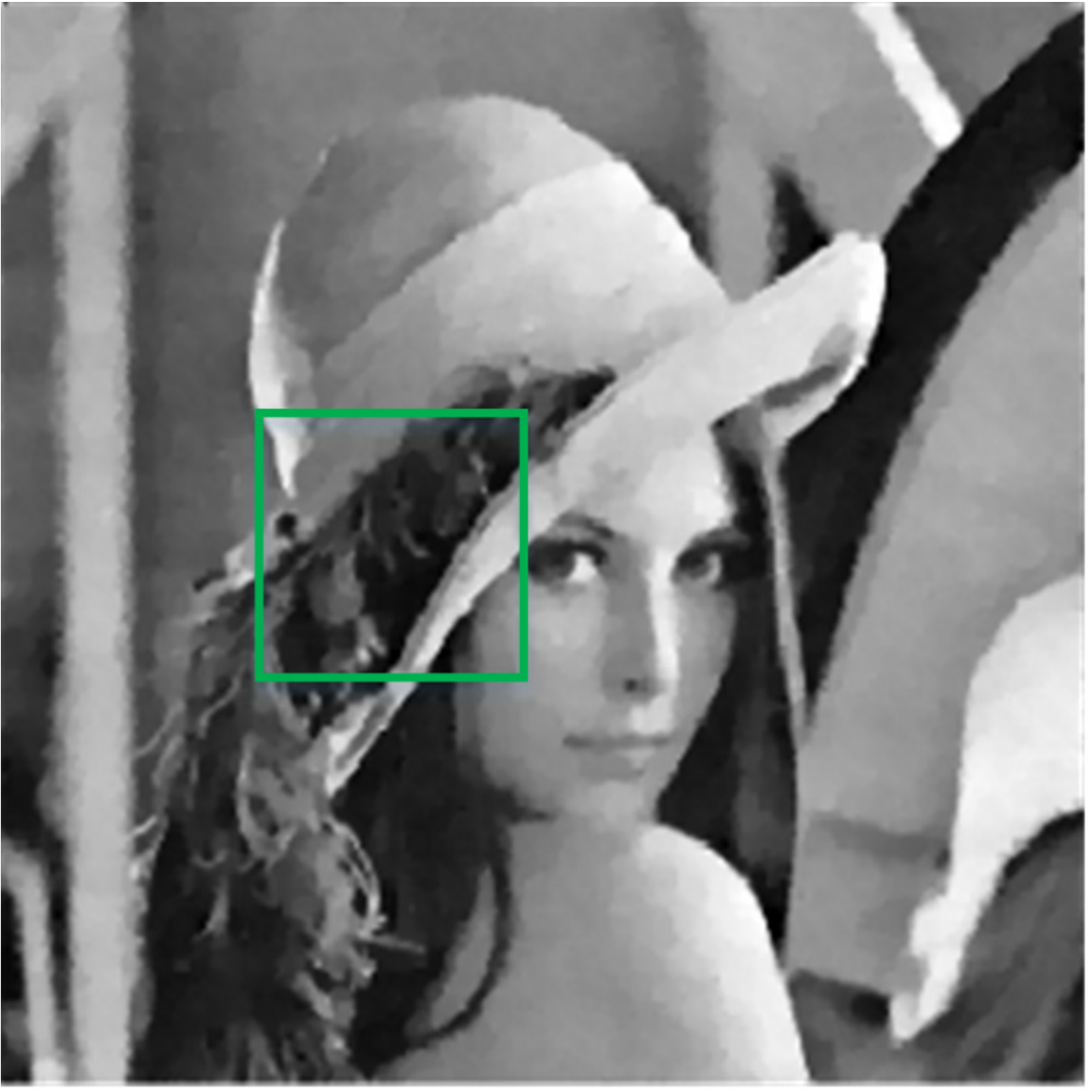}}\hspace{-1ex}
      \subfigure{
			\includegraphics[width=0.16\linewidth]{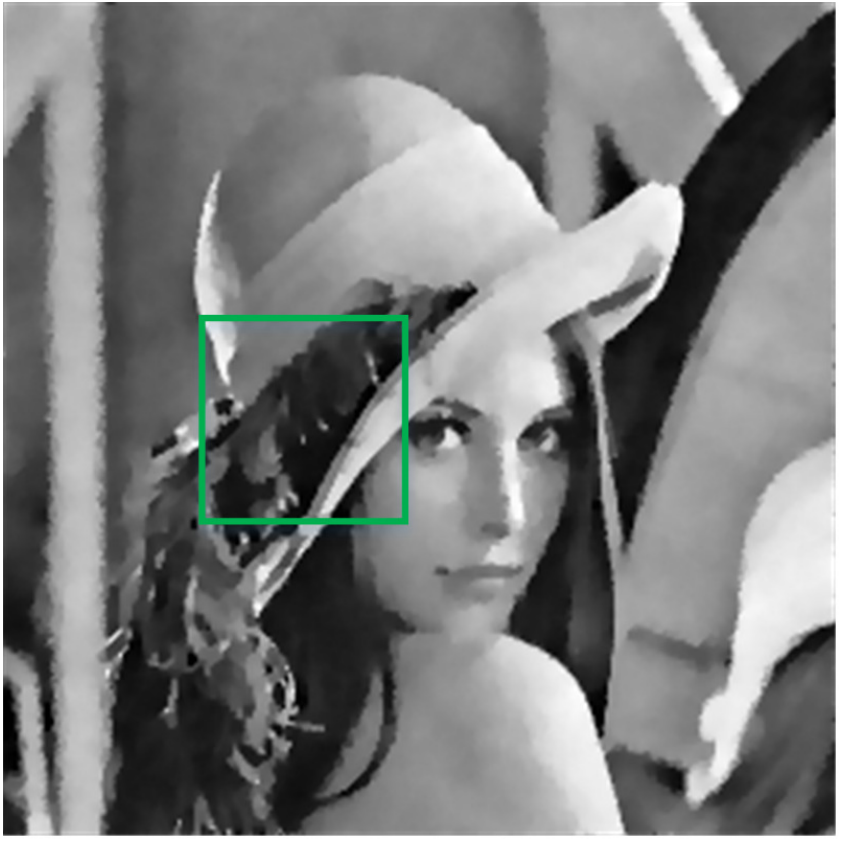}}\hspace{-1ex}
      \subfigure{
			\includegraphics[width=0.16\linewidth]{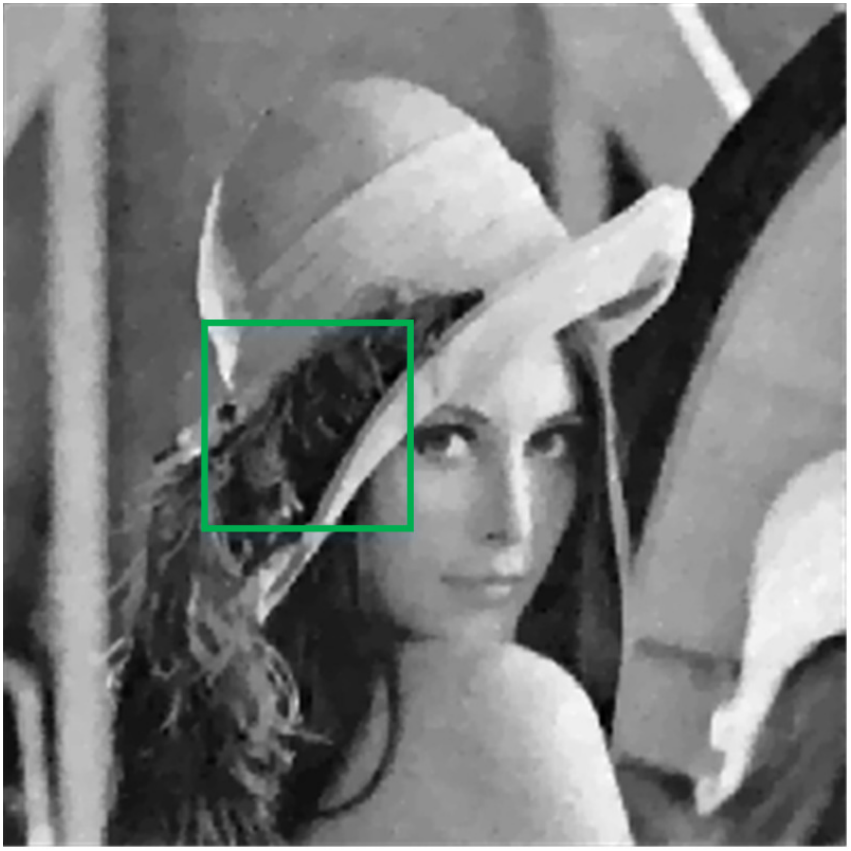}}\hspace{-1ex}
      \subfigure{
			\includegraphics[width=0.16\linewidth]{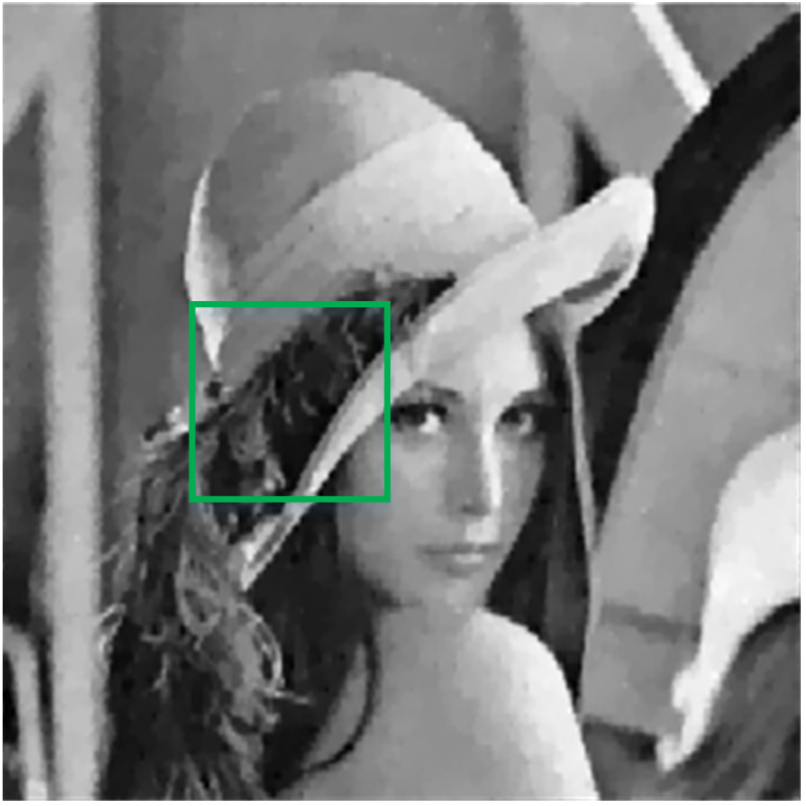}}
\setcounter{subfigure}{0}
      \subfigure[TV]{
            \includegraphics[width=0.16\linewidth]{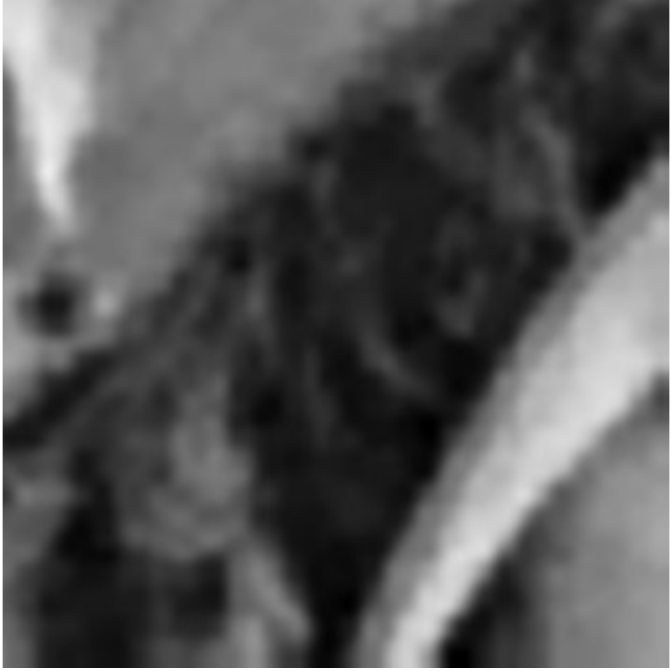}}\hspace{-1ex}
      \subfigure[Euler]{
            \includegraphics[width=0.16\linewidth]{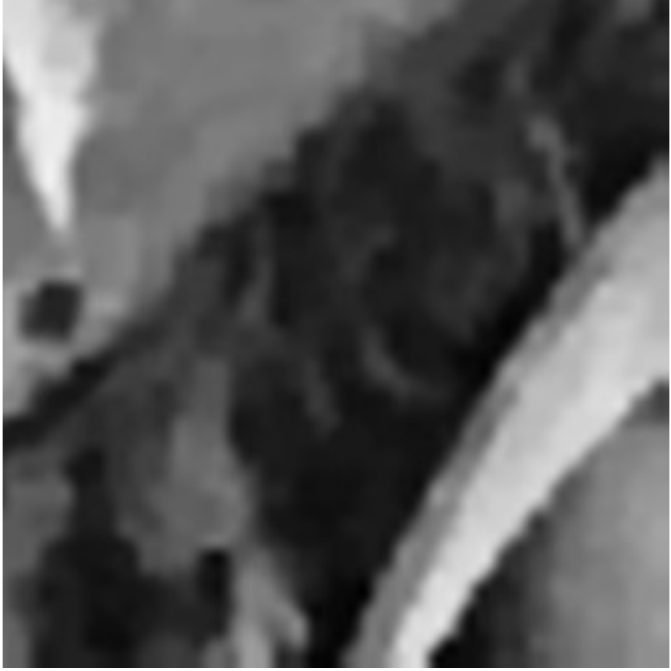}}\hspace{-1ex}
      \subfigure[TGV]{
            \includegraphics[width=0.16\linewidth]{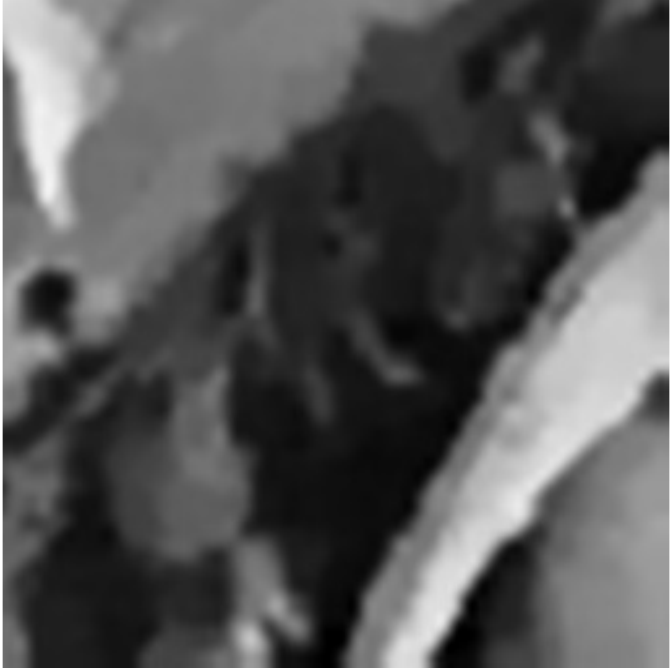}}\hspace{-1ex}
      \subfigure[MEC]{
            \includegraphics[width=0.16\linewidth]{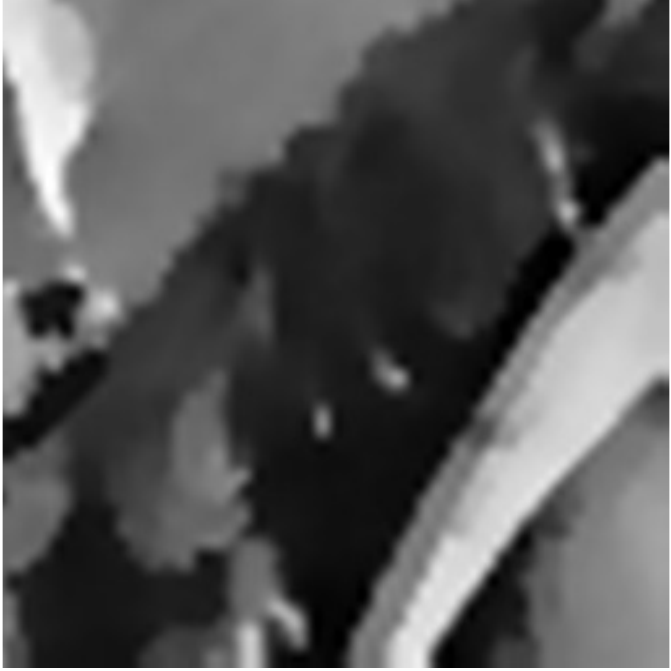}}\hspace{-1ex}
      \subfigure[TAC-MC]{
			\includegraphics[width=0.16\linewidth]{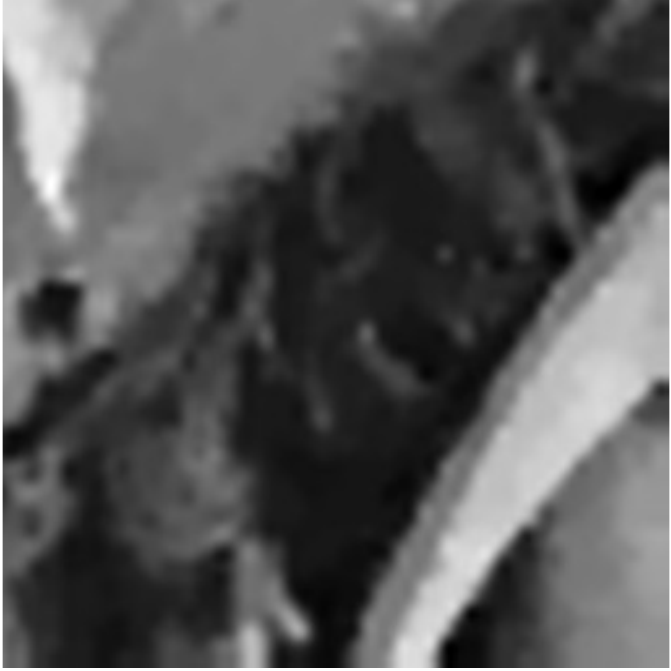}}\hspace{-1ex}
      \subfigure[TAC-GC]{
            \includegraphics[width=0.16\linewidth]{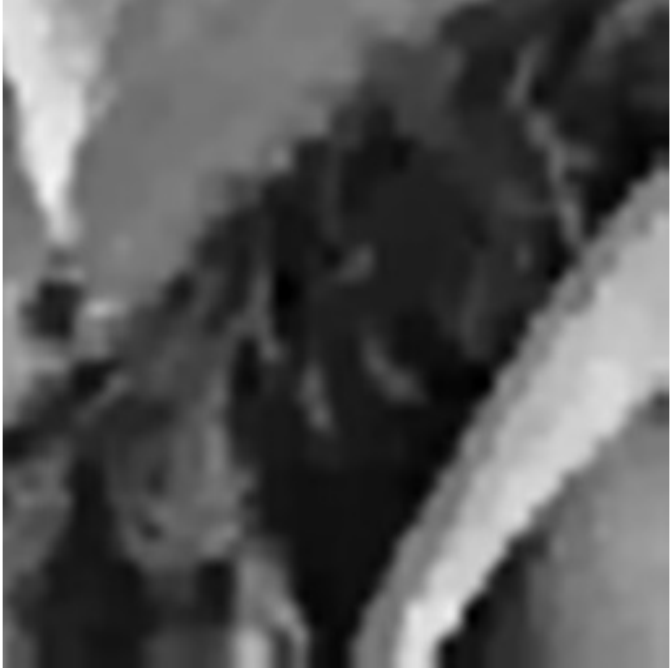}}
	\caption{The denoising results of `Lena' (top) and the corresponding local magnification views (bottom) by the comparative methods.}
	\label{DenoisingLena}
\end{figure*}

\begin{figure*}[t]
      \centering
      \subfigure{
			\includegraphics[width=0.16\linewidth]{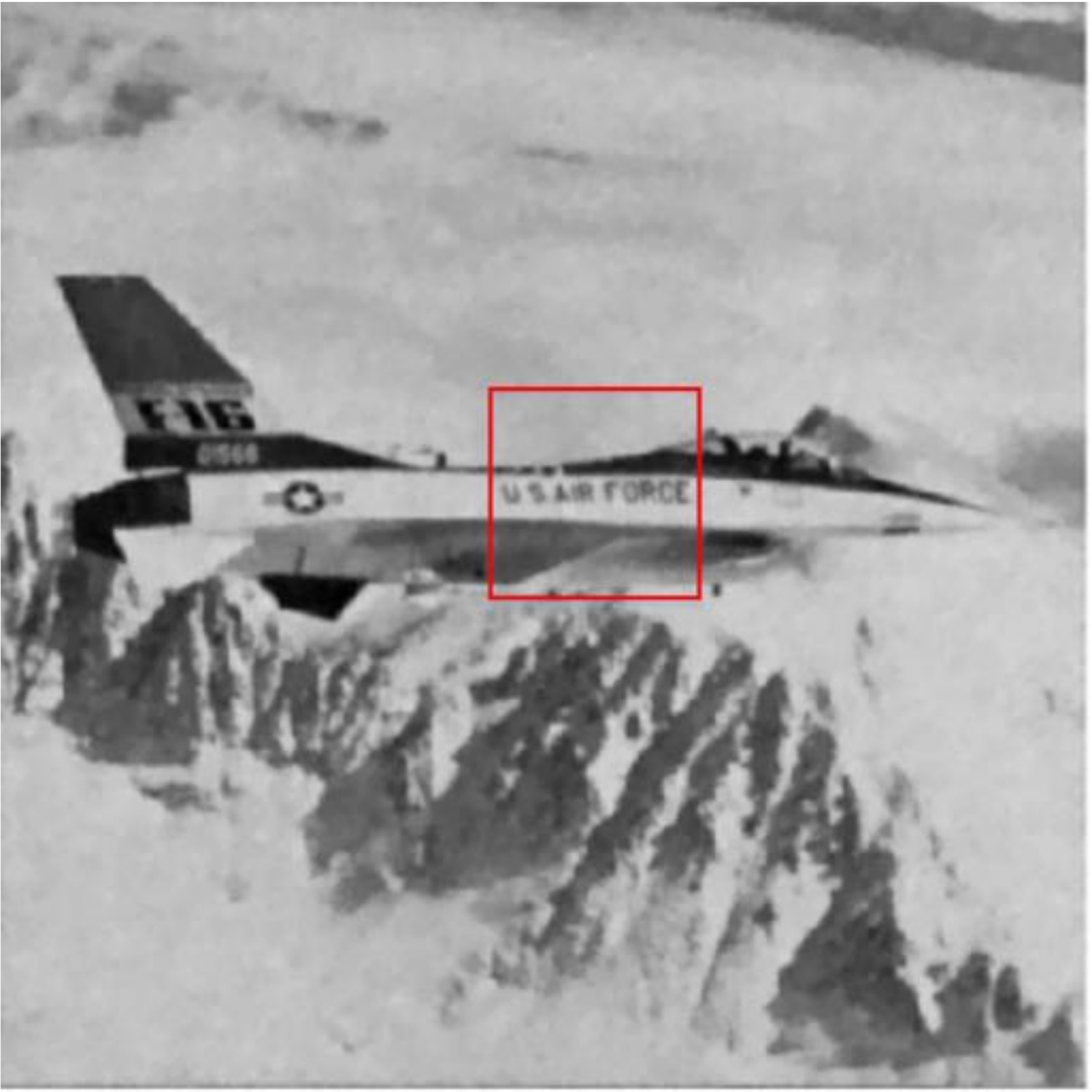}}\hspace{-1ex}
      \subfigure{
            \includegraphics[width=0.16\linewidth]{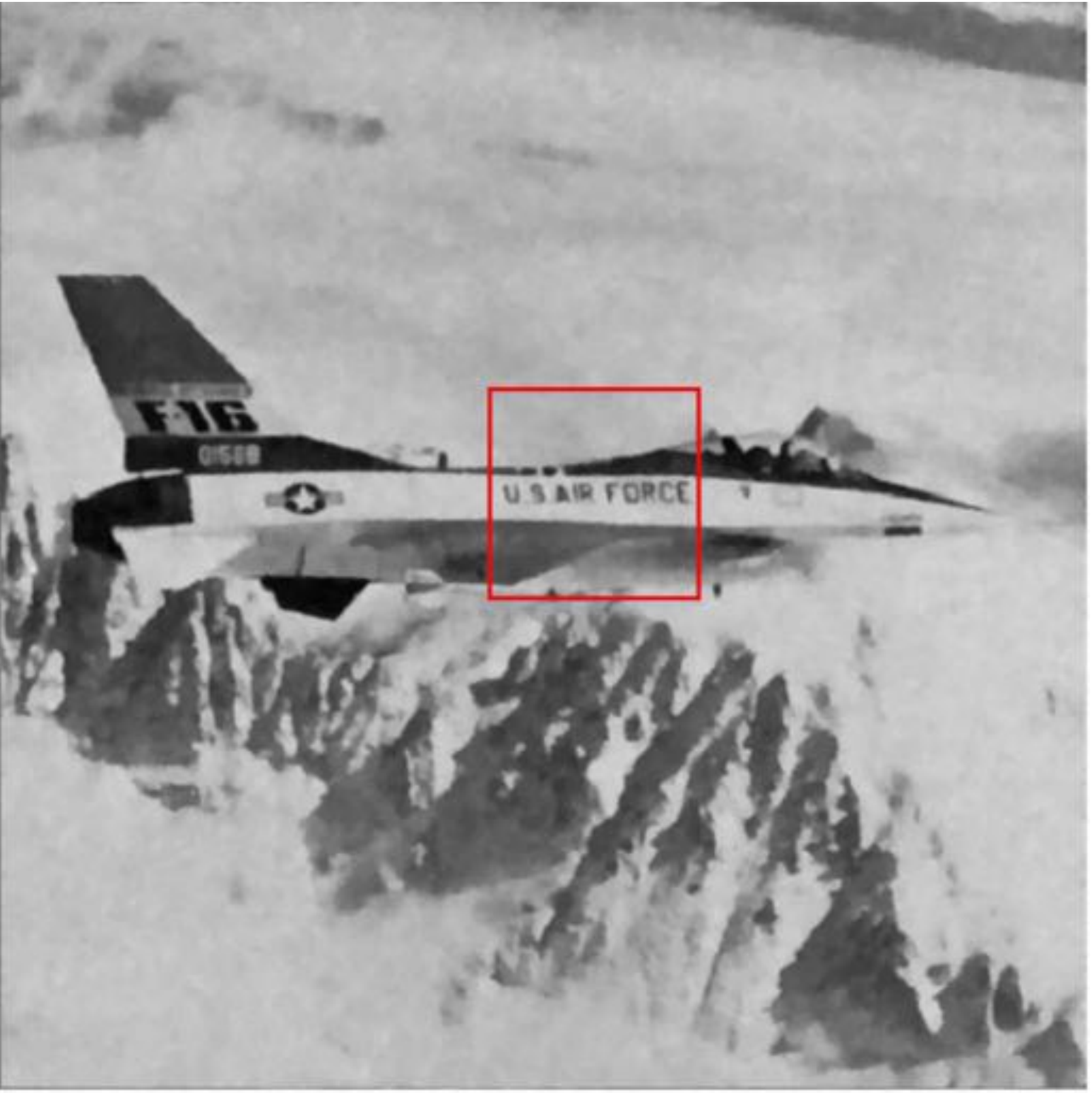}}\hspace{-1ex}
      \subfigure{
			\includegraphics[width=0.16\linewidth]{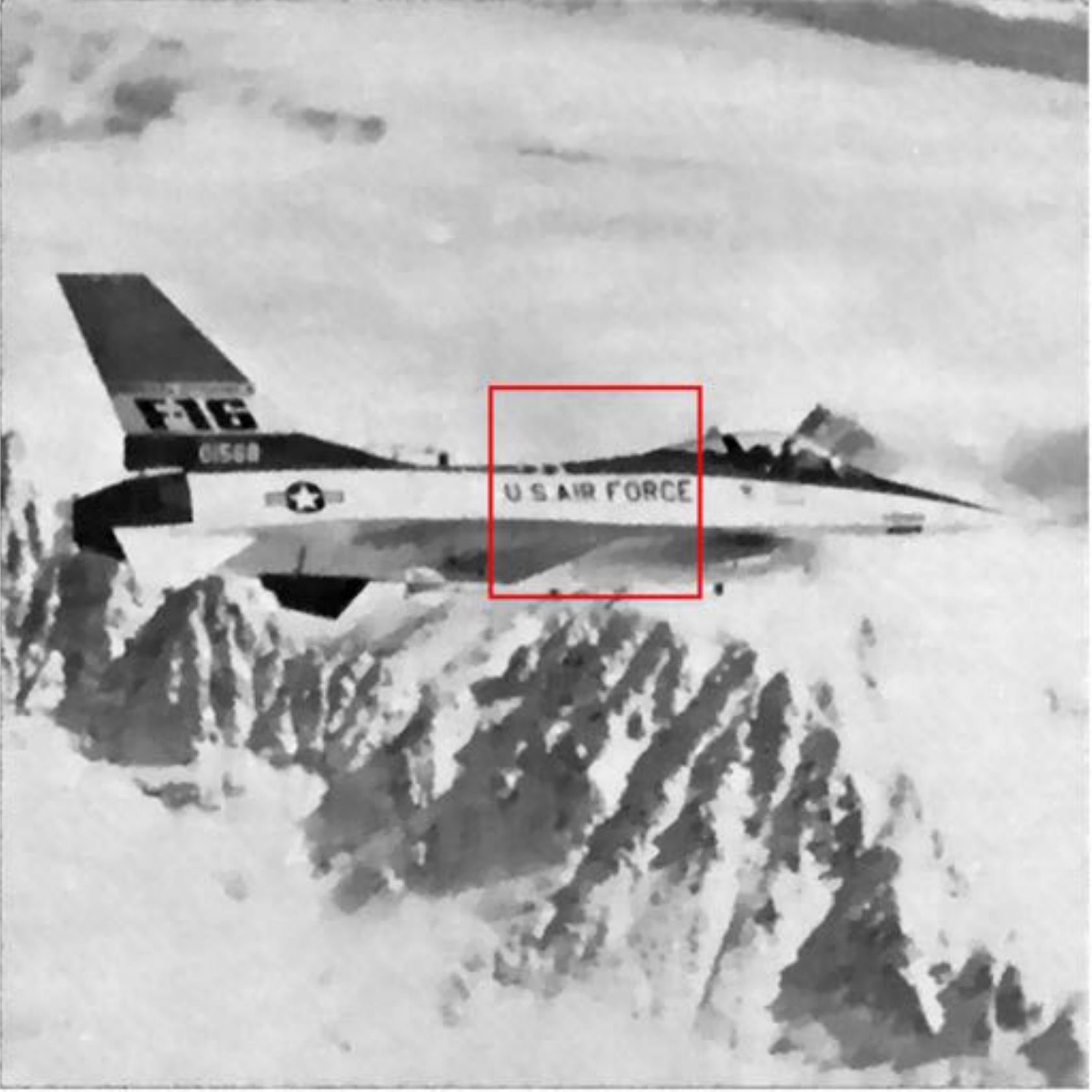}}\hspace{-1ex}
      \subfigure{
			\includegraphics[width=0.16\linewidth]{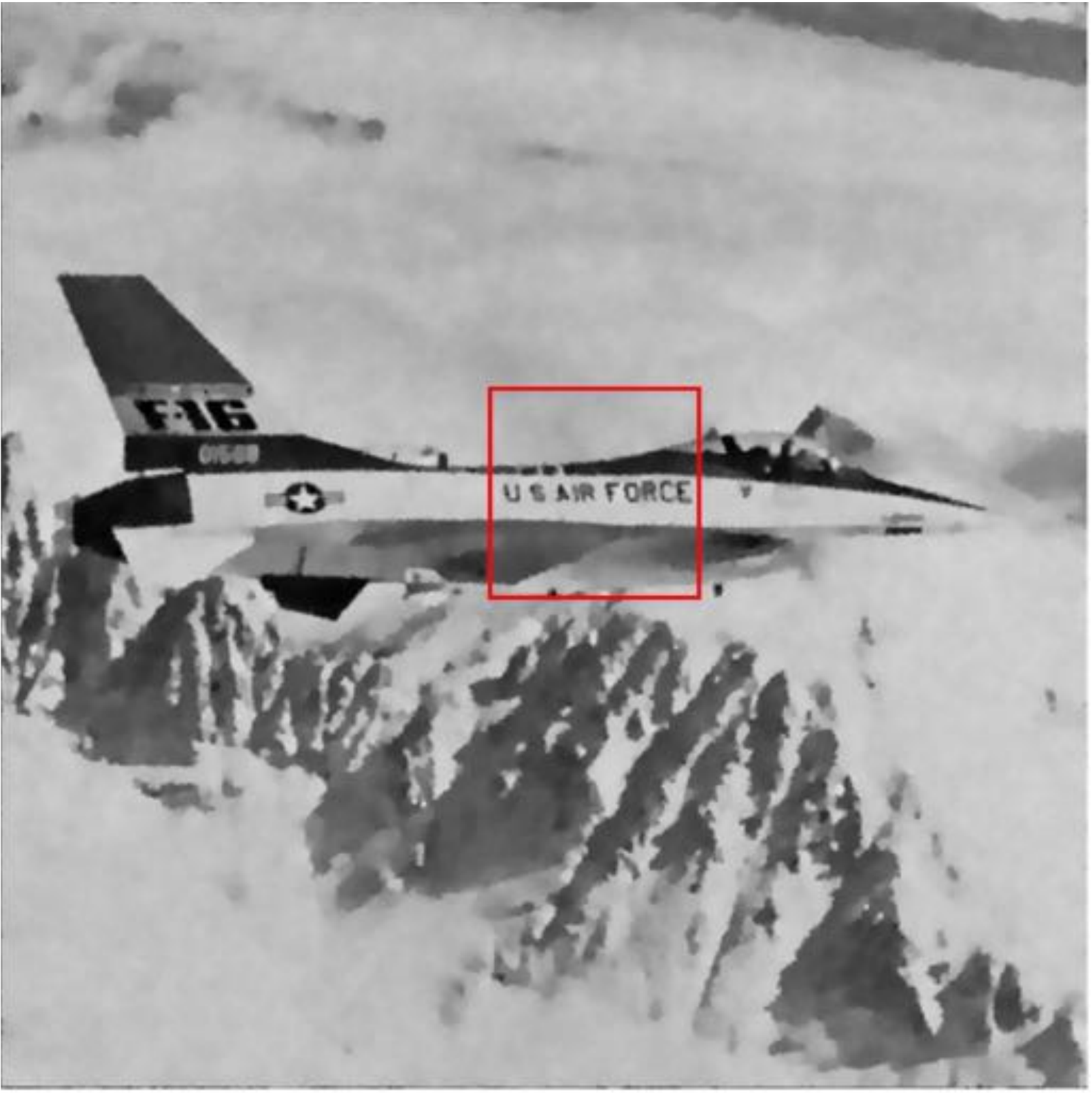}}\hspace{-1ex}
      \subfigure{
			\includegraphics[width=0.16\linewidth]{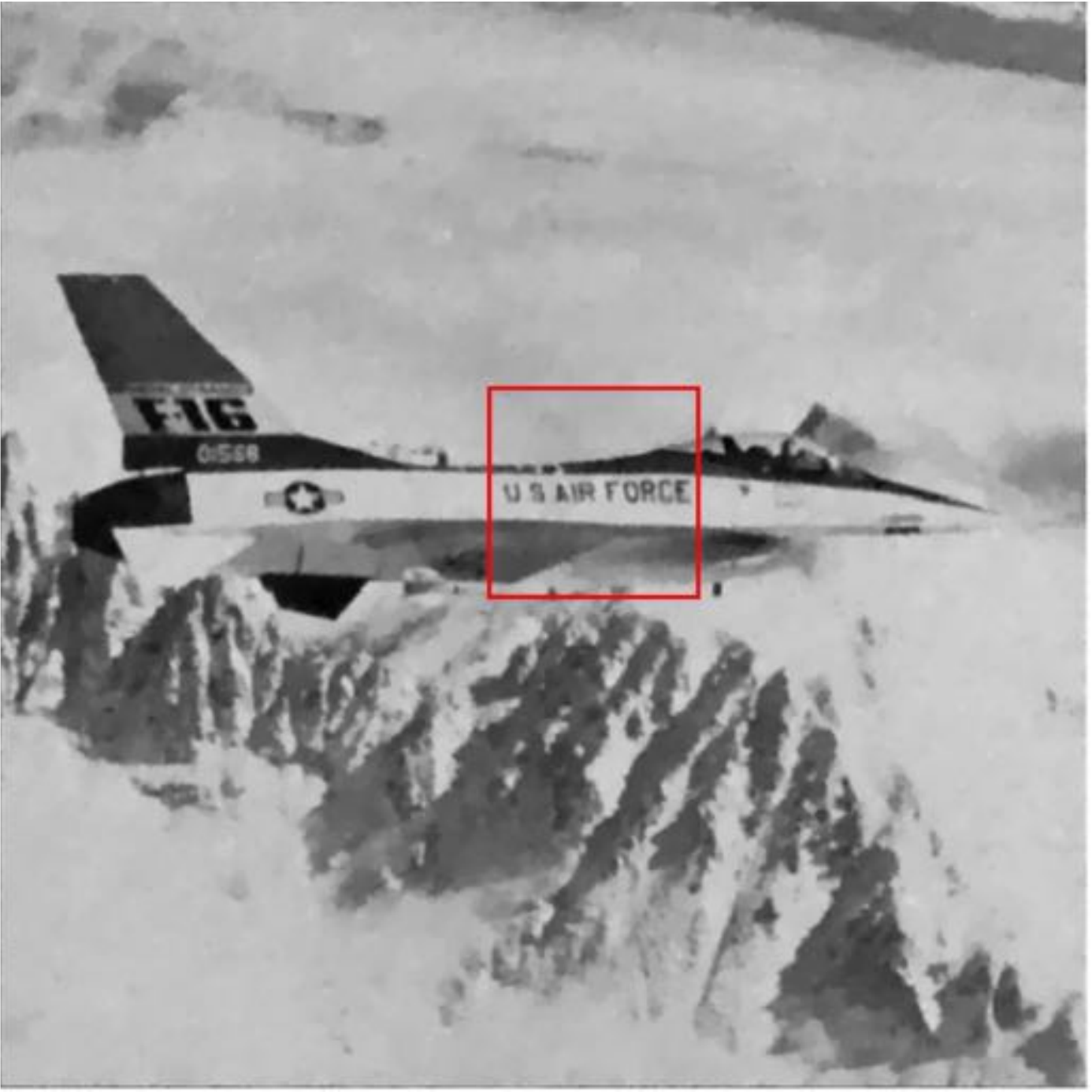}}\hspace{-1ex}
      \subfigure{
			\includegraphics[width=0.16\linewidth]{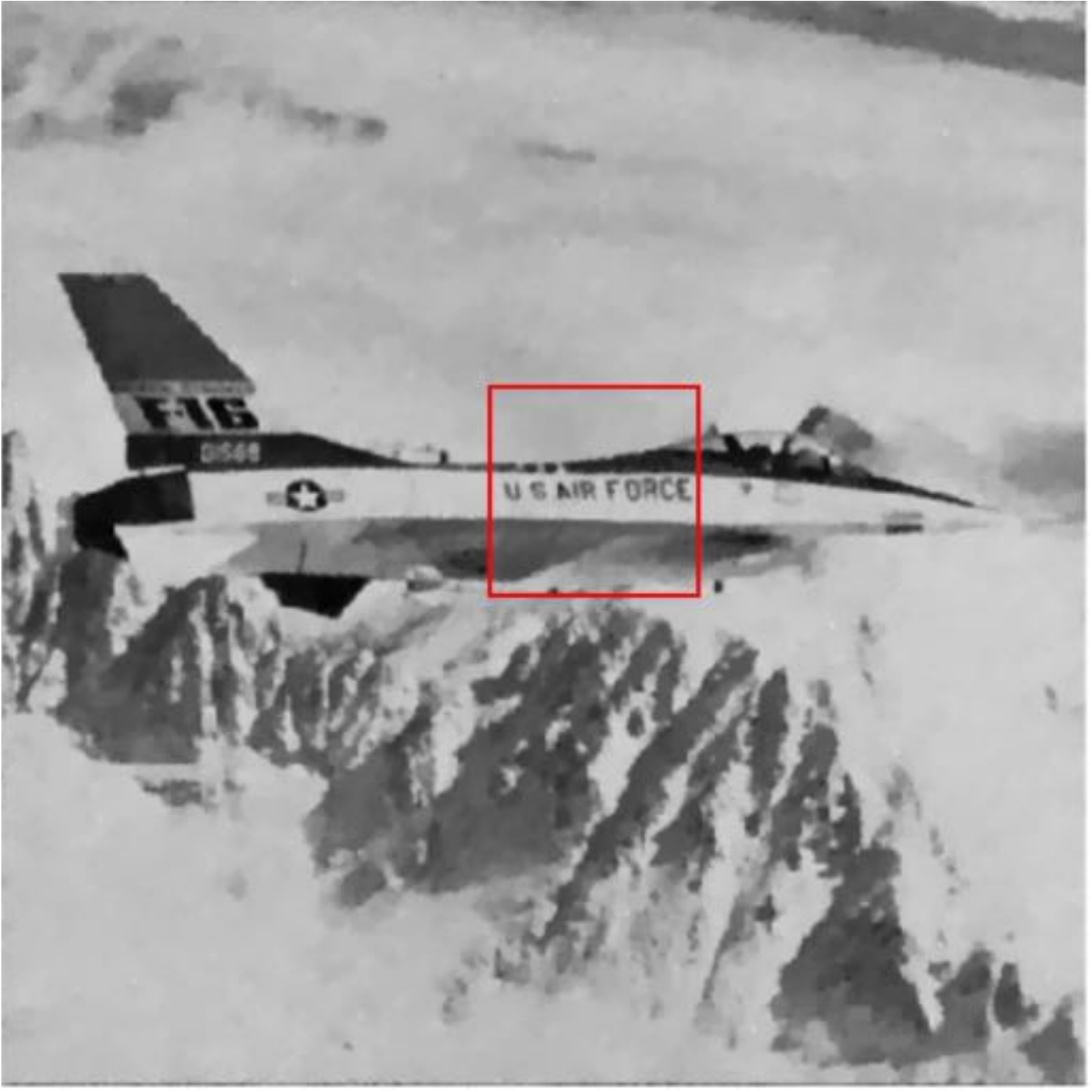}}
      \setcounter{subfigure}{0}
      \subfigure[TV]{
            \includegraphics[width=0.16\linewidth]{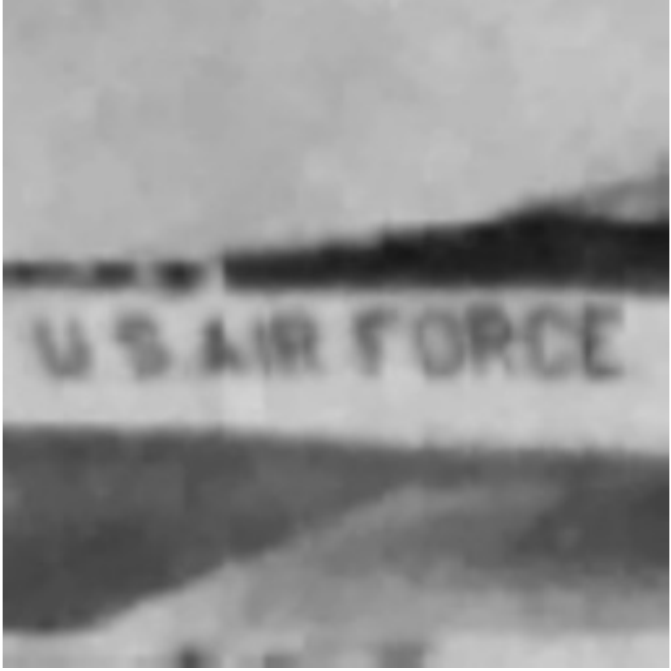}}\hspace{-1ex}
      \subfigure[Euler]{
            \includegraphics[width=0.16\linewidth]{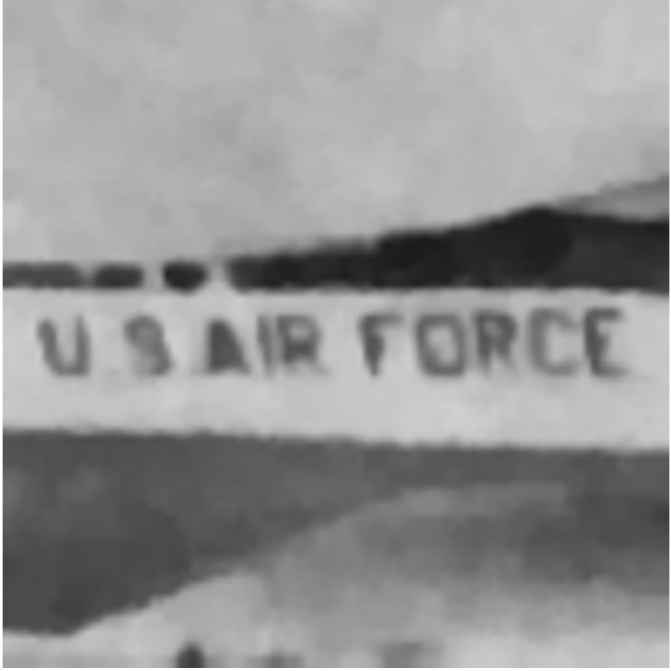}}\hspace{-1ex}
      \subfigure[TGV]{
            \includegraphics[width=0.16\linewidth]{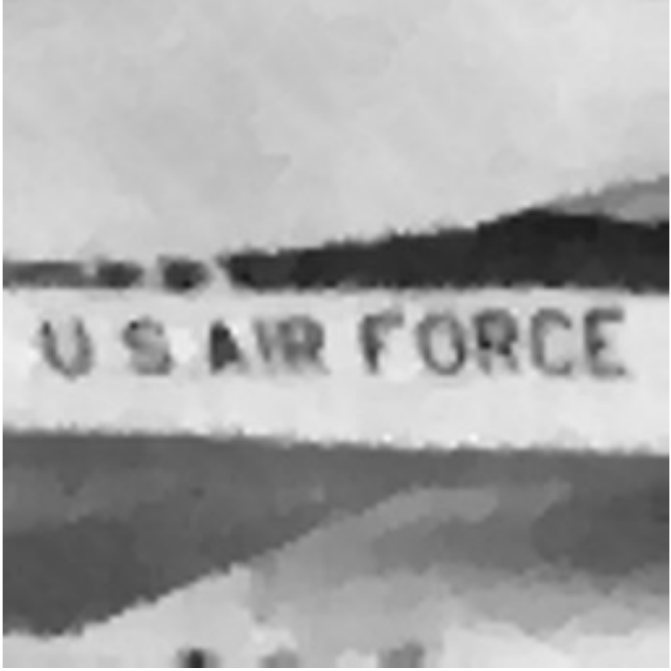}}\hspace{-1ex}
      \subfigure[MEC]{
            \includegraphics[width=0.16\linewidth]{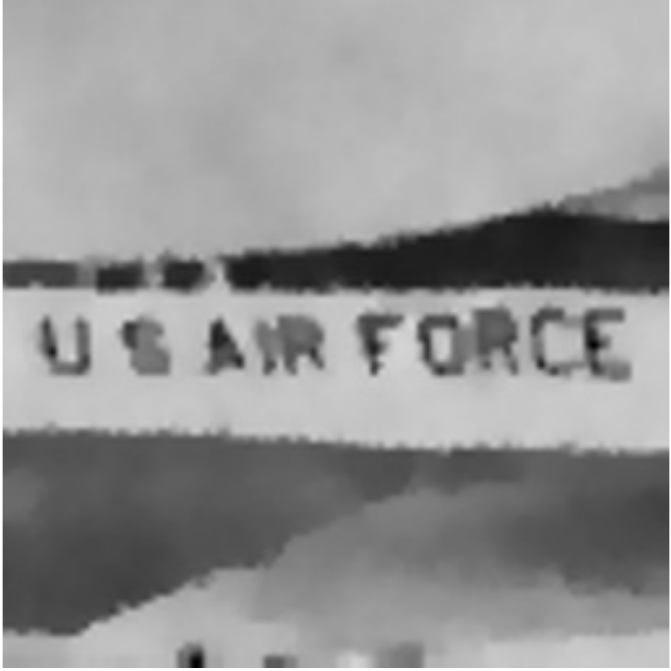}}\hspace{-1ex}
      \subfigure[TAC-MC]{
			\includegraphics[width=0.16\linewidth]{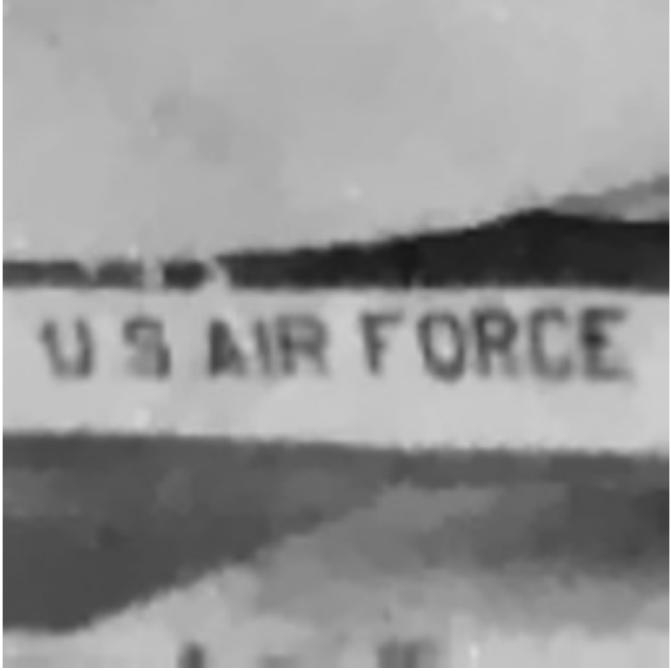}}\hspace{-1ex}
      \subfigure[TAC-GC]{
            \includegraphics[width=0.16\linewidth]{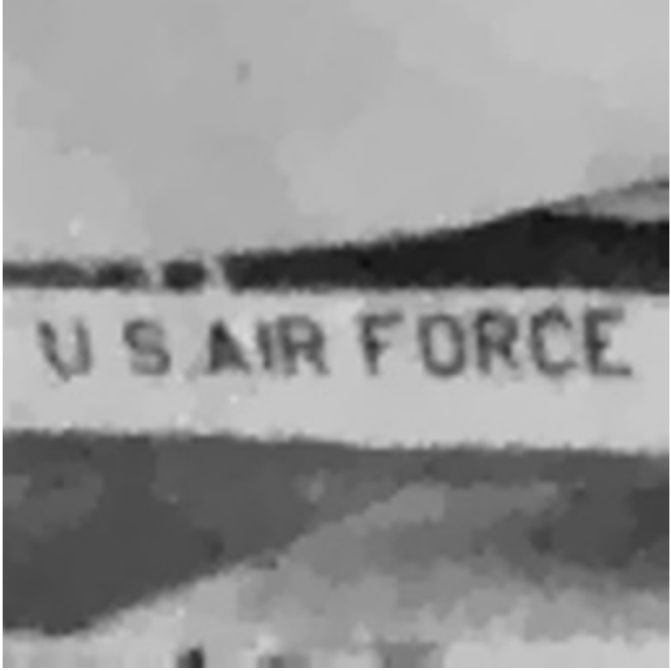}}
	\caption{The denoising results of `Plane' (top) and the corresponding local magnification views (bottom) by the comparative methods.}
	\label{DenoisingPlane}
\end{figure*}

\begin{figure*}[htp]
      \centering
      \subfigure[Lena]{
			\includegraphics[width=0.35\linewidth]{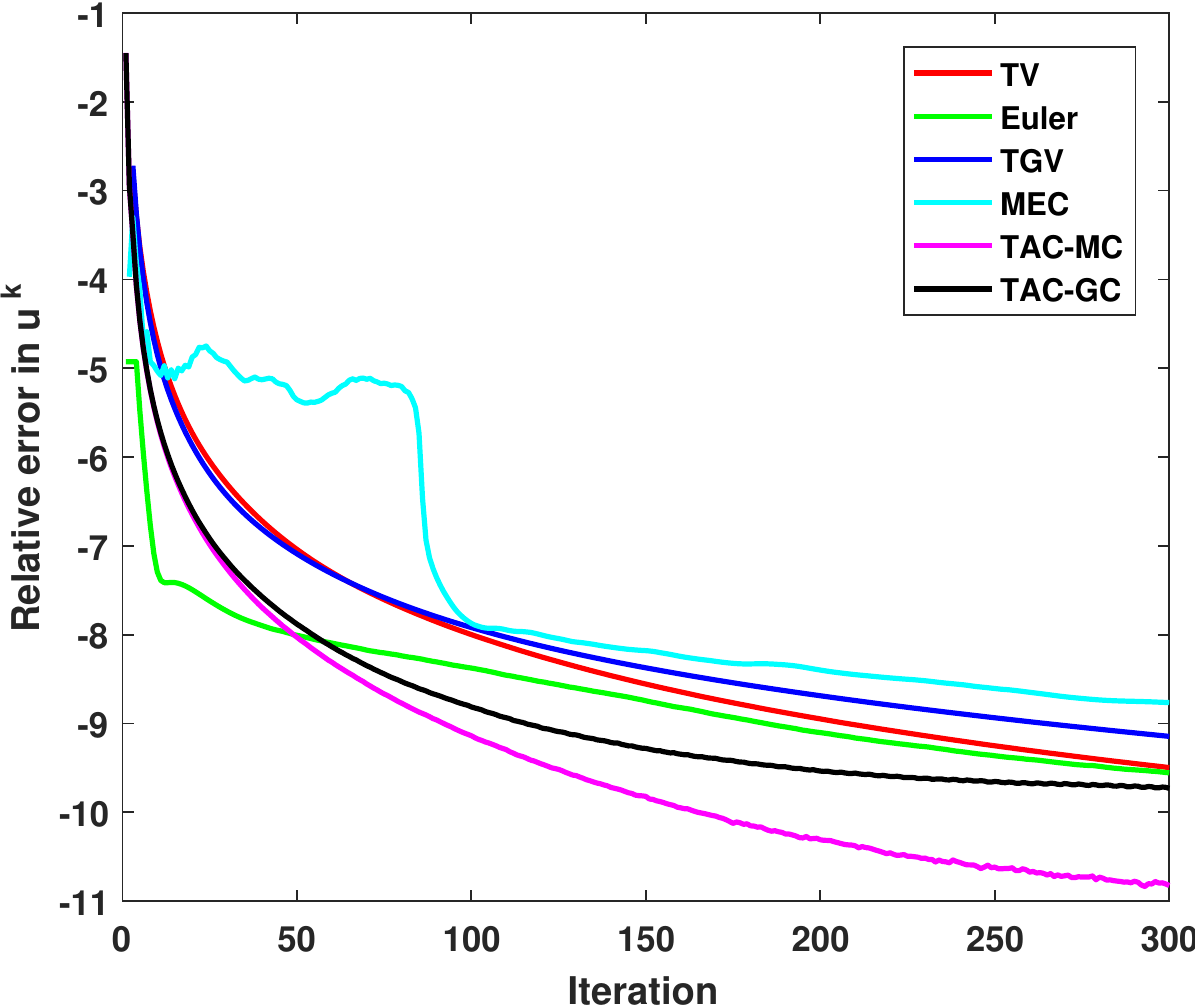}}
      \hspace{.2in}
      \subfigure[Plane]{
			\includegraphics[width=0.35\linewidth]{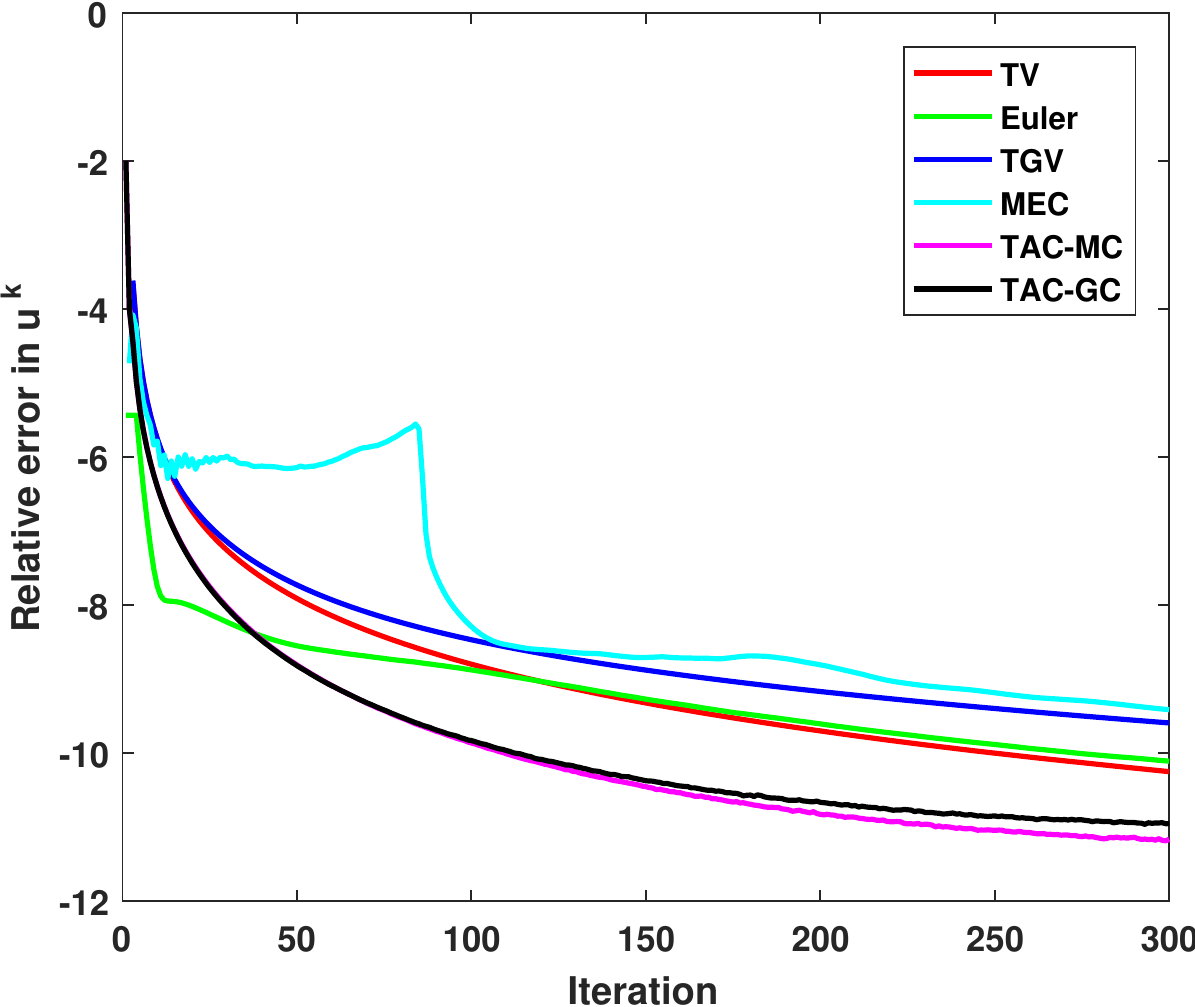}}
	\caption{Relative errors in $u^k$ of `Lena' and `Plane' by the comparative methods.}
	\label{ReErrLenaPlane}
\end{figure*}

\begin{figure*}[t]
      \centering
      \subfigure[MC-Cameraman-N]{
			\includegraphics[width=0.22\linewidth]{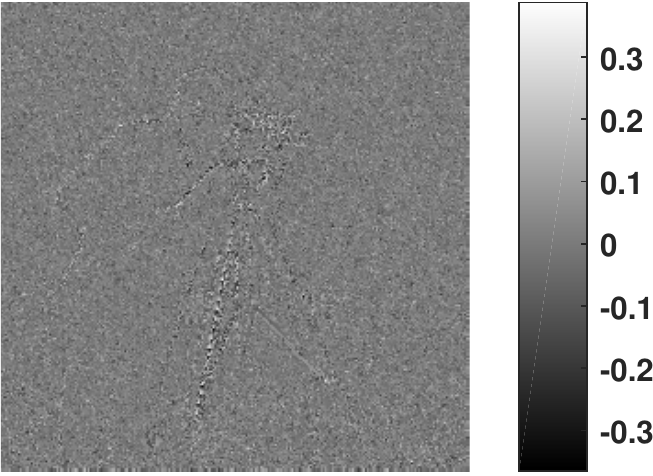}}
      \subfigure[MC-Lena-N]{
			\includegraphics[width=0.22\linewidth]{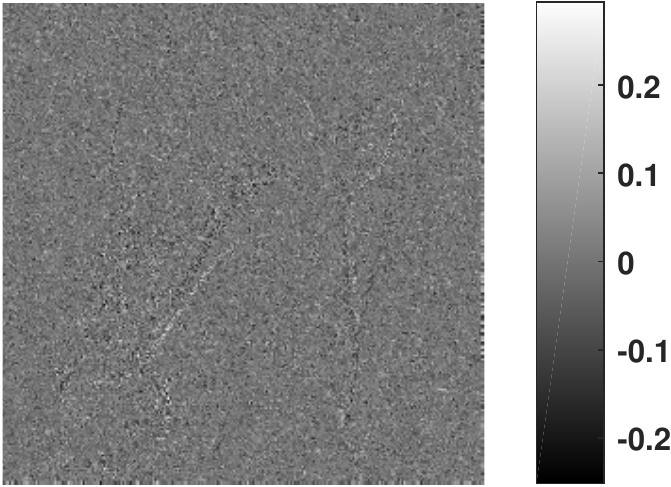}}
      \subfigure[GC-Cameraman-N]{
			\includegraphics[width=0.22\linewidth]{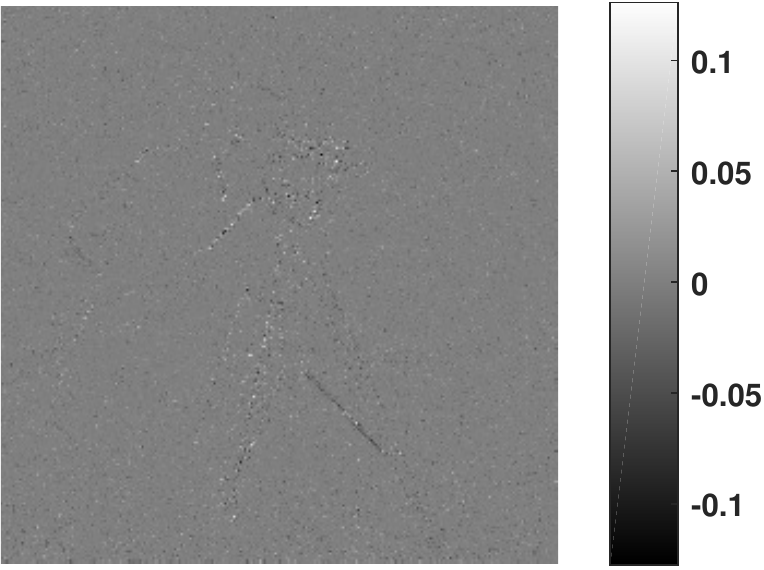}}
      \subfigure[GC-Lena-N]{
			\includegraphics[width=0.23\linewidth]{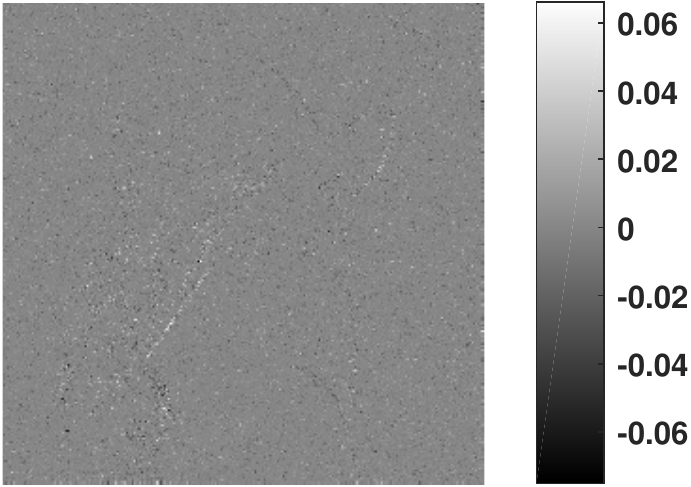}}
      \subfigure[MC-Cameraman-R]{
			\includegraphics[width=0.22\linewidth]{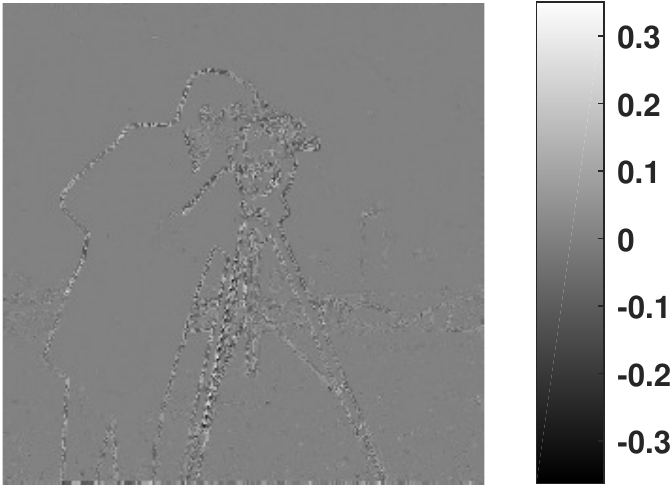}}
      \subfigure[MC-Lena-R]{
			\includegraphics[width=0.22\linewidth]{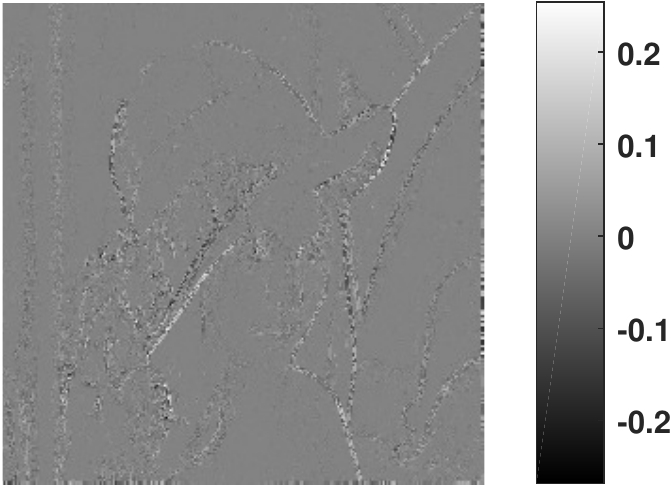}}
      \subfigure[GC-Cameraman-R]{
			\includegraphics[width=0.22\linewidth]{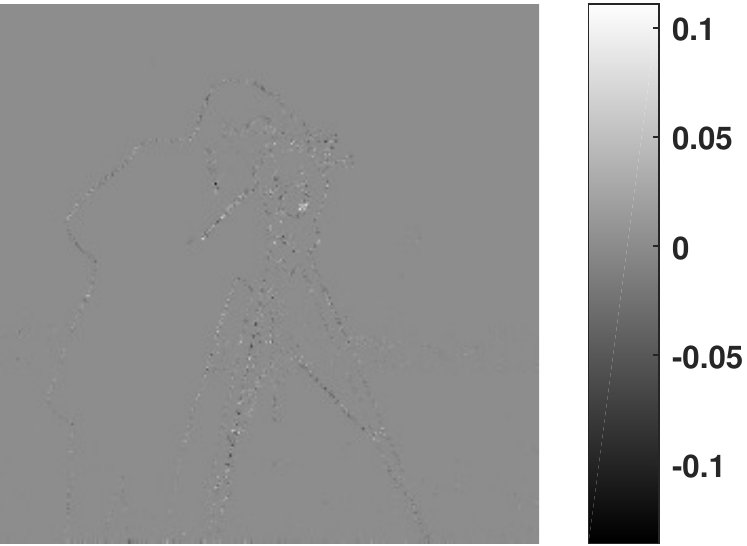}}
      \subfigure[GC-Lena-R]{
			\includegraphics[width=0.23\linewidth]{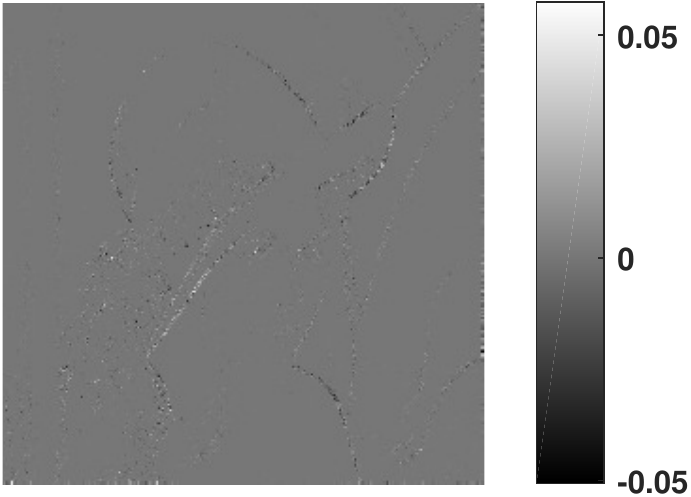}}
      \subfigure[MC-Cameraman-C]{
			\includegraphics[width=0.22\linewidth]{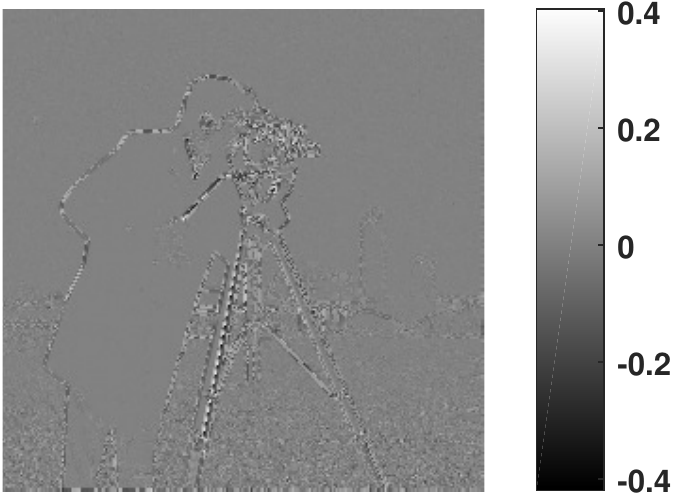}}
      \subfigure[MC-Lena-C]{
			\includegraphics[width=0.22\linewidth]{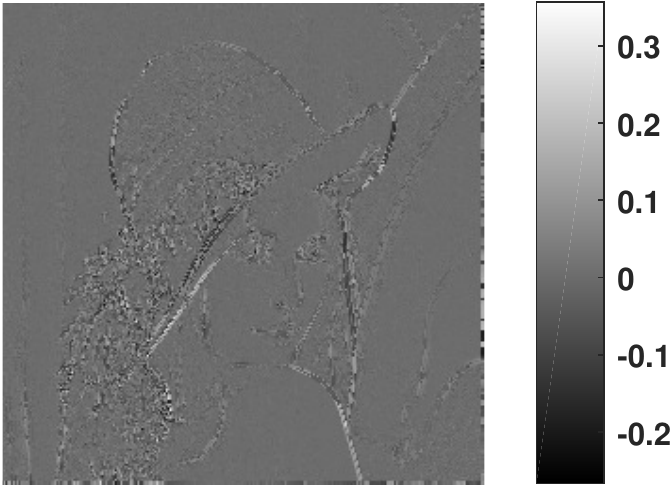}}
      \subfigure[GC-Cameraman-C]{
			\includegraphics[width=0.22\linewidth]{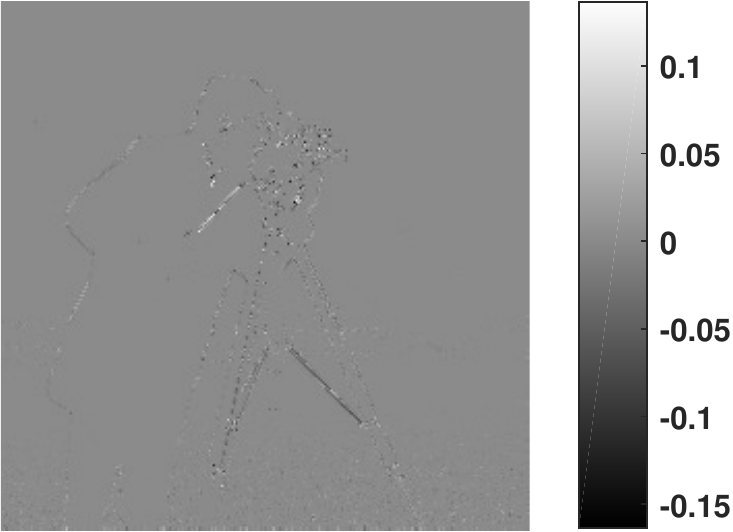}}
      \subfigure[GC-Lena-C]{
			\includegraphics[width=0.23\linewidth]{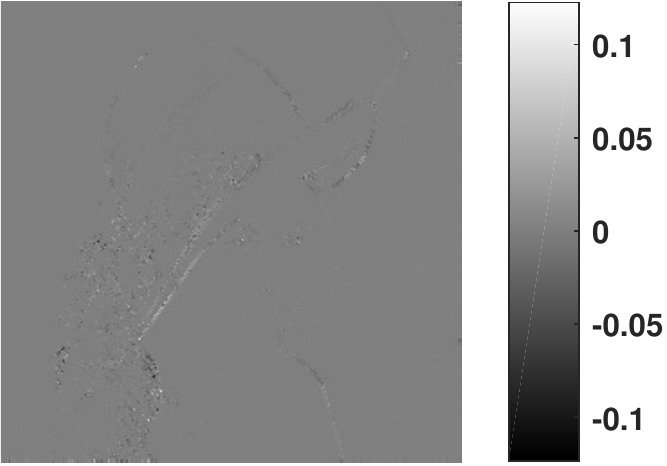}}
	\caption{The numerical MC and GC of the noisy images (N), the restoration results (R) obtained by our proposals and clean images (C) on the first, second and third row, respectively.}
	\label{MCGCimages}
\end{figure*}

We compare the restoration results both quantitatively and qualitatively. The recovery results and the residual images $f-u$ of `Cameraman' and `Triangle' are visually exhibited in FIG. \ref{DenoisingCameraman} and FIG. \ref{DenoisingTriangle}, while the denoising images and the selected local magnification views of `Lena' and `Plane' are shown in FIG. \ref{DenoisingLena} and FIG. \ref{DenoisingPlane}. In general, all methods can remove the noises and recover the major structures and features quite well. However, the TV model suffers from obvious staircase-like artifacts such that lots of image details and textures are observed in the residual images. The Euler's elastica, TGV and MEC method can overcome the staircase effects and preserve image details to some extent due to the high-order regularizer. And our TAC-MC and TAC-GC models still give better recovery results, which produce the smooth and clean images with fine details and textures. In addition, Table \ref{denoise2} presents the PSNR and SSIM in this experiment, which shows our TAC-GC model gives the overall best recovery results. We also record the CPU time in Table \ref{efficiency}, which also illustrates that our TAC-MC and TAC-GC models outperform other high-order methods, significantly faster than the Euler's elastica and mean curvature model.

\begin{figure*}[t]
      \centering
      \subfigure[$\lambda=0.035$]{
			\includegraphics[width=0.3\linewidth]{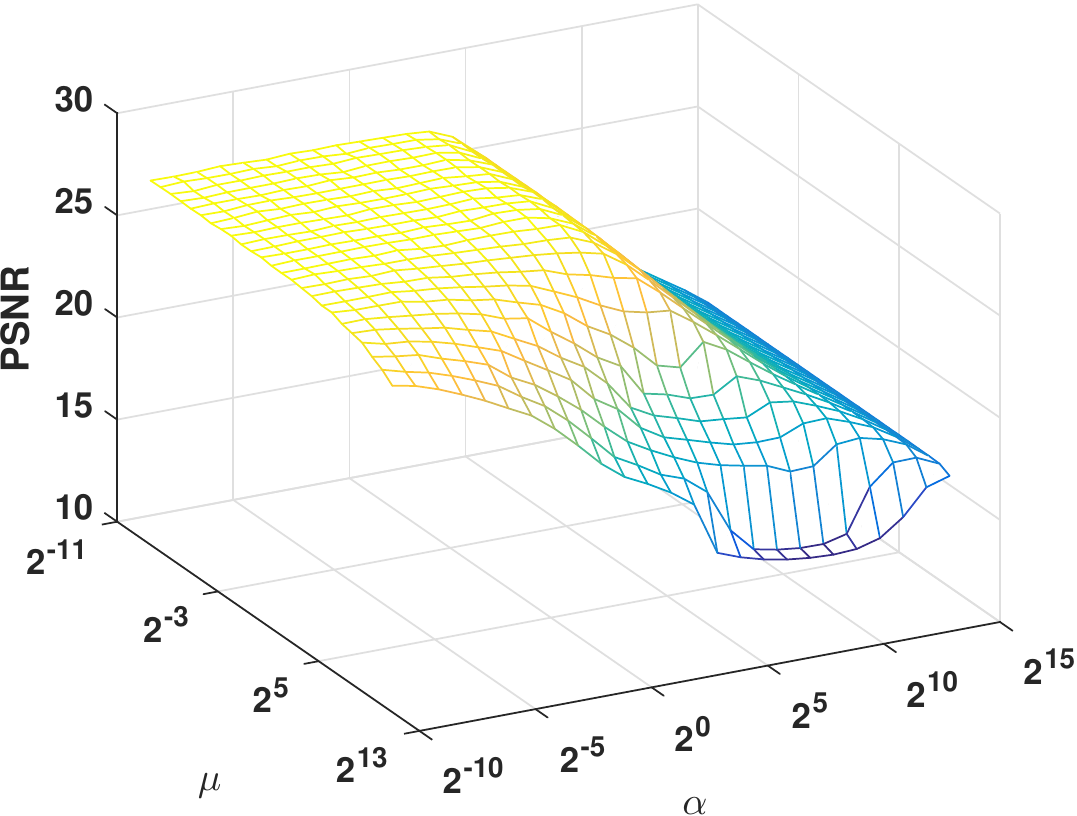}}
      \subfigure[$\lambda=0.07$]{
			\includegraphics[width=0.3\linewidth]{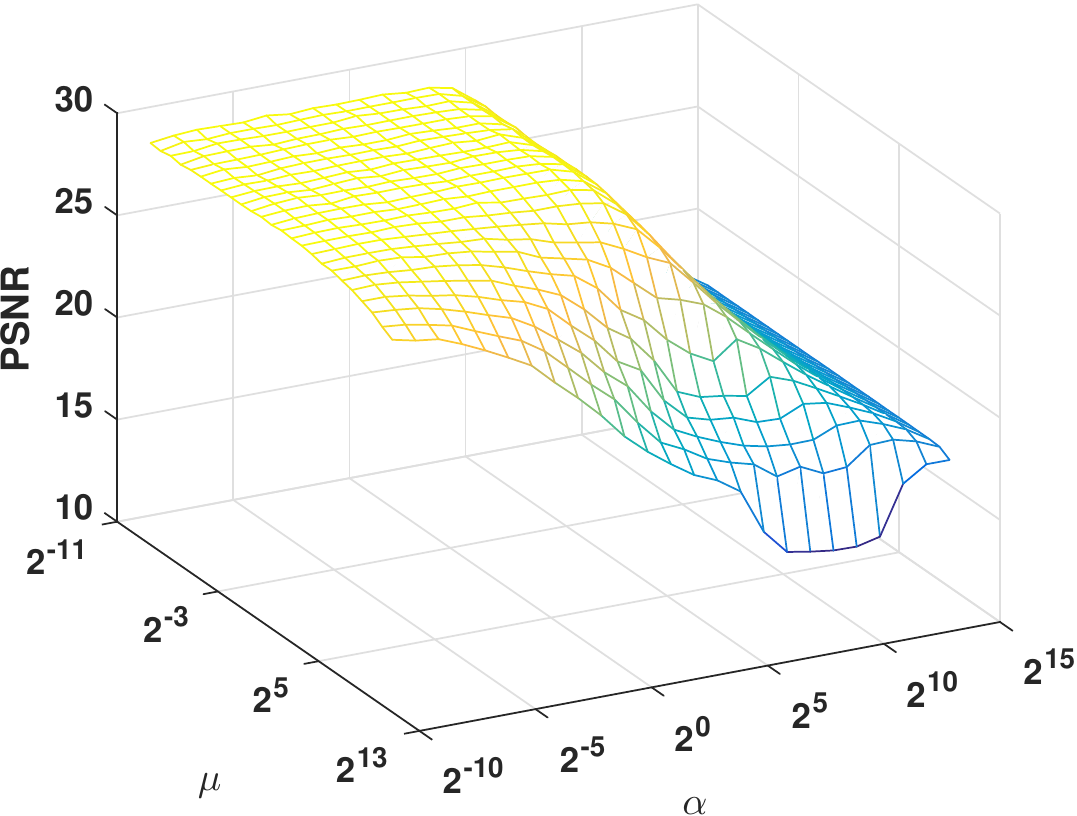}}
      \subfigure[$\lambda=0.14$]{
            \includegraphics[width=0.3\linewidth]{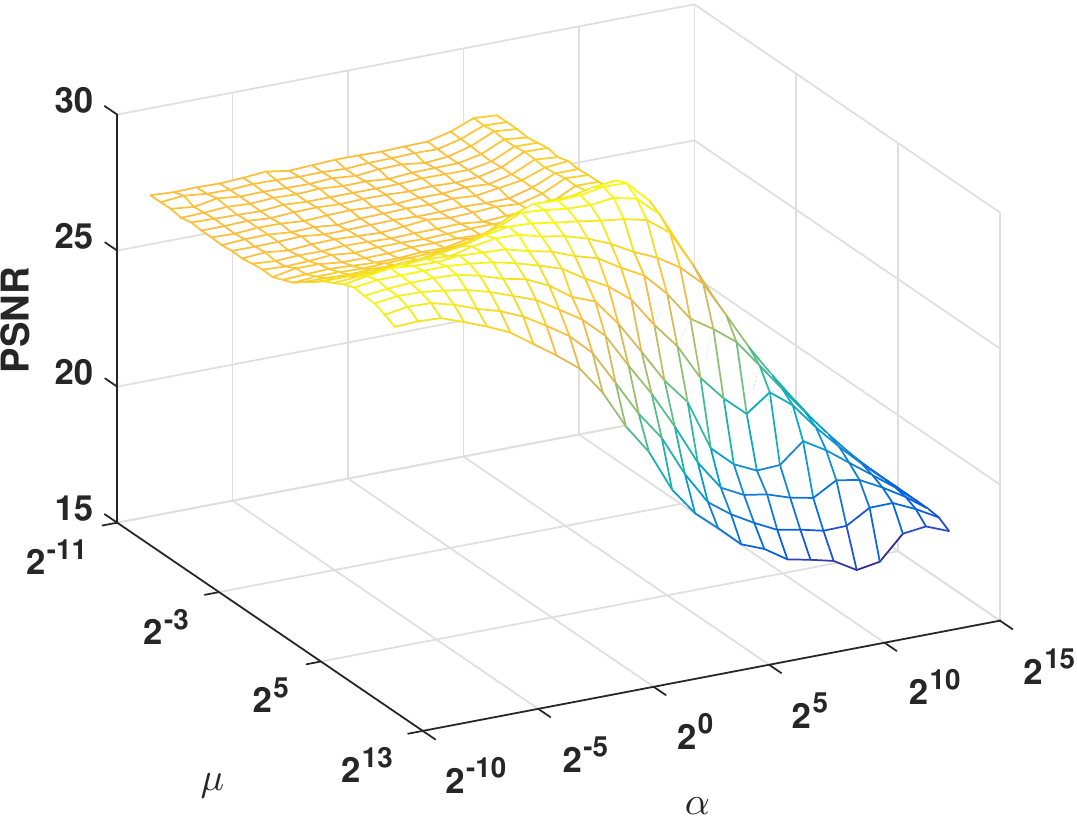}}
      \subfigure[PSNR=27.11]{
			\includegraphics[width=0.25\linewidth]{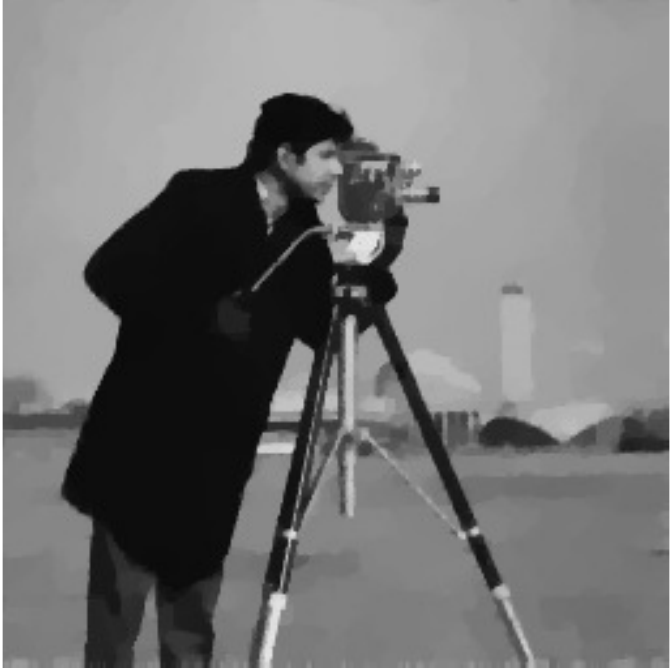}}
      \hspace{.3in}
      \subfigure[PSNR=28.92]{
			\includegraphics[width=0.25\linewidth]{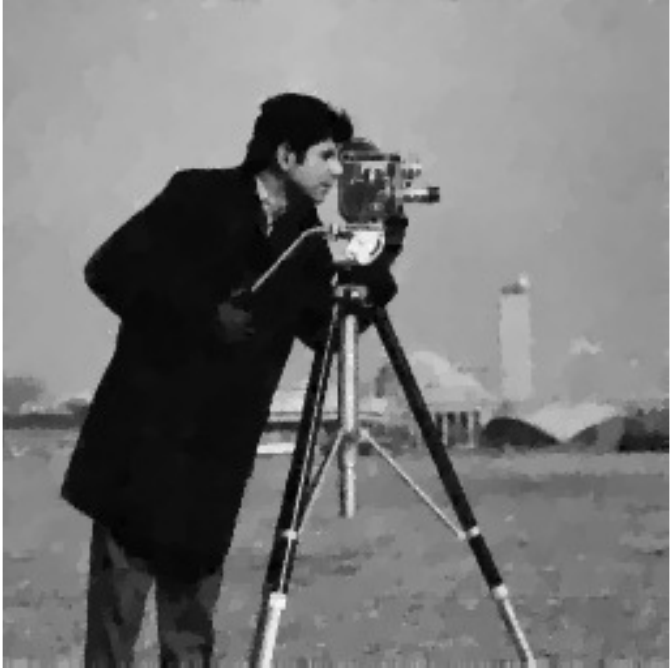}}
      \hspace{.3in}
      \subfigure[PSNR=27.85]{
            \includegraphics[width=0.25\linewidth]{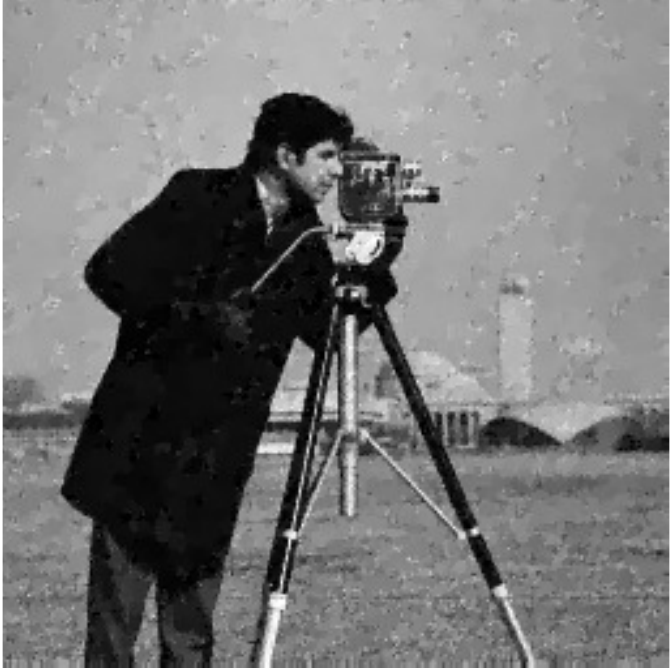}}
	\caption{The PSNR evolutions of `Cameraman' obtained by different combinations of the parameter $r$ and $\alpha$ with fixed regularized parameters $\lambda$ in TAC-GC method.}
	\label{parameters}
\end{figure*}

Furthermore, we track the decay of the relative residuals, relative errors in $\bm\Lambda^k$, relative errors in $u^k$ and the numerical energies of the TAC-MC and TAC-GC methods, which are displayed using log-scale in FIG. \ref{EvaluationsCameraman} and FIG. \ref{EvaluationsTriangle}. These plots demonstrate the convergence of the iterative process and the stability of the proposed methods. As shown, the TAC-GC model usually converges  to a lower numerical energy. To better visualize the convergence of the comparative methods, we plot the relative errors in $u^k$ of `Lena' and `Plane' of these methods in FIG. \ref{ReErrLenaPlane}. Although the relative error of the Euler's elastica energy drops faster at the beginning, our TAC-MC and TAC-GC models can attain smaller relative errors as iteration increases. Thus, our proposal always converges faster than others when a stringent relative error is given as the stopping cretiera.

The visual illustrations of the numerical MC and GC of `Cameraman' and `Lena' estimated on the noisy images, restoration images and the clean images are presented in FIG. \ref{MCGCimages}. Significant noises can be observed in the curvature images of noisy images, while the curvature images of the recovery images are noiseless and only jumps on edges. Indeed, the MC and GC images of the restorations are much alike to the ones obtained by the clean images in visual perceptions. It reveals that the TAC-MC and TAC-GC models successfully reduce the noises contained in MC and GC images, which indicates the reasonability and effectiveness of our proposed models. Through in-depth comparison between the curvature images of the recovery images and clean images, we have the following two observations:
\begin{itemize}
\item Only main edges are presented in the GC images.  The GC measures $\kappa_1\kappa_2$ and has a small magnitude, as long as one principal curvature is small. It well explains why GC regularized model gives lower numerical energy. Thus, minimizing GC allows for fine details and structures, which is more suitable for natural images such as `Lena' and `Cameraman' etc.
\item More details and small edges exist in the MC images. By minimizing the total MC of the noisy image, some tiny structures in the MC images will be smoothed out. Thus, the MC regularity works better for images containing large homogeneous or slowly varying regions e.g., the smooth images in FIG. \ref{smoothimages}.
\end{itemize}

In order to analyze the impact of the parameter $\lambda$, $\alpha$ and $\mu$ in our algorithm, we select the image `Cameraman' as example and test the denoising performance with different combinations of parameters. We select $\alpha$ from three different values $\lambda \in \{0.035, 0.07, 0.14\}$. For each $\lambda$, we vary the parameters $(\alpha,\mu) \in\{\alpha^0\times2^{-\delta}, \alpha^0\times2^{-\delta+1},\cdots, \alpha^0\times2^{\delta-1}, \alpha^0\times2^{\delta}\}\times \{\mu^0\times2^{-\delta}, \mu^0\times2^{-\delta+1},\cdots, \mu^0\times2^{\delta-1}, \mu^0\times2^{\delta}\}$ with $\alpha^0=5$, $\mu^0=2$ and $\delta=12$. In FIG. \ref{parameters}, we plot the PSNR values with different parameters and present the best recovery results for $\lambda = \{0.035, 0.07, 0.14\}$, respectively. As shown, there are relative large intervals for $\alpha$ and $\mu$ to generate good recovery results for fixed $\lambda$. And too small $\lambda$ results in over smoothed recovery results with some details missing, while too large $\lambda$ leads to nonsmooth recovery results with some noise remaining. Therefore, the choice of $\lambda$ is the most important consideration to achieve a high-quality restoration result, which should be tuned according to the noise levels of the test images.

\subsection{Salt $\&$ pepper and Poisson denoising}

In this subsection, both the salt $\&$ pepper and Poisson denoising experiments are operated to further illustrate the excellent performance of our curvature model.
According to the statistical properties of the salt $\&$ pepper noise, we adopt the $L^1$-norm data fidelity term instead of the $L^2$-norm one \cite{dong2009efficient,nikolova2004variational,goldstein2009split}, which gives

\begin{equation}\label{L1norm}
  \min_{u} \sum_{1\leq i,j\leq m} g(\kappa_{i,j}) |\nabla u_{i,j}| + {\lambda}\|u-f\|_1.
\end{equation}
To deal with the above minimization problem, two auxiliary variables are introduced to rewrite the above minimization problem into the following constrained one
\begin{equation*}\label{L1normconstrain}
\begin{split}
& \min_{u,\bm v,w}~ \sum_{1\leq i,j\leq m}g(\kappa_{i,j})|\bm v_{i,j}| + {\lambda}\|w\|_1 \\
& ~\mathrm{s.t.}~~~\bm v=\nabla u,~~w=u-f.
\end{split}
\end{equation*}
More details for solving such constrained minimization problem can be referred to \cite{yang2009efficient,goldstein2009split}.

\begin{figure*}[t]
      \centering
      \subfigure{
			\includegraphics[width=0.135\linewidth]{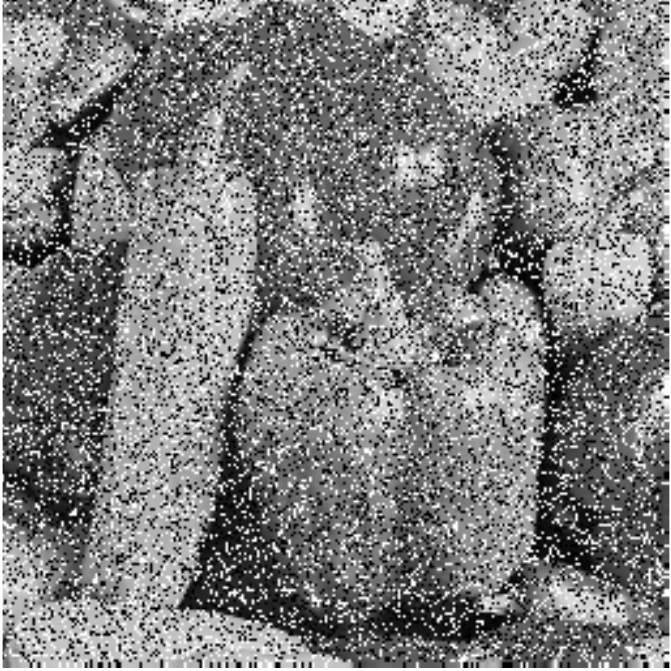}}\hspace{-1ex}
      \subfigure{
			\includegraphics[width=0.135\linewidth]{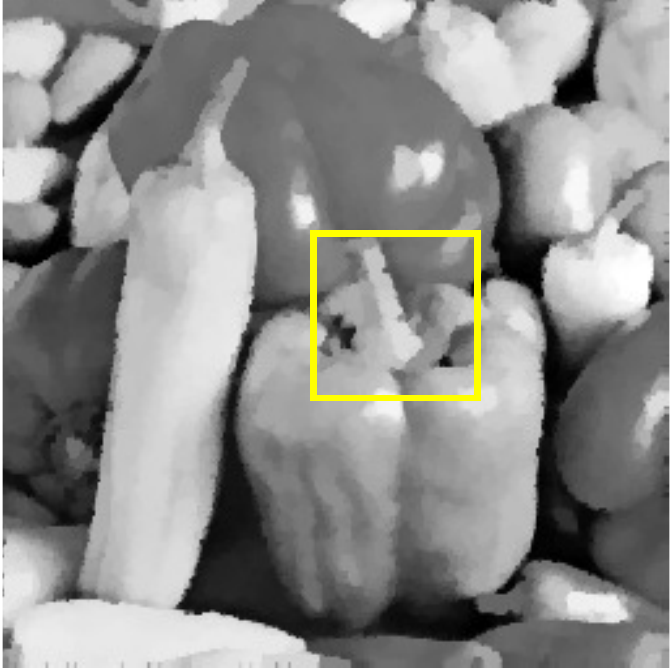}}\hspace{-1ex}
      \subfigure{
			\includegraphics[width=0.135\linewidth]{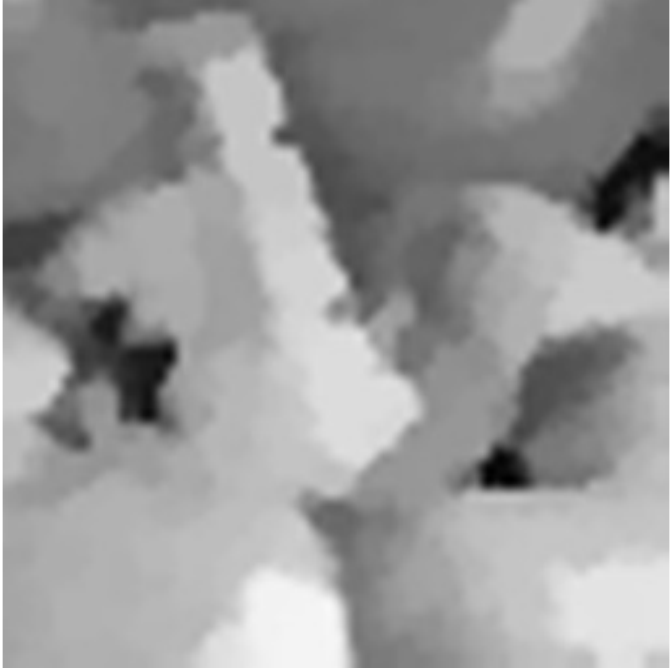}}\hspace{-1ex}
      \subfigure{
			\includegraphics[width=0.135\linewidth]{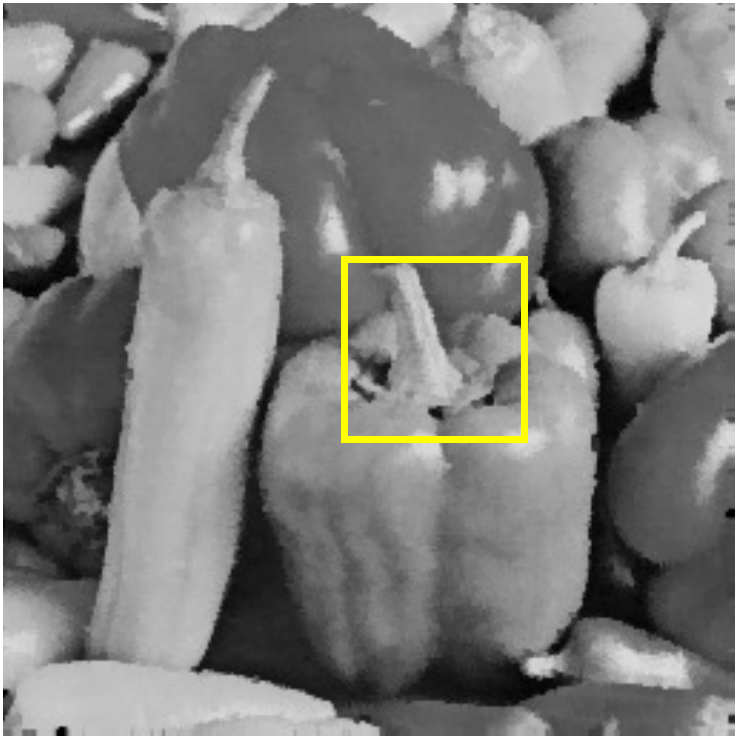}}\hspace{-1ex}
      \subfigure{
			\includegraphics[width=0.135\linewidth]{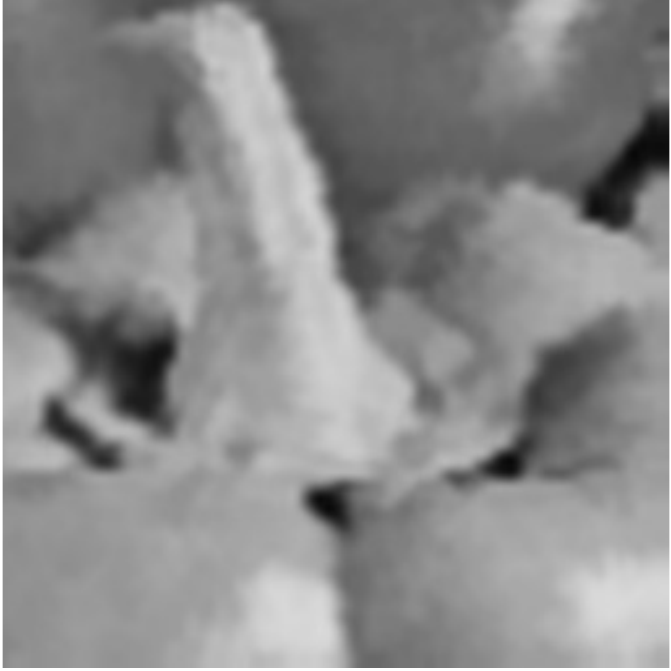}}\hspace{-1ex}
      \subfigure{
			\includegraphics[width=0.135\linewidth]{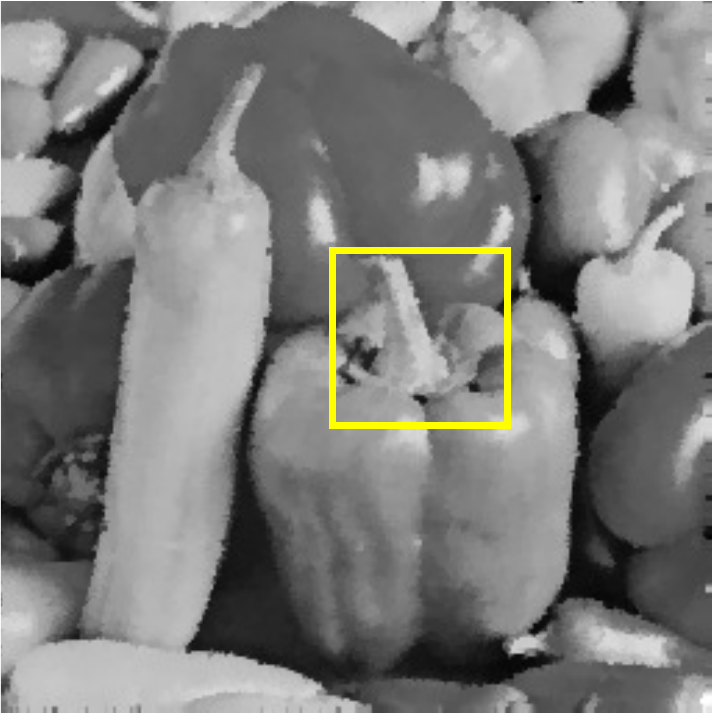}}\hspace{-1ex}
      \subfigure{
			\includegraphics[width=0.135\linewidth]{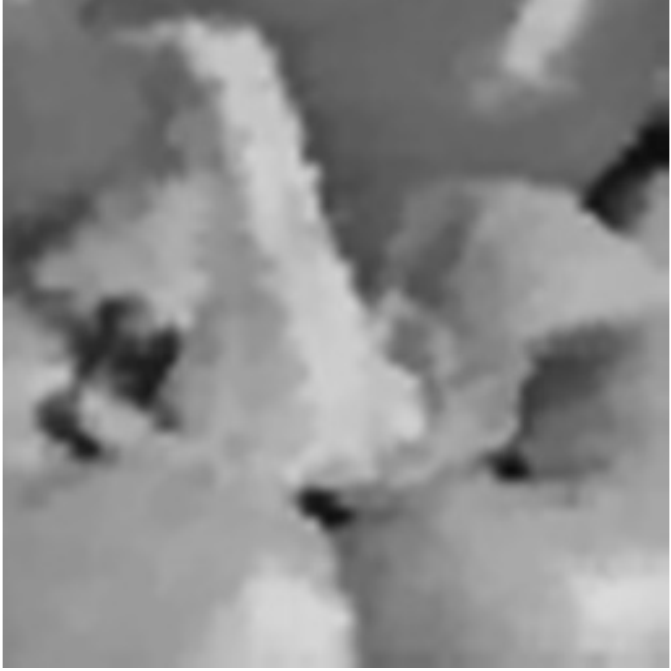}}
      \setcounter{subfigure}{0}
      \subfigure[Noisy]{
			\includegraphics[width=0.135\linewidth]{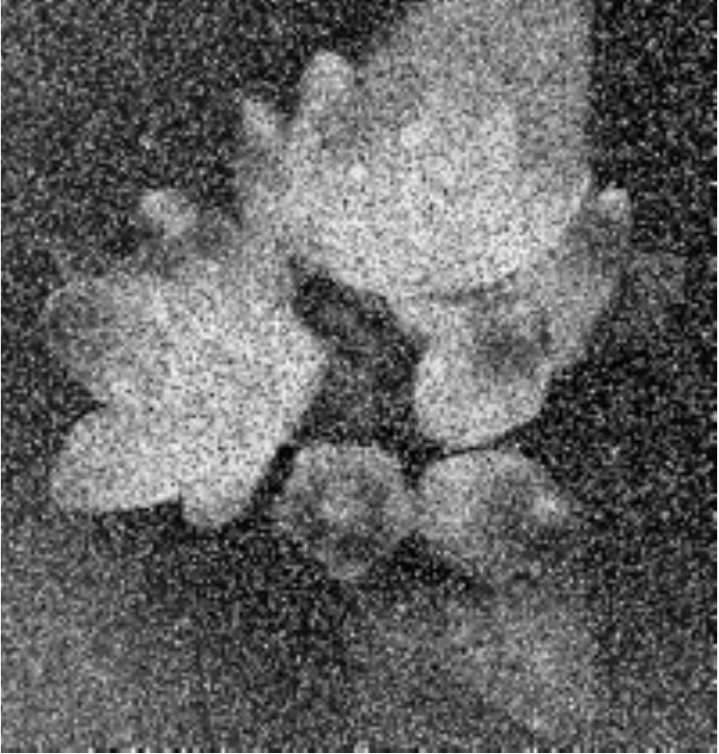}}\hspace{-1ex}
      \subfigure[TV]{
			\includegraphics[width=0.135\linewidth]{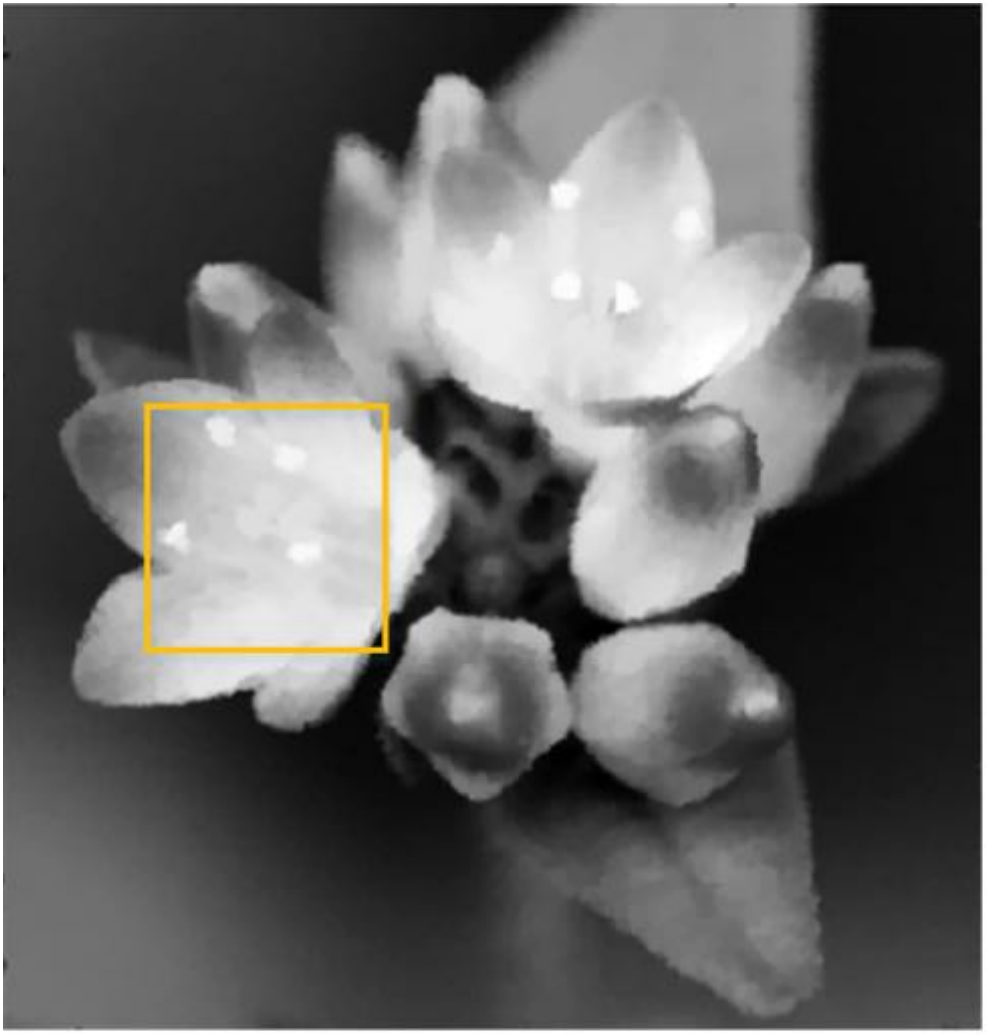}}\hspace{-1ex}
      \subfigure[zoomed]{
			\includegraphics[width=0.135\linewidth]{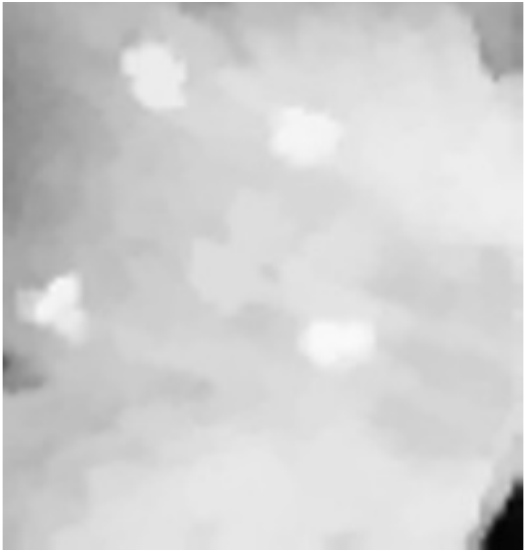}}\hspace{-1ex}
      \subfigure[Euler]{
			\includegraphics[width=0.135\linewidth]{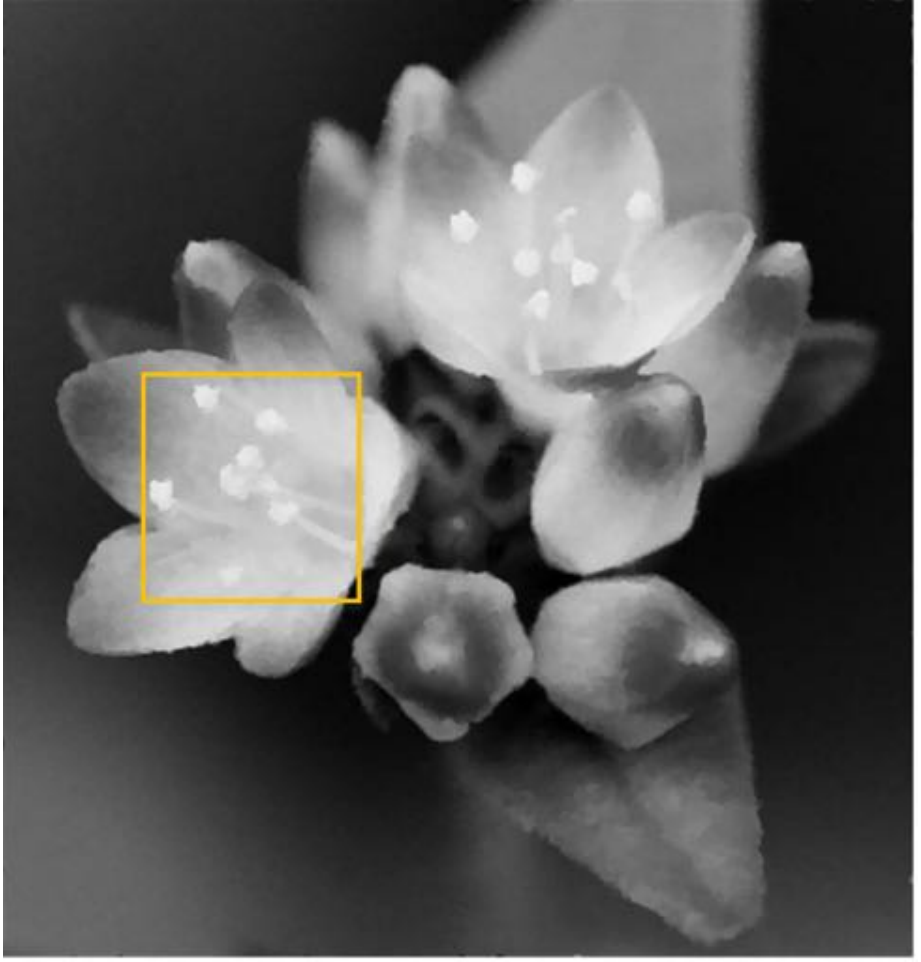}}\hspace{-1ex}
      \subfigure[zoomed]{
			\includegraphics[width=0.135\linewidth]{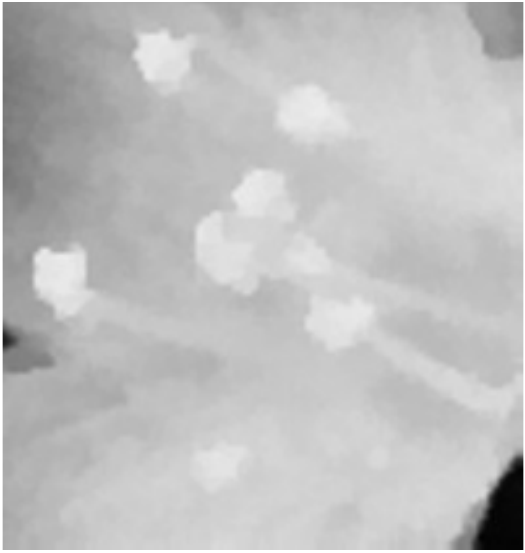}}\hspace{-1ex}
      \subfigure[TAC-GC]{
			\includegraphics[width=0.135\linewidth]{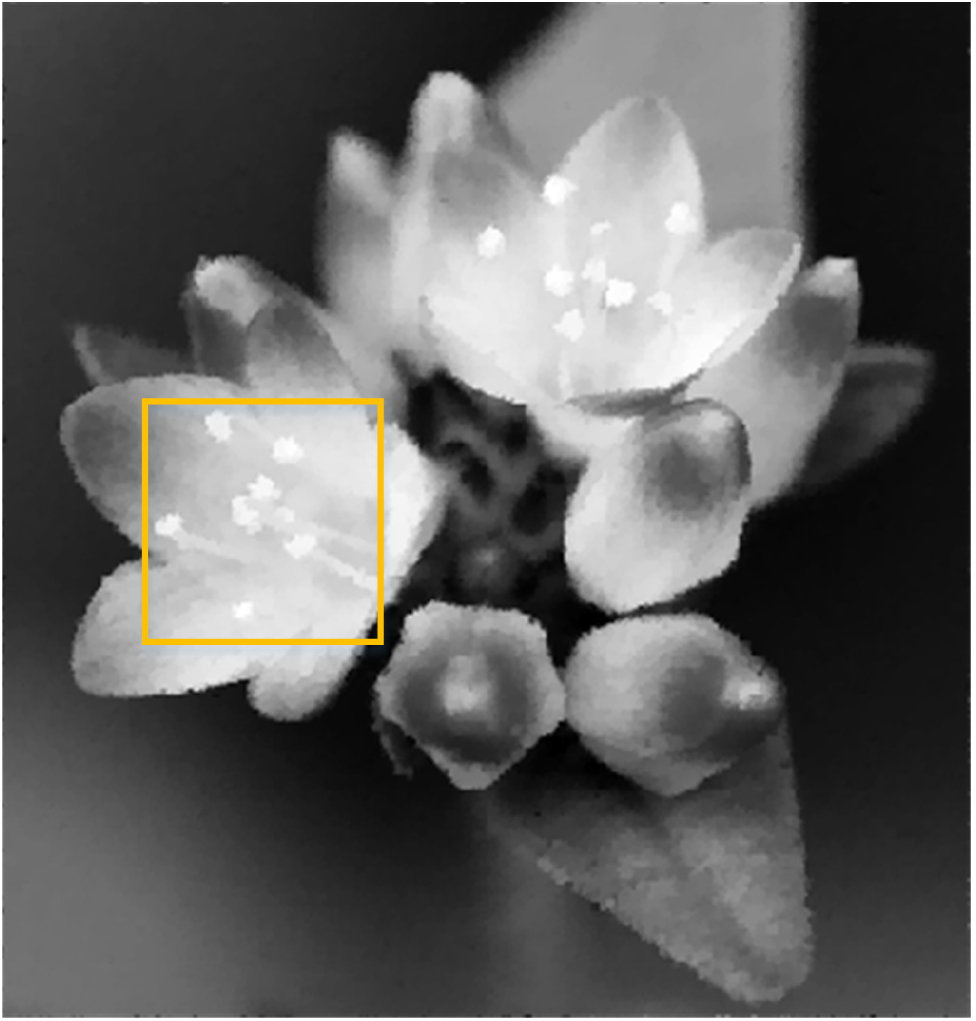}}\hspace{-1ex}
      \subfigure[zoomed]{
			\includegraphics[width=0.135\linewidth]{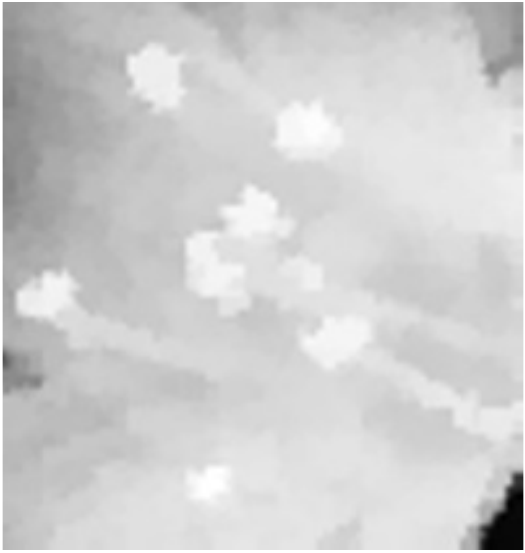}}
	\caption{The Salt $\&$ pepper denoising results of `Peppers' and `Realtest' obtained by the TV, Euler's elastica and TAC-GC methods. The parameters are set as (b) TV: $r_1=5$, $r_2=20$ and $\lambda=15$; (d) Euler: $\alpha=20$, $r_1=1$, $r_2=7\cdot10^2$, $r_3=10^2$, $r_4=5\cdot10^2$ and $\eta=20$; (f) TAC-GC: $\alpha=20$, $\mu_1=30$, $\mu_2=120$ and $\lambda=1.6$.}
	\label{implusenoise}
\end{figure*}

\begin{table*}[!htp]
\caption{The evaluations of salt \& pepper noise removal for the TV, Euler's elastica and TAC-GC methods.}
\label{denoise3}
\tiny
\begin{center}
\begin{tabular}{c|c|c|c|c|c|c|c|c|c|c}
\hline\hline
Methods & \multicolumn{2}{|c|}{TV} & \multicolumn{4}{|c|}{Euler's elastica} & \multicolumn{4}{|c}{TAC-GC}  \\
\hline
Images & PSNR & SSIM & PSNR & SSIM & Iterations & Time & PSNR & SSIM & Iterations & Time \\
\hline\hline
Peppers  & 24.08 & 0.8452 & 25.01 & 0.8689 & 261 & 18.44 & \bf{25.17} & \bf{0.8716} & \bf{193} & \bf{12.40}\\
\hline
Realtest & 31.64 & 0.8862 & \bf{33.05} & \bf{0.9104} & 265 & 67.89 & 32.97 & 0.9067 & ${\bm{172}}$ & ${\bm{44.10}}$\\
\hline\hline
\end{tabular}
\end{center}
\end{table*}

We use two grayscale test images `Peppers' $(256\times256)$ and `Realtest' $(400\times420)$, both of which are corrupted by $30\%$ salt $\&$ pepper noise. The parameters are set as $\lambda=1.6$, $\alpha=20$, and $\mu_1=30$, $\mu_2=120$, while the termination criteria is $\epsilon=3\times10^{-5}$. We compare the TAC-GC model with both the TV and Euler's elastica method. FIG. \ref{implusenoise} shows the recovery results and their local magnification views obtained by the TV, Euler's elastica and our TAC-GC method. It can be observed that the recovery of the TV model tends to lose image details and features due to the apparent staircase-like artifacts in smooth regions, while both Euler's elastica and TAC-GC method can preserve fine image details and textures to a certain extent. Table \ref{denoise3} illustrates that the TAC-GC model can obtain higher PSNR and SSIM than TV model while give the similar PSNR and SSIM as Euler's elastica model. More importantly, we find out that the TAC-GC model always approaches to the stopping criteria with fewer iterations and less computational costs than the Euler's elastica model.

\begin{figure*}[t]
      \centering
      \subfigure{
			\includegraphics[width=0.155\linewidth]{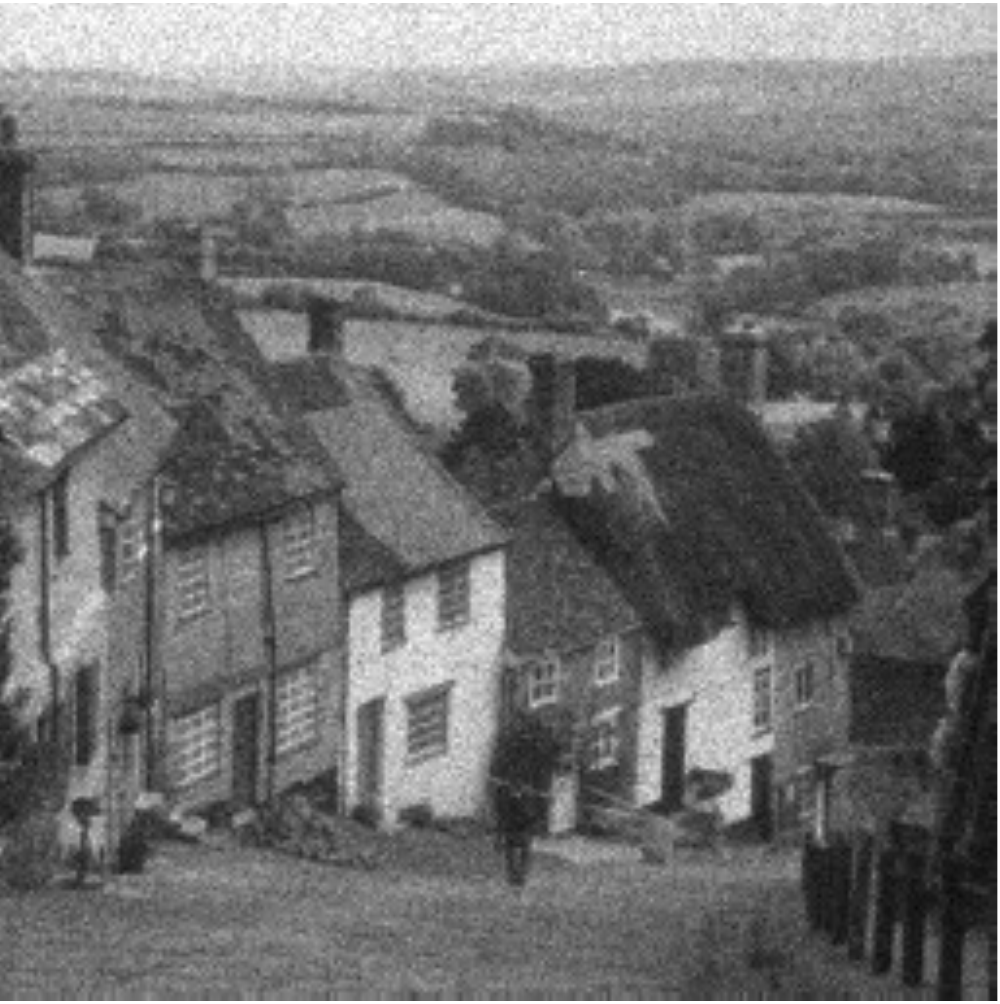}}\hspace{-1.0ex}
      \subfigure{
            \includegraphics[width=0.155\linewidth]{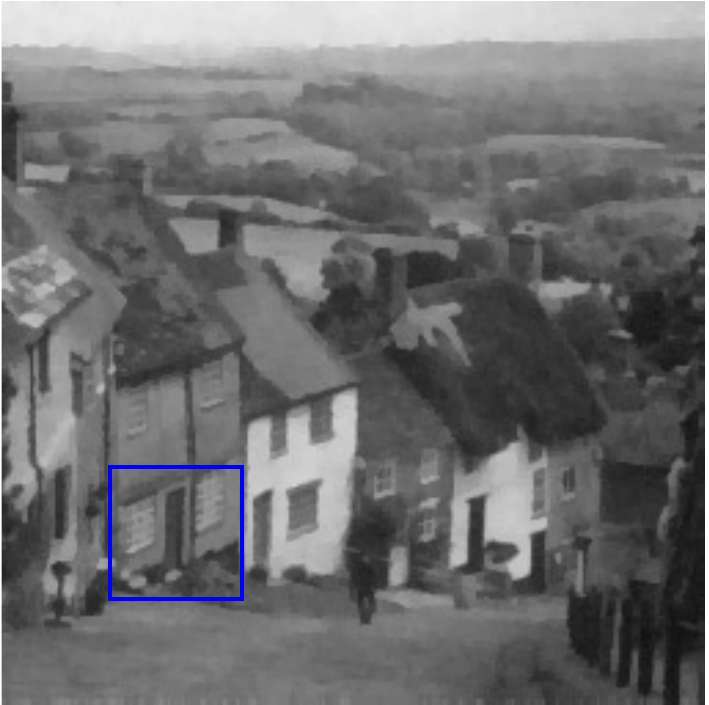}}\hspace{-1.0ex}
      \subfigure{
			\includegraphics[width=0.155\linewidth]{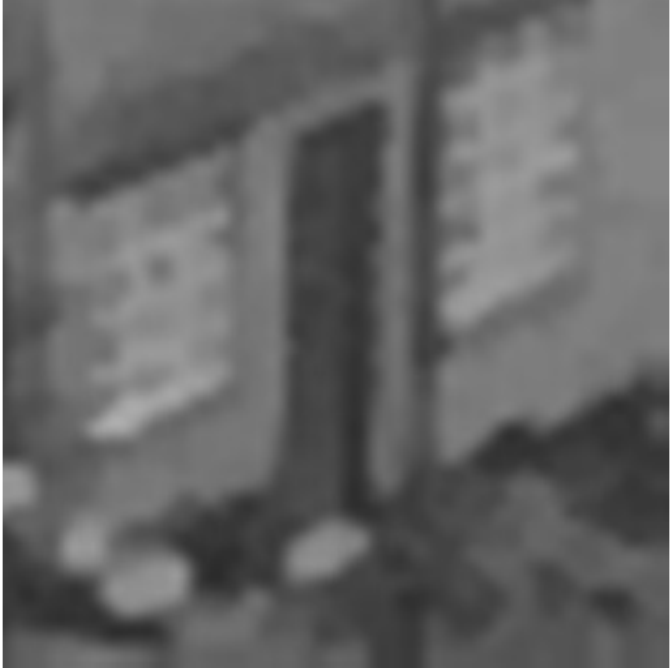}}\hspace{-1.0ex}
      \subfigure{
			\includegraphics[width=0.155\linewidth]{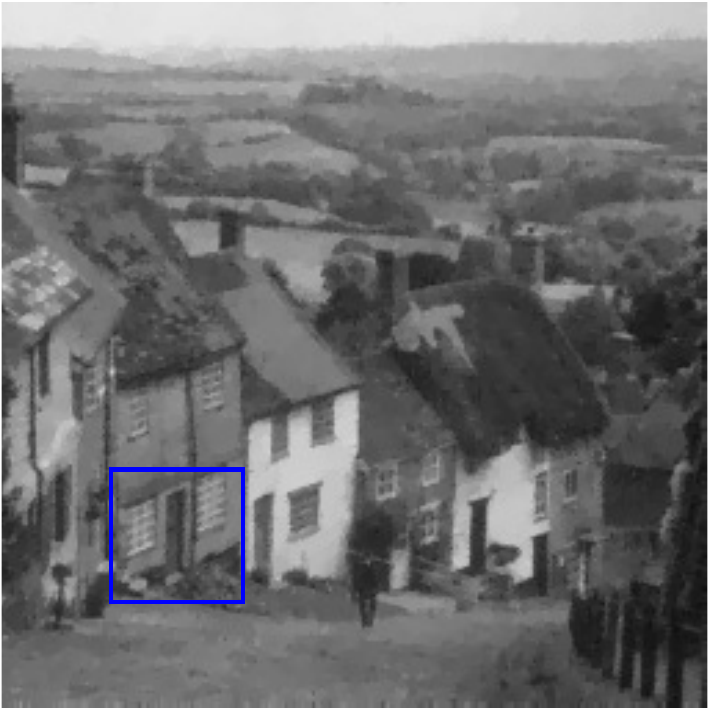}}\hspace{-1.0ex}
      \subfigure{
			\includegraphics[width=0.155\linewidth]{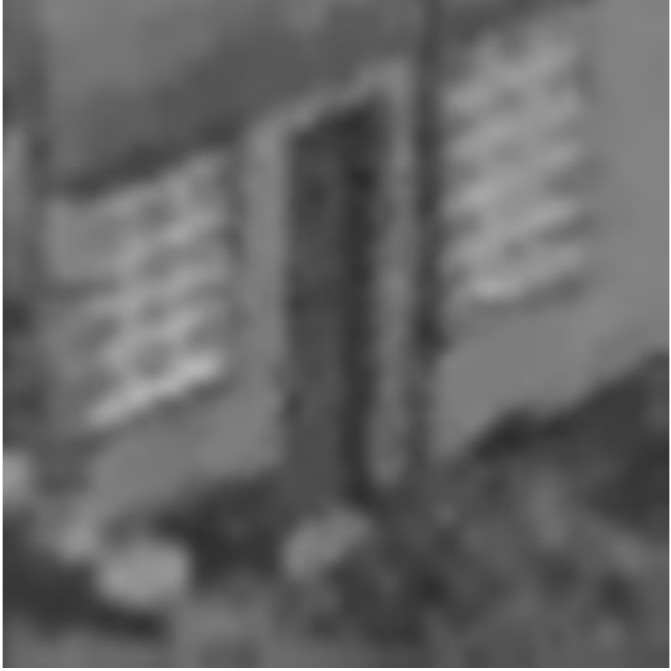}}\hspace{-1.0ex}
      \subfigure{
			\includegraphics[width=0.18\linewidth]{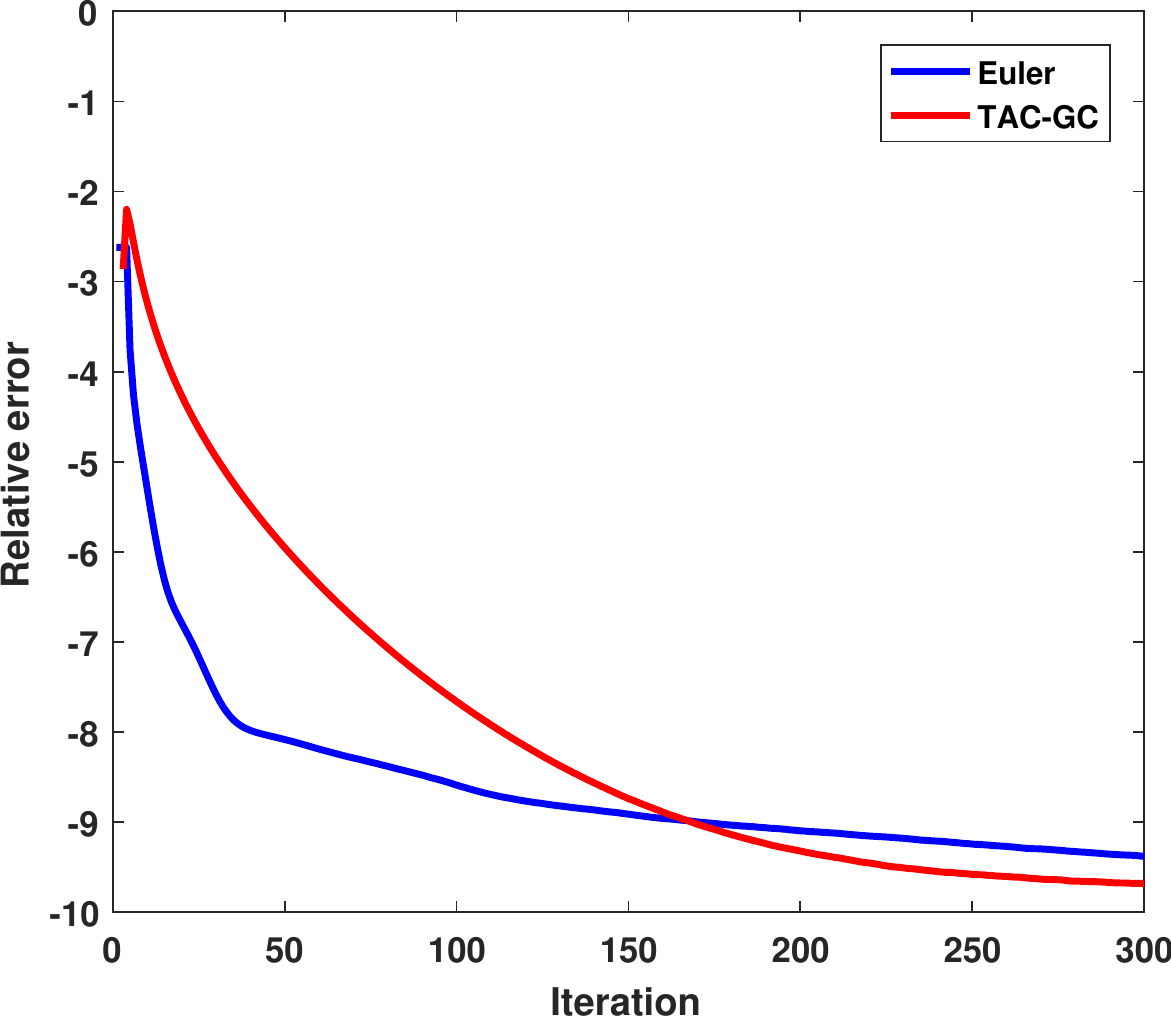}}
\setcounter{subfigure}{0}
      \subfigure[Noisy]{
			\includegraphics[width=0.155\linewidth]{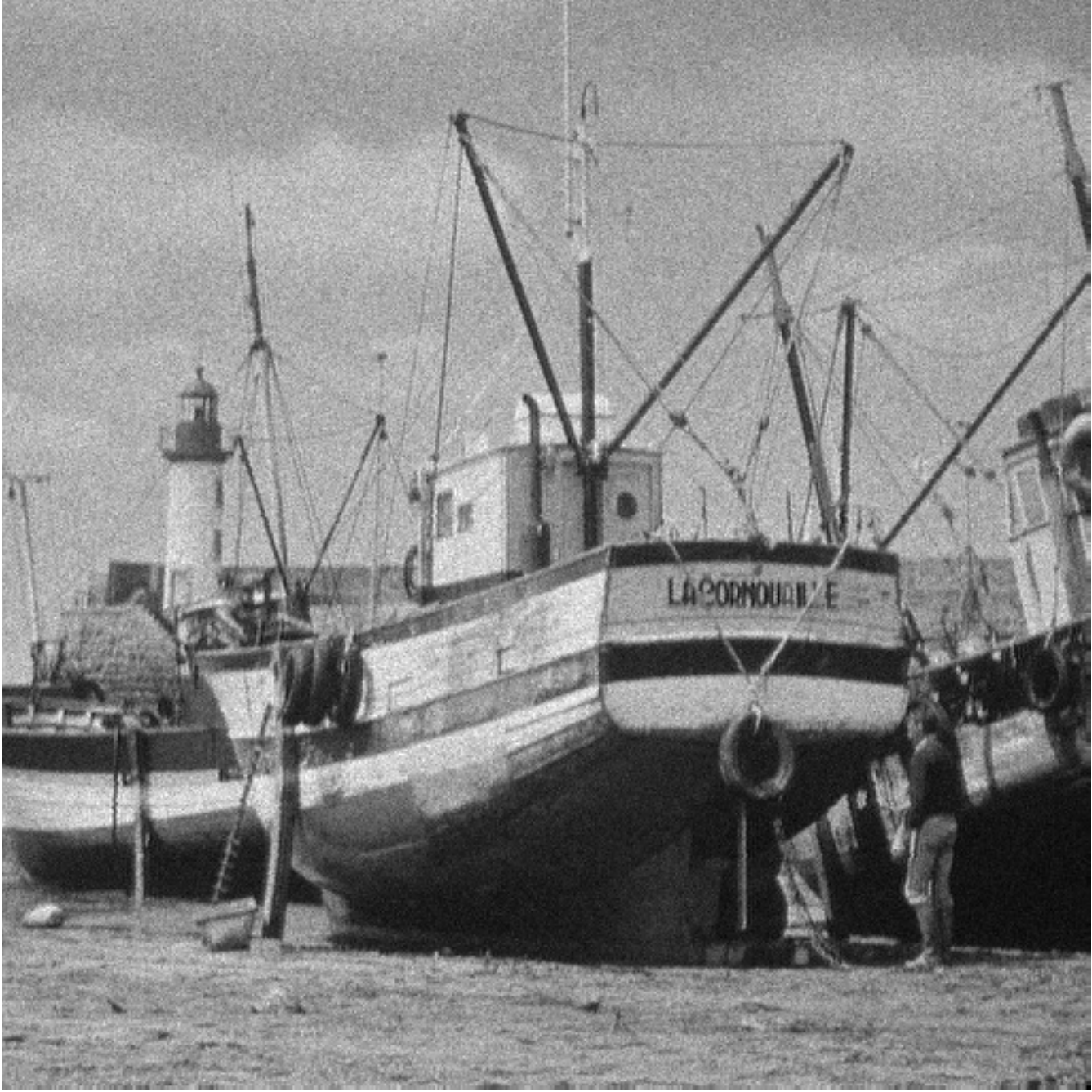}}\hspace{-1.0ex}
      \subfigure[Euler]{
            \includegraphics[width=0.155\linewidth]{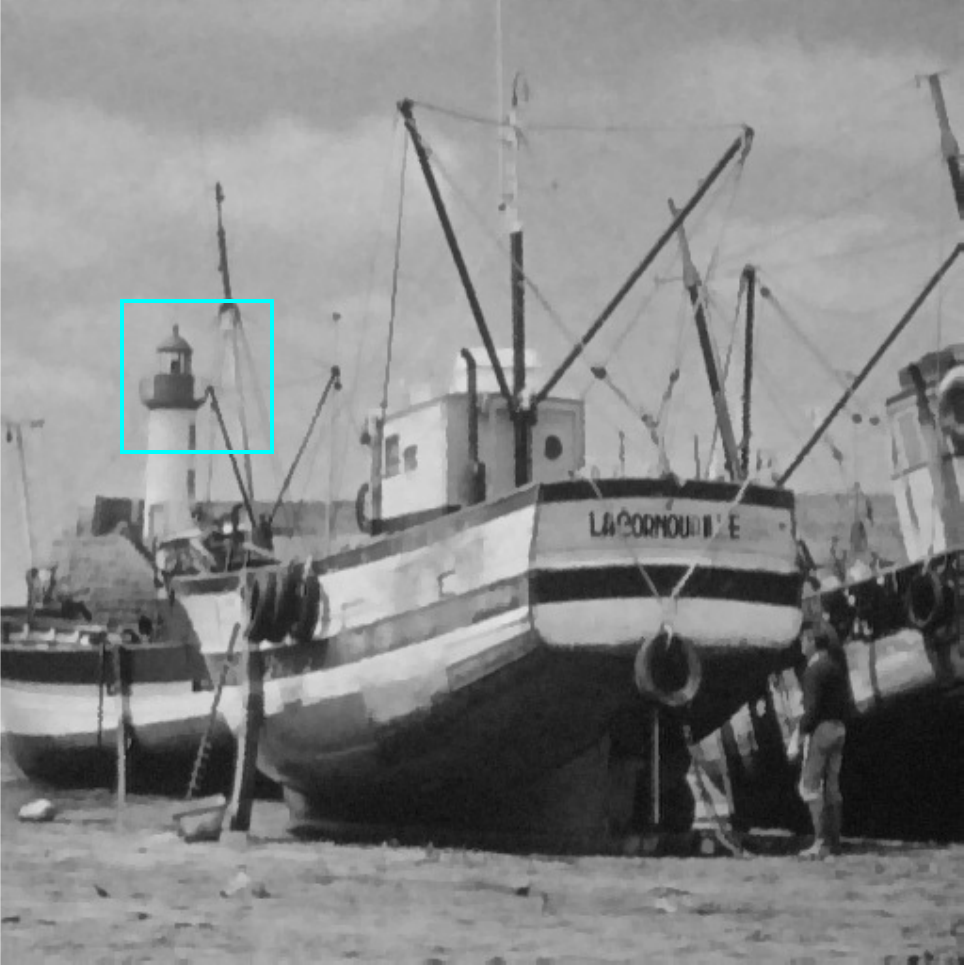}}\hspace{-1.0ex}
      \subfigure[zoomed]{
			\includegraphics[width=0.155\linewidth]{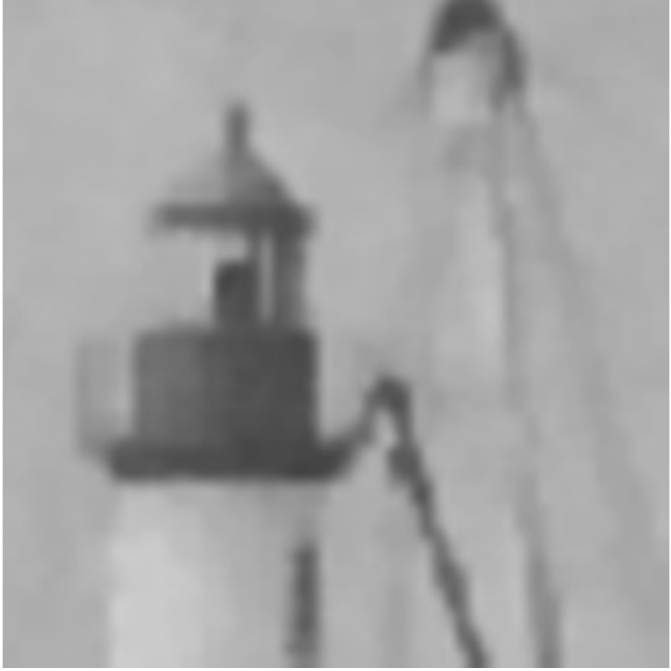}}\hspace{-1.0ex}
      \subfigure[TAC-GC]{
			\includegraphics[width=0.155\linewidth]{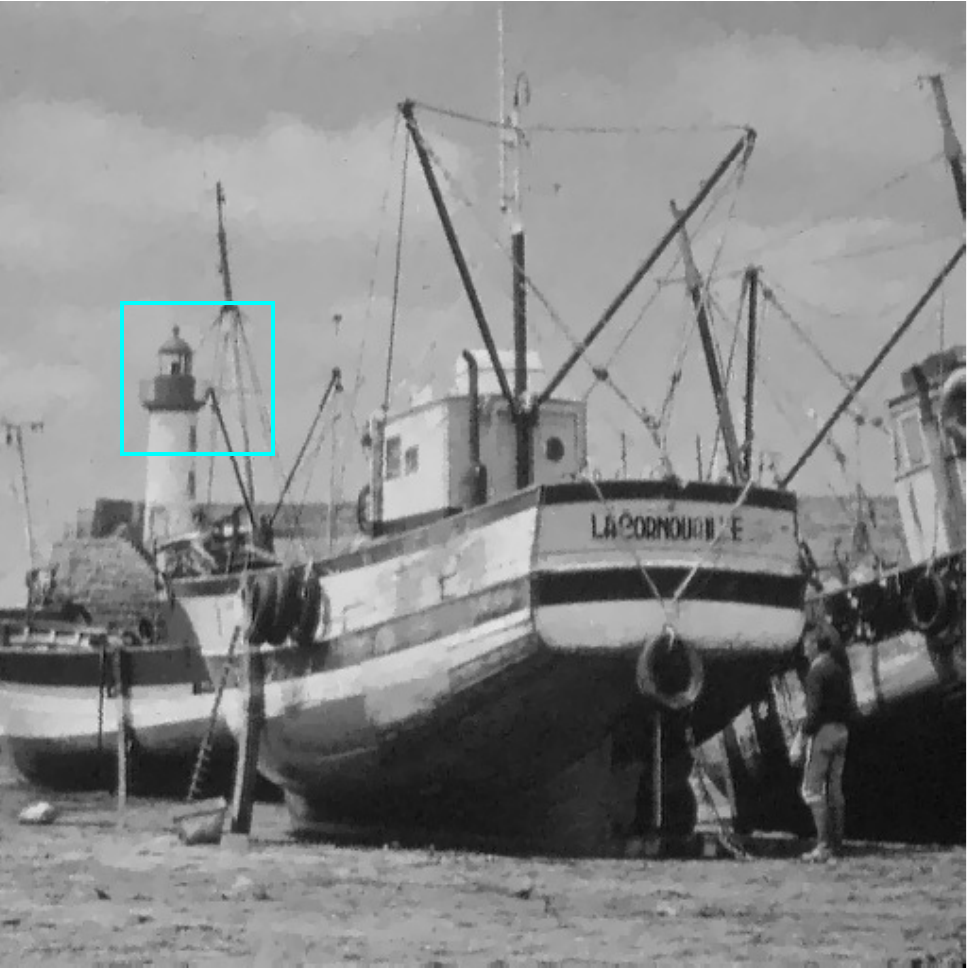}}\hspace{-1.0ex}
      \subfigure[zoomed]{
			\includegraphics[width=0.155\linewidth]{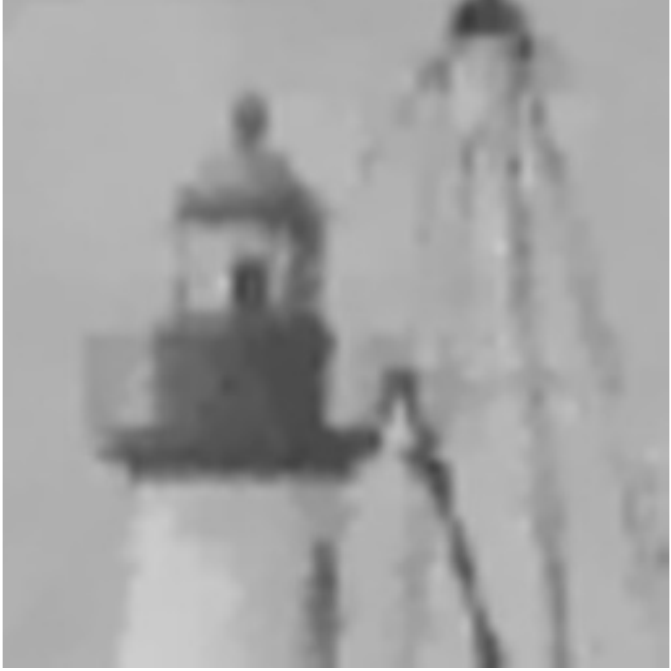}}\hspace{-1.0ex}
      \subfigure[Relative Error]{
			\includegraphics[width=0.18\linewidth]{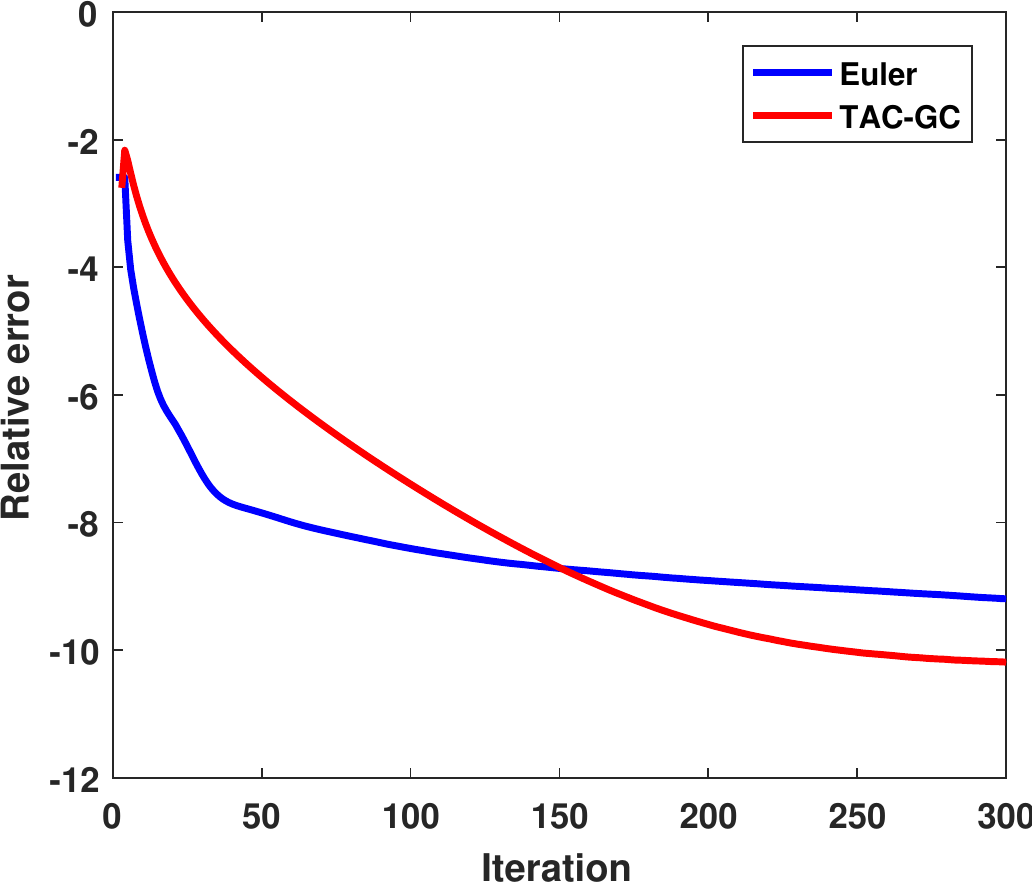}}
	\caption{The Poisson denoising results of `Hill' and `Boats' obtained by the Euler's elastica and TAC-GC methods. The parameters are set as (b) Euler: $\alpha=15$, $r_1=2$, $r_2=5\cdot10^2$, $r_3=10^2$, $r_4=5\cdot10^2$ and $\eta=2\cdot10^2$; (d) TAC-GC: $\alpha=15$, $\mu_1=2$, $\mu_2=4$ and $\lambda=25$.}
	\label{PoissonDenoising}
\end{figure*}

\begin{table*}[t]
\caption{The evaluations of Poisson noise removal for the Euler's elastica and TAC-GC methods.}
\label{denoise4}
\tiny
\begin{center}
\begin{tabular}{c|c|c|c|c|c}
\hline\hline
Images & Methods & PSNR & SSIM & Iterations & Time \\
\hline\hline
Goldhill & Euler's elastica & 31.53 & 0.8689 & 285 & 20.40 \\
$256\times256$ & TAC-GC & \bf{32.16} & \bf{0.8812} & \bf{216} & \bf{14.42} \\
\hline
Boats & Euler's elastica & 31.90 & 0.8772 & 254 & 116.84 \\
$512\times512$ & TAC-GC & \bf{32.76} & \bf{0.8958} & \bf{198} & \bf{79.25} \\
\hline\hline
\end{tabular}
\end{center}
\end{table*}

We also conduct the examples of Poisson noise removal, the variational model of which can be formalized by integrating the Kullback-Leibler (KL) fidelity as
\begin{equation}\label{KLnorm}
   \min_{u} \sum_{1\leq i,j\leq m} g(\kappa_{i,j}) |\nabla u_{i,j}| + \lambda \sum_{1\leq i,j\leq m}(u_{i,j}-f_{i,j}\log u_{i,j}).
\end{equation}
More detailed implementation of \eqref{KLnorm} can be found in \cite{le2007variational,wu2011augmented}.

The Poisson noise is introduced into two clean images, i.e., `Goldhill' $(256\times256)$ and `Boats' $(512\times512)$. We set the parameters in our model as $\lambda=25$, $\alpha=15$, $\mu_1=2$, $\mu_2=4$, and stopping criteria is given as $\epsilon=7\times10^{-5}$. The restoration results of our TAC-GC model are compared with the Euler's elastica model as illustrated in FIG. \ref{PoissonDenoising} and Table \ref{denoise4}.
As shown in FIG. \ref{PoissonDenoising}, the TAC-GC model can preserve more image details and features than the Euler's elastica model, e.g., the window area in `Goldhill' and the mast area in `Boats'. The results are further verified by the PSNR and SSIM in Table \ref{denoise4}. Similar to the previous experiment, our TAC-GC method converges faster than the Euler's elastica method using the same stopping criteria, which demonstrates that our curvature model outperforms the Euler's elastica method in both quality and efficiency in Poisson noise removal.

\begin{figure*}[t]
      \centering
      \subfigure[Noisy($\theta=20$)]{
			\includegraphics[width=0.18\linewidth]{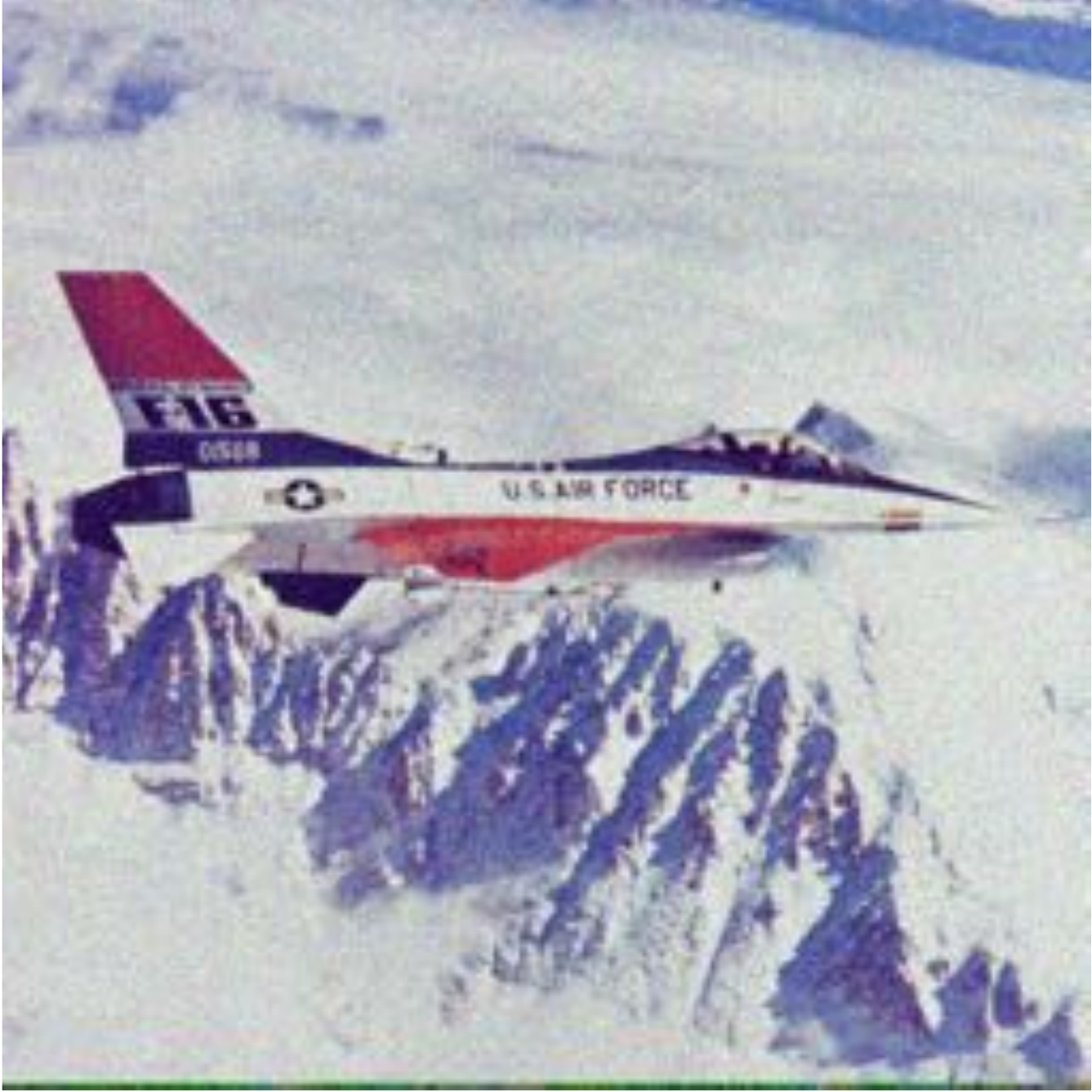}}\hspace{-1ex}
      \subfigure[Euler]{
			\includegraphics[width=0.18\linewidth]{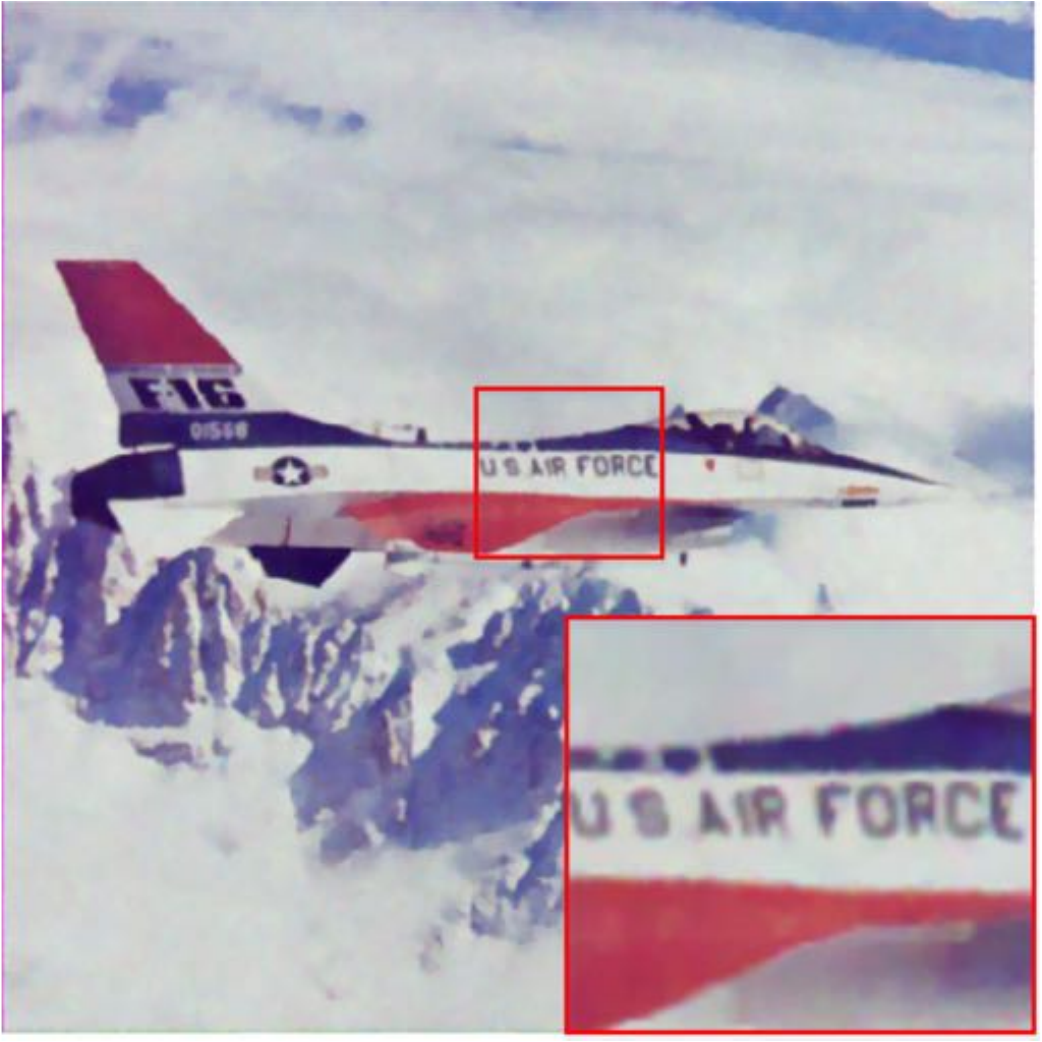}}\hspace{-1ex}
      \subfigure[TAC-GC]{
			\includegraphics[width=0.18\linewidth]{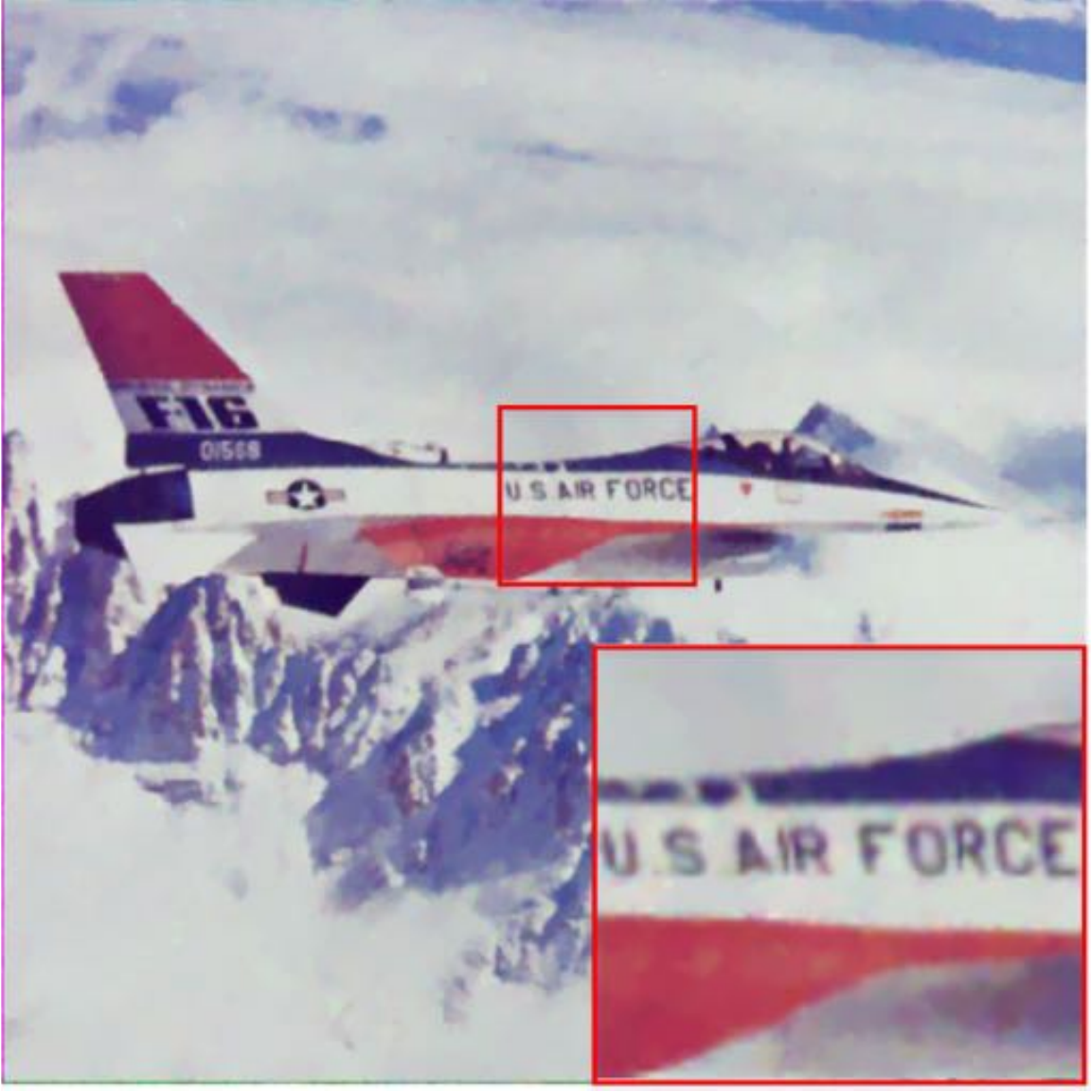}}\hspace{-1ex}
      \subfigure[Energy]{
            \includegraphics[width=0.2\linewidth]{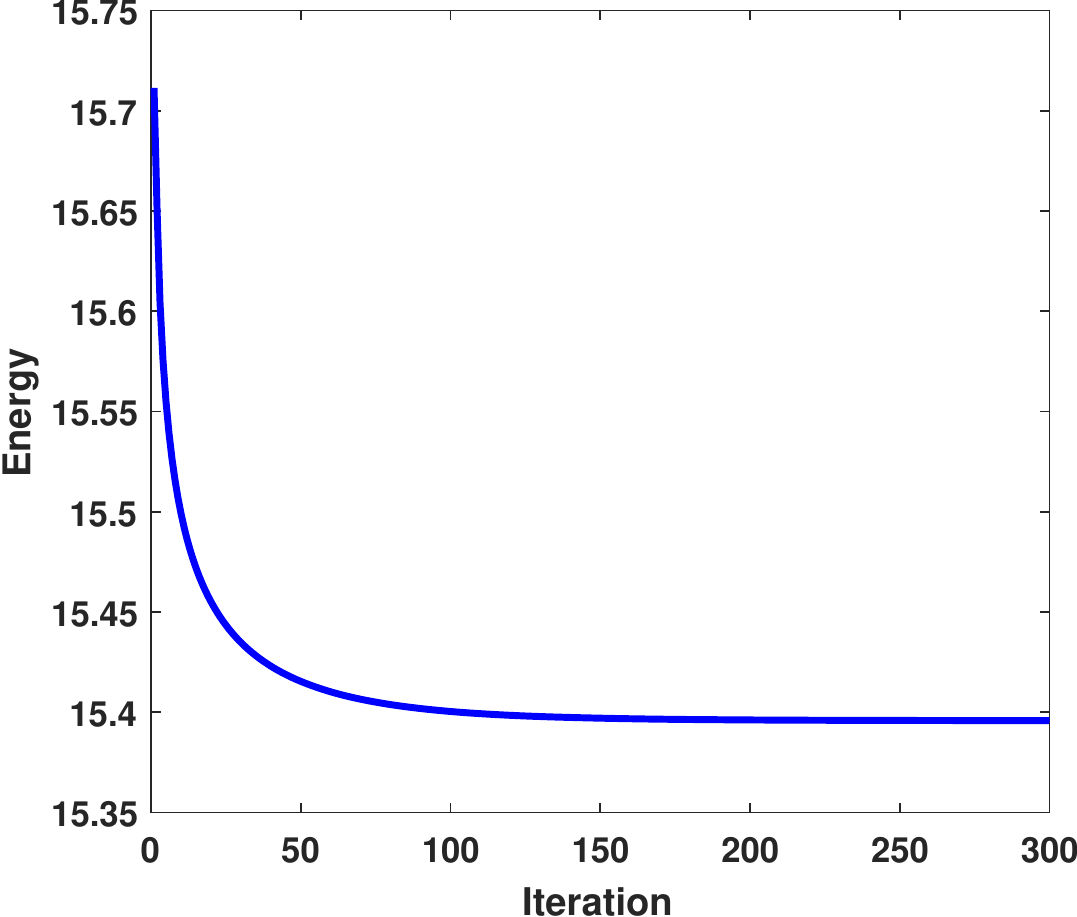}}\hspace{-1ex}
      \subfigure[Relative error]{
			\includegraphics[width=0.21\linewidth]{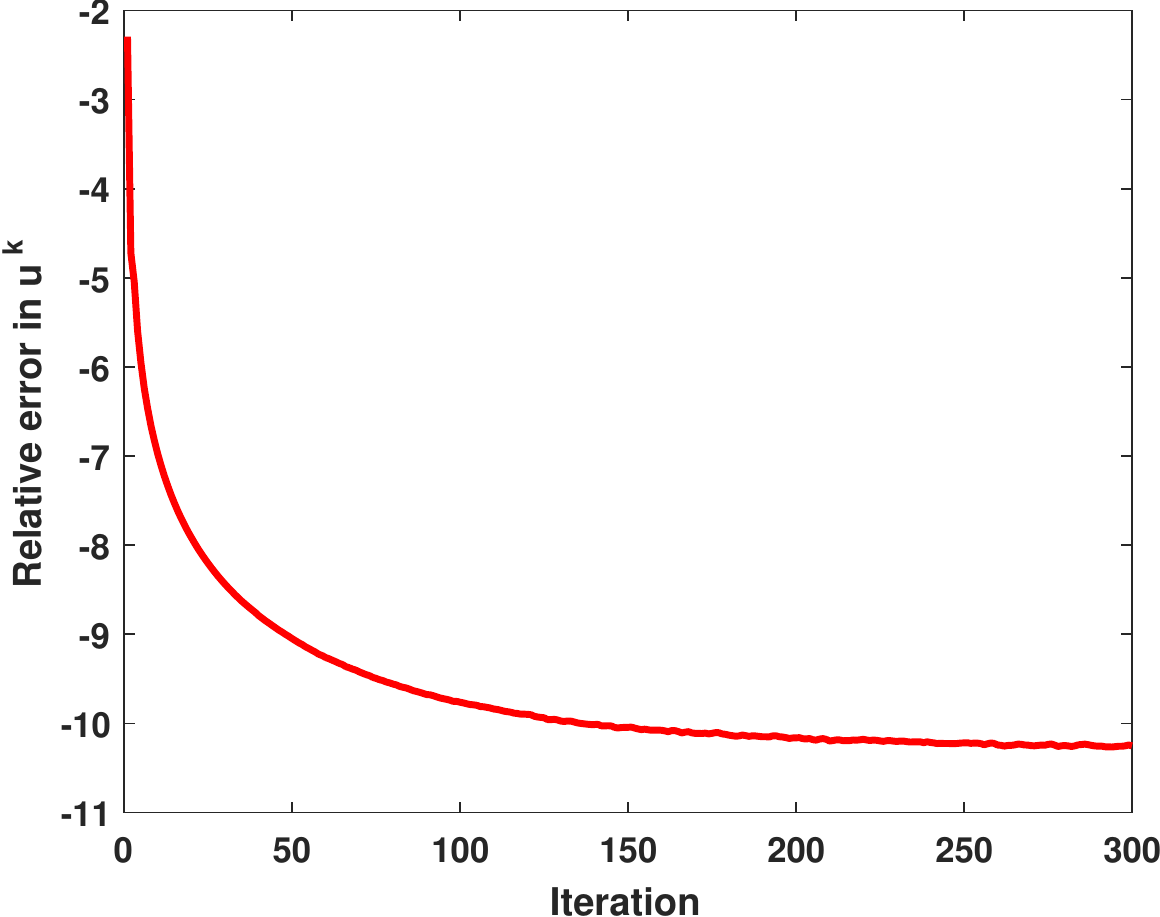}}
	\caption{The denoising results of `Airplane' by the Euler's elastica and TAC-GC methods.}
	\label{airplane}
\end{figure*}

\begin{figure*}[t]
      \centering
      \subfigure[Noisy($\theta=30$)]{
			\includegraphics[width=0.18\linewidth]{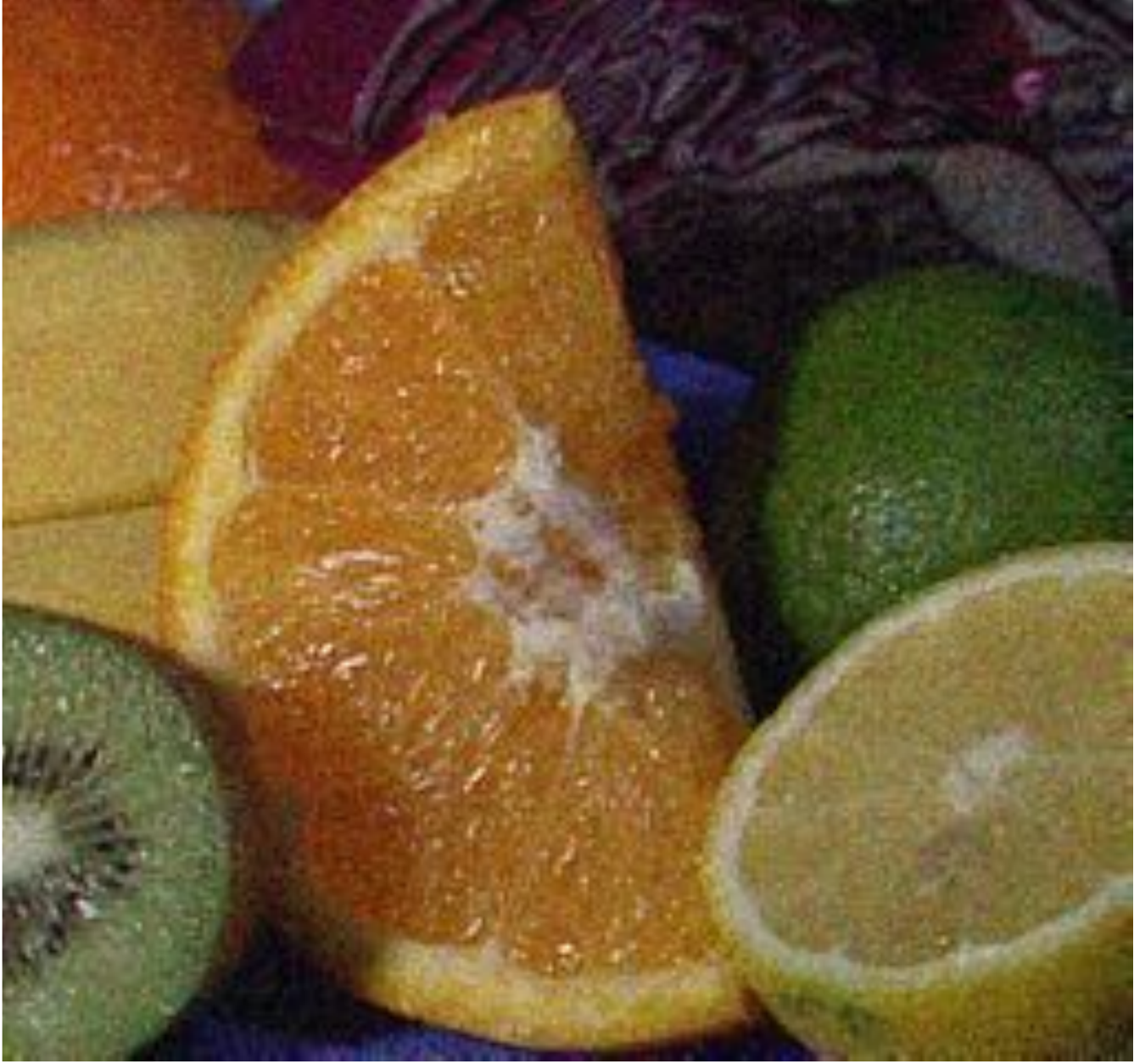}}\hspace{-1ex}
      \subfigure[Euler]{
			\includegraphics[width=0.18\linewidth]{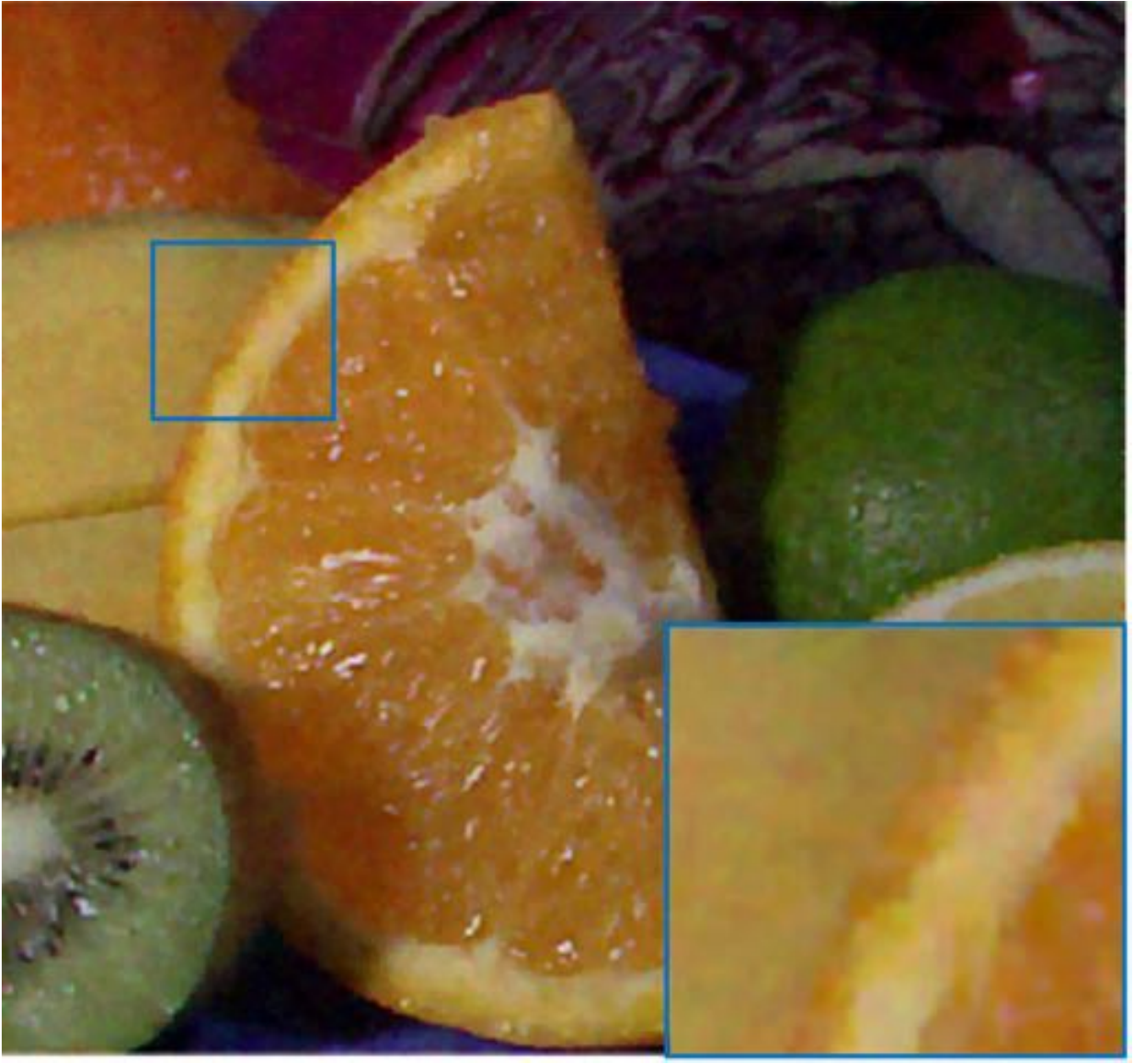}}\hspace{-1ex}
      \subfigure[TAC-GC]{
			\includegraphics[width=0.18\linewidth]{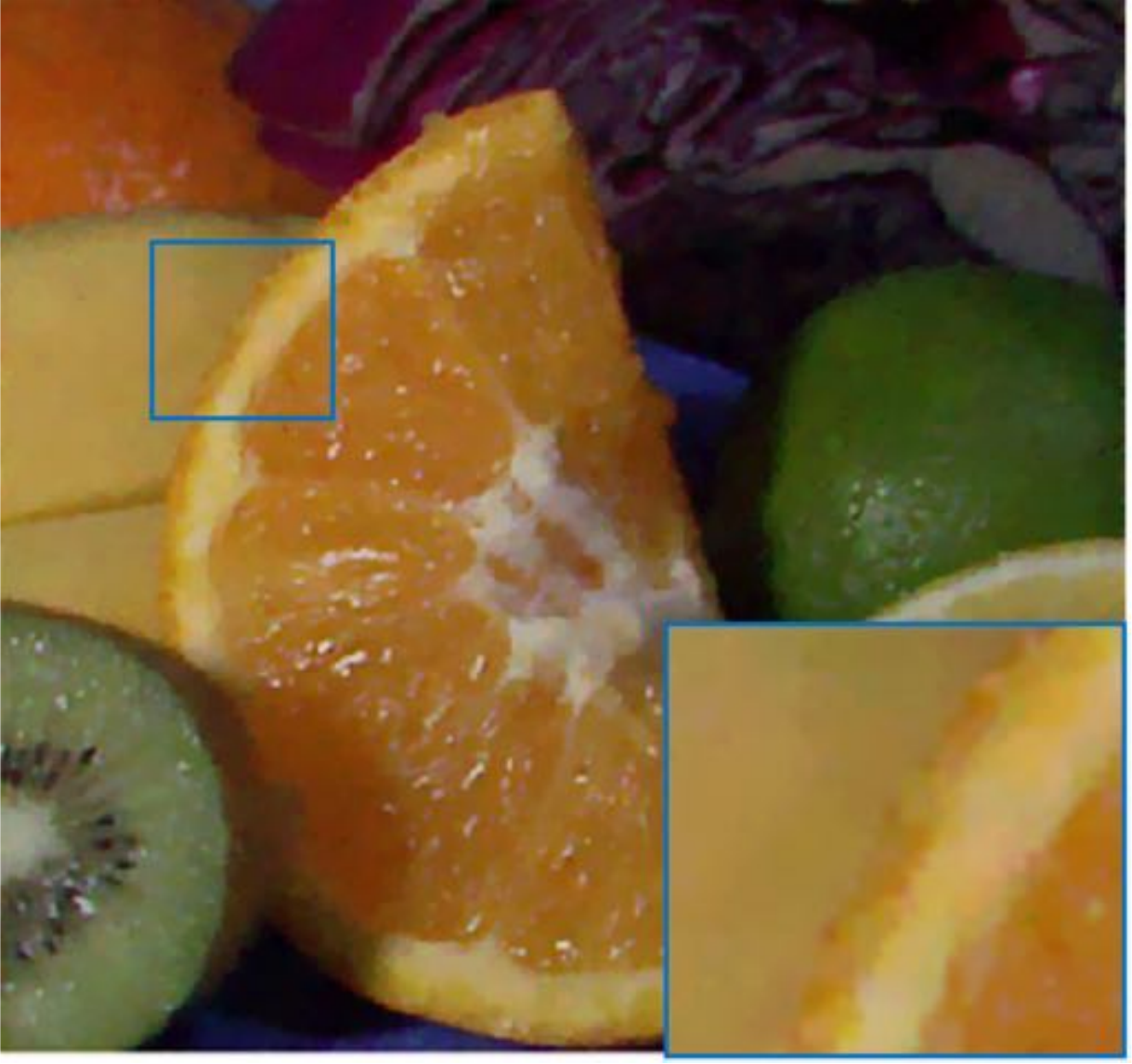}}\hspace{-1ex}
      \subfigure[Energy]{
            \includegraphics[width=0.21\linewidth]{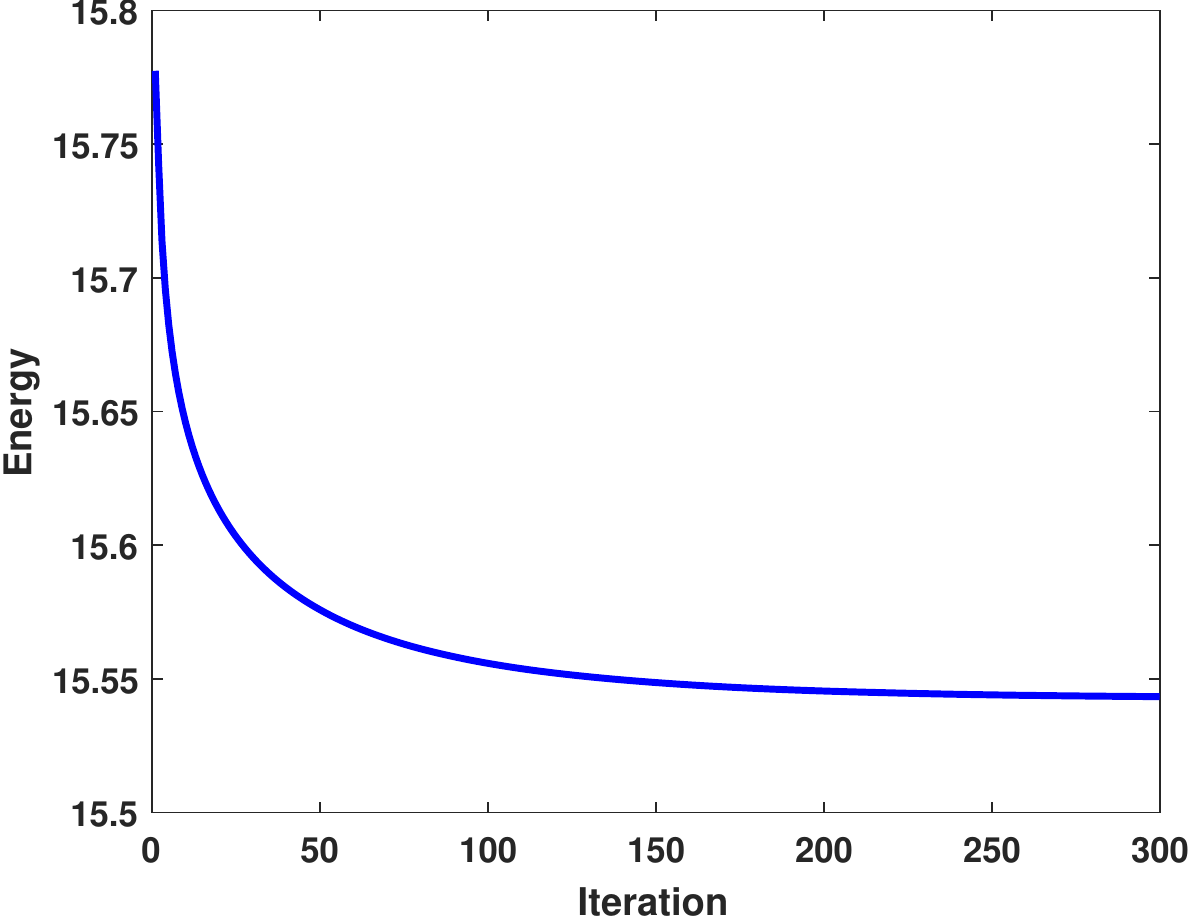}}\hspace{-1ex}
      \subfigure[Relative error]{
			\includegraphics[width=0.2\linewidth]{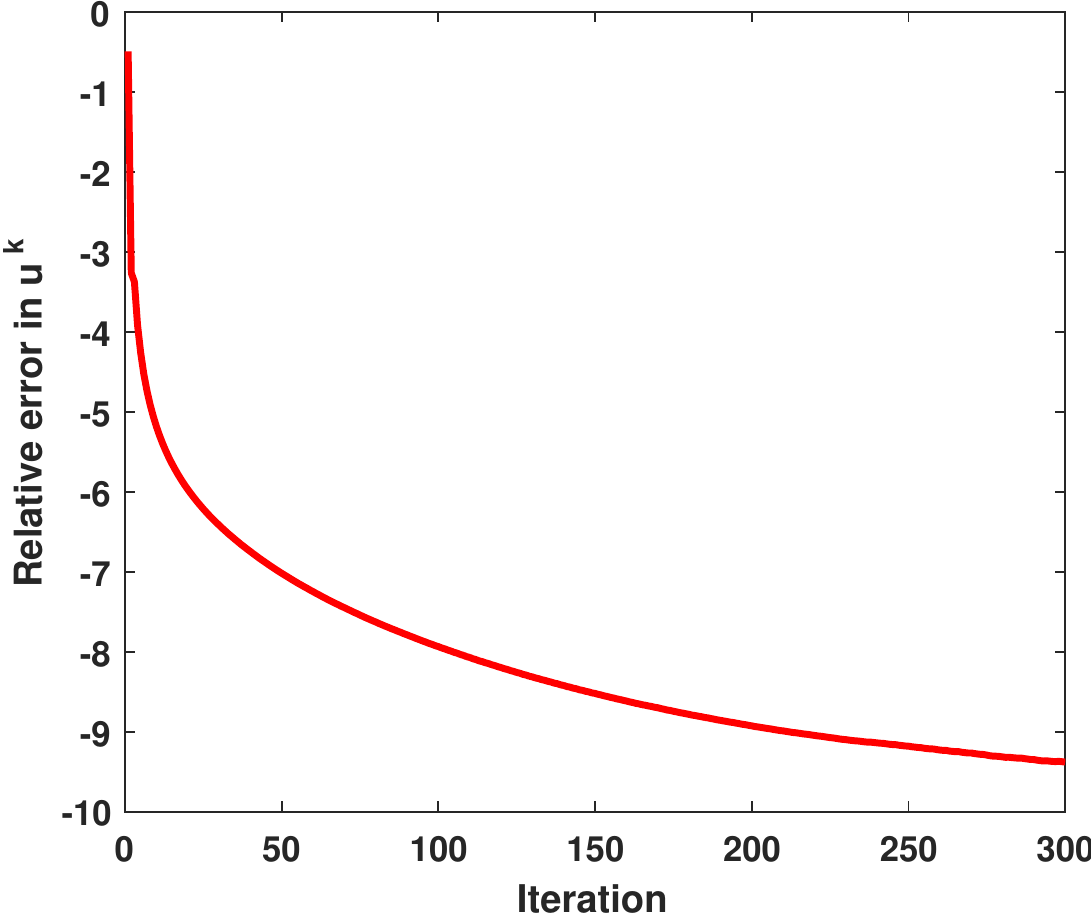}}
	\caption{The denoising results of `Fruits' by the Euler's elastica and TAC-GC methods.}
	\label{fruits}
\end{figure*}

\begin{figure*}[!ht]
      \centering
      \subfigure[Noisy($\theta=40$)]{
			\includegraphics[width=0.18\linewidth]{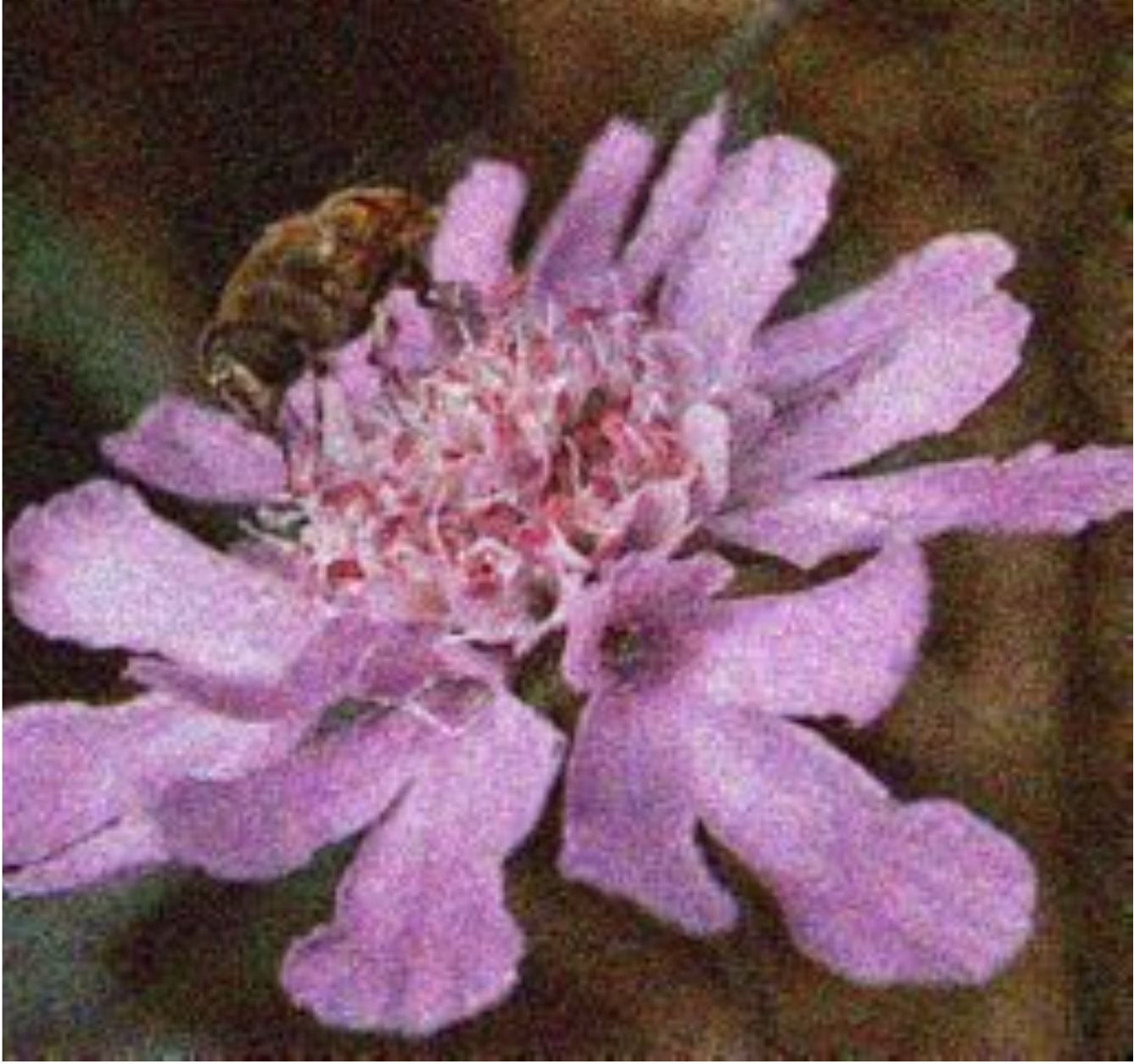}}\hspace{-1ex}
      \subfigure[Euler]{
			\includegraphics[width=0.18\linewidth]{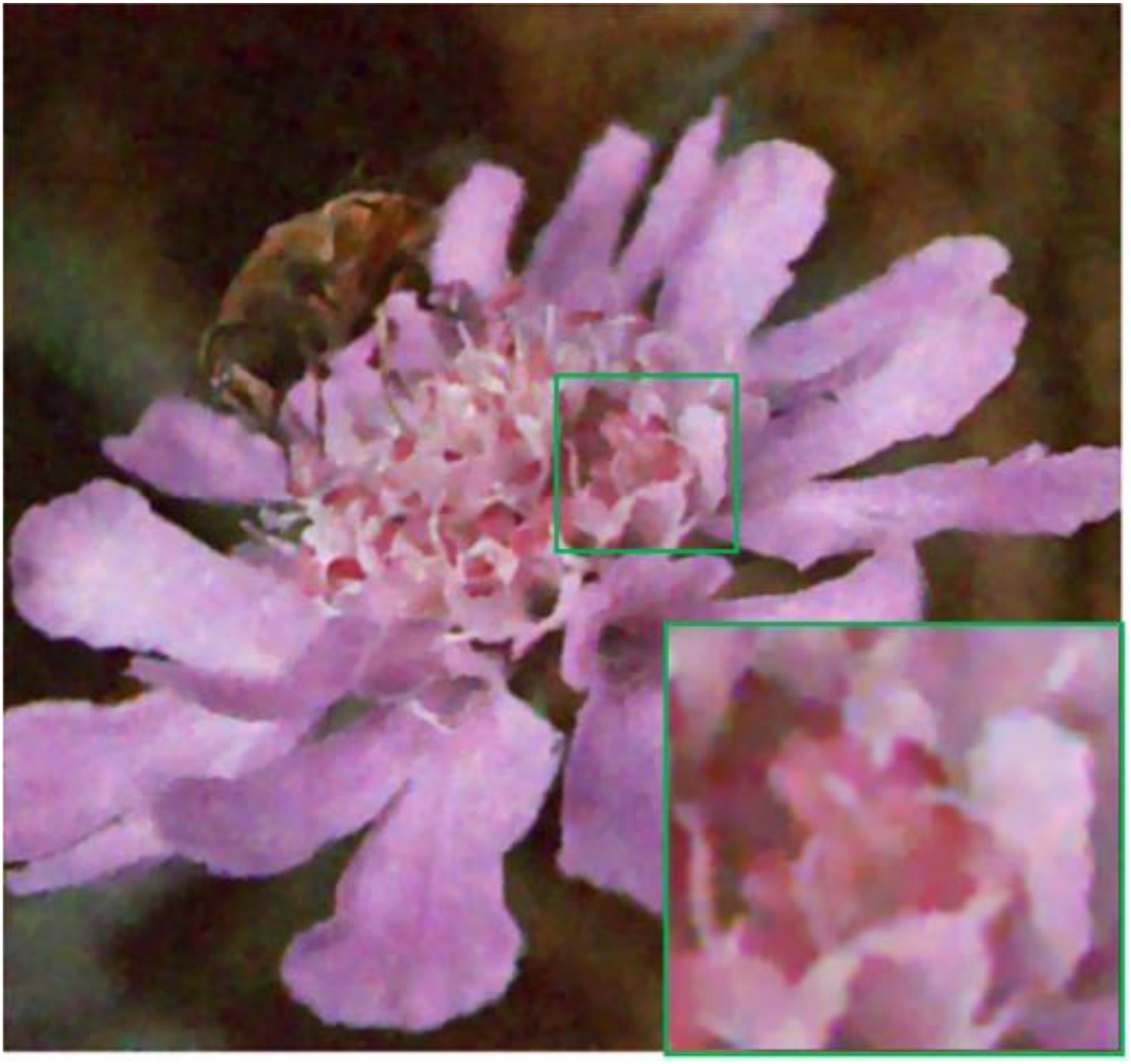}}\hspace{-1ex}
      \subfigure[TAC-GC]{
			\includegraphics[width=0.18\linewidth]{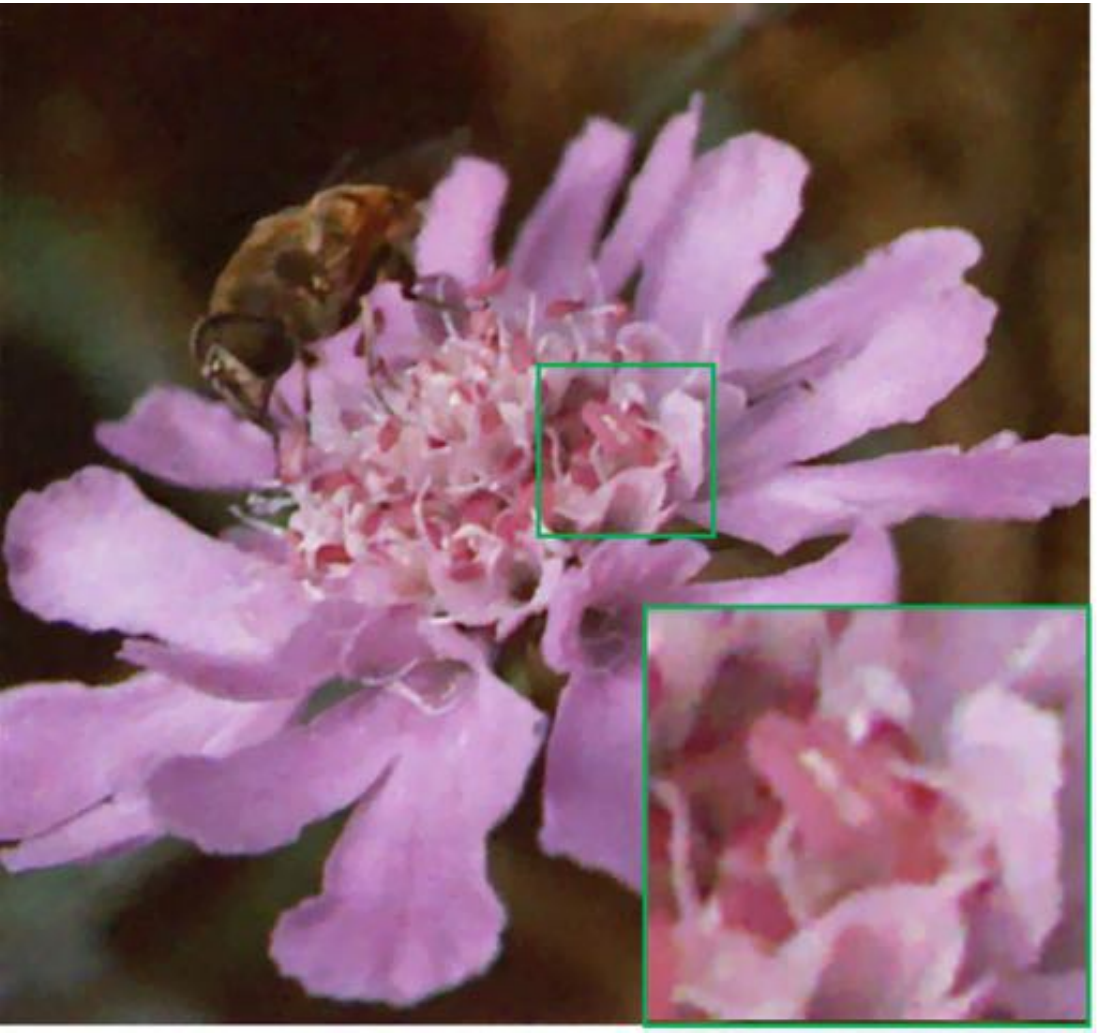}}\hspace{-1ex}
      \subfigure[Energy]{
            \includegraphics[width=0.2\linewidth]{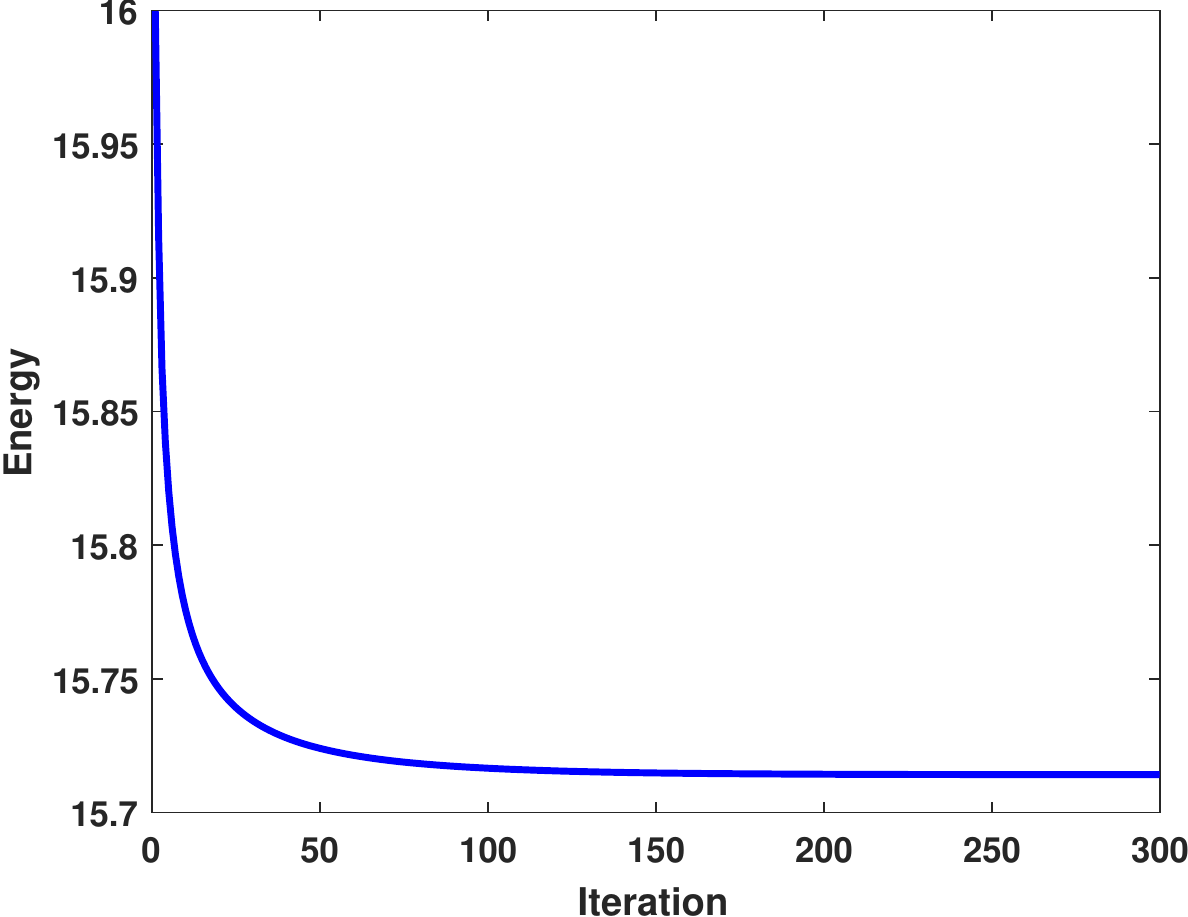}}\hspace{-1ex}
      \subfigure[Relative error]{
			\includegraphics[width=0.2\linewidth]{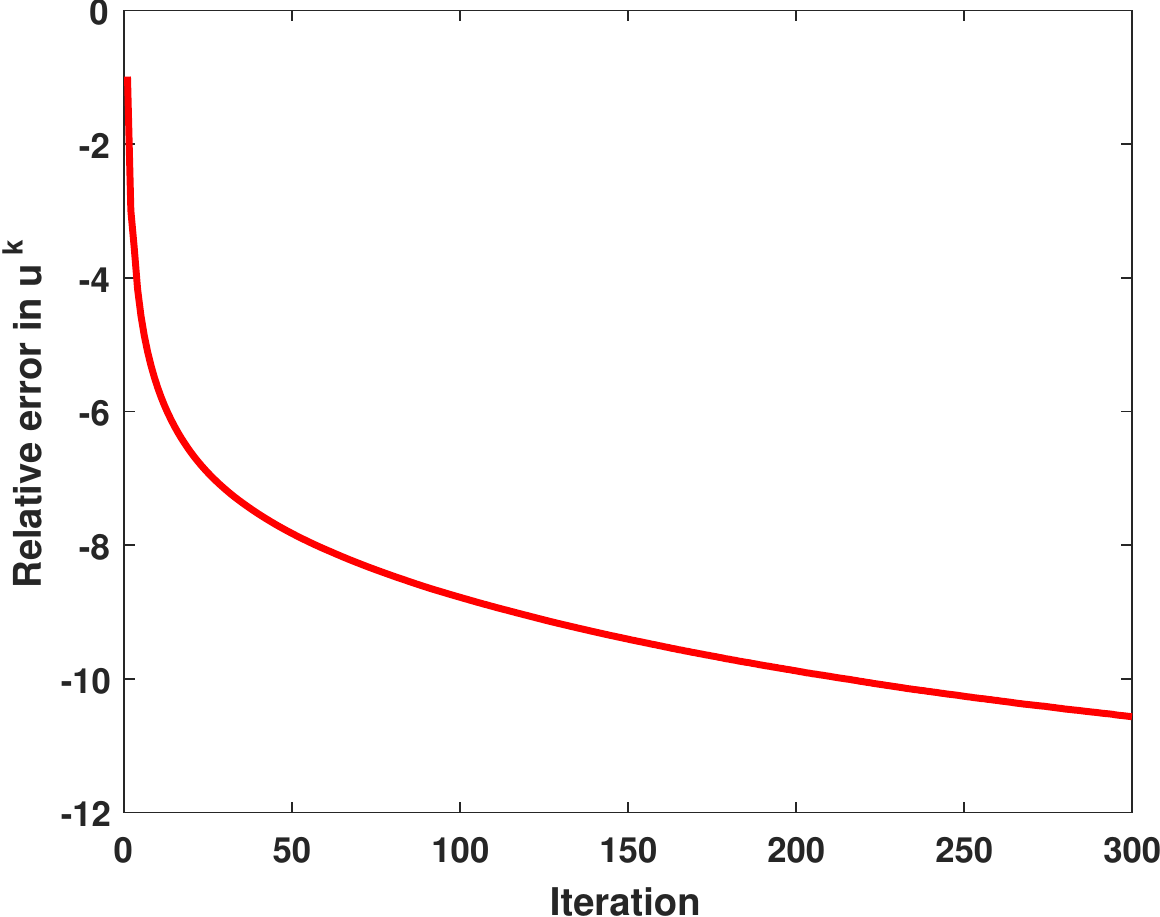}}
	\caption{The denoising results of `Flower' by the Euler's elastica and TAC-GC methods.}
	\label{flower}
\end{figure*}

\subsection{Color images denoising}

\begin{table*}[!ht]
\caption{The evaluations of color image noise removal for the Euler's elastica and TAC-GC methods.}
\label{denoise5}
\tiny
\begin{center}
\begin{tabular}{c|c|c|c|c|c}
\hline\hline
Images & Methods & PSNR & SSIM & Iterations & Time \\
\hline\hline
Airplane($\theta=20$) & Euler's elastica & 30.55 & 0.8955 & 192 & 336.91 \\
$512\times512$ & TAC-GC & \bf{31.07} & \bf{0.9071} & \bf{171} & \bf{205.47} \\
\hline
Fruits($\theta=30$) & Euler's elastica & 28.34 & 0.9156 & 264 & 481.74 \\
$480\times512$ & TAC-GC & \bf{28.96} & \bf{0.9228} & \bf{233} & \bf{241.54} \\
\hline
Flower($\theta=40$) & Euler's elastica & 28.32 & 0.9282 & 245 & 466.01 \\
$480\times512$ & TAC-GC & \bf{28.92} & \bf{0.9379} & \bf{210} & \bf{218.47}\\
\hline\hline
\end{tabular}
\end{center}
\end{table*}

In this subsection, we extend our TAC-GC model to color image restoration \cite{wen2008efficient,mairal2007sparse}. Without loss of generality, we consider a vectorial function ${\bf u}=(u^r,u^g,u^b):{\rm{\Omega}}\rightarrow\mathbb{R}^3$ defined on a bounded open domain ${\rm{\Omega}}\subset \mathbb{R}^2$. For the sake of simplicity, we propose to independently recover each RGB channel of color images, and then generate the final restored image by combining the RGB channels together. Thus, the corresponding color image denoising model with the $L^2$-norm data fidelity term can be described as
\begin{equation}\label{L2colornorm}
   \min_{{\bf u}\in\mathbb{R}^3}~\sum_{\sigma}\sum_{1\leq i,j\leq m} g(\kappa^\sigma_{i,j}) |\nabla u^\sigma_{i,j}|+ \frac{\lambda}{2}\sum_{\sigma}\|u^{\sigma}-f^{\sigma}\|^2.
\end{equation}
where $\sigma \in\{r,g,b\}$. We plan to extend our curvature models to the color TV model \cite{blomgren1998color} and Beltrami color image model \cite{rosman2011semi,rosman2010polyakov} as our further work.

Three different color images are selected as examples to demonstrate the efficiency and superiority of our TAC-GC model, which are `Airplane', `Fruits' and `Flower' degraded by the Gaussian noise with zero mean and the standard deviation $\theta=\{20, 30, 40\}$, respectively. The parameters are set as $\lambda=\{0.07, 0.05, 0.03\}$, $\alpha=5$, $\mu=2$ and $\epsilon=\{4.8\times10^{-5}, 1.0\times10^{-4}, 3.6\times10^{-4}\}$ for different noise levels accordingly to guarantee satisfactory restoration results be achieved. On the other hand, the parameters of the Euler's elastica model are set as $\eta=\{2\cdot10^2, 1.5\cdot10^2, 10^2\}$, $\alpha=10$, $r_1=1$, $r_2=2\cdot10^2$ and $r_4=5\cdot10^2$ for the three images, respectively.

As shown in FIG. \ref{airplane}-FIG. \ref{flower}, the proposed TAC-GC model can preserve sharper image edges and smoother homogenous regions, and the energy curves becomes stable after certain number of iterations. To further evaluate the denoising performance, quantitative results with different degradations are summarized in Table \ref{denoise5}, which obviously shows the TAC-GC model outperforms the Euler's elastica model in both recovery quality and computational efficiency.

\subsection{Image Inpainting}
Last but not least, we demonstrate some examples of our TAC-GC model on applications of image inpainting \cite{shen2003euler}. In general, the task of image inpainting is to reconstruct a missing part of an image using information from the given region. The missing part of the image is called the inpainting domain, denoted by $D\subseteq {\rm{\Omega}}$. In this case, we can formulate the curvature-based model as follows
\begin{equation}\label{Inpaintingnorm}
   \min_{u}~ \sum_{1\leq i,j\leq m} g(\kappa_{i,j}) |\nabla u_{i,j}|+ \frac{\lambda}{2}\|u-f\|^2_{{\rm{\Omega}\setminus D}}.
\end{equation}
More details of the implementation can be found in \cite{yashtini2016fast}.

In FIG. \ref{Inpainting}, three contaminated images (i.e., A1, B1 and C1) are considered, where A2, B2 and C2 are the reconstruction results of our TAC-GC model.  It seems that the reconstructed results are quite natural and extremely similar to the original images. Table \ref{inpainting} records the quantitative numerical results, where the PSNR and SSIM indicate the excellent inpainting performance of the proposed TAC-GC method in inpainting applications.

\begin{figure*}[t]
      \centering
      \subfigure[A1]{
			\includegraphics[width=0.15\linewidth]{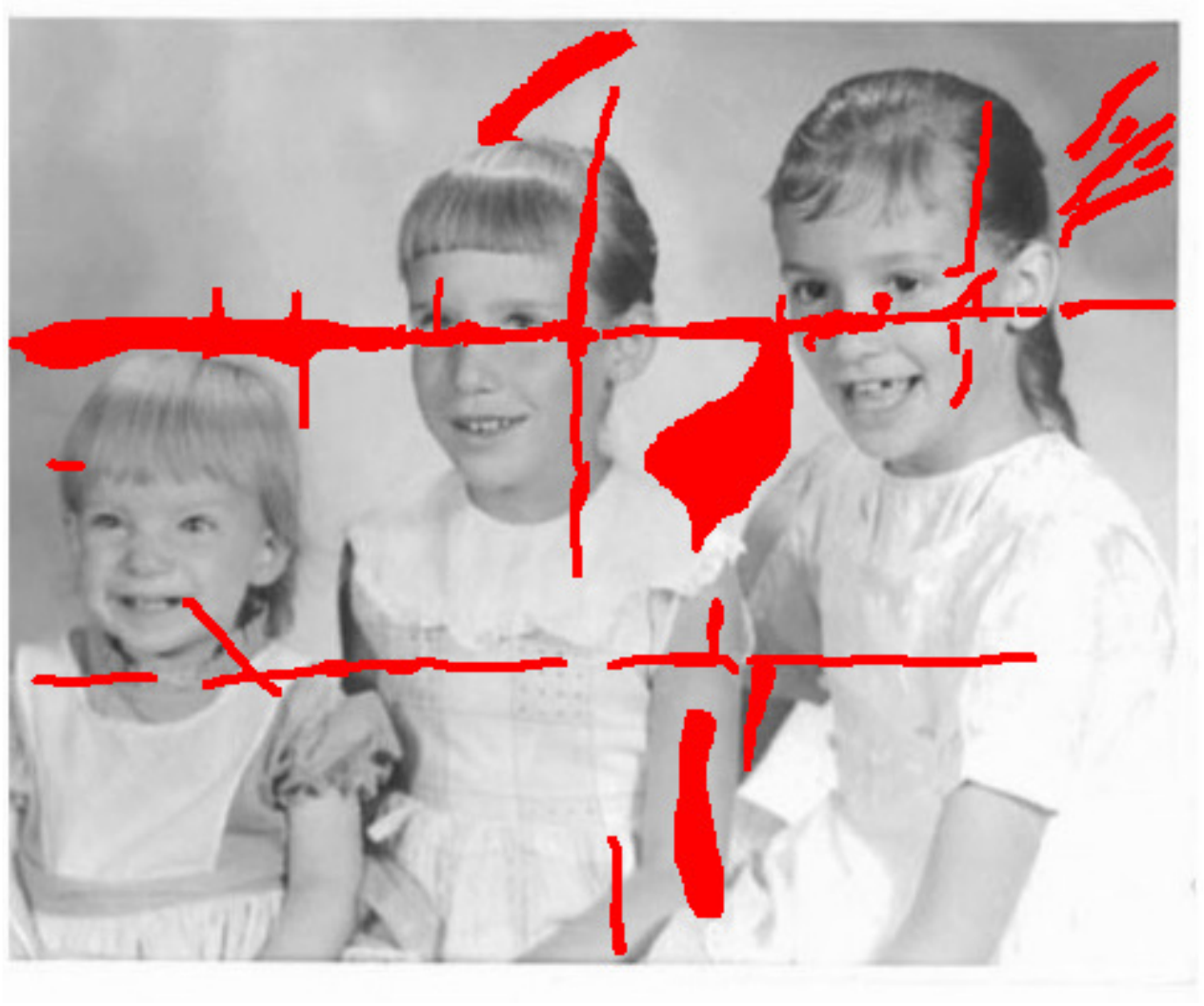}}
      \subfigure[A2]{
			\includegraphics[width=0.15\linewidth]{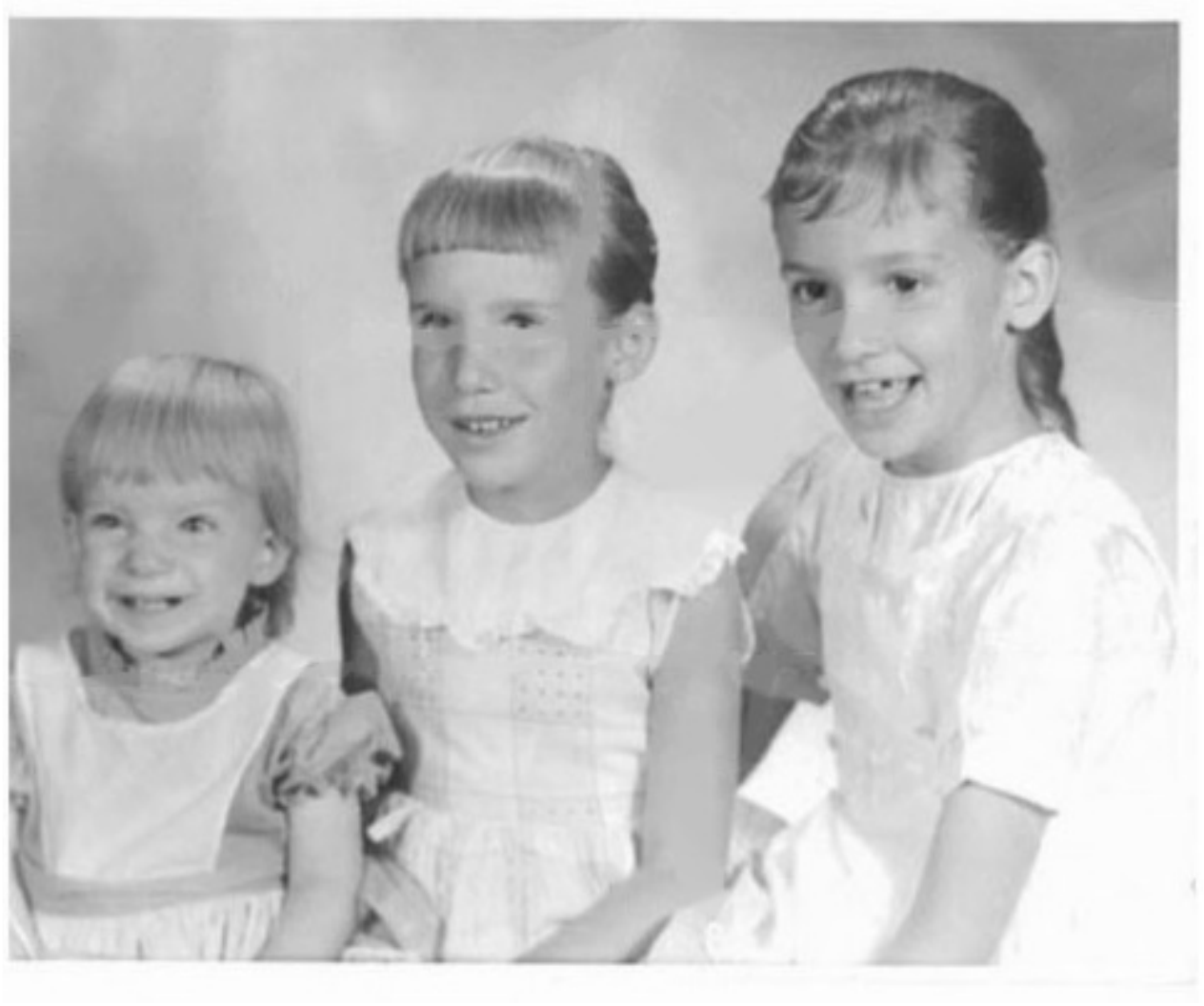}}
      \subfigure[B1]{
			\includegraphics[width=0.15\linewidth]{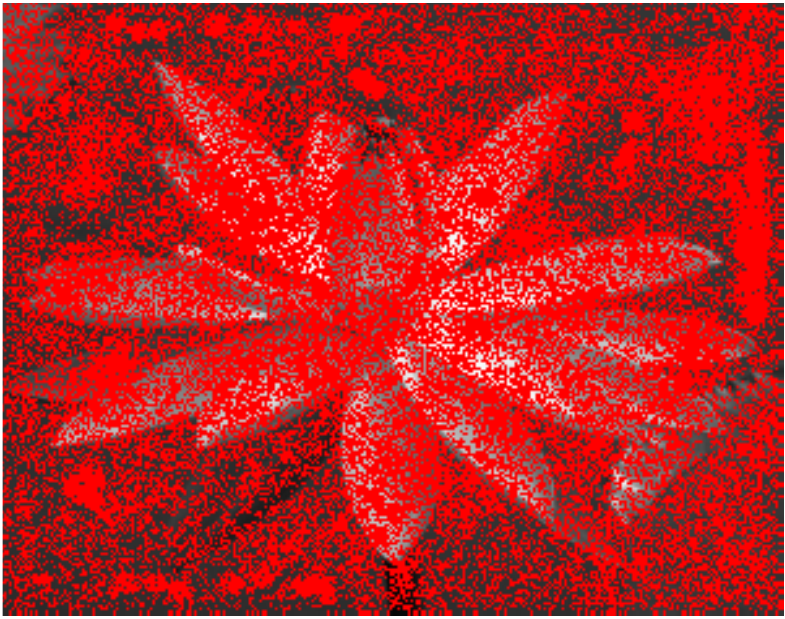}}
      \subfigure[B2]{
			\includegraphics[width=0.15\linewidth]{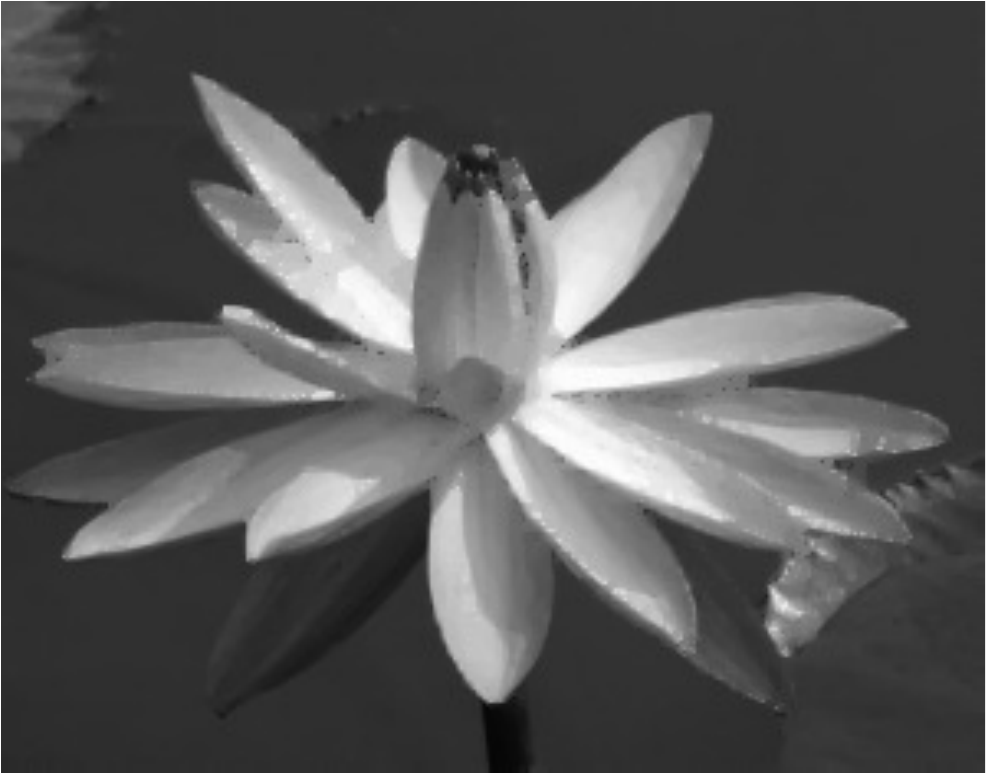}}
      \subfigure[C1]{
            \includegraphics[width=0.15\linewidth]{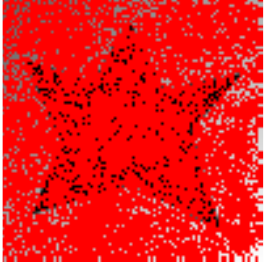}}
      \subfigure[C2]{
			\includegraphics[width=0.15\linewidth]{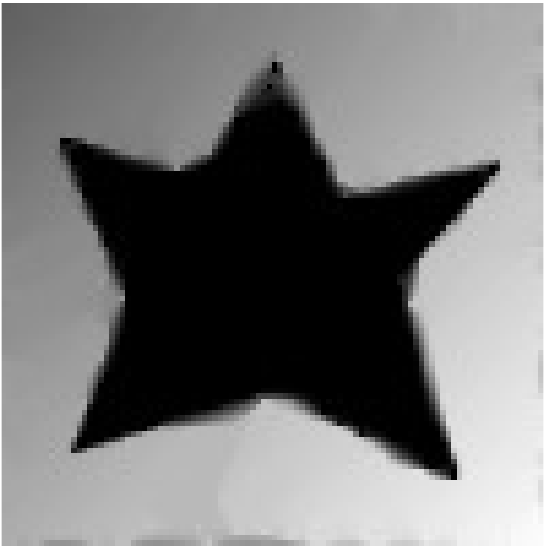}}
	\caption{The inpainting results of different images in TAC-GC method, where the parameters are adopted as $\lambda=10$, $\alpha=5$, $\mu_1=0.5$, $\mu_2=0.1$ and $\epsilon=5\times10^{-4}$.}
	\label{Inpainting}
\end{figure*}

\begin{table*}[t]
\caption{The details of inpainting results for different images.}
\label{inpainting}
\tiny
\begin{center}
\begin{tabular}{c|c|c|c|c|c|c|c}
\hline\hline
Images & Size & Unknowns& PSNR & SSIM & Iterations & Time & Percentage of unknowns\\
\hline\hline
A1 & $484\times404$ & 14258 & 36.42 & 0.9739 & 260 & 75.08 & $7.29\%$\\
\hline
B1 & $300\times235$ & 42114 & 31.25 & 0.9531 & 172 & 15.95 & $59.74\%$\\
\hline
C1 & $100\times100$ & 8496 & 25.76 & 0.9087 & 259 & 4.81 & $84.96\%$\\
\hline\hline
\end{tabular}
\end{center}
\end{table*}

\section{Conclusions}
\label{sect5}
In this work, we proposed the discrete curvature-based regularizers for image reconstruction problems. Both MC and GC were derived and investigated using the normal curvatures in a local window based on the differential geometry. Our proposed model can be regarded as a re-weighted TV model, which was solved by the proximal ADMM-based algorithm. We briefly discussed the convergence of the proximal ADMM-based algorithm under certain assumptions. Numerical experiments on both gray and color images have illustrated the efficacious and superior performance of our proposed method in terms of quantitative and qualitative evaluations. Apparently, the proposed method can be used for other practical applications in image processing and computer vision, for instance image segmentation, image registration, image super-solution etc.


\section*{Acknowledgment}

The authors would like to thank Dr. Gong and Prof. Sbalza-rini for sharing the MATLAB code of curvature filter. The work was partially supported by National Natural Science Foundation of China (NSFC 11701418), Major Science and Technology Project of Tianjin 18ZXRHSY00160 and Recruitment Program of Global Young Expert. The second author was supported by NSFC 11801200 and a startup grant from HUST.

\bibliographystyle{siam}
\bibliography{RefsCurvature}

\end{document}